\documentclass[11pt, a4paper]{article}
\title{Paralinearization and extended lifespan \\
for solutions of the $ \alpha $-SQG sharp front equation }
\author{
    Massimiliano Berti\thanks{Scuola Internazionale Superiore di Studi Avanzati (SISSA), Trieste, Italy. \href{mailto:berti@sissa.it}{\texttt{berti@sissa.it}}}
    \and
    Scipio Cuccagna\thanks{Università degli Studi di Trieste, Dipartimento di Matematica, Informatica e Geoscienze, Italy. \href{mailto:scuccagna@units.it}{\texttt{scuccagna@units.it}}}
    \and
    Francisco Gancedo\thanks{University of Seville, Spain. \href{mailto:fgancedo@us.es}{\texttt{fgancedo@us.es}}}
    \and
    Stefano Scrobogna\thanks{Università degli Studi di Trieste, Dipartimento di Matematica, Informatica e Geoscienze, Italy. \href{mailto:stefano.scrobogna@units.it}{\texttt{stefano.scrobogna@units.it}}}
}

\usepackage[a4paper]{geometry}
\usepackage{empheq}
\usepackage{times}
\usepackage{stmaryrd}
\usepackage{fourier}
\usepackage[T1]{fontenc}
\usepackage{amssymb, amsmath,  bm, mathrsfs}
\usepackage{amsfonts}
\usepackage[english]{babel}
\usepackage{amssymb,amsthm}
\usepackage{mathtools}
\DeclareMathAlphabet{\mathcal}{OMS}{cmsy}{m}{n}
\DeclareFontFamily{U}{mathc}{}
\DeclareFontShape{U}{mathc}{m}{it}
{<->x*[1.03] mathc10}{}
\DeclareMathAlphabet{\mathscr}{U}{mathc}{m}{it}

\DeclareMathAlphabet{\mathpzc}{OT1}{pzc}{m}{it}

\usepackage{hyperref}
\usepackage[capitalise]{cleveref}
\usepackage{cite}
\allowdisplaybreaks[1]
\usepackage{xfrac}
\usepackage[utf8]{inputenc}
\usepackage[T1]{fontenc}
\usepackage{enumerate}
\usepackage{accents}
\usepackage{bbm}
\usepackage{thmtools}

\usepackage{todonotes}

\usepackage{tikz}
\usetikzlibrary{arrows, automata, backgrounds, shadows, patterns, calc, hobby, plotmarks, shapes}
\usetikzlibrary{shapes.misc}

\tikzset{cross/.style={cross out, draw=black, minimum size=2*(#1-\pgflinewidth), inner sep=0pt, outer sep=0pt},
cross/.default={1pt}}
\allowdisplaybreaks

\makeatletter
\@ifpackageloaded{stix}{%
}{%
  \DeclareFontEncoding{LS2}{}{\noaccents@}
  \DeclareFontSubstitution{LS2}{stix}{m}{n}
  \DeclareSymbolFont{stix@largesymbols}{LS2}{stixex}{m}{n}
  \SetSymbolFont{stix@largesymbols}{bold}{LS2}{stixex}{b}{n}
  \DeclareMathDelimiter{\lBrace}{\mathopen} {stix@largesymbols}{"E8}%
                                            {stix@largesymbols}{"0E}
  \DeclareMathDelimiter{\rBrace}{\mathclose}{stix@largesymbols}{"E9}%
                                            {stix@largesymbols}{"0F}
}
\makeatother

\DeclareSymbolFontAlphabet{\amsmathbb}{AMSb}%

\usepackage{listings}
\usepackage{color}

\definecolor{dkgreen}{rgb}{0,0.6,0}
\definecolor{gray}{rgb}{0.5,0.5,0.5}
\definecolor{mauve}{rgb}{0.58,0,0.82}

\makeatletter
\def\maketag@@@#1{\hbox{\m@th\normalfont\normalsize#1}}
\makeatother

\newcommand{\dd}{\textnormal{d}}

\newcommand{\sgn}{\textnormal{sgn}}

\newcommand{\Vol}{\textnormal{Vol}}

\newcommand{\pare}[1]{\left( #1 \right)}

\newcommand{\angles}[1]{\left\langle #1 \right\rangle}

\newcommand{\norm}[1]{\left\| #1 \right\|}
\newcommand{\av}[1]{\left| #1 \right|}
\newcommand{\bra}[1]{\left[ #1 \right]}

\newcommand{\pbra}[2]{\set{ #1 \big. , \  #2 }}
\newcommand{\comm}[2]{\left\llbracket #1 \,  , \,   #2 \right\rrbracket}
\newcommand{\Ad}{\textnormal{Ad}}
\newcommand{\xfrac}[2]{\left. #1 \middle/ #2 \right. }
\newcommand{\set}[1]{\left\{ #1 \right\}}
\newcommand{\system}[1]{\left\{ #1 \right.}

\newcommand{\ceil}[1]{\left\lceil #1 \right\rceil}

\newcommand{\OpBW}[1]{\textnormal{Op}^{BW}\bra{#1}}

\newcommand{\OpW}[1]{\textnormal{Op}^{W}\bra{#1}}

\newcommand{\ddt}{\frac{\textnormal{d}}{\textnormal{d}t}}

\newcommand{\Id}{\textnormal{Id}}

\newcommand{\defeq}{\vcentcolon=}

\newcommand{\Cast}[2]{C^{#1}_\ast \pare{I; H^{#2}_0\pare{\bT;\bR}}}
\newcommand{\Ball}[2]{B^{#1}_{#2, \bR}\pare{I;\epsilon_0}}

\def\fint{\mathop{\,\rlap{--}\!\!\!\int}\nolimits}

\renewcommand{\Im}{\textnormal{Im}}

\newcommand{\RN}[1]{%
  \textup{\uppercase\expandafter{\romannumeral#1}}%
}

\newcommand{\KF}{K\cF}
\newcommand{\KR}{K\cR}
\newcommand{\KM}{K\mathcal{M}}

\newcommand{\cC}{\mathcal{C}}
\newcommand{\cI}{\mathcal{I}}
\newcommand{\cL}{\mathcal{L}}
\newcommand{\cT}{\mathcal{T}}

\newcommand{\cF}{\mathcal{F}}
\newcommand{\cJ}{\mathcal{J}}

\newcommand{\cR}{\mathcal{R}}
\newcommand{\bR}{\mathbb{R}}
\newcommand{\bN}{\mathbb{N}}
\newcommand{\bM}{\mathbb{M}}

\newcommand{\cO}{\mathcal{O}}
\newcommand{\cK}{\mathcal{K}}
\newcommand{\sG}{\mathsf{G}}

\newcommand{\fV}{\mathfrak{V}}

\newcommand{\cP}{\mathcal{P}}

\newcommand{\cM}{\mathcal{M}}

\newcommand{\bZ}{\mathbb{Z}}
\newcommand{\bC}{\mathbb{C}}

\newcommand{\bT}{\mathbb{T}}

\newcommand{\bV}{\mathbb{V}}

\newcommand{\cg}{\mathpzc{g}}
\newcommand{\cy}{\mathpzc{y}}

\newcommand{\sT}{\mathsf{T}}
\newcommand{\sH}{\mathsf{H}}
\newcommand{\sK}{\mathsf{K}}
\newcommand{\sR}{\mathsf{R}}
\newcommand{\sM}{\mathsf{M}}
\newcommand{\sx}{\mathsf{x}}
\newcommand{\sX}{\mathsf{X}}

\newcommand{\st}{\mathsf{t}}

\newcommand{\sa}{\mathsf{a}}

\newcommand{\ii}{{\rm i}}

\newcommand{\be}{\begin{equation}}
\newcommand{\ee}{\end{equation}}
\newcommand{\vr}{\varrho}
\newcommand{\N}{\mathbb N}
\newcommand{\R}{\mathbb R}
\newcommand{\T}{\mathbb T}
\newcommand{\C}{\mathbb C}
\newcommand{\Z}{\mathbb Z}
\newcommand{\di}{{\rm d}}
\newcommand{\im}{{\rm i}}
\newcommand{\cU}{\mathcal{U}}
\newcommand{\ov}{\overline}
\def\wt{\widetilde}

\newcommand{\la}{\langle}
\newcommand{\ra}{\rangle}
\newcommand{\x}{\xi}
\newcommand{\s}{\sigma}
\newcommand{\pa}{\partial}
\newcommand{\sign}{{\rm sign}}
\newcommand{\mM}{\mathcal{M}}
\newcommand{\mR}{\mathcal{R}}
\newcommand{\Gt}[2]{{\tilde{\Gamma}^{#1}_{#2}}}

\newcommand{\Tu}{{{\mathbb T}^1}}

\newcommand{\Gr}[2]{{{\Gamma}^{#1}_{#2}[\epsilon_0]}}
\newcommand{\Br}[2]{{B^{#1}_{{s_0}_{#2}}(I;\epsilon_0)}}
\newcommand{\Lcal}{{\mathcal L}}

\newcommand{\Rcal}{{\mathcal R}}
\newcommand{\sg}[3]{{{\Sigma\Gamma}^{#1}_{#2}[\epsilon_0,#3]}}
\newcommand{\sr}[3]{{{\Sigma\Rcal}^{#1}_{#2}[\epsilon_0 ,#3]}}

\newcommand{\sFR}[2]{{{\Sigma\Fcal^{\R}_{#1}[\epsilon_0 ,#2]}}}
\newcommand{\Fcal}{{\mathcal F}}
\newcommand{\sF}[2]{{{\Sigma\Fcal}_{#1}[\epsilon_0 ,#2]}}

\theoremstyle{theorem}
\newtheorem{theorem}{Theorem}[section]

\newtheorem*{theorem*}{Theorem}
\newtheorem{prop}[theorem]{Proposition}
\newtheorem{proposition}[theorem]{Proposition}
\newtheorem{lemma}[theorem]{Lemma}

\theoremstyle{definition}
\newtheorem{definition}[theorem]{Definition}

\newtheorem{rem}[theorem]{Remark}
\newtheorem{remark}[theorem]{Remark}
\newtheorem{notation}[theorem]{Notation}

\newtheorem{step}{Step}
\makeatletter
\@addtoreset{step}{subsection}
\makeatother

\makeatletter
\@addtoreset{substep}{step}
\makeatother

\makeatletter
\@addtoreset{proofpart}{theorem}
\makeatother

\makeatletter
\@addtoreset{proofsubpart}{part}
\makeatother

\numberwithin{equation}{section}

\begin{document}

\maketitle

\noindent
\begin{abstract}
In this paper we 
paralinearize  the contour dynamics equation for sharp-fronts % vortex patches 
of $\alpha$-SQG, 
for any 
$ \alpha \in (0,1) \cup (1,2) $, 
% for %  the evolution of 
close to a circular vortex. This turns out to 
be a quasi-linear Hamiltonian PDE.
After deriving 
the asymptotic expansion of the linear frequencies of oscillations at the 
vortex disk and verifying 
the absence of three wave interactions,   
we prove that, in the most singular cases $ \alpha \in (1,2) $, any initial vortex patch which is  
$ \varepsilon $-close
to the disk exists for a time interval of size %  magnitude
at least $  \sim \varepsilon^{-2} $. 
This quadratic lifespan result relies 
on a paradifferential Birkhoff 
normal form reduction % approach 
and exploits cancellations arising from the Hamiltonian nature of the equation. 
This is the first normal form long time existence result of sharp fronts.
% result % of  this kind  for the $\alpha$-SQG equations. 
\end{abstract}
% The result is proved by normal form methods for quasi-linear PDEs. 
%\\[1mm]
{\small
{\it Keywords:} $ \alpha $-SQG equations, vortex patches, paradifferential calculus, 
Birkhoff normal form.}

	{\small\tableofcontents}

	\allowdisplaybreaks

\section{Introduction and main results}

In this paper we consider the generalized surface
quasi-geostrophic $\alpha$-SQG equations
\begin{align}\label{transport}
\partial_t \theta(t,\zeta)+u(t,\zeta)\cdot \nabla \theta(t,\zeta)=0 \, ,
\quad \pare{t,\zeta}\in \bR \times  \bR^2 \, , 
\end{align}
with  velocity field
\begin{equation}\label{gBS}
u \defeq\nabla^\perp \av{D}^{-2+\alpha}\theta \, , \qquad  \av{D} \defeq(-\Delta)^\frac12 \, ,
\qquad  \alpha\in (0,2) \, .
\end{equation}
These class of
active scalar  equations have been introduced in   \cite{PHS94,CFMR05} and,
for  $ \alpha \to 0 $,
formally reduce  to the 2D-Euler equation in vorticity formulation
(in this case $\theta$ is the vorticity of the fluid).
The case $ \alpha = 1 $
is the surface quasi-geostrophic (SQG) equation
in \cite{CMT1994_2} which models the evolution of 
the temperature $ 	\theta $ for atmospheric and oceanic flows. 

For the 2D Euler equation %the solutions 
global-in-time  well-posedness 
results are well known  for either regular initial data, 
see e.g. \cite{MB2002,Chemin98}, 
as well as for 
$L^1\cap L^\infty$ initial vorticities, thanks to  the celebrated 
Yudovich Theorem \cite{Yudovich1963}. 
This result is  based on the fact that the vorticity is  
transported by the particles of fluid
along the velocity field, 
which turns out to be log-Lipschitz, and thus it defines a global flow on the plane.  
On the other hand, for $\alpha > 0$, an analogous result does not hold 
because the velocity field $  u $  in \eqref{gBS} is more singular 
and does not define a flow. Nevertheless
local in time smooth solutions exist 
thanks to a nonlinear commutator structure of the vector field
for $ \alpha=1 $,
 in 
\cite{CMT1994_1},  and for $ \alpha\in\pare{1, 2} $, in \cite{CCCG2012}. For $ \alpha =1 $, the works \cite{CF02,Cordoba98} rule out the possible formation of certain 
kind of singularities but
the question of whether a finite-time singularity may develop from a smooth initial datum remains open.  In this context  we mention  the construction in  
\cite{HK21} 
of solutions 
that must either exhibit infinite in time growth of derivatives or blow up in finite time.

Existence of global weak $  L^p $ solutions  has been obtained by energy methods    
for $ \alpha=1 $,  if $ p > 4 / 3 $, in  \cite{Resnick95,Marchand08}, 
and for $ \alpha\in\pare{1, 2} $ if $ p=2 $, in \cite{CCCG2012}.
For $ \alpha\in (0,1] $ global weak solutions exist 
also in $ L^1\cap L^2 $ as proved in \cite{LX2019}. 
We also mention that non-unique weak solutions of SQG
have been constructed by  convex integration techniques in \cite{BSV19,IM2021}.

A particular type of weak solutions are the 
{\it vortex 
patches} -also called {\it sharp fronts}- 
which are given by the characteristic function of an evolving domain
\begin{equation}\label{front}
\theta\pare{t, \zeta}\defeq \begin{cases}
	1\quad\text{if } \zeta\in D(t) \, ,\\
	0\quad\text{if } \zeta\notin D(t) \, .
\end{cases} \qquad D(t)\subset \bR^2 \, .
\end{equation}
The  vortex patch  
problem \eqref{front}
can be described by the evolution of the interface $ \partial D(t)$ 
 only.  
The simplest example of a finite energy 
vortex patch is the circular ``Rankine" vortex
which is the circular steady solution with $ D(t) = D(0) = \{ |\zeta | \leq  1 \}  $ at any time $ t $.
On the other hand, 
since for $ \alpha\in\pare{0, 2} $  there is no analogue of Yudovich theorem, 
also to establish the local existence theory for sharp fronts nearby is a difficult task.
In the last few 
years special 
global in time sharp-front 
solutions  of $ \alpha$-SQG 
close to the Rankine vortex  have been constructed: 
the uniformly rotating $ V$-states  in 
\cite{HH2015,CCG2016_1,CCG2016_2,GomezSerrano2019}, 
as well as time  quasi-periodic solutions in \cite{HHM2021} for 
$ \alpha \in \pare{0, \tfrac12}$, and in \cite{GIP} for $ \alpha \in (1,2) $.
We quote further
 literature after the statement of  Theorem \ref{thm:main}. 

In this work we  prove the first  long-time existence result of sharp fronts
of $ \alpha $-SQG, in the more singular cases $ \alpha \in (1,2) $,
for any initial interface $ \partial D(0) $ sufficiently smooth and close to a circular Rankine vortex, see
Theorem \ref{thm:main}.  This is achieved thanks to the paralinearization 
result of the $ \alpha$-SQG sharp front equation in Theorem \ref{prop:paralinearization_1}
for any $ \alpha \in (0,1) \cup (1,2) $, that 
we consider of independent interest in itself. 

Let us present precisely our main results. 
The evolution of the boundary of the vortex patch 
is governed by the {\it Contour Dynamics Equation} 
for a parametrization  
  $ X:\bT\to \bR^2 $, $ x \mapsto X(t, x) $,  with $ \bT \defeq \bR \slash 2\pi \bZ $,
  of the boundary $ \partial D(t) $ of the vortex patch.
The Contour Dynamics Equation  for the $ \alpha $-SQG patch
--also called sharp-fronts equation-   is
\begin{equation}\label{eq:SQGpatch} 
\partial_t X\pare{t, x}  = \frac{c_\alpha}{2\pi} \  \int \frac{X'\pare{t, x } - X'\pare{t, y}}{\av{X \pare{t, x} - X \pare{t, y}}^\alpha} \, \dd y \, , \quad \alpha\in\pare{0, 2}  \, ,
\end{equation}
where $ ' $ denotes the derivative with respect to $ x $, 
\begin{equation}\label{eq:calpha}
c_\alpha \defeq \frac{\Gamma\pare{\frac{\alpha}{2}}}{2^{1-\alpha} \Gamma\pare{1-\frac{\alpha}{2}}}  
\end{equation}
and   $\Gamma ( \cdot )$ is the Euler-Gamma function. 
 The local solvability of \Cref{eq:SQGpatch} in Sobolev class 
has been proved 
in \cite{Gancedo2008}  for $ \alpha\in\left(0, 1\right] $, if the 
initial datum belongs to $ H^s $, $ s\geq 3 $
 and in \cite{GNP2021,GP2021} for 
less regular initial data  (see \cite{Rodrigo2005}
for  $ C^\infty $ data). The uniqueness has been established in \cite{CCG18}.
For $ \alpha\in\pare{1, 2} $ the local existence and uniqueness theory 
has been proved in \cite{CCCG2012,GP2021}  for initial data in $ H^s $, $  s\geq 4 $, see also \cite{KYZ17,AA2022}. 
In the very recent work \cite{KL2023} it is proved that the $ \alpha $-patch problem 
is ill posed in $ W^{k, p} $ if $ p\neq 2 $.

Very little is known concerning long time existence results. 
Actually highly unstable dynamical behaviour could emerge. 
In this context we mention the remarkable work \cite{KRLY16} 
where two smooth patches of 
opposite sign 
develop 
 a finite time particle collision. We also quote 
the numerical study \cite{Scott11} which provides 
some evidence of the development of filaments, pointing to a possible formation of  singularities via a self-similar filament cascade.

In this paper we consider sharp fronts
of $ \alpha $-SQG that are a radial perturbation of the unitary circle, i.e.
\begin{equation}\label{circle}
X\pare{x} = \pare{1+h\pare{ x}}\vec \gamma\pare{x} \, , \qquad \quad \vec \gamma (x) 
\defeq (\cos (x), \sin (x)) \, .
\end{equation}
Since
only the normal component
of the velocity field deforms the patch,   one derives
from \eqref{eq:SQGpatch} a scalar evolution equation for  $ h (x) $.
Multiplying \eqref{eq:SQGpatch} by the
 normal vector
 $ n (x) = h' (x) \vec  \gamma' (x) - \pare{1+h(x)} \vec \gamma (x) $  to the boundary
 of the patch at $ X(x) $, 
  we deduce that  $ h(t,x)$ solves the equation
\begin{equation}
\label{eq:elev_funct_eq}
\begin{aligned}
 - \pare{1+h\pare{x}} \partial_t h \pare{x} 
=  & \  \frac{c_\alpha}{2\pi} \  \int \frac{ \cos\pare{x-y}\bra{\pare{1+h\pare{x}}h'\pare{y} - \pare{1+h\pare{y}}h'\pare{x}}   }{\bra{ \pare{1+h\pare{x}}^2 + \pare{1+h\pare{y}}^2-2\pare{1+h\pare{x}}\pare{1+h\pare{y}}\cos\pare{x-y}}^{\frac{ \alpha }{2}}} \dd y \\
& \ +  \frac{c_\alpha}{2\pi} \   \int \frac{ \sin\pare{x-y}\bra{\pare{1+h\pare{x}}\pare{ 1+ h\pare{y} } + h'\pare{x} h'\pare{y}} }{\bra{ \pare{1+h\pare{x}}^2 + \pare{1+h\pare{y}}^2-2\pare{1+h\pare{x}}\pare{1+h\pare{y}}\cos\pare{x-y}}^{\frac{ \alpha }{2}}} \dd y \, .
  \end{aligned}
\end{equation}
In view of  \cite{CCCG2012,GP2021}  
if $ h_0\in H^s $,
for any
$ s \geq 4 $,
 there exists  a unique solution $ h\in\cC\pare{\bra{0, T}; H^s} $ of \eqref{eq:elev_funct_eq}
defined up to a time
$
T > \frac{1}{C_s \ \norm{h_0}_{H^s}} $.
The following result  extends the local-existence result for longer times.

\begin{theorem}[Quadratic life-span]\label{thm:main}
Let $ \alpha \in \pare{1, 2} $.
There exists  $ s_0 > 0 $ such that for any $ s \geq s_0 $,
 there are $ \varepsilon_0  > 0 $, $ c_{s,\alpha} > 0 $,
 $ C_{s,\alpha} > 0 $  such that, for any
 $ h_0 $ in $ H^s \pare{ \bT ; \bR }$ satisfying  $ \norm{h_0}_{H^s} \leq
  \varepsilon < \varepsilon_0  $,
 the 
 equation \eqref{eq:elev_funct_eq} with initial condition $ h(0) = h_0 $
 has a unique classical solution 
  \begin{equation}\label{timeexi}
h\in\cC\pare{\bra{- T_{s, \alpha} , T_{s, \alpha}}; H^s\pare{\bT;\bR}} \qquad \text{with} \qquad
 T_{s, \alpha}  > c_{s,\alpha} \varepsilon^{-2}  \, , 
\end{equation}
satisfying 
$ \norm{h\pare{t}}_{H^s} \leq C_{s, \alpha} \, \varepsilon  $, for any 
$  t\in\bra{-T_{s, \alpha}, T_{s, \alpha}}  $.
\end{theorem}

\Cref{thm:main} is proved by normal form arguments for quasi-linear 
Hamiltonian PDEs. The first important step is the {\it paralinearization} of 
 \eqref{eq:elev_funct_eq} once it has been written in Hamiltonian form, 
 see \Cref{prop:paralinearization_1}. The 
 paralinearization formula \eqref{eq:paralinearized_1}
of the $ \alpha $-SQG equations 
holds for any $ \alpha \in (0,1) \cup (1,2) $. It is a major result of
this paper, that we expect  to be used also in other contexts. 

In order to prove  \Cref{thm:main} 
we reduce the paralinearized equation \eqref{eq:paralinearized_1}, for any $ \alpha \in (1,2) $, 
to Birkhoff normal form up to cubic smoothing
 terms. 
This requires to 
prove the absence of {\it three wave interactions}, which is  verified 
in Lemma \ref{lem:nonres_cond}
by showing the convexity of the  
linear normal frequencies of 
 the $ \alpha $-SQG equation at the circular vortex patch. 
 
Theorem \ref{thm:main} is the first Birkhoff normal form  results for sharp fronts 
equations. 

In recent years several 
advances have been obtained concerning long time existence of solutions for 
quasi-linear equations in fluids dynamics on $ \bT $, namely 
with periodic boundary conditions, as  the water waves equations.
Quadratic life span of small amplitude solutions have been obtained in 
\cite{BFF2021,IT2017,HIT2015, HIT2016,Wu2009,IP2015,IP2019,AIT2019},
extended to longer times in 
\cite{BD2018,BFP2018,BFF2021,BMM2022,DIP2022,WU2020,Zheng2022,BMM2},
by either introducing quasi-linear modified energies or using 
Birkhoff normal form techniques. 
We also quote the long time existence result 
\cite{CCZ2021} for solutions of  SQG close to the infinite energy 
radial solution $ \av{\zeta} $.

Before explaining 
in detail 
the main ideas of proof we present further results in literature about 
$ \alpha $-SQG. 
\\[1mm]
{\it Further literature.}
Special infinite energy 
global-in-time sharp front solutions   have been constructed 
in  \cite{CGI2019, HSZ20, HSZ2021} 
if the initial patch is a small perturbation of the half-space, 
by exploiting dispersive techniques.   In \cite{CCG2019,CCG2020} special global smooth solutions in the cases $ \alpha= 0,1 $ are obtained using bifurcation theory.
Concerning the possible formation of 
singularities, we mention that \cite{KRLY16,KYZ17} 
constructed
special initial sharp fronts in the half-space 
which develop 
a splash singularity in finite time  if  $\alpha \in \pare{ 0,\frac{1}{12} } $,  
 later extended 
in  \cite{GP2021} for $\alpha \in \pare{ 0,\frac{1}{3} }$. 
We also mention that 
 \cite{GS14,KL21} have proved that, if $  \alpha\in (0,1]  $, 
 the sharp fronts equation in the whole space 
 does not generate finite-time singularities of splash type.
\\[1mm]
{\it V-states.}
   The existence of uniformly rotating 
    $V$-states close to the disk was first numerically investigated in \cite{DZ1978} and analytically proved in \cite{Burbea1982} for the Euler equations,
    recently extended to global branches in 
    \cite{HMW2020}.  
For $ \alpha $-SQG  equations, as already mentioned,   
local bifurcation results  
of sharp fronts from the disk have been obtained in 
\cite{HH2015,CCG2016_1,CCG2016_2,CCG2019,GomezSerrano2019,CCG2020}. 
Smooth $ V $-states 
bifurcating from different steady configurations have been constructed 
for  $ \alpha \in [0,2) $ in
   \cite{DHH2018, Renault2017, HM2017, HM2016_1, HM2016_2, DHMV2016, DHH2016, HMV2015, HMV2013}. We refer to the introductions in 
   \cite{HHM2021,GIP} for more references. 
\\[1mm]
{\it Quasi-periodic solutions.}
As already mentioned 
global in time quasi-periodic solutions of the $ \alpha $-SQG 
vortex patch equation 
close to the circle \eqref{circle}
have been recently constructed   in \cite{HHM2021} for 
$ \alpha \in \pare{0, \tfrac12} $ 
and 
 in \cite{GIP} for $ \alpha \in (1,2)$.    
The result  \cite{HHM2021} holds for 
``most" values of $ \alpha \in \pare{0, \tfrac12} $ (used a parameter to impose non-resonance conditions) 
whereas   \cite{GIP} holds for any $ \alpha \in (1,2)$, 
using the initial conditions as parameters, 
 via a Birkhoff normal form analysis.   
The $ 2D$-Euler equation is more degenerate  
and in this case quasi-periodic solutions have been constructed 
in \cite{BHM2022} close to the family of 
Kirkhoff ellipses, not only close to the disk
(we refer to \cite{BHM2022} for a wider introduction to the field and literature
about quasi-periodic solutions). 
%  for most values of the eccentricities. 
 These results build on on KAM techniques  \cite{BBHM2018, BBM2016_2, BFM2021,BM2020,FG2020} developed for the water waves equations.

\subsection*{Ideas of the proof}

{\bf The average-preserving unknown and the Hamiltonian formulation.}
The equation
\eqref{eq:elev_funct_eq}  for the
unknown $ h (x) $ is {\it not} convenient
because its evolution does not preserve  the average and
 it is {\it not} Hamiltonian.
This problem is overcome 
by reformulating \eqref{eq:elev_funct_eq} 
in terms of the  variable
\begin{equation}\label{eq:def_f}
f\pare{x} \defeq  h\pare{x} + \tfrac12 h^2\pare{x} \, . 
\end{equation}
 Indeed, symmetrizing in the $ x,y $ variables, we get 
$ \int_{\bT} \, 	\text{r.h.s.  \eqref{eq:elev_funct_eq}} \, \dd y
=0 $ 
and then 
$ \ddt \int_{\bT} \pare{ h\pare{t, x} + \tfrac{1}{2} h^2\pare{t, x}} \dd x =0 $. 
Thus the average of  
 $f\pare{x} $ in \eqref{eq:def_f}
is preserved along the patch evolution.
Note that, inverting \eqref{eq:def_f} for small $ \| f \|_{L^\infty} $ and $ \| h \|_{L^\infty} $,
we have
$ h\pare{x}  = \sqrt{1+2 f\pare{x}} - 1 $
and the Sobolev norms of $ f \pare{x} $ and $ h \pare{x}  $ are equivalent
\begin{equation}\label{eq:f-h}
\norm{f}_{s}\sim \norm{h}_{s} \, , \qquad 
 \forall s > \tfrac12  \, .
\end{equation}

\begin{rem}
There is a  deep connection between the conservation of the average of $ f(x)  $ and the incompressibility of the flow generated by the $ \alpha $-SQG patch.
Actually  the Lebesgue measure $ \Vol \pare{t} $ of the finite region of $ \bR^2 $ enclosed by the  patch is, passing to polar coordinates,
\begin{equation*}
\Vol\pare{t} = \int _{-\pi}^{\pi} \int _0 ^{1+h\pare{t, x}} \rho \ \dd \rho \ \dd x = \pi + \int _{-\pi} ^\pi \pare{ h\pare{t, x} + \frac{h^2\pare{t, x}}{2} } \dd x = \pi + \int _{-\pi}^{\pi} f\pare{t, x} \dd x
\end{equation*}
and therefore the conservation of the average $ \int_\T f \pare{x} \ \dd x $
amounts to the conservation of
$ \Vol\pare{t} $.
\end{rem}

The variable \eqref{eq:def_f}
 has been used in \cite{HHM2021} where it is also
proved that the evolution equation for $ f $ has a Hamiltonian structure, 
see also \cite{GIP}.
The following result is  \cite[Proposition 2.1]{HHM2021}:

\begin{prop}[Hamiltonian formulation of \eqref{eq:elev_funct_eq}]
Let $ \alpha\in\pare{0, 2} $.
If $ h $ is a solution of  \cref{eq:elev_funct_eq} then the variable $ f $ defined  in \eqref{eq:def_f}  solves the Hamiltonian equation
\begin{equation}\label{eq:SQG_Hamiltonian}
\pa_t f = \partial_x \,  \nabla E_{\alpha} \pare{f}
\end{equation}
where  $ E_\alpha \pare{ f } $ is the {\it pseudo-energy} of the patch
whose $ L^2 $-gradient $ \nabla E_\alpha \pare{f} $ is
\begin{equation}
\label{eq:gradient-pseudoenergy}
\nabla E_\alpha \pare{f} \\ = \frac{c_\alpha}{2\pare{1-\frac{\alpha}{2}}}  \fint \frac{1+2f\pare{y} +\sqrt{1+2f\pare{x}} \ \partial_y \bra{\sqrt{1+2f\pare{y}}\sin\pare{x-y}}  }{\bra{ 1+2f\pare{x} + 1+2f\pare{y} -2\sqrt{1+2f\pare{x}}\sqrt{1+2f\pare{y}} \cos\pare{x-y}}^{\frac{\alpha}{2}} } \ \dd y
\, .
\end{equation}
\end{prop}
Note that the evolution equation \eqref{eq:SQG_Hamiltonian} is
translation-invariant because 
 $  E_{\alpha} \circ \st_\varsigma =  E_{\alpha} $ for any 
 $ \varsigma \in \mathbb{R} $, where $\st_\varsigma f (x) := f (x+ \varsigma )$. %  {X.tra0}.
Moreover, in view of  the presence of the Poisson tensor $ \partial_x $ in
\eqref{eq:SQG_Hamiltonian} it is evident that the space average
of $  f $ is a prime integral of \eqref{eq:SQG_Hamiltonian}. In the sequel we 
assume the space average of $ f $ to be zero.
\\[1mm]
{\bf Paralinearization of \eqref{eq:SQG_Hamiltonian} for $ \alpha \in (0,1) 
\cup (1,2) $.}
 \Cref{sec:paralinearization} is dedicated to write the Hamiltonian equation \eqref{eq:SQG_Hamiltonian} in paradifferential form and to provide the detailed
structure of the principal and subprincipal symbols in the expansion of the paradifferential operator, obtaining
\begin{equation}
\label{eq:exp_approx}
\pa_t f + \partial_x \circ \OpBW{
\pare{ 1+\nu\pare{f; x} }  L_{\alpha}\pare{\av{\xi}} + V\pare{f; x}  + P \pare{f; x , \xi}
} f = \textnormal{smoothing terms} 
\end{equation}
where (see \Cref{prop:paralinearization_1} for a detailed statement)
\begin{itemize} 
\item $ \pare{1+\nu\pare{f; x} }  L_{\alpha}\pare{\av{\xi}} + V\pare{f; x} $
is a real symbol of order $ \max \{\alpha-1,0\}$ and $ \nu\pare{f; x} $, 
$ V\pare{f; x} $ are real functions vanishing for $  f = 0 $;
\item  the symbol $ P \pare{f; x , \xi} $ has order $ - 1 $ and vanishes for $  f = 0 $.
\end{itemize}
We note that in  \eqref{eq:exp_approx} the operator $ \OpBW { \ } $ 
is the paradifferential quantization  according to Weyl 
(see Definition \ref{quantizationtotale}) 
and  thus 
$  \OpBW{  {1+\nu\pare{f; x} }  L_{\alpha}\pare{\av{\xi}} + V\pare{f; x}  } $ 
is self-adjoint. As a consequence  the linear Hamiltonian operator
$ \pa_x \circ  \OpBW{  {1+\nu\pare{f; x} }  L_{\alpha}\pare{\av{\xi}} + V\pare{f; x}  } $ 
is   skew-self-adjoint at positive orders. This is the % ultimate
 reason why the unbounded  quasi-linear
vector field
$ \pa_x \circ  \OpBW{  {1+\nu\pare{f; x} }  L_{\alpha}\pare{\av{\xi}} + V\pare{f; x}  } f $  
admits energy estimates  in Sobolev spaces $ H^s $ via commutator estimates
(actually existence and unicity of the 
solutions of 
% the paralinearized equation 
\eqref{eq:exp_approx} would follow % arguing
 as in \cite{BMM2021}). 
We remark the absence in \eqref{eq:exp_approx} of %paradifferential 
operators like 
 $ \partial_x\circ\OpBW{ \text{symbol of order } (\alpha - 2)} $. 
The  cancellations of such terms  are verified in \Cref{sec:Hamiltonian_identity}
by a direct calculus 
and it is ultimately a consequence of the Hamiltonian structure of the equation
\eqref{eq:SQG_Hamiltonian}. 

We also note that 
the equation \eqref{eq:exp_approx} can be written,  in homogeneity degrees, as  
\begin{equation}
\label{lineazero}
\pa_t f + \omega_\alpha (D) f = \cO \pare{ f^2 } \, \qquad \text{where}
\qquad \omega_\alpha (D) := \partial_x \circ L_{\alpha}\pare{\av{D}}  
\end{equation}
is the unperturbed dispersion relation.  

Let us 
 explain how we deduce the paralinearization formula 
\eqref{eq:exp_approx} in Section 
\ref{sec:paralinearization}. 
The nonlinear 
term  $ \nabla E_\alpha\pare{f} $ in \eqref{eq:SQG_Hamiltonian} can be written as a 
convolution operator 
\begin{equation*}
\nabla E_\alpha\pare{f}\pare{x}
=
\int _{-\pi}^{\pi} K\pare{f;x,z} \frac{f\pare{x} - f\pare{x-z}}{\av{ z}^\alpha} \, \dd z   
\end{equation*}
with a nonlinear real valued convolution kernel $ K\pare{f;x,z}$. 
By Taylor expanding the kernel $ z\mapsto K\pare{f;x,z} $ at $ z = 0 $ 
(provided $ f $ is sufficiently regular) and expanding in paraproducts the arguments of the above integral,  we obtain an expansion of the form
\begin{subequations}
\begin{align}\label{eq:Taylor_pseudoenergy}
\nabla E_\alpha\pare{f}\pare{x}
 = & \ 
\sum_{\mathsf{j}=0}^{\mathsf{J}}
\OpBW{K_\mathsf{j}\pare{f, \ldots , f^{\pare{\mathsf{j+1}}}; x}}
\int _{-\pi}^{\pi} \pare{f\pare{x} - f\pare{x-z}} \ \frac{z^{\mathsf{j}}}{\av{ z}^\alpha} \ \dd z \\
& +
\int_{-\pi}^{\pi} \OpBW {R\pare{f, \ldots , f^{\pare{\mathsf{j+1}}}; x, z}}\pare{f\pare{x} - f\pare{x-z}} \dd z 
+\textnormal{ smoothing terms } ,   \label{eq:Taylor_pseudoenergy2}
\end{align}
\end{subequations}
where $ R\pare{f, \ldots , f^{\pare{\mathsf{j+1}}}; x, z} = 
\mathpzc{o} \pare{ \av{z}^{\mathsf{J}-\alpha}}   $ as $ z\to 0 $
being the Taylor remainder at order $ \mathsf{J} $
(here  $ f^{\pare{\mathsf{j}}} (x) $ denotes the $ \mathsf{j} $-derivative of 
$ f(x) $).
% is an $ \mathpzc{o}\pare{\av{z}^{-\alpha}z^{\mathsf{J}}} $ function. 
The terms  in the finite sum \eqref{eq:Taylor_pseudoenergy}
are particularly simple paradifferential operators. Indeed, provided $ \alpha < 2 $, 
$$ 
\int _{-\pi}^{\pi} \pare{f\pare{x} - f\pare{x-z}} \ \frac{z^{\mathsf{j}}}{\av{ z}^\alpha} \ \dd z = \bV_{\alpha-\mathsf{j}} f  + m_{\alpha-\pare{\mathsf{j}+1}}\pare{D}f  
$$
 where $ \bV_{\alpha-\mathsf{j}} $ is a real constant 
 and $  m_{\alpha-\pare{\mathsf{j}+1}}\pare{\xi} $ is a 
 Fourier multiplier of order $ \alpha-\pare{\mathsf{j}+1} $, as follows 
 by standard asymptotics of singular integral operators, see  \cite{Stein1993}. 
Thus, by symbolic calculus, % composition results,  
\begin{equation*}
\begin{small}
\OpBW {K_\mathsf{j} } % \pare{f, \ldots , f^{\pare{\mathsf{j+1}}}; x}}
\int _{-\pi}^{\pi} \pare{f\pare{x} - f\pare{x-z}} \ \frac{z^{\mathsf{j}}}{\av{ z}^\alpha} \ \dd z
=
\OpBW {V_{\alpha-\mathsf{j}}\pare{f, \ldots , f^{\pare{\mathsf{j+1}}}; x} 
 + K_\mathsf{j}\pare{f, \ldots , f^{\pare{\mathsf{j+1}}}; x}  
 m_{\alpha-\pare{\mathsf{j}+1}}\pare{\xi}} f
 +\textnormal{l.o.t.} , 
 \end{small}
\end{equation*}
where  $ V_{\alpha-\mathsf{j}}  $ are real functions. 
The unbounded terms 
$ \partial_x\circ \OpBW{K_\mathsf{j}\pare{f, \ldots , f^{\pare{\mathsf{j+1}}}; x}  m_{\alpha-\pare{\mathsf{j}+1}}\pare{\xi}} f $, $ \mathsf{j}=0,1 $,  
would induces a loss of derivatives  in the  $ H^s $ energy estimates 
if  the imaginary part 
$ \Im  \, m_{\alpha-\pare{\mathsf{j}+1}}\pare{\xi}\neq 0 $.
Therefore a detailed analysis of these symbols  is essential. 
The highest order Fourier multiplier $ m_{\alpha-1}(\xi) $ 
turns out to be real. %  and even with respect to $ \xi $. 
Concerning the subprincipal symbol 
$ K_{1}\pare{ f, f'; x }   m_{\alpha-2}(\xi) $, it turns out that 
 $  m_{\alpha-2}(\xi) $ has a non-zero imaginary part but 
a subtle 
nonlinear cancellation 
reveals that the corresponding 
coefficient $ K_{1}\pare{ f, f'; x }  $
 is identically zero, as verified in \cref{sec:Hamiltonian_identity}. 
Such a structure, which ultimately stems by the Hamiltonian nature of \eqref{eq:SQG_Hamiltonian}, 
could be proven up to an arbitrary negative order. 

Concerning  the first term in \eqref{eq:Taylor_pseudoenergy2}, we use   
that $ R\pare{f, \ldots , f^{\pare{\mathsf{j+1}}}; x, z} $ is $ \mathpzc{o}\pare{ \av{z}^{\mathsf{J}-\alpha}} $ as $ z\to 0 $ so that, % Consequently, 
modulo regularizing operators, it can be expressed as a paradifferential operator of order \( \alpha - (\mathsf{J} + 1) \), which is a bounded vector field 
taking $ \mathsf{J}\geq 2 $, see \Cref{prop:reminders_integral_operator}. 
\\[1mm]
{\bf Reduction of \eqref{eq:exp_approx} to  Birkhoff normal form up to cubic terms.}
In Section \ref{sec:constant_coeff} we  first conjugate the  paradifferential equation
\eqref{eq:exp_approx} 
into an equation with constant coefficient symbols, modulo smoothing  operators, 
\begin{equation}
\label{eq:exp_approxredu}
\pa_t \cg + \partial_x \circ \OpBW{ \pare{1+\mathpzc{c}_0\pare{f}} 
L_\alpha\pare{\av{\xi}} + \mathsf{H}_{\alpha}\pare{f;  \xi} } \cg = \text{smoothing terms}
\end{equation}
where $ \mathpzc{c}_0\pare{f}  $ is the average of a real nonlinear function of 
$ \nu{\pare{f;x}} $
and $ \mathsf{H}_{\alpha}\pare{f;  \xi} $ is a $ x $-independent symbol with imaginary part 
$ \Im  \mathsf{H}_{\alpha}\pare{f;  \xi} $ of order $ - 1 $, see 
Proposition \ref{prop:cc_ao}. Thus 
\eqref{eq:exp_approxredu} is still Hamiltonian up to order zero and 
thus it satisfies  $ H^s $-energy estimates. 
The unknowns $\cg\pare{t}$ and $f\pare{t}$ have equivalent Sobolev norms 
$\norm{\cg\pare{t}}_s \sim_{s,\alpha} \norm{f\pare{t}}_s$.
We remark that in \eqref{eq:exp_approxredu}
the constant $ \mathpzc{c}_0\pare{f} $ and the symbol 
$ \mathsf{H}_{\alpha}\pare{f;  \xi} $ vanish 
quadratically at $ f = 0 $ and thus 
the only term which can disturb the quadratic life span 
of the solution $ \cg \pare{t} $
is the smoothing operator $ R_1 \pare{f}  $ in the decomposition 
$$
\text{smoothing terms} = R_1 \pare{f} \cg + R_{\geq 2} \pare{f} \cg \, .  
$$
Then in Lemma \ref{lem:BNF1step}  % In \Cref{sec:quadratic_normal_forms} 
we implement a  Birkhoff 
 normal form step  to cancel  $ R_1 \pare{f} \cg $.  
An algebraic  ingredient is to verify the absence of three wave interactions, namely that,   
for any  $ n,j,k\in\bZ\setminus \set{0} $ satisfying $ k = j + n $, 
$$
\av{\omega_\alpha\pare{k} -  \omega_\alpha\pare{j} -  \omega_\alpha\pare{n} } 
\geq c  > 0   \, ,
$$
where $\omega_\alpha \pare{j} $ are the normal $ \alpha$-SQG frequencies in 
\eqref{lineazero}. Such a property follows by proving the {\it convexity}  of the 
the map  $ \omega_\alpha \pare{j} $ for $ j \in \N  $, 
see \cref{lem:nonres_cond}. %  (actually also stronger lower bounds hold). 

The final outcome is an {\it energy estimate} for any 
small enough 
solution of \eqref{eq:SQG_Hamiltonian} of the form 
\[
  \norm{f\pare{t}}^{2}_{H^s}   \lesssim_{s,\alpha}   \norm{f\pare{0}}^{2}_{H^s} 
+   \int_0^t \norm{f\pare{\tau}}^{4}_{H^s} \, \dd \tau \, , \qquad t > 0 \, , 
\]
which implies Theorem \ref{thm:main}.
\\[1mm]
{\bf Structure of the manuscript.}
Section \ref{sec:preliminaries} contains the % main functional  setting of 
paradifferential calculus used along the paper. 
In Section \ref{sec:para}
we report the main results in \cite{BD2018,BMM2022}. Then 
in Section \ref{sec:paradiff}  we introduce a $ z $-dependent 
paradifferential calculus used for the paralinearization 
of \eqref{eq:exp_approx}
in \Cref{sec:paralinearization}. 
\Cref{sec:linearized} is dedicated to  the linearization of \eqref{eq:SQG_Hamiltonian} at the stationary state $ f \equiv 0 $. 
\Cref{lem:linearization,prop:Lalpha_asymptotic} extend
to any $ \alpha\in\pare{0, 2} $ the asymptotic expansions of the 
normal frequencies $ \omega_\alpha \pare{j} $ 
proved in \cite{HHM2021} for $ \alpha \in (0,1)$. 
In \Cref{sec:paralinearization} we provide 
the paralinearization \eqref{eq:exp_approx} of the Hamiltonian equation 
\eqref{eq:SQG_Hamiltonian} for any $ \alpha \in (0,1) \cup (1,2) $. 
In Section \ref{sec:constant_coeff} we  conjugate the  paradifferential equation
\eqref{eq:exp_approx} into an equation with constant coefficients, modulo smoothing  operators. 
In \Cref{sec:quadratic_normal_forms} we perform the 
Birkhoff  normal form step and prove Theorem \ref{thm:main}.
\\[1mm]
{\bf Notation.}
We denote with $ C $ a positive constant which does not depend on any parameter of the problem. We write $ A\lesssim_{c_1, \ldots, c_M} B $ if $ A\leq C \pare{c_1, \ldots, c_M} B $ and $ A\sim_{c_1, \ldots, c_M} B $ if $ A\lesssim_{c_1, \ldots, c_M} B $ and $ B\lesssim_{c_1, \ldots, c_M} A $.  
We denote with $ \bN = 1, 2, \ldots $ the set of natural numbers and 
$ \bN_0\defeq \bN\cup\set{0} $. For any  $ x \geq 0 $  we denote 
 $ \ceil{x} \defeq \min\set{n\in\bN_0 \ \middle| \ x\leq n} $.
We denote $ \bT \defeq \bR \setminus (2 \pi \bZ ) $ the one-dimensional torus
with norm  $ \av{x}_{\bT} \defeq \inf_{j \in \bZ} \av{x + 2 \pi j } $.
We denote $ D = - \ii \partial_x $ and  $ \comm{A}{B} 
  $ the commutator $ \comm{A}{B}  = A B- B A =: \Ad_A B $.
Given a linear real self-adjoint operator $ A $ any operator of the form $ \partial_x \circ A $ will be referred as {\it linear Hamiltonian}. 
We denote  $ \fint \bullet \dd x = \frac{1}{2\pi  } \int_{\bT} \bullet \dd x $.

\section{Functional setting}\label{sec:preliminaries}

Along the paper we deal with real parameters 
\begin{equation}\label{eq:parameters}
s\geq s_0  \gg K \gg \rho \gg N  \geq 0
\end{equation}
where $ N \in \N $. 
The values of $ s, s_0, K $ and $ \rho $ may vary from line to line while still being true the relation \eqref{eq:parameters}. 
For the proof of Theorem \ref{thm:main}  we shall take  $ N = 1 $.
 \smallskip

We expand a  $2\pi$-periodic function $u(x)$ in  $ L^2 (\T;\C)$ in Fourier series as
\begin{equation}\label{Fourierser}
u(x)= \sum_{j \in \mathbb{Z}} \hat{u}\pare{j} e^{\im j x}\, ,
\qquad \hat{u}\pare{j} \defeq {\cal F}_{x \to j}\pare{j} \defeq u_j \defeq \frac{1}{2\pi}\int_{\mathbb{T}}u(x) e^{- \im j x }\,\di x \, . 
\end{equation}
A function $ u(x) $ is real if and only if  $ \overline{u_j}  = u_{-j}  $, for any $ j \in \Z $. 
\noindent
For any $ s\in\bR $ we define the Sobolev space $  H^{s} \defeq H^{s}(\T;\C) $ 
with norm
\begin{equation*}
\norm{u}_{s} \defeq \norm{u}_{H^{s}} = \pare{
\sum_{ j \in \bZ } \angles{j}^{2s}  \av{\hat u\pare{j}}^2 
} ^{\frac12} \, , \qquad \angles{j} \defeq \max (1, |j|) \, . 
\end{equation*}
We define $ \Pi_0 u \defeq \hat u_0 $  the average of $ u $ and 
\begin{equation}\label{eq:Pi0bot}
\Pi_0^\bot  \defeq  \Id - \Pi_0 \, . 
\end{equation}
We define $ H^s_0 $ the subspace of zero average functions of $ H^s $, for which
 we also denote $ \norm{u}_s = \norm{u}_{H^s}  = \norm{u}_{H^s_0} $.
Clearly  $ H^0_0(\T;\C) = L^2_0(\T;\C)  $
with scalar product, for any $ u, v \in  L^2_0 (\T;\C)$,
\begin{equation}\label{scpr12hom}
\la u, v \ra_{ L^2_0 } = \int_\T  u(x)\, \overline{ {  v(x)}} \, \di x \, .
\end{equation}
Given an interval $ I\subset \R$ symmetric with respect to $ t = 0 $
and $s\in \R$, we define the space
$$
C_*^K \pare{ I;H_0^s\pare{ \mathbb{T};\mathbb{X} } } \defeq
\bigcap_{k=0}^K C^k \pare{ I; H_0^{s- \alpha k}\pare{ \mathbb{T};\mathbb{X} }  } \, ,
\qquad \quad
\mathbb{X}=\bR \, , \ \bC \, ,
$$
resp. $  C^{K}_{*}(I;H^{s}(\T;\mathbb{X}))  $,  
endowed with the norm
\begin{equation} \label{Knorm}
\sup_{t\in I} \norm{ u (t, \cdot)}_{K,s} \qquad
{\rm where} \qquad
\norm{ u(t, \cdot)}_{K,s}\defeq \sum_{k=0}^K \norm{ \partial_t^k u(t, \cdot)} _{H ^{s- \alpha k}} \, .
\end{equation}
We denote $B^K_s(I;\epsilon_0)$, resp. $B^K_{s,\R}(I;\epsilon_0)$,
 the ball of radius $\epsilon_0 > 0 $ in $C_*^K(I,H_0^s\left(\mathbb{T};\mathbb{C})\right)$,
resp.  in $C_*^K(I,H_0^s\left(\mathbb{T};\mathbb{R})\right)$. We 
also we define $ B_{C^{K}_{*}(I,H^{s}(\T;\C))}\pare{0;\epsilon_0} $ the ball of center zero and radius $ \epsilon_0 $ in $ C^{K}_{*}(I,H^{s}(\T;\C)) $.

\begin{remark}
The parameter $ s $ in \eqref{Knorm} denotes the spatial Sobolev regularity of the solution $ u(t, \cdot) $
and $ K $  its regularity in the time variable.
The $ \alpha$-SQG vector field
 loses $ \alpha $-derivatives,
and therefore,  differentiating the solution $ u(t) $ for  $k$-times in the time variable,
there is a loss of $ \alpha k$-spatial derivatives.
The parameter $\rho$ in \eqref{eq:parameters} denotes the order where we decide to stop our regularization of the system.
\end{remark}

We  set some further notation.
For $n\in \mathbb{N}$ we denote by $\Pi_n$ the orthogonal projector from
$L^2(\mathbb{T};\mathbb{C})$ to the linear subspace spanned by
 $\{ e^{\im nx}, e^{-\im nx}\}$,
$  (\Pi_n u)(x) \defeq
  \hat{u}(n) e^{\im nx} + \hat{u}(-n) e^{-\im nx} $. 
 If $ \, \cU=(u_1, \dots , u_p)$ is a $p$-tuple of functions and $\vec{n}=(n_1,\dots,n_p)\in \mathbb{N}^p $, we set
 $ \Pi_{\vec{n}} \cU \defeq \big( \Pi_{n_1}u_1,\dots,\Pi_{n_p}u_p \big) $ and 
$ \st_\varsigma \cU \defeq \big(  \st_\varsigma u_1,\dots,  \st_\varsigma u_p \big) $, 
 where $ \st_\varsigma $ is the translation operator
 \begin{equation}\label{X.tra0}
\st_\varsigma \colon u(x) \mapsto u(x + \varsigma) \, .
\end{equation}
For $ \vec{\jmath}_p = (j_1,\dots,j_p) \in \Z^p$
we  denote
$ |\vec{\jmath}_p | \defeq \max(|j_1|, \ldots, |j_p| ) $ and
$ u_{\vec{\jmath}_p} \defeq u_{j_1} \dots u_{j_p}  $. 
Note that 
the Fourier coefficients of $\st_\varsigma u$ are 
$ (\st_\varsigma u)_j  = e^{\im  j \varsigma} u_j $. 

A vector field $ X(u) $ is {\it translation invariant} if
$ X \circ \st_\varsigma = \st_\varsigma \circ X $ 
for any $  \varsigma \in \R $.

Given a linear operator  $ R(u) [ \cdot ]$ acting on $ L^2_0(\T;\C)$
we associate the linear  operator  defined by the relation
$ \ov{R(u)} v \defeq \ov{R(u) \ov{v} } $ for any  $ v \in L^2_0(\T;\C)  $
An operator $R(u)$ is {\em real } if $R(u) = \ov{R(u)} $ for any $ u $ real. 

\subsection{Paradifferential calculus}\label{sec:para}

We introduce paradifferential  operators (Definition \ref{quantizationtotale})
following \cite{BD2018}, with minor modifications
due to the fact that
we deal with a scalar equation and not a system, and the fact that 
we consider operators acting on $  H_0^s $ and $ H^s $ and not 
on  homogenous spaces $ \dot H^s $. In this way we will mainly rely on 
results in \cite{BD2018,BMM2022}.

\paragraph{Classes of symbols.}
Roughly speaking the class $\wt{\Gamma}_p^m$ contains symbols of order $m$ and homogeneity $p$ in $u$, whereas the class $\Gamma_{K,K',p}^m$ contains non-homogeneous symbols of order $m$ that vanish at degree at least $p$ in $u$ and that are $(K-K')$-times differentiable in $t$. We can think the parameter $K'$ like the number of time derivatives of $u$ that are contained in the symbols.
We denote $ H_0^{\infty}(\mathbb{T};\mathbb{C})
\defeq \bigcap_{s \in \R} H_0^{s}(\mathbb{T};\mathbb{C})$.

\begin{definition}[Symbols]\label{def:symbols}
Let $m\in \R$, $p,N\in \N_0 $,
$ K, K' \in \N_0 $ with $ K' \leq K  $, and $ \epsilon_0>0$.
\begin{enumerate}[i)]

\item $p$-{\bf homogeneous symbols.} We denote by $\wt{\Gamma}^m_p$ the space of symmetric $p$-linear maps from $ \pare{H_0^{\infty}\pare{\mathbb{T};\mathbb{C}}}^p$ to the space of $ \cC^\infty $ functions from $\mathbb{T}\times \R$ to $\mathbb{C}$,
$ (x, \xi) \mapsto a(\cU;x,\xi)$,  satisfying the following: there exist $\mu \geq 0$ and, for any $\gamma, \beta\in \N_0$,
there is a constant $C>0$ such that
\begin{equation}\label{homosymbo}
\av{\partial_x^{\gamma}\partial_{\xi}^{\beta}
a\pare{  \Pi_{\vec n} \cU;x,\xi }}
\leq
C |\vec{n}|^{\mu+\gamma} \langle \xi \rangle^{m-\beta}
\prod_{j=1}^p  \norm{\Pi_{n_j} u_j }_{L^2}
\end{equation}
for any $ \cU = (u_1,\dots,u_p)\in \pare{ H_0^{\infty}\pare{\mathbb{T};\mathbb{C}} }^p$ and $\vec{n}=(n_1,\dots,n_p)\in \mathbb{N}^p$.
Moreover we assume that, if for some $(n_0,\dots, n_p)\in \N_0\times \N^p$, $\Pi_{n_0}a\left( \Pi_{\vec n} \cU ;\cdot\right)\not=0$, then there exists a choice of signs $ \eta_j 
\in \{ \pm 1 \} $ such that $\sum_{j=1}^p  \eta_j n_j=n_0 $.  In addition we require the translation invariance property
\begin{equation} \label{mome}
a\left( \st_{\varsigma} \cU; x,\xi\right)= a\left( \cU; x+\varsigma, \xi\right),\quad \forall
\varsigma\in \R \, ,
\end{equation}
where $\st_\varsigma$ is the translation operator in \eqref{X.tra0}.

For $ p = 0 $ we denote by $\wt{\Gamma}^m_0 $ the space of constant coefficients symbols $ \xi \mapsto a(\xi) $ which satisfy \eqref{homosymbo} with $ \gamma = 0 $ and the right hand side replaced by $ C \la \xi \ra^{m - \beta} $ and we call them Fourier multipliers.

\item  {\bf Non-homogeneous symbols. }   We denote by $\Gamma_{K,K',p}^m[\epsilon_0]$ the space of functions  $ a(u;t, x,\xi) $,
defined for $ u \in B_{s_0}^{K'}(I;\epsilon_0)$ for some $s_0$ large enough, with complex values, such that for any $0\leq k\leq K-K'$, any $s\geq s_0$, there are $C>0$, $0<\epsilon_0(s)<\epsilon_0$ and for any $ u \in B_{s_0}^K\pare{ I;\epsilon_0(s) }\cap C_{*}^{k+K'}\pare{I, H_0^{s}\pare{\mathbb{T};\mathbb{C}}}$ and any $\gamma,\beta \in \N_0$, with $\gamma \leq s-s_0$ one has the estimate
\begin{equation}\label{nonhomosymbo}
\av{ \partial_t^k\partial_x^\gamma \partial_\xi^\beta a\pare{ u;t, x,\xi }}  \leq C \langle \xi \rangle^{m-\beta} \| u \|_{k+K',s_0}^{p-1}\|u\|_{k+K',s} \, .
\end{equation}
If $ p = 0 $ the right hand side has to be replaced by $ C \langle \xi \rangle^{m-\beta} $. We say that a non-homogeneous symbol  $a(u;x,\xi) $ is \emph{real} if it is real valued for any
$ u \in B^{K'}_{s_0,\R}(I;\epsilon_0)$.

\item
{\bf Symbols.} We denote by $\Sigma \Gamma_{K,K',p}^m[\epsilon_0,N]$ the space of 
symbols 
$$
a(u;t, x,\xi)= \sum_{q=p}^{N} a_q\pare{ u, \ldots, u;x,\xi } + a_{>N}(u;t, x,\xi) 
$$
where $a_q $, $q=p,\dots, N$
are homogeneous symbols in $ \wt{\Gamma}_q^m $ and  
$a_{>N} $ is 
a non-homogeneous symbol in $ \Gamma_{K,K',N+1}^m $. 

We say that a symbol  $a(u;t, x,\xi) $ is \emph{real} if it is real valued for any
$ u \in B^{K'}_{s_0,\R}(I;\epsilon_0)$.
\end{enumerate}
\end{definition}

\begin{notation}\label{notation:multilinear_polyniomials}
If $ a ( \cU; \cdot  )$  is a $ p $-homogenous symbol we also denote $ a (u) \defeq
a(u, \ldots, u; \cdot ) $ the corresponding polynomial and we identify the $ p $-homogeneous monomial $a (u;\cdot) $ with the $ p $-linear symmetric form $  a ( \cU; \cdot  ) $.
\end{notation}

Actually also the non-homogeneous component of the symbols that we will 
encounter in Section \ref{sec:paralinearization}  depends
on time and space only through $ u $, but since this information 
is not needed it is not included in Definition \ref{def:symbols} (as in \cite{BD2018}).

\begin{rem}\label{rem:simb1}
 If $ a ( \cU; \cdot  )$ is a homogeneous
symbol in $ \widetilde \Gamma_p^m $ then
$ a (u, \ldots, u; \cdot ) $ belongs to  $\Gamma^m_{K,0,p} [\epsilon_0] $, for any $ \epsilon_0 > 0 $.
\end{rem}

\begin{rem}
 If $a $ is a symbol in $ \Sigma \Gamma^m_{K,K',p}[\epsilon_0,N] $
then $ \partial_x a  \in \Sigma \Gamma^{m}_{K,K',p}[\epsilon_0,N]   $ and
$ \partial_\xi a \in  \Sigma \Gamma^{m-1}_{K,K',p}[\epsilon_0,N]   $.
If in addition $ b $ is a symbol in $ \Sigma \Gamma^{m'}_{K,K',p'}[\epsilon_0,N]  $ then
$a b \in \Sigma \Gamma^{m+m'}_{K,K',p+p'}[\epsilon_0,N]  $.
\end{rem}

\begin{remark}[ Fourier representation of symbols]\label{rem:symbol}
 The translation invariance property \eqref{mome}
 means that the dependence with respect to the variable $x$ of a symbol
$a(\cU;x,\xi)$  enters only through the functions $\cU(x)$,
implying that a symbol
$ a_q(u;x,\xi)$ in  $\wt{\Gamma}_q^m$, $ m\in \mathbb{R} $, has the form 
\begin{equation}\label{sviFou}
a_q(u;x,\xi)=  \!\!\!\!
\sum_{\vec \jmath_q \in {(\Z \setminus
 \{ 0 \}})^{q}}  \! \! \! \! \!\!\!\!\!  \left( a_q\right)_{\vec \jmath_q}(\xi)
u_{j_1} \cdots u_{j_q}
 e^{\im \pare{ j_1+ \cdots + j_q } x}
\end{equation}
where $ (a_q)_{\vec \jmath}(\xi) \in \C $ are  Fourier multipliers of order $m$ satisfying: there exists $ \mu \geq 0  $, and
for any $ \beta \in \N_0 $, there is $ C_\beta > 0 $ such that
\begin{equation}\label{eq:fourier_char_homsymbols}
  \av{ \pa_\xi^\beta\left( a_q\right)_{\vec \jmath_q}(\xi) } \leq C_\beta
\av{   \vec \jmath_q} ^\mu \angles{ \xi }^{m-\beta} ,
\quad
\forall  \vec \jmath_q \in (\Z \setminus \{0\})^q   \, .
\end{equation}
A symbol
$ a_q(u;x,\xi) $ as in \eqref{sviFou} is  real if
\begin{equation}\label{realsim} 
\overline{\left( a_q\right)_{\vec \jmath_q}(\xi)} =
\left( a_q\right)_{- \vec \jmath_q}(\xi) 
\end{equation} 
By \eqref{sviFou}
a symbol
$ a_{1} $ in  $\widetilde{\Gamma}_{1}^{m}$
can be written as
$ a_{1}(u;x,\x)=
\sum_{ \substack{j \in  \Z \setminus \{0\} } }
(a_{1})_{j}(\x)
u_{j} e^{\ii  j x} $, and therefore,
if  $ a_1 $ is \text{independent} of $x$, it  is actually $ a_1\equiv0$.
\end{remark}

We also define classes of functions in analogy with our classes of symbols.

\begin{definition}[Functions] \label{def:functions}
Let $p, N \in \N_0 $,
 $K,K'\in \N_0$ with $K'\leq K$, $\epsilon_0>0$.
We denote by $\widetilde{\mathcal{F}}_{p}$, resp. $\mathcal{F}_{K,K',p}[\epsilon_0]$,
 $\Sigma\mathcal{F}_{K,K',p}[\epsilon_0,N]$,
the subspace of $\widetilde{\Gamma}^{0}_{p}$, resp. $\Gamma^0_{K,K',p}[\epsilon_0]$,
resp. $\Sigma\Gamma^{0}_{K,K',p}[\epsilon_0,N]$,
made of those symbols which are independent of $\xi $.
We write $\widetilde{\mathcal{F}}^{\R}_{p}$,   resp. $\mathcal{F}_{K,K',p}^{\R}[\epsilon_0]$,
$\Sigma\mathcal{F}_{K,K',p}^{\R}[\epsilon_0,N]$,  to denote functions in $\widetilde{\mathcal{F}}_{p}$,
resp. $\mathcal{F}_{K,K',p}[\epsilon_0]$,    $\Sigma\mathcal{F}_{K,K',p}[\epsilon_0,N]$,
which are real valued for any $ u \in B^{K'}_{s_0,\R}(I;\epsilon_0)$.
\end{definition}

The above class of symbols is closed under composition by a change of variables,
see  \cite[Lemma 3.23]{BD2018}.
\begin{lemma}\label{lem:closure_comp_symbols}
Let $K'\leq K\in \N$, $m\in \R$, $p\in\N_0$, $N\in \N$ with $p\leq N$, $\epsilon_0>0$ small enough. Consider
a symbol $ a $ in
$\sg{m}{K,K',p}{N}$ and functions $ b, c $ in $\sFR{K,K',1}{N}$. Then
$  a\bigl(v;t,x+b(v;t, x),\xi( 1 + c(v;t, x))\bigr) $
is in $\sg{m}{K,K',p}{N}$. In particular, if
$a$  is a function
in $\sF{K,K',p}{N}$, then  $a(v;t,x+b(v;t, x))$ is in $\sF{K,K',p}{N}$.
\end{lemma}

The following result is  \cite[Lemma 3.21]{BD2018}.
\begin{lemma}[Inverse diffeomorphism] \label{lem:LemA3}
Let $ 0 \leq K' \leq K $ be in $ \N $ and $  \beta (f; t,  x ) $ be a real function
$ \beta (f; t, \cdot )  $ in $ \Sigma {\cal F}^\R_{K,K',1}[\epsilon_0,N] $.
If $s_0$ is large enough, and $ f \in B^K_{s_0}\pare{I;\epsilon_0} $ then the
map
$  \Phi_f : x\to x+ \beta (f;t,x) $
is, for $\epsilon_0$ small enough, a diffeomorphism of  $ \, \Tu$, and
its  inverse diffeomorphism may be written as
$ \Phi_f^{-1} : y\to y + \breve{\beta} (f;t,y) $
for some $ \breve{\beta}  $ in $\sFR{K,K',1}{N}$.
\end{lemma}

\paragraph{Paradifferential quantization.}
Given $p\in \N $ we consider  {\it admissible cut-off} \, functions
  $\psi_{p}\in C^{\infty}(\R^{p}\times \R;\R)$ and $\psi\in C^{\infty}(\R\times\R;\R)$,
  even with respect to each of their arguments, satisfying, for $0<\delta\ll 1$,
\begin{align}
&{\rm{supp}}\, \psi_{p} \subset\set{(\xi',\xi)\in\R^{p}\times\R; |\xi'|\leq\delta \langle\xi\rangle } \, ,\qquad \psi_p (\xi',\xi)\equiv 1\,\,\, \rm{ for } \,\,\, |\xi'|\leq \delta \langle\xi\rangle / 2 \, ,
\label{admcutoff1} 
\\
&\rm{supp}\, \psi \subset\set{(\xi',\xi)\in\R\times\R; |\xi'|\leq\delta \langle\xi\rangle } \, ,\qquad \quad
 \psi(\xi',\xi) \equiv 1\,\,\, \rm{ for } \,\,\, |\xi'|\leq \delta   \langle\xi\rangle / 2 \, . \label{admcutoff2}
\end{align}
For $p=0$ we set $\psi_0\equiv1$.
We assume moreover that
\begin{equation}\label{admispapb}
\av{\partial_{\xi}^{\gamma}\partial_{\xi'}^{\beta}\psi_p(\xi',\xi)}\leq C_{\gamma,\beta}\langle\xi\rangle^{-\gamma-|\beta|} \, , \  \forall \gamma \in \N_0, \,\beta\in\N_0^{p} \, ,
\ \
\av{\partial_{\xi}^{\gamma}\partial_{\xi'}^{\beta}\psi(\xi',\xi)} \leq C_{\gamma,\beta}\langle\xi\rangle^{-\gamma-\beta}, \  \forall \gamma, \,\beta\in\N_0 \, .
\end{equation}
If $ a (x, \xi) $ is a smooth symbol
we define its Weyl quantization  as the operator
acting on a
$ 2 \pi $-periodic function
$u(x)$ (written as in \eqref{Fourierser})
 as
\begin{equation}\label{Opweil}
{\rm Op}^{W}\bra{a}u=\sum_{k\in \Z}
\pare{ \sum_{j\in\Z}\hat{a} \pare{ k-j, \frac{k+j}{2} } \hat{u}\pare{ j }  }
e^{\im k x}
\end{equation}
where $ \hat{a}(k,\xi) $ is the $ k$-Fourier coefficient of the $2\pi-$periodic function $x\mapsto a(x,\xi)$.

\begin{definition}[Bony-Weyl quantization]
\label{quantizationtotale}
If $ a $ is a symbol in $\widetilde{\Gamma}^{m}_{p}$,
respectively in $\Gamma^{m}_{K,K',p}[\epsilon_0]$,
we set
$$
\begin{aligned}
& a_{\psi_{p}}(\mathcal{U};x,\xi) \defeq \sum_{\vec{n}\in \N^{p}}\psi_{p}\left(\vec{n},\xi \right) \ a \pare{ \Pi_{\vec{n}}\mathcal{U};x,\xi }  \, , \\
& a_{\psi}(u;t, x,\xi) \defeq\frac{1}{2\pi}\int_{\mathbb{R}}
\psi (\xi',\xi )\hat{a} \pare{ u;t,\xi',\xi } e^{\im \xi' x}\di \xi'  \, ,
\end{aligned}
$$
where  $  \hat a $ stands for the Fourier transform with respect to the $ x $ variable, and
we define the \emph{Bony-Weyl} quantization of $ a $ as
\begin{equation}\label{BW}
\OpBW{a(\mathcal{U};\cdot)} = {\rm Op}^{W} \bra{ a_{\psi_{p}} \pare{\mathcal{U};\cdot} } \, ,\qquad
\OpBW{a(u;t,\cdot)} = {\rm Op}^{W} \bra{a_{\psi}\pare{u;t,\cdot} } \, .
\end{equation}
If  $a$ is a symbol in  $\Sigma\Gamma^{m}_{K,K',p}[\epsilon_0,N]$,
we define its \emph{Bony-Weyl} quantization
$$
\OpBW{a(u;t,\cdot)} =\sum_{q=p}^{N}
\OpBW{a_q(u,\ldots,u;\cdot)} + \OpBW{a_{>N}(u;t,\cdot) }  \, .
$$
\end{definition}

\begin{rem} \label{rem:OpBW_firstproperties}

$ \bullet $
 The operator
$ \OpBW{a} $
maps functions with zero average in functions with zero average, 
and $ \Pi_0^\bot \OpBW{a} = \OpBW{a}\Pi_0^\bot $. 

$ \bullet $
If $ a$ is a homogeneous  symbol, 
the two definitions  of quantization in \eqref{BW} differ 
by a  smoothing operator according to
Definition \ref{def:smoothing} below. 

$ \bullet $
Definition \ref{quantizationtotale}
is  independent of the cut-off functions $\psi_{p}$, $\psi$,
up to smoothing operators (Definition \ref{def:smoothing}).

$ \bullet $
The action of
$ \OpBW{a} $ on  the spaces $ H^s_0 $ only depends
on the values of the symbol $  a(u;t, x,\xi)$
for $|\xi|\geq 1$.
Therefore, we may identify two symbols $ a(u;t, x,\xi)$ and
$ b(u;t, x,\xi)$ if they agree for $|\xi| \geq 1/2$.
In particular, whenever we encounter a symbol that is not smooth at $\xi=0 $,
such as, for example, $a = g(x)|\x|^{m}$ for $m\in \R\setminus\{0\}$, or $ \sign (\xi) $,
we will consider its smoothed out version
$\chi(\xi)a$, where
$\chi\in  C^{\infty}(\R;\R)$ is an even and positive cut-off function satisfying
\begin{equation}\label{eq:chi}
\chi(\x) =  0 \;\; {\rm if}\;\; |\x|\leq \tfrac{1}{8}\, , \quad
\chi (\x) = 1 \;\; {\rm if}\;\; |\x|>\tfrac{1}{4} \, ,
\quad  \pa_{\x}\chi(\x)>0\quad\forall  \x\in \big(\tfrac{1}{8},\tfrac{1}{4} \big) \, .
\end{equation}
\end{rem}

\begin{rem}
Given  a paradifferential  operator
$ A = \OpBW{a(x,\xi)} $ it results
\begin{equation}\label{A1b}
\ov{ A} = \OpBW{\overline{a(x, - \xi)}} \, , \quad
A^\intercal = \OpBW{a(x, - \xi)} \, , \quad
A^*= \OpBW{\overline{a(x,  \xi)}} \, ,
\end{equation}
where $ A^\intercal $  is the transposed  operator with respect to the real scalar product
$ \la u, v \ra_r = \int_\T  u(x)\,  {  v(x)} \, \di x $, and
$ A^* $ denotes the adjoint  operator  with respect to the complex
scalar product of $  L^2_0 $ in \eqref{scpr12hom}. It results $ A^* = \ov{A}^\intercal $.

$ \bullet $
A paradifferential operator $A= \OpBW{a(x,\xi)} $ is {\it real} (i.e. $A = \ov{A} $) if
\begin{equation}\label{areal} 
\ov{a(x,\xi)}= a(x,-\xi)   \, . 
\end{equation}
It is {\it symmetric} (i.e. $A = A^\intercal $)
 if $  a(x,\xi) = a(x,-\xi) $.
A operator $ \pa_x \OpBW{a(x,\xi)} $ is Hamiltonian if and only if
 \begin{equation}\label{Hamassy}
 a(x, \xi ) \in \R  \qquad \text{and} \qquad
 a(x, \xi ) = a(x, - \xi ) \quad \text{is \ even \ in \ } \xi \, .
 \end{equation}
\end{rem}

We now provide the action of a paradifferential operator on Sobolev spaces, cf.  \cite[Prop. 3.8]{BD2018}.

\begin{lemma}[Action of a paradifferential operator]
  \label{prop:action}
    Let $ m \in \R $.
  
  \begin{enumerate}[i)]
  \item \label{item:OpBWmaps1}
If $ p \in \N $, there is $ s_0 > 0 $ such that for any symbol $ a $ in $\Gt{m}{p}$,
there is  a constant $ C > 0 $, depending only on $s$ and on \eqref{homosymbo}
with $ \gamma = \beta = 0 $,
such that, for any $ (u_1, \ldots, u_p ) $, for $ p \geq 1  $,
\begin{equation*}
  \norm{\OpBW{a(u_1, \ldots, u_p;\cdot)} u_{p+1}}_{H^{s-m}_0}\leq C
% \prod_{j=1}^p 
\norm{ u_1 }_{H^{s_0}_0}  \cdots \norm{ u_p }_{H^{s_0}_0}
\norm{u_{p+1}}_{H^{s}_0} \, .
\end{equation*}
If  $ p = 0 $
the above bound holds replacing the right hand side with
$ C \norm{u_{p+1}}_{H^{s}_0} $.
\item \label{item:OpBWmaps2}
Let $\epsilon_0>0$, $p\in \N $, $K'\leq K\in \N $, $a$ in $\Gr{m}{K,K', p}$.
There is $ s_0 > 0 $,
and  a constant $ C $, depending only on $s$, $\epsilon_0 $, and on \eqref{nonhomosymbo} with $ 0 \leq \gamma  \leq 2, \beta = 0$,
such that, for any $ t $ in $ I $, any $ 0\leq k\leq K-K'$, any $ u $ in $\Br{K}{}$,
$$
  \norm{\OpBW{ \partial_t^ka(u;t,\cdot) } }_{\Lcal(H^{s}_0,H^{s-m}_0)}\leq C
 \| u(t, \cdot)\|_{k+K',s_0}^p \, ,
$$
so that
$\| \OpBW{a(u;t,\cdot)} v(t) \|_{K-K', s-m} \leq C  \| u(t, \cdot)\|_{K,s_0}^p
\| v(t) \|_{K- K', s} $.

  \end{enumerate}
\end{lemma}

\paragraph{Classes of $m$-Operators and smoothing Operators.}
Given integers $(n_1,\ldots,n_{p+1})\in \N^{p+1}$, we denote by $\max_{2}(n_1 ,\ldots, n_{p+1})$
the second largest among  $ n_1,\ldots, n_{p+1}$.

We now define  $ m $-operators.
The class $\widetilde{\mathcal{M}}^{m}_{p}$ denotes multilinear
 operators that lose $m$ derivatives
 and are $p$-homogeneous in $u $,
while the class $\mathcal{M}_{K,K',p}^{m}$ contains non-homogeneous
operators  which lose $m$ derivatives,
vanish at degree at least $ p $ in $ u $, satisfy tame estimates
 and are $(K-K')$-times differentiable in $ t $.
The  constant $ \mu $ in \eqref{eq:bound_fourier_representation_m_operators} takes into account possible loss of derivatives in
the ``low" frequencies. The following definition is a small 
adaptation of  \cite[Def. 2.5]{BMM2022} as it defines $ m$-operators acting 
on $ H^{\infty}(\T;\C) $
and not  $	\dot H^{\infty}(\T;\C^2)$ (and we state it directly in Fourier 
series representation).

\begin{definition}[ Classes of $m$-operators]\label{def:moperators}
Let  $ m \in \R $,  $p,N\in \N_0 $,
$K,K'\in\N_0$ with $K'\leq K$, and $ \epsilon_0 > 0 $.

\begin{enumerate}[i)]

\item \label{item:maps1}
{\bf $p$-homogeneous $m$-operators.}
We denote by $\widetilde{\mathcal{M}}^{m}_{p}$
 the space of $ (p+1)$-linear translation invariant operators from 
$ \pare{H^\infty \pare{\bT;\bC}}^{p} \times H^\infty \pare{\bT;\bC} $ to 
$ H^\infty \pare{\bT;\bC} $, symmetric in $ (u_1, \ldots, u_p ) $,  
with Fourier expansion 
\begin{equation} \label{smoocara0}
M(u)v \defeq M\pare{u, \ldots, u} v = \sum_{\substack{(j_1, \ldots, j_{p},j,k) \in \Z^{p+2} \\  
j_1 + \ldots + j_p + j = k} } M_{j_1, \ldots, j_{p}, j,k} \ 
u_{j_1} \ldots u_{j_p} v_{j}  e^{\ii  k x}  \, ,
 \end{equation}
with coefficients $ M_{j_1, \ldots, j_{p}, j,k} $ symmetric in $ j_1, \ldots, j_{p} $, 
 satisfying the following:  there are $\mu \geq0$, $C>0$ such that, for any 
 $  j_1, \ldots, j_{p}, j, k \in \Z^{p+2} $, it results 
 \begin{equation}\label{eq:bound_fourier_representation_m_operators}
\av{ M_{ j_1, \ldots, j_{p}, j, k} } \leq C \   {\rm max} _2 \set{\angles{j_1}, \ldots,  \angles{j_{p}}, \angles{ j} } ^\mu \   \max\set{ \angles{j_1}, \ldots , \angles{j_{p}},\angles{ j} }^m  \, , 
\end{equation}
and the reality condition holds: 
\begin{equation}
\label{eq:reality_cond_reminder}
\overline{ M_{\vec \jmath_{p}, j,k}  } = M_{- \vec \jmath_{p}, -j, -k}  \, ,
\qquad \forall
 \vec \jmath_{p} = \pare{ j_1, \ldots, j_{p} } \in \bZ^p, \pare{ j,k }  \in \bZ^{2} \, .
\end{equation}
If $ p=0 $ the right hand side of \eqref{smoocara0} must be substituted with $ \sum _{j\in \bZ} M_j v_j e^{\ii j x} $ with $ \av{M_j}\leq C \angles{ j}^m $. 
 \item
 {\bf Non-homogeneous $m$-operators.}
  We denote by  $\mathcal{M}^{m}_{K,K',p}[\epsilon_0]$
  the space of operators $(u,t,v)\mapsto M(u;t) v $ defined on { $B_{C^{K'}_{*}(I,H^{s_0}(\T;\C))}\pare{0;\epsilon_0}\times I \times C^0_{*}(I,H^{s_0}(\T;\C))$} for some $ s_0 >0  $,
  which are linear in the variable $ v $ and such that the following holds true.
  For any $s\geq s_0$ there are $C>0$ and
  $\epsilon_0(s)\in]0,\epsilon_0[$ such that for any
  { $u \in B_{C^{K'}_{*}(I,H^{s_0}(\T;\C))}\pare{0;\epsilon_0} \cap C^K_{*}(I,H^{s}(\T;\C))$,
  any $ v \in C^{K-K'}_{*}(I,H^{s}(\T;\C))$}, any $0\leq k\leq K-K'$, $t\in I$, we have that { 
\begin{equation}
\label{piove}
\norm{{\partial_t^k\left(M(u;t)v\right)}}_{s- \alpha k-m}
 \leq C   \sum_{k'+k''=k}  \pare{ \|{v}\|_{k'',s}\|{u}\|_{k'+K',{s_0}}^{p}
 +\|{v}\|_{k'',{s_0}}\|u\|_{k'+K',{s_0}}^{p-1}\| u \|_{k'+K',s} } \, .
\end{equation}}
In case $ p = 0$ we require  the estimate
$ \|{\partial_t^k\left(M(u;t) v \right)}\|_{s- \alpha k-m}
 \leq C  \| v \|_{k,s}$.  
  We say that a non-homogeneous $ m $-operator  $M\pare{u;t} $ is \emph{real} if it is real valued for any
$ u \in B_{C^{K'}_{*}(I,H^{s_0}(\T;\R))}\pare{0;\epsilon_0} $.

 \item
 {\bf $m$-Operators.}
We denote by $\Sigma\mathcal{M}^{m}_{K,K',p}[\epsilon_0,N]$
the space of operators 
\begin{equation}
\label{maps}
M(u;t)v =\sum_{q=p}^{N}M_{q}(u,\ldots,u)v+M_{>N}(u;t)v 
\end{equation}
where $M_{q} $ are homogeneous $m$-operators in $ \widetilde{\mathcal{M}}^{m}_{q}$, $q=p,\ldots, N$  and
$M_{>N}$  is a non--homogeneous $m$-operator
in $\mathcal{M}^{m}_{K,K',N+1}[\epsilon_0]$.  We say that a  $ m $-operator  $M\pare{u;t} $ is \emph{real} if it is real valued for any
$ u \in B_{C^{K'}_{*}(I,H^{s_0}(\T;\R))}\pare{0;\epsilon_0}$.

\item
{\bf Pluri-homogeneous $m$-Operator. }
We denote 
by $ \Sigma_p^N \wt {\cM}^{m}_q $ the 
pluri-homogeneous $m$-operators  of the form \eqref{maps}
with $ M_{>N} = 0 $.
\end{enumerate}

We denote with $ \dot{\wt \cM}^m_p, \ \dot \cM^m_{K, K', p}\bra{\epsilon_0} $ and $ \Sigma \dot \cM^m_{K, K', p}\bra{\epsilon_0,N}  $ the subspaces of $m$-operators 
in $ {\wt \cM}^m_p $, respectively  $ \cM^m_{K, K', p}\bra{\epsilon_0} $ and $ \Sigma  \cM^m_{K, K', p}\bra{\epsilon_0,N}  $, 
 defined on zero-average functions 
 taking value $ M(u)v $  in zero-average functions. 
\end{definition}

\begin{rem}\label{rem:first_properties_moperators}
 By  \cite[Lemma 2.8]{BMM2022},
if $ M( u_1,\dots , u_p)$ is a $p$--homogeneous
 $m$-operator in $  \widetilde \mM_p ^m$ then
  $ M(u) = M(u, \ldots, u) $ is a  non-homogeneous $ m$-operator in
 $  {\cal M}^m_{K,0,p}[\epsilon_0] $ for any $\epsilon_0>0$ and $K\in \N_0$.
 We shall say that $ M(u) $ is in $  \widetilde \mM_p ^m $. 
  \end{rem}

\begin{rem}\label{item:propMop2} The multiplication operator  $ v \mapsto \frac{1}{1+2f} v $
belongs to $ \Sigma \cM^0_{K, 0, 0}\bra{\epsilon_0, N} $. 
 \end{rem}

If $m  \leq 0 $ the  operators in $ \Sigma \mM^{m}_{K,K',p}[\epsilon_0,N]$ are referred to as smoothing operators.

 \begin{definition}[Smoothing operators] \label{def:smoothing}
Let $ \rho\geq0$. A $ (-\rho)$-operator $R(u)$ belonging to $ \Sigma \mM^{-\rho}_{K,K',p}[\epsilon_0,N]$ 
is called  a smoothing operator. 
We also denote
\begin{align*}
 \widetilde{\mathcal{R}}^{-\rho}_{p}\defeq \widetilde{\mathcal{M}}^{-\rho}_{p} \, ,
&& 
 \mathcal{R}^{-\rho}_{K,K',p}[\epsilon_0]\defeq\mathcal{M}^{-\rho}_{K,K',p}[\epsilon_0] \, , 
 && 
  \Sigma\mathcal{R}^{-\rho}_{K,K',p}[\epsilon_0,N]\defeq\Sigma\mathcal{M}^{-\rho}_{K,K',p}[\epsilon_0,N] \, .
\end{align*}
We define $ \dot{\wt \cR}^{-\rho}_p = \dot{\wt \cM}^{-\rho}_p $, 
$ \dot \cR^{-\rho}_{K, K', p}\bra{\epsilon_0}  = 
\dot \cM^{-\rho}_{K, K', p}\bra{\epsilon_0} $ and 
$ \Sigma \dot \cR^{-\rho}_{K, K', p}\bra{\epsilon_0,N}  = \Sigma \dot \cM^{-\rho}_{K, K', p}\bra{\epsilon_0,N} 
$ as in \Cref{def:moperators}. 
\end{definition}

If $ R (u) $ is a homogenous smoothing operator in $ {\wt \cR}^{\, -\rho}_p $
then $ \Pi_0^\bot R(u) $, where $ \Pi_0^\bot $ is defined 
\eqref{eq:Pi0bot}, restricted to zero average functions 
$ u $, belongs  to $  \dot{\wt \cR}^{-\rho}_p $.

\begin{rem} \label{rem:smoo}
$ \bullet $
Lemma \ref{prop:action} implies that,  if  $a(u; t, \cdot)$ is  in $\Sigma\Gamma^{m}_{K,K',p}\bra{\epsilon_0 , N}$, 
 $m\in\bR $,  then $ \OpBW{a(u; t, \cdot)} $ defines a $ m $-operator in
$\Sigma  \cM^{m}_{K,K',p}\bra{\epsilon_0 , N}$.

$ \bullet $
The composition of smoothing operators $ R_1 \in \sr{-\rho}{K,K',p_1}{N}$
 and  $ R_2  \in \sr{-\rho}{K,K',p_2}{N} $ is a smoothing operator $ R_1 R_2 $  in 
$ \sr{-\rho}{K,K',p_1+p_2}{N} $.  This is a particular case of Proposition \ref{compositionMoperator}-($i)$.
\end{rem}

\begin{lemma}
\label{lem:MtoFunctions}
Let $ m\in \bR$, $ \epsilon_0>0 $, $ K, K'\in\bN_0 $, $ K'\leq K $, $ N, p\in \bN_0  $, 
$ u \in\Ball{K}{s} $
and  $ M(u;t) $ be a real operator in $ \Sigma \cM^m_{K, K', p}\bra{\epsilon_0, N} $. Then 
$M(u;t) u  $ is a real function in $ \Sigma \cF^\R_{K, K', p+1} \bra{\epsilon_0, N+1} $ according to \Cref{def:functions}. 
\end{lemma}

\begin{proof}
We decompose $ M (u;t)= \sum_{q=p}^N M_q (u)+ M_{>N} (u; t)  $
in the usual homogeneous and non-homogeneous components. 
We assume $ u $ is in $ \Ball{K}{s} $ so that $ u $ has zero average. 
We now prove that $  M_q (u) u $ is a function in $ \wt\cF^\bR_{q+1}  $. 
For any  zero average function $ u $, according
 to \eqref{smoocara0} we have
\begin{equation*}
(M_q (u) u) \pare{x} = \sum_{\substack{(j_1, \ldots, j_{q},j) \in \pare{ \Z\setminus \set{0} }^{q+1}  \\  
j_1 + \ldots + j_p + j = k} } M_{j_1, \ldots, j_{p}, j,k} \ 
u_{j_1} \ldots u_{j_q} u_{j} \  e^{\ii  \pare{j_1+\ldots + j_q + j} x}  .
\end{equation*}
Moreover, by \eqref{eq:bound_fourier_representation_m_operators}, for any 
$ (j_1, \ldots, j_{p}, j) = \pare{\vec{\jmath}_q, j}\in \pare{\bZ\setminus \{ 0\}}^{q+1} $, we have 
\begin{align*}
\av{M_{j_1, \ldots, j_{p}, j,k}} \lesssim & \  {\rm max}_2 \set{\angles{j_1}, \ldots, \angles{j_q}, \angles{j}}^\mu  {\rm max} \set{\angles{j_1}, \ldots, \angles{j_q}, \angles{j}}^m
\\
\lesssim & \  {\rm max} \set{\angles{j_1}, \ldots, \angles{j_q}, \angles{j}}^{2\max\set{\mu,m}} 
\lesssim  \  \av{\pare{\vec{\jmath}_q, j}}^{2\max\set{\mu,m}} \, , 
\end{align*}
and, in view of \eqref{eq:fourier_char_homsymbols}, 
we thus obtain that $ M_q (u)u $ is a function in $ \wt\cF_{q+1} $. In view of  \eqref{eq:reality_cond_reminder}  the function $ M_q (u) u $ is real.

We now prove that $  \pare{M_{>N}\pare{u;t}u}(t,x) $ is a function in 
$  \cF^\bR_{K,K',N+2}\pare{\epsilon_0} $. 
Let $ s_0 \defeq 1 +\alpha \pare{ K -K' } + m $. 
For any $ 0\leq k \leq K-K' $, for any $ s\geq s_0 $,  and $ 0\leq \gamma \leq s-s_0 $ we have that 
$ s - \alpha k - m > \gamma +1 $, and 
\begin{equation*}
\begin{aligned}
 \av{\partial_t^k\partial_x^\gamma ( M_{> N}\pare{u;t} u) }
&  \lesssim \norm{\partial_t^k (M_{>N}\pare{u;t}u)}_{\gamma + 1} 
 \leq \norm{\partial_t^k (M_{>N}\pare{u;t}u)}_{s-\alpha k - m}
\stackrel{\eqref{piove}}\lesssim \norm{u}_{k+K', s_0}^{N+1} \norm{u}_{k+K', s} 
\end{aligned}
\end{equation*}
proving, in view of 
\Cref{def:symbols,def:functions}, 
that $ M_{>N}\pare{u;t}u $ is a function in 
$ \cF_{K,K',N+2}\pare{\epsilon_0} $. The reality condition is verified since $ M_{>N} $ is a real $ m $-operator per hypothesis. 
\end{proof}

\paragraph{Symbolic calculus.}
Let
$ \s(D_{x},D_{\x},D_{y},D_{\eta}) \defeq D_{\x}D_{y}-D_{x}D_{\eta}  $
where $D_{x}\defeq\frac{1}{\ii}\pa_{x}$ and $D_{\x},D_{y},D_{\eta}$ are similarly defined.
The following is   Definition 3.11 in \cite{BD2018}.

\begin{definition}[Asymptotic expansion of composition symbol]
\label{def:as.ex}
Let $ p $, $ p' $ in $\N_0 $, $ K, K' \in \N_0 $ with $K'\leq K$,  $ \rho  \geq 0 $, $m,m'\in \R$, $\epsilon_0>0$.
Consider symbols $a \in \Sigma\Gamma_{K,K',p}^{m}[\epsilon_0,N]$ and $b\in \Sigma \Gamma^{m'}_{K,K',p'}[\epsilon_0,N]$. For $u$ in $B_{\s}^{K}(I;\epsilon_0)$
we define, for $\rho< \s- s_0$, the symbol
\begin{equation}\label{espansione2}
(a\#_{\rho} b)\pare{ u;t, x,\x } \defeq\sum_{k=0}^{\rho}\frac{1}{k!}
\left(
\frac{\ii}{2}\s\pare{ D_{x},D_{\x},D_{y},D_{\eta} } \right)^{k}
\Big[a(u;t, x,\x)b(u;t,y,\eta)\Big]_{|_{\substack{x=y, \x=\eta}}}
\end{equation}
modulo symbols in $ \Sigma \Gamma^{m+m'-\rho}_{K,K',p+p'}[\epsilon_0,N] $.
\end{definition}

The symbol $ a\#_{\rho} b $ belongs  to $\Sigma\Gamma^{m+m'}_{K,K',p+p'}[\epsilon_0,N]$.
Moreover
\begin{equation} \label{asharpb}
a\#_{\rho}b = a b + \frac{1}{2 \ii }\{a,b\} 
\end{equation} 
up to a symbol in $\Sigma\Gamma^{m+m'-2}_{K,K',p+p'}[\epsilon_0,N]$,
where
$$
\{a,b\}  \defeq  \pa_{\xi}a \ \pa_{x}b -\pa_{x}a \ \pa_{\xi}b
$$
denotes the Poisson bracket.
\smallskip
The following result is proved in Proposition $3.12$ in \cite{BD2018}.

\begin{proposition}[Composition of Bony-Weyl operators] \label{prop:composition_BW}
Let $p,q,N, K, K'  \in \N_0 $ with $ K' \leq K $,  $\rho \geq 0 $, $m,m'\in \R$, $\epsilon_0>0$.
Consider  symbols
$a\in \Sigma {\Gamma}^{m}_{K,K',p}[\epsilon_0,N] $ and $b\in \Sigma {\Gamma}^{m'}_{K,K',q}[\epsilon_0, N]$.
Then
\begin{equation}\label{smoospec}
\OpBW{a(u;t, x,\x)}\circ\OpBW{b(u;t, x,\x)} - \OpBW{(a\#_{\rho} b)(u;t, x,\x)}
\end{equation}
is a smoothing operator in $ \Sigma \dot{\mathcal{R}}^{-\rho+m+m'}_{K,K',p+q}[\epsilon_0,N]$.
\end{proposition}

  We have the following result, see e.g. Lemma 7.2 in \cite{BD2018}.
\begin{lemma}[Bony paraproduct decomposition]
\label{lem:paraproduct_Weyl}
Let  $u_1, u_2$ be functions in $H^\s(\T;\C)$   with
$\s >\frac12$. Then
\begin{equation}\label{bonyeq}
u_1 u_2  =  \OpBW{u_1 }u_2 + \OpBW{u_2 }u_1 +
R_1(u_1)u_2 +  R_2(u_2)u_1
\end{equation}
where for $ \mathsf{j} =1, 2 $, $R_\mathsf{j}$ is a homogeneous smoothing operator in $ \widetilde \mR^{-\rho}_{1}$ for any $ \rho \geq 0$. 
\end{lemma}

We now state other composition results for  $m$-operators which follow as in 
\cite[Proposition 2.15]{BMM2022}.

\begin{proposition}[Compositions of $m$-operators] \label{compositionMoperator}
Let $p, p', N, K, K' \in \N_0$ with $K'\leq K$ and $\epsilon_0>0$. Let $m,m' \in \R$.
Then
\begin{enumerate}[i)]
\item
 If  $M(u;t) $ is  in
$ \Sigma\mathcal{M}^{m}_{K,K',p}[\epsilon_0,N]$ and $M'(u;t) $ is  in
$ \Sigma\mathcal{M}^{m'}_{K,K',p'}[\epsilon_0,N] $ then the composition
$ M(u;t)\circ M'(u;t) $
is  in $\Sigma\mathcal{M}^{m+\max(m',0)}_{K,K',p+p'}[\epsilon_0,N]$.

\item
If  $M (u) $ is a homogeneous $ m$-operator in $  \widetilde{\mathcal{M}}_{p}^{m}$
and $M^{(\ell)}(u;t)$, $\ell=1,\dots,p+1$, are matrices of  $ m_\ell $-operators  in
$ \Sigma \mM^{m_\ell}_{K,K',q_\ell}[\epsilon_0,N]$ with  $m_\ell \in \R$,
$q_\ell\in \N_0$,
then
$$
M \pare{ M^{(1)}(u;t)u, \ldots,  M^{(p)}(u;t)u }M^{(p+1)}(u;t)
$$
belongs to $ \Sigma\mM_{K,K',p+\bar q}^{m+ \bar m}[\epsilon_0,N]$ with $ \bar m\defeq\sum_{\ell=1}^{p+1} \max(m_\ell,0)$ and $\bar q\defeq \sum_{\ell=1}^{p+1}q_\ell$.

\item
If   $M(u;t) $ is in $ {\mathcal{M}}_{K,0,p}^{m}[\breve \epsilon_0]$ for any $\breve \epsilon_0\in \R^+$ and $ \bM_0(u;t) $ belongs to  $ \mM^0_{K,K',0}[\epsilon_0] $,
then $M(\bM_0(u;t)u;t)$ is in  $  \mM^{m}_{K,K',p}[\epsilon_0]$.

%\item
% Let $c$ be a homogeneous symbol in $\widetilde{\Gamma}_p^{m}$  and $M^{(\ell)} (u;t) $, $\ell=1,\dots, p $, be operators in
%$ \Sigma\mM_{K,K',q_\ell}[\epsilon_0,N] $ with $ q_\ell \in \N_0 $.
%Then
%$ b(u;t, x,\x)\defeq c(M^{(1)}(u;t)u,\ldots,M^{(p)}(u;t)u;x,\x) $
%is a symbol in $\Sigma\Gamma^{m}_{K,K',p+\bar q}[\epsilon_0,N]$ with $\bar q\defeq q_1+\dots+q_p$ and
%$$
%\left. \OpBW{ c\pare{ w_1,\ldots,w_p;t, x,\x }}\right|_{w_\ell=M^{(\ell)}(u;t)u}=
%\OpBW{b(u;t, x,\x)} +R(u;t)
%$$
%where $R(u; t) $ is a smoothing operator in
%$\Sigma \mathcal{R}^{-\rho}_{K,K', p + \bar q}[\epsilon_0,N] $
%for any $\rho\geq0$.
\item
Let $a$ be a symbol in $\sg{m}{K,K',p}{N}$ with $m\geq 0$ and $R$ a smoothing operator in $\sr{-\rho}{K,K',p'}{N}$. Then 
$$
\OpBW{a(u; t, \cdot)} \circ R(u;t) \, , \quad R(u;t) \circ \OpBW{a(u; t, \cdot)} \, , 
$$
are in $\sr{-\rho+m}{K,K',p+p'}{N}$.
\end{enumerate}
\end{proposition}

\begin{notation}
 In the sequel if $ K'= 0 $ we  denote a 
symbol $ a (u; t, x, \xi) $ in $\Gamma_{K,0,p}^m[\epsilon_0]$  simply as
$ a (u; x, \xi)  $,   
and a smoothing operator in $ R (u;t) $ 
in 
$  \Sigma\mathcal{R}^{-\rho}_{K,0,p}[\epsilon_0,N] $ simply as $ R(u)  $, 
without writing the $ t $-dependence.
\end{notation}

We finally provide the Bony paralinearization formula
of the composition operator. 

\begin{lemma}[Bony Paralinearization formula] % of the composition operator]
  \label{lem:Bony_paralinearization_W}
Let $F$ be a smooth $\C$-valued function defined on a neighborhood of zero in $\C $, vanishing at zero at order $q\in
\N $. Then there is $ \epsilon_0 >  0 $
 and a  smoothing operator  $ R\pare{u} $
in \, $\Sigma \cR^{-\rho}_{K,0,q'}\bra{\epsilon_0, N}$, $q' \defeq \max(q-1,1)$, for any $\rho$,  such that
\begin{align}
  \label{eq:2313}
    F(u) =  \  \OpBW{ F'\pare{u} } u + R\pare{u}u  \,      .
\end{align}
\end{lemma}

\begin{proof}
The  formula
follows by combination  of \cite[Lemmata 3.19 and 7.2]{BD2018}.
\end{proof}

\subsection{$ z$-dependent paradifferential calculus}  \label{sec:paradiff}

Along the 
paralinearization process of the $\alpha$-SQG equation 
in Section \ref{sec:paralinearization} we 
shall encounter parameter dependent paradifferential operators depending on a  $ 2\pi $-periodic variable $ z $.
The following  ``Kernel-functions"
have to be considered as Taylor remainders of maps of the form $ F\pare{u; x, z} $ at $ z=0 $
which are smooth  in $ u $ and 
with finite regularity in $ x $ and $ z $. We are interested in the behavior of such 
functions close to $ z=0 $.

\begin{definition}[Kernel functions]\label{def:kernel_functions}
Let $n\in \R$,  $p,N\in \N_0 $,  
$ K \in \N_0 $, and $ \epsilon_0>0 $.
\begin{enumerate}[i)]
\item $p$-{\bf homogeneous Kernel-functions.} If $ p \in \N $
we denote  $\wt{\KF}^n_p $ the space of 
$ z $-dependent, $ p $-homogeneous maps 
from $ H^\infty_0 \pare{\bT;\bC} $ to the space of 
$ x $-translation invariant real functions  
$ \varrho (u;x,z) $ of class $ \cC^\infty $ 
in $ (x, z ) \in \bT^2 $ with Fourier expansion 
\begin{equation}\label{homoresz}
\varrho (u;x,z)=  
\sum_{j_1, \ldots, j_p \in {\Z \setminus
 \{ 0 \}}}  \! \varrho_{j_1, \ldots, j_p} (z) 
u_{j_1} \cdots u_{j_p}
 e^{\im (j_1+ \cdots + j_p) x} \, , \quad z \in \T\setminus \{ 0 \}  \, , 
\end{equation} 
with coefficients $ \varrho_{j_1, \ldots, j_p} (z) $ of class  $ \cC^\infty\pare{\bT; \C }  $,  
symmetric in $ (j_1, \ldots, j_p ) $, 
satisfying the reality condition 
$ \overline{\varrho_{j_1, \ldots, j_p}}\pare{z} =
\varrho_{-j_1, \ldots, -j_p} \pare{z} $ 
and the 
following: for any $ l \in \N_0 $, there exist $ \mu > 0 $ and 
a constant $ C > 0 $ such that
\begin{equation}\label{homores1}
\av{\partial_{z}^{l} \varrho_{j_1, \ldots, j_p} (z) }
\leq C \av{\vec{\jmath}}^{\mu}\  \av{z}^{n-l}_{\bT}  \, , \quad
\forall  \vec \jmath = (j_1, \ldots, j_p) \in (\Z \setminus \{0\})^p   \, .
\end{equation}
For $ p = 0 $ we denote by $\wt{\KF}^n _0 $ the space of 
maps $ z \mapsto \varrho (z) $ which satisfy 
$ \av{\partial_{z}^{l} \varrho (z) } \leq C \,  \av{z}^{n-l}_{\bT} $.
\item  {\bf Non-homogeneous Kernel-functions.}   We denote by $\KF _{K,0,p}^{n}[\epsilon_0]$ the space of $ z $-dependent, 
 real functions  $ \varrho (u;x,z) $,
defined for $ u \in B_{s_0}^0 (I;\epsilon_0)  $ for some $s_0$ large enough, such that for any $0\leq k\leq K $ and $ l\leq  \max\set{0, \ceil{1+ n }\big. } $, any $s\geq s_0$, there are $C>0$, $0<\epsilon_0(s)<\epsilon_0$ and for any $ u \in B_{s_0}^K\pare{ I;\epsilon_0(s) }\cap C_{*}^{k}\pare{ I, H_0^{s} \pare{ \mathbb{T};\mathbb{C} } }$ and any $\gamma \in \N_0$, with $\gamma\leq s-s_0,  $ one has the estimate
\begin{equation}\label{homores1.1}
\av{ \partial_t^k\partial_x^\gamma\partial_z ^l \varrho \pare{ u;x, z } }  \leq C  \| u \|_{k,s_0}^{p-1}\|u\|_{k,s} \ \av{z} ^{n-l}_{\bT} \, , \qquad z\in\bT\setminus \{ 0 \}   \, .
\end{equation}
If $ p = 0 $ the right hand side in \eqref{homores1.1} 
has to be replaced by $   \av{z} ^{n-l}_{\bT}  $.

\item
{\bf Kernel-functions.} We denote by $\Sigma \KF _{K,0,p}^{n}[\epsilon_0,N]$ the space of real functions  
of the form 
\begin{equation}\label{espsymbol}
\varrho (u;x,z)= \sum_{q=p}^{N} 
\varrho _q\pare{ u;x,z } + \varrho _{>N}(u;x,z )  
\end{equation}
where $\varrho_q \pare{ u;x,z }  $, $q=p,\dots, N$ 
are homogeneous Kernel functions in $ \wt{\KF}_q^n $,  and 
$\varrho _{>N}(u;x,z )    $ 
is a non-homogeneous Kernel function in $ \KF _{K,0,N+1}^{n} [\epsilon_0] $.

A Kernel function  $\varrho(u;x,z) $ is \emph{real} if it is real valued for any
$ u \in  B_{s_0,\R}^0 (I;\epsilon_0) $.
\end{enumerate}
\end{definition}

In view of Remark \ref{rem:simb1}, 
 a homogeneous Kernel function $ \varrho ( u; x,z) $ in $ \wt{\KF}^n_p $ 
defines  a non-homogenous  Kernel function  in $ \KF _{K,0,p}^{n}[\epsilon_0]  $
for any $ \epsilon_0 >0 $. 

\begin{rem}\label{rem:firstremaKernelfunctions}
Let   $ \vr \pare{u;x, z} $ be 
a Kernel function in $ \Sigma \cF^{n}_{K, 0, p}\bra{\epsilon_0, N} $ with $ n\geq 0 $, 
which admits a continuous  extension in $ z=0 $. Then its trace 
 $  \vr \pare{u;x, 0} $ is a function in $ \Sigma \cF^{\bR}_{K, 0, p}\bra{\epsilon_0, N} $. 
\end{rem}

\begin{rem} \label{rem:2quantiz} 
If $ \varrho (u; x, z)$ is a homogeneous   Kernel function  $ \wt{\KF}^n_p $, 
the two definitions  of quantization in \eqref{BW} differ 
by a  Kernel smoothing operator in $ \wt{\KR}^{-\rho,n}_p $, for any $ \rho > 0 $, 
according to
Definition \ref{def:KR} below. 
\end{rem}

\begin{rem} \label{rem:operation_z}  
If $ \vr_1 (u;x,z) $ is a  Kernel function in 
$ \Sigma \KF ^{n_1}_{K, 0, p_1}\bra{\epsilon_0, N} $ 
and $ \vr_2 (u;x,z) $  in
$ \Sigma \KF ^{n_2}_{K, 0, p_2}\bra{\epsilon_0, N} $, 
then the sum 
$ (\vr_1 + \vr_2 )(u;x,z)$ is a  Kernel function in $ \Sigma \KF^{\min\set{n_1, n_2}}_{K, 0, \min\set{p_1, p_2}}\bra{\epsilon_0, N} $ and the product
$ ( \vr_1  \vr_2 )(u;x,z)$ is a  Kernel function in 
$ \Sigma \KF^{n_1+ n_2}_{K, 0, p_1 + p_2}\bra{\epsilon_0, N} $.
\end{rem}

\begin{rem}\label{prop:integral_kernelfunction}
Let $ \varrho \pare{u; x, z} $ be a Kernel function  in $ \Sigma \KF^{n}_{K, 0, p}\bra{\epsilon_0, N} $ with $  n>-1 $. Then 
$ \fint \varrho\pare{u; x, z}\dd z $ is a function in 
$ \Sigma \cF^{\bR}_{K, 0, p}\bra{\epsilon_0, N}$. 
This follows directly integrating  \eqref{homores1} and \eqref{homores1.1} in $ z $.
\end{rem}

The $ m $-Kernel-operators  defined below are a $ z $-dependent family of $  m $-operators with coefficients small as $ |z|^n_{\bT}  $. They appear for example as 
smoothing operators in the  composition of Bony-Weyl quantizations of 
Kernel-functions.

\begin{definition}\label{def:KM}
Let  $ m,n \in \R $,  $p,N\in \N_0 $,
$K\in\N_0$ with $ \epsilon_0 > 0 $.
\begin{enumerate}[i)]
\item
{\bf $p$-homogeneous $m$-Kernel-operator.}
We denote by $\widetilde{\KM}^{m,n}_{p}$ 
 the space  of 
$ z $-dependent, $ x $-translation invariant  homogeneous $ m $-operators 
according to  \Cref{def:moperators}, \cref{item:maps1},  in which the constant 
$ C $  is substituted with $ C\av{z}^n_{\bT} $, equivalently
\begin{equation}\label{homoresMz}
M( u ;z)v \pare{x} = 
\sum_{\substack{(\vec \jmath_{p},j,k) \in \Z ^{p+2} \\  
j_1 + \ldots + j_p + j = k } } M_{\vec \jmath_{p}, j,k}\pare{z} u_{j_1} \ldots u_{j_p} v_{j}  
e^{\ii  k x}  \, ,  \qquad 
  z \in \bT\setminus \{ 0 \}  \, , 
\end{equation} 
 with coefficients satisfying
\be \label{kernel1smo}
\big| M_{\vec \jmath _p, j, k}\pare{z} \big| \leq C \   {\rm max} _2 \set{\angles{j_1}, \ldots,  \angles{j_{p}}, \angles{j}} ^\mu \ \max\set{\angles{j_1}, \ldots , \angles{j_{p}}, \angles{j}} ^m \av{z}^n_{\bT} \, . 
\ee
If $ p=0 $ the right hand side  of \eqref{homoresMz} is replaced by
 $ \sum _{j\in \bZ} M_j \pare{z} v_j e^{\ii j x} $ with $ \big| M_j \pare{z} \big| 
 \leq C \angles{ j}^m \av{z}^n_{\bT} $. 
 \item
 {\bf Non-homogeneous $m$-Kernel-operator.}
  We denote by  $\KM ^{m,n}_{K,0,p}[\epsilon_0]$
  the space of $ z $-dependent, 
  non-homogeneous operators $ M(u;z) v  $ defined for any 
$ z  \in\bT\setminus \{ 0 \}   $,  such that for any $ 0\leq k\leq K $
\begin{equation}\label{piove2}
\norm{{\partial_t^k\left(M(u;z)v\right)}} _{s- \alpha k-m} \\
 \leq C  \av{z}^n_{\bT} \   \sum_{k'+k''=k}  \pare{ \|{v}\|_{k'',s}\|{u}\|_{k',{s_0}}^{p}
 +\|{v}\|_{k'', {s_0}}\|u\|_{k', {s_0}}^{p-1}\| u \|_{k', s} } \, .
\end{equation}
 \item
 {\bf $m$-Kernel-Operator.}
We denote by $\Sigma\KM^{m,n}_{K,0,p}[\epsilon_0,N]$
the space of operators 
of the form 
\begin{equation}\label{mapz}
M(u;z)v =\sum_{q=p}^{N}M_{q}(u,\ldots,u)v+M_{>N}(u;z)v 
\end{equation}
where   $M_{q} $  are homogeneous $m$-Kernel 
 operators  in $ \widetilde{\KM}^{m,n}_{q}$, $q=p,\ldots, N$  and $M_{>N}$ is a non--homogeneous $m$-Kernel-operator  in 
$\mathcal{M}^{m,n}_{K,0,N+1}[\epsilon_0]$. 

\item {\bf Pluri-homogeneous $m$-Kernel-Operator. } We denote 
by $ \Sigma_p^N \wt {\cM}^{m}_q $ the 
pluri-homogeneous $m$-operators  of the form \eqref{mapz}
with $ M_{>N} = 0 $.
\end{enumerate}
\end{definition}

%Notice  that the dependence on the time $ t $ in the % non-homogenous 
%$m$-operators  enters only though $ u $.

\begin{rem}\label{rem:OpBWisanMop-z}
Given $ \varrho\pare{u;x, z}\in \Sigma \KF^{n}_{K,0,p}\bra{\epsilon_0, N} $ then $ \OpBW{\varrho\pare{u;x, z}} \in \Sigma \KM^{0, n}_{K, 0, p}\bra{\epsilon_0, N} $.
\end{rem}

\begin{definition}[Kernel-smoothing operators]\label{def:KR}
Given 
$ \rho > 0 $ we define the homogeneous and non-homogeneous Kernel-smoothing operators as
\begin{align*}
\wt \KR ^{-\rho,n}_p \defeq \wt \KM ^{-\rho,n}_p , 
&&
\KR^{-\rho,n}_{K, 0, p}\bra{\epsilon_0} \defeq \KM^{-\rho,n}_{K, 0, p}\bra{\epsilon_0}, 
&&
\Sigma \KR^{-\rho,n}_{K, 0, p}\bra{\epsilon_0, N}
\defeq \Sigma \KM^{-\rho,n}_{K, 0, p}\bra{\epsilon_0, N}. 
\end{align*}
\end{definition}

In view of \cite[Lemma 2.8]{BMM2022}, 
if $ M\pare{u, \ldots , u;z} $ is a homogenous 
$ m$-Kernel operator in $  \wt\KM^{m,n}_p $ then $  M\pare{u, \ldots , u; z}  $
defines a non-homogenous  $ m$-Kernel operator in $ \KM^{m,n}_{K, 0, p}\bra{\epsilon_0} $ for any $ \epsilon_0 > 0 $ and $ K\in \bN_0 $.

\begin{prop}[Composition of $ z $-dependent operators] \label{prop:comp_z} Let $ m,n,m',n'\in\bR $, and integers 
$ K, p, p', N\in\bN_0 $ with $ p, p'\leq N $.
\begin{enumerate}
\item \label{item:OpOp_ext_z} Let $ \vr \pare{u;x,z} \in \Sigma \KF ^{n}_{K, 0, p}\bra{\epsilon_0, N} $ and $  \vr' \pare{u;x,z} \in \Sigma \KF ^{n'}_{K, 0, p'}\bra{\epsilon_0, N}  $ 
be Kernel functions. Then 
\begin{equation*}
\OpBW{\vr \pare{u;x,z}}\circ \OpBW{\vr'\pare{u;x,z}} = \OpBW{\vr \ \vr' \pare{u;x,z}} + R\pare{u;z}
\end{equation*}
where $ R\pare{u;z} $ is a Kernel-smoothing operator in $ \Sigma\KR^{-\rho, n+n'}_{K, 0, p+p'}\bra{\epsilon_0, N} $  for any $ \rho \geq 0 $; 

\item  \label{item:MM_ext_z} Let $ M \pare{u;  z} $ be a $ m$-operator in 
$ \Sigma \KM^{m, n}_{K, 0 ,p}\bra{\epsilon_0, N} $ and 
$ M' \pare{u;  z} $ be an $ m'$-operator in 
$ \Sigma \KM^{m', n'}_{K, 0,p'}\bra{\epsilon_0, N} $. Then $ M\pare{u;  z} \circ M' \pare{u;  z} $ belongs to 
$ \Sigma \KM^{m+\max\pare{m',0}, n+n'}_{K, 0 , p+p'}\bra{\epsilon_0, N} $; 

\item  \label{item:OpR_ext_z} Let $ \varrho \pare{u;x,z} $ be a Kernel function in 
$ \Sigma\KF^n_{K, 0, p}\bra{\epsilon_0, N} $ and $ R\pare{u;z} $ 
be a Kernel smoothing operator in $ \Sigma\KR^{-\rho, n'}_{K, 0, p'}\bra{\epsilon_0, N} $ then $ \OpBW{\varrho\pare{u; x, z}}\circ R\pare{u;z} $ and $  R\pare{u;z} \circ \OpBW{\varrho\pare{u; x, z}} $   are a Kernel smoothing operator  in $ \Sigma \KR^{-\rho, n+n'}_{K, 0, p+p'}\bra{\epsilon_0, N} $; 

\item  \label{item:MM_int_z} 
Let $ M \pare{u;z} $ be an homogeneous $ m$-Kernel operator 
in $ \wt \KM^{m,n}_1 $, and $ M' \pare{u;z} $ in $ \Sigma \KM^{0,0}_{K,0,0}\bra{\epsilon_0, N} $ then  $ M\pare{M'\pare{u;z}u;z} \in \Sigma \KM^{m,n}_{K,0,1}\bra{\epsilon_0, N} $.
%{\color{blue} Let $ M \pare{u;z} $ be a $ m$-Kernel operator 
%in $ \Sigma \KM^{m,n}_{K,0,p}\bra{\epsilon_0 , N} $, and $ M' \pare{u;z} $ in $ \Sigma \KM^{0,0}_{K,0,0}\bra{\epsilon_0', N} $ there exists an $ \epsilon_0'' > 0 $ such that $ M''\pare{u;z} \defeq M\pare{M'\pare{u;z}u;z} \in \Sigma \KM^{m,n}_{K,0,p}\bra{\epsilon_0'', N} $.}
\end{enumerate}
\end{prop}

\begin{proof}
The proof of item \ref{item:OpOp_ext_z} 
is performed in  \cite[Proposition 3.12]{BD2018} keeping track of the dependence in the variable $ z $ of the symbols as in \eqref{homores1},
 \eqref{homores1.1} when $ \gamma =0 $. 
More precisely $ \varrho $ and $ \varrho' $ satisfy $ z $-dependent inequalities (cf. \eqref{homosymbo}, \eqref{nonhomosymbo})
\begin{align*}
\av{\partial_x^\alpha \varrho_q\pare{\Pi_{\vec n}\cU; x,z}} \leq \  C \av{z}^n_{\bT} \av{\vec n}^{\mu+\alpha}\prod_{j=1}^p \norm{\Pi_{n_j} u_j}_{L^2} \, , 
&& 
\av{\partial_t^k \partial_x^\alpha \varrho\pare{u; x, z}} \leq   \ C \av{z}^n_{\bT} \norm{u}_{k, s_0}^{p-1}\norm{u}_{k, s} \, ,  
\end{align*}
and,  in the proof of \cite[Proposition 3.12]{BD2018},  the seminorm of the composed symbol always appear as a product of the seminorms of the factor symbols. 
The proof of item \ref{item:MM_ext_z}  is the same as in \cite[Proposition 2.15-i]{BMM2022}, keeping track of the dependence in $ z $ of the $ m $-operators. 
For 
item \ref{item:OpR_ext_z}, see \cite[Proposition 2.19-i]{BMM2022} factoring 
the dependence on $ z $. 
Item  
\ref{item:MM_int_z}  is a consequence of  \cite[Proposition 2.15-ii]{BMM2022} factoring 
 the dependence on $ z $.
\end{proof}

Finally  integrating \eqref{kernel1smo} and \eqref{piove2} in $ z $ we deduce the following lemma.

\begin{lemma}\label{prop:action_z_smoothing}
 Let $ R\pare{u;z} $ be a Kernel smoothing  operator in $ \Sigma \cR^{-\rho, n}_{K,0, p}\bra{\epsilon_0, N} $ 
 with $ n >  - 1 $.
  Then 
\begin{equation*}
\fint R\pare{u;z} g\pare{x-z} \dd z = R_1\pare{u}g \, , \qquad 
\fint R\pare{u;z}  \dd z = R_2\pare{u} \, , 
\end{equation*}
where $ R_1 \pare{u} $, $ R_2 \pare{u} $ are smoothing operators 
 in $ \Sigma \cR^{-\rho}_{K,0,p}\bra{\epsilon_0, N} $. 
\end{lemma}

The following proposition will be crucial  in Section \ref{sec:paralinearization}.

\begin{prop}
\label{prop:reminders_integral_operator}
Let  $ n > -1 $ and $ \varrho \pare{u; x , z} $ be a Kernel-function  in 
$ \Sigma\KF^{n}_{K,0,p}\bra{\epsilon_0 , N} $. 
Let  us define the operator, for any $ g \in H^s_0 (\bT) $, $ s \in \bR  $,   
\begin{equation}\label{defTf}
\pare{\cT_\varrho g} 
\pare{x} \defeq \fint \OpBW{\varrho \pare{u; \bullet , z}} g\pare{x-z} \dd z   \, .
\end{equation}
Then there exists
\begin{itemize}
\item   a symbol $ a \pare{u; x ,  \xi}  $ 
in $ \Sigma\Gamma^{-\pare{1+n}}_{K,0, p}\bra{\epsilon_0 , N} $ satisfying \eqref{areal};

\item a pluri-homogeneous smoothing operator $R\pare{u}  $ 
in $ \Sigma_p^N \wt {\cR}^{-\rho}_q $ 
for any $ \rho > 0 $;
\end{itemize}
 such that $ \cT_\varrho g = \OpBW{ a \pare{u; x,  \xi}} g + R\pare{u} g $.
\end{prop}

\begin{proof} 
In view of Definition \ref{quantizationtotale} and  \cref{rem:2quantiz} we have that
\begin{equation}\label{diffqua}
\OpBW{\vr\pare{u;x,z}} - \OpW{  \vr_\psi \pare{u;x,z}}  =: R \pare{u;z} 
\end{equation}
is a pluri-homogeneous Kernel smoothing operator in 
$ \Sigma_p^N \wt {\KR}^{-\rho,n}_q $ for any $ \rho $.  
Since $ n > - 1 $, integrating in $ z $, 
we deduce that $ \fint R \pare{u;z} g(x-z) \dd z  = R \pare{u} g $ where 
$ R \pare{u} $
 is a pluri-homogeneous smoothing operator
in $ \Sigma_p^N \wt {\cR}^{-\rho}_q $ (cf.
\Cref{prop:action_z_smoothing}). 

In view of  \eqref{BW} and \eqref{Opweil} we compute for any $ \nu\in\bZ $
\begin{equation*}
\cF_{x\to \nu}\pare{ \fint\OpW{\vr_\psi \pare{u; x, z}}  g\pare{x-z} \dd z } \pare{\nu}
\\
=
\sum _{k \in \bZ} \psi\pare{\nu-k, \frac{\nu+k}{2}} \fint  \hat{\varrho} \pare{u; \nu-k, z} e^{-\ii kz} \dd z \ \hat{g}\pare{k} 
\end{equation*}
where $ \psi  (\xi', \xi) $ is an  admissible cut-off function, namely satisfying
 \eqref{admcutoff2}-\eqref{admispapb}. 
Introducing  another  admissible cut-off function $ \tilde{\psi} (\xi', \xi) $   
 identically equal to one on the support of 
$ \psi (\xi', \xi) $, and since $ \widehat{g}(0)=0 $, 
\begin{multline}\label{concutoff}
\cF_{x\to \nu}\pare{ \fint\OpW{\vr_\psi \pare{u; x, z}} \dd z } \pare{\nu} 
\\
=
\sum  _{k\in \mathbb{Z}} \psi\pare{\nu-k, \frac{\nu+k}{2}} \tilde \psi\pare{\nu-k, \frac{\nu+k}{2}} \chi\pare{2k} \fint  \hat{\varrho} \pare{u; \nu-k, z} e^{-\ii kz} \dd z \ \hat{g}\pare{k} 
\end{multline}
where $ \chi ( \cdot ) $ is the $ \cC^\infty $ function defined 
in \eqref{eq:chi}. 
Introducing a $ \cC^\infty $ function $ \eta : \bR \to [0,1]  $ with compact support 
such that 
\begin{align}\label{etacut}
\eta (z) = 1 \, , \ \forall |z| \leq \frac{\pi}{2} \, , && 
\eta (z) = 0 \, ,  \ \forall |z| \geq \frac{3\pi}{2} \, , &&  
\sum_{j \in \bZ} \eta ( z + 2 \pi j) = 1 \, , \ \forall z \in \bR \, ,  
\end{align}
we may write the integral on $ \bT $ as
\begin{equation}\label{intsuR}
\fint  \hat{\varrho} \pare{u; \nu-k, z} e^{-\ii kz} \dd z = 
\left. \frac{1}{2\pi} \int_{\bR}  \hat{\varrho} \pare{u; j, z} \eta (z) e^{-\ii  \xi z} \dd z \  \right|_{\pare{ j, \xi }= \pare{ \nu-k , k }} \, . 
% I(u; \nu-k,k)
\end{equation}
Therefore by \eqref{diffqua}, \eqref{concutoff} and \eqref{intsuR}
the operator $ \cT_\varrho $ in \eqref{defTf} is equal to 
\begin{equation}\label{loper}
\cT_\vr = \OpW{a_\psi\pare{u;x,\xi}} = \OpBW{a\pare{u;x, \xi}} + R\pare{u} 
\qquad \text{where} \qquad 
R \pare{u} \in \Sigma_p^N \wt {\cR}^{-\rho}_q  
\end{equation}
and
\begin{equation}\label{simboloa}
a\pare{u; x , \xi}  = \sum_{j \in \Z} \hat{a}\pare{u; j , \xi} e^{\ii j x}  \, , \quad 
\hat{a}\pare{u; j , \xi} \defeq  \tilde{\psi}\pare{j, \xi} \chi\pare{2\xi - j}  
\frac{1}{2\pi}\int_\bR  \hat{\varrho} \pare{u; j, z} \eta (z)  e^{-\ii\pare{ \xi - \frac{j}{2} } z} \dd z  \, . 
\end{equation}
In order to prove the lemma, in view of \eqref{loper}, 
it is sufficient to show that  
$ a \pare{u; x , \xi} $ defined in \eqref{simboloa} is a symbol in 
$ \Sigma\Gamma^{-\pare{1+n}}_{K,0, p}\bra{\epsilon_0 , N} $
according to \Cref{def:symbols}.
Notice that 
$ a\pare{u;x, \xi} $ satisfies the reality condition \eqref{areal}. 
Moreover, in view of the support properties of
$ \tilde{\psi}\pare{j, \xi} $ and $ \chi\pare{2\xi - j} $, it results that 
\begin{equation}\label{eq:restriction_freq}
\hat{a}\pare{u; j , \xi} \neq 0 \qquad \Longrightarrow \qquad 
\av{\xi-\frac{j}{2}} \sim \av{\xi}  \, , \quad \av{\xi} \gtrsim 1 \, , \quad 
\av{\xi}\sim\angles{\xi}  
\, .
\end{equation}
We decompose the Kernel function 
$$ 
\varrho \pare{u; x , z} \in\Sigma\KF^{n} _{K,0,p}\bra{\epsilon_0 , N} 
 \qquad \text{as} \qquad 
\vr \pare{u;x,z} = \sum_{q=p}^N 
\vr_q \pare{u; x , z} + \vr _{> N}\pare{u;x, z} \, , 
$$ 
where 
$ \varrho_q \pare{u; x , z} $ are homogenous Kernel functions in $ \widetilde{\KF}^{n}_q $ and 
$ \varrho_{> N} \pare{u; x , z} $ is a non-homogenous Kernel function
in $ \KF^{n}_{K,0,N+1} [\epsilon_0] $. Accordingly we decompose the 
symbol $a \pare{u; x, \xi } $ in \eqref{simboloa}  as
$$
a \pare{u;x,\xi} = \sum_{q=p}^{N} a_q \pare{u;x,\xi}  + a_{>N} \pare{u;x,\xi}
$$
where
\begin{footnotesize}
\begin{equation}\label{simbolsqN}
\begin{aligned}
a_q \pare{u;x,\xi} = \sum_{j \in \Z} \hat{a}_q \pare{u; j, \xi} e^{\ii j x} \, ,  & \quad  
\hat{a}_q \pare{u; j, \xi} \defeq \tilde{\psi}\pare{j, \xi} \chi\pare{2\xi - j}  \frac{1}{2\pi} \int_{\bR} 
 \hat{\varrho}_q  \pare{u; j, z} \eta (z) e^{-\ii\pare{ \xi - \frac{j}{2} } z} \dd z \, , \\
a_{>N} \pare{u;x,\xi} 
= \sum_{j \in \Z} \hat{a}_{>N} 
\pare{u; j, \xi} e^{\ii j x} \, , & \quad  
\hat{a}_{>N} \pare{u; j, \xi} \defeq \tilde{\psi}\pare{j, \xi} \chi\pare{2\xi - j} 
\frac{1}{2\pi} \int_{\bR}  \hat{\varrho}_{>N}  \pare{u; j, z} 	\eta (z) e^{-\ii\pare{ \xi - \frac{j}{2} } z} \dd z \, . 
\end{aligned}
\end{equation}
\end{footnotesize}
We now prove that, according to \Cref{def:symbols}, 
\begin{align}
\label{eq:res_h}
a_q 
\in & \ \widetilde{\Gamma}^{-\pare{1+n}}_q \,, \quad \forall q = p, \ldots, N \, , 
\\
\label{eq:res_nh}
a_{>N} \in & \ \Gamma^{-\pare{1+n}} _{K,0,N+1}\bra{\epsilon_0}.
\end{align}
\begin{step}[Proof of \eqref{eq:res_h}]
In view of \eqref{homoresz}  
the $ q $-homogeneous component $ a_q  (u;x, \xi ) $ in 
\eqref{simbolsqN} has an expansion as in 
\eqref{sviFou} (recall the notation $\vec \jmath_q = (j_1, \ldots, j_q) $)
\begin{align*}
a_{{\vec \jmath}_q} (\xi)  
= 
\tilde{\psi}\pare{j, \xi} \chi\pare{2\xi - j}  \frac{1}{2\pi} \int_{\bR}  \varrho_{j_1, \ldots, j_q} (z) \, \eta (z) \,  
 e^{-\ii \xi  z} \dd z   \, , &&  j = j_1 + \ldots + j_q \, . 
\end{align*}
Let us prove it satisfies \eqref{eq:fourier_char_homsymbols} with $ m = - (1+n ) $. 
Decomposing $ 1 = \chi_1 (\cdot ) + \chi_2 (\cdot) $
where  $\chi_1 : \R \to [0,1] $ is a smooth cutoff  function 
supported and equal to $1$ near $0$,  
 we decompose % (here  $ j = j_1 + \ldots + j_q $)
\begin{equation} \label{eq:oscillatory_integral}
a_{{\vec \jmath}_q} (\xi) =  
a_{{\vec \jmath}_q}^{(1)} (\xi)  
+  a_{{\vec \jmath}_q}^{(2)} (\xi)  =
 \sum _{\mathsf{j}=1}^2
\tilde{\psi}\pare{j, \xi} \chi\pare{2\xi - j}  
\frac{1}{2\pi} \int_{\bR} \chi _\mathsf{j} \pare{\angles{\xi} z}
\varrho_{j_1, \ldots, j_q} (z) \, \eta (z) \,  e^{-\ii \pare{ \xi - \frac{j}{2} } z} \dd z \, . 
\end{equation} 
By \eqref{homores1} (with $ l = 0 $)
and since  $n >-1 $ we deduce
\begin{equation}\label{eq:I1_est}
\av{a_{{\vec \jmath}_q}^{(1)} \pare{ \xi }}
 \lesssim
\int_{\av{z}\lesssim \xfrac{1}{\angles{\xi}}}
\av{{\vec \jmath}_q}^{\mu}  |z|^n 
\dd z \lesssim \av{{\vec \jmath}_q}^{\mu}  \angles{\xi}^{-\pare{1+n}} \, . 
\end{equation}
We now estimate $ a_{\vec \jmath}^{(2)} (\xi) $.  From
$ e^{-\ii z \pare{ \xi - \frac{j}{2}  } } = 
\bra{-\ii\pare{ \xi - \frac{j}{2} }}^{-l}\partial_z^l \Big(  e^{-\ii\pare{ \xi - \frac{j}{2} } z} \Big)  $ for any $ l\in\bN_0 $,
we obtain, by an integration by parts (use \eqref{etacut} and that 
$ \chi _2\pare{\angles{\xi} z} $ vanishes near zero), that 
\begin{multline}\label{a2si}
a_{{\vec \jmath}_q}^{(2)} (\xi) = 
\bra{-\ii\pare{ \xi - \frac{j}{2} }}^{-l} \tilde{\psi}\pare{j, \xi}
\chi\pare{2\xi - j} 
\sum_{l_1+l_2+l_3=l} c_{l_1, l_2, l_3}
\frac{1}{2\pi} \int_{\bR} e^{-\ii\pare{ \xi - \frac{j}{2} } z}  \,   Y_{l_1, l_2, l_3} (z)   \,  \dd z  
\\
 \text{where} \quad 
Y_{l_1, l_2, l_3} (z)  \defeq \angles{\xi}^{l_1}  \pare{ \partial_z^{l_1} \chi _2 }\pare{ z}  \, \partial_z^{l_2}\eta (z) \,  \partial_z^{l_3}\varrho_{{\vec \jmath}_q} (z)  \, . 
\end{multline}
% where $ Y \defeq \chi _2\pare{\angles{\xi} z} \varrho_{\vec \jmath} (z) $. 
Since $ \varrho_q \pare{u;x,z} $ is a Kernel function 
in $\wt{\KF}^n_p $, using \eqref{homores1} 
and exploiting that $\angles{\xi} ^{-1}\sim |z|$ 
on the support of $ \chi _2^{\pare{l_1}}  \pare{\angles{\xi} z} $ for any $ l_1 \ge 1 $, we get 
\begin{equation}\label{palder}
\begin{aligned}
 \int_\bR \av{   Y_{l_1, l_2, l_3} (z)} \dd z  \lesssim \av{\vec{\jmath}_q}^{\mu }
     \angles{\xi}^{l_1}\ \int _{ \frac{1}{\angles{\xi}}\sim |z| }  \av{z}^{n-l_3}  \dd z \lesssim \av{\vec \jmath_q}^\mu \angles{\xi}^{l_1+l_3 - \pare{n+1}}
   .
  \end{aligned} 
\end{equation}
When $ l_1 = 0 $ we have that
\begin{equation}\label{eq:l1zero}
\int_\bR \av{   Y_{0, l_2, l_3} (z)} \dd z  \lesssim \av{\vec\jmath_q}^\mu
\int_{\frac{\delta}{\angles{\xi}}\leq \av{z}\leq \frac{3\pi}{2} }
   \av{ 
  z }^{n-l_3} \dd z
  \lesssim \av{\vec\jmath_q}^\mu
\pare{1+\angles{\xi}^{l_3-\pare{n+1}}} . 
\end{equation}
Then by \eqref{a2si}, \eqref{palder}, \eqref{eq:l1zero}
and \eqref{eq:restriction_freq},  we deduce, for $ l > n + 1 $,  
\begin{equation}\label{bounda2}
\big| a_{{\vec \jmath}_q}^{(2)} (\xi) \big|
\lesssim_l \av{\vec{\jmath}_q}^{\mu } \angles{\xi}^{-\pare{1+n}} \, . 
\end{equation}
The bounds \eqref{eq:I1_est} and \eqref{bounda2} prove 
that $ a_{{\vec \jmath}_q} (\xi) $  in \eqref{eq:oscillatory_integral} satisfies
the estimate \eqref{eq:fourier_char_homsymbols} (for $ \beta = 0 $ and $ m = 
- (1+ n) $).  
Since, for any $ \beta \in \N $, 
\begin{align}\label{rem:simplification_oscint}
\partial_\xi^ \beta a_{{\vec \jmath}_q} (\xi) = 
\sum_{\beta_1+ \beta_2+ \beta_3 = \beta}
C_{\beta_1,\beta_2,\beta_3}
\pa_{\xi}^{\beta_1} \tilde{\psi}\pare{j, \xi} 
\pa_{\xi}^{\beta_2}  \chi\pare{2\xi - j}  
\int _{\bR} \varrho_{{\vec \jmath}_q} (z) 
 \pare{-\ii z}^{\beta_3}   e^{-\ii\pare{ \xi - \frac{j}{2} } z} \dd z \, , 
\end{align}
using \eqref{admispapb}, \eqref{eq:restriction_freq}, the fact that $ \chi $ is supported near $ 0 $, 
 that 
 $ z^{\beta_3}  \varrho_{{\vec \jmath}_q} (z)  $ satisfies \eqref{homores1} 
 with ($ n $ replaced by $ n+\beta_3$, cf. Remark \ref{rem:operation_z})
and repeating the bound obtained for $ \beta_3 =0 $ for the integral term in \cref{rem:simplification_oscint}, we obtain 
$$
\av{\partial_\xi^\beta a_{{\vec \jmath}_q} (\xi)}  \lesssim_\beta
\sum_{\beta_1+\beta_2+\beta_3 = \beta} \angles{\xi}^{- {\beta_1} }
\angles{\xi}^{- {\beta_2} }
\av{\vec{\jmath}_q}^{\mu } \angles{\xi}^{- {(1 +
 n + \beta_3)} } 
 \lesssim_\beta 
\av{\vec{\jmath}_q}^{\mu } \angles{\xi}^{- {(1 +
 n + \beta)} } \, .  
$$ 
Note that actually for any $ j \in \Z $, $ \beta_2 \geq 1 $,  the derivative
$ \partial_\xi^{{\beta}_2} \chi \pare{  2 \xi - j }   = 0 $, 
for any $ |\xi| \geq  2 $. 
This concludes the proof of \eqref{eq:res_h}.
\end{step}

\begin{step}[Proof of \eqref{eq:res_nh}]
We argue similarly to the previous step. 
Recalling \eqref{simbolsqN},  for any $ 0 \leq k \leq K $ and $ \gamma \in \N_0 $, 
we decompose, with $ \chi_\mathsf{j}, \ \mathsf{j}=1, 2 $ defined as in \eqref{eq:oscillatory_integral}, 
\begin{equation} \label{eq:anonhom}
\begin{aligned}
& \partial_t^k \partial_x^\gamma a_{>N}\pare{u; x,\xi}  =  \  I_1 + I_2 \qquad \text{where}  \\
& I_\mathsf{j} \defeq  \sum _{j\in \bZ}   \tilde{\psi}\pare{j, \xi} \chi\pare{2\xi-j}  \frac{1}{2\pi} \int_\bR  \chi_\mathsf{j} \pare{z\angles{\xi}}  \widehat{ \partial_t^k   \partial_x^\gamma\varrho}_{>N} \pare{u;j,z} \eta\pare{z} e^{-\ii \pare{ \xi -\frac{j}{2} } z} \dd z \ e^{\ii j x} \, .
\end{aligned} 
\end{equation}
Fix $ \mu_0 > 1 $, let $ s_0 > 0 $ associated to $ \vr_{>N} $ as per \Cref{def:kernel_functions}, let $ \gamma \leq s-\pare{s_0+\mu_0} $. 
The term $ I_1 $ can be  estimated using \eqref{homores1.1}, 
the fact that $ \chi_1 (z \angles{\xi} ) $ is supported for $ \av{z}\lesssim 1/\angles{\xi} $
and  $ n > - 1 $, as
\begin{equation}
\label{I1bound}
\av{I_1} \lesssim \sum_{j\in \bZ} \angles{j}^{-\mu_0} \int _{\av{z}\lesssim 1/\angles{\xi}} \av{ \widehat{ \partial_t^k \partial_x^\gamma \pare{1-\partial_x^2}^{\frac{\mu_0}{2}} \varrho_{>N}}\pare{u;j,z} } \dd z
\\
\lesssim
\norm{u}_{k, s_0 +\mu_0}^{N}\norm{u}_{k, s}\angles{\xi}^{-\pare{1+n}} .
\end{equation}
Next we estimate $ I_2 $. After an integration by parts,  setting $  l = \max \set{0, \ceil{1+n}} $, we have from \cref{eq:anonhom}
\begin{equation}
\label{passs1}
\begin{aligned}
\av{I_2} 
\lesssim & \ \sum_{j \in\bZ} \av{\tilde{\psi}\pare{j, \xi}\chi\pare{2\xi-j}} \av{\xi-\frac{j}{2}}^{-l} \sum_{l_1+l_2+l_3 = l}  \int _\bR
\av{Z_{l_1, l_2, l_3}\pare{z}}\dd z \, . 
\end{aligned}
\end{equation}
where
\begin{equation}
\label{passs3}
Z_{l_1, l_2, l_3}\pare{z} \defeq  \angles{\xi}^{l_1}\pare{ \partial_z^{l_1} \chi_2 }\pare{z\angles{\xi}}  \ \partial_z^{l_2}\eta\pare{z} \ \partial_t^k \widehat{\partial_x^\gamma \partial_z^{l_3}  \vr} _{> N}\pare{u;j, z} \, . 
\end{equation}
For any  $ j\in \bZ $
\begin{equation}\label{passs2}
\av{ \widehat{ \partial_t^k  \partial_x^\gamma  \partial_z^{l_3} \varrho}_{>N}\pare{u;j,z}}
\lesssim \angles{j}^{-\mu_0} \sup _{x\in\bT} \av{\partial_t^k \partial_x^{\gamma } \partial_z^{l_3}  \pare{1-\partial_x^2}^{\frac{\mu_0}{2}} \varrho _{>N}\pare{u;x,z}}
\lesssim \angles{j}^{-\mu_0} \norm{u}_{k, s_0+\mu_0}^{N} \norm{u}_{k , s}\av{z}^{n - l_3} \, .
\end{equation}
With computations analogous to the ones performed in \cref{palder,eq:l1zero} we obtain using \cref{passs2,passs3}, that
\begin{equation}
\label{passs4}
\begin{aligned}
\int _{\bR}\av{Z_{l_1, l_2, l_3} \pare{z}} \dd z \lesssim & \ \angles{j}^{-\mu_0} \norm{u}_{k, s_0+\mu_0}^{N} \norm{u}_{k , s} \angles{\xi}^{\pare{l_1+l_3}-\pare{n+1}}, 
& \text{ if } & l_1 \neq 0 , \\
%------------------------------------------------
\int _{\bR}\av{Z_{0, l_2, l_3} \pare{z}} \dd z \lesssim & \ \angles{j}^{-\mu_0} \norm{u}_{k, s_0+\mu_0}^{N} \norm{u}_{k , s} \pare{ 1+ \angles{\xi}^{l_3 -\pare{n+1}} }. 
\end{aligned}
\end{equation}
Since $ l = l_1+l_2+l_3 > 1+n $ and $ \mu_0 > 1 $ we obtain, by \cref{passs1,passs4,eq:restriction_freq} that
\begin{equation}\label{I2bound}
\av{I_2} \lesssim 
 \norm{u}_{k , s_0+\mu_0}^{N} \norm{u}_{k , s}  \angles{\xi}^{-\pare{1+n}} \sum_{j\in\bZ} \angles{j}^{-\mu_0} . 
\end{equation}
Inserting \cref{I1bound,I2bound} in \eqref{eq:anonhom} we conclude that
\begin{equation*}
\av{\partial_t^k \partial_x^\gamma  a_{>N}\pare{u; x,\xi} }   \lesssim   \norm{u}_{k , s_0+\mu_0}^{N} \norm{u}_{k , s}   \angles{\xi}^{-\pare{1+n}} .
\end{equation*}
Arguing as in \eqref{rem:simplification_oscint} we thus obtain that, for any $ \beta \in \N $, 
$ \av{\partial_t^k \partial_x^\gamma  \partial_\xi^\beta a_{>N}\pare{u; x,\xi} }   \lesssim   \norm{u}_{k , s_0+\mu_0}^{N} \norm{u}_{k, s}   \angles{\xi}^{-\pare{1+n+\beta}} $
concluding the proof of \eqref{eq:res_nh}.
\end{step}
\end{proof}

\section{The linearized problem at $ f = 0 $}\label{sec:linearized}

The  linearized equation  \eqref{eq:SQG_Hamiltonian}
 at $ f = 0 $ is
\begin{equation}\label{flin0}
\pa_t f = \partial_x \, \dd \nabla E_{\alpha}(0) f \, .
\end{equation}
In this section we  prove that the linear Hamiltonian operator
$ \partial_x \,  \dd \nabla E_\alpha(0) $
is a Fourier multiplier with symbol 
$ -\ii \, \omega_\alpha \pare{j}   $ 
and we provide its 
asymptotic expansion, see Lemmas \ref{lem:linearization} and \ref{prop:Lalpha_asymptotic}. 
%In   lem:conve} 
We also prove a convexity  property of the frequency map 
$ j \mapsto \omega_\alpha \pare{j}  $  and that $ \omega_\alpha \pare{j} $ are positive 
for any $ j \geq	 2 $, whereas $ \omega_\alpha (0) =   \omega_\alpha (1) = 0 $, cf. Remark 
\ref{rem:regmis}.   These latter results do not rely  on
oscillatory integrals
expansions %(cf.  \cite{Stein1993})  
and enable to prove 
the absence of three-wave resonances in \Cref{lem:nonres_cond}.

The following result extends
the computations %performed
 in \cite[Section 3]{HHM2021}, valid
for  $ \alpha\in\pare{0, 1} $, to the whole range  $ \alpha\in\pare{0, 2} $.

\begin{lemma}[Linearization of $  \nabla E_\alpha $ at zero]\label{lem:linearization}
For any $ \alpha\in\pare{0, 2} $,  
it results that 
\begin{equation}\label{diffinzero}
\dd \nabla E_\alpha(0)  = - L_\alpha\pare{\av{D}}
\, ,
\end{equation}
where  $ L_\alpha\pare{\av{D}} $ is the Fourier multiplier operator 
\begin{equation}\label{eq:Lalpha}
 L_\alpha\pare{\av{D}} \defeq  \
\frac{c_\alpha}{2\pare{1-\frac{\alpha}{2}}}  \bra{ \sT^1_\alpha\pare{\av{D}} - \sT^2_\alpha\pare{\av{D}} - \frac{\Gamma\pare{2-\alpha }}{\Gamma\pare{1 - \frac{\alpha}{2}}^2}}
\end{equation}
with 
\begin{align}  
 \label{eq:T1alpha} \sT^1_\alpha \pare{\av{j}} \defeq & \ \frac{ \Gamma\pare{2-\alpha} }{\Gamma\pare{1-\frac{\alpha}{2}} \Gamma \pare{ \frac{\alpha}{2}  } }  \sum_{k=0}^{\av{j}-1} \frac{\Gamma \pare{ \frac{\alpha}{2} +k }}{\Gamma \pare{1-\frac{\alpha}{2} +k} } \frac{1}{1-\frac{\alpha}{2} + k} \,, & & \sT^1_\alpha \pare{0} = 0 \, ,  \\
\label{eq:T2alpha} \sT^2 _\alpha \pare{\av{\xi}} \defeq & \ \frac{ \Gamma\pare{2-\alpha}}{\Gamma\pare{1-\frac{\alpha}{2}} \Gamma\pare{\frac{\alpha}{2}} }    \frac{ \Gamma\pare{\frac{\alpha}{2} +\av{\xi}} }{ \Gamma\pare{1-\frac{\alpha}{2} +\av{\xi}} }  &=& \ \bra{\av{\xi}^2 - \pare{1-\frac{\alpha}{2}}^2}\sM_\alpha\pare{\av{\xi}} , \\
\sM_\alpha\pare{\av{\xi}} \defeq &\frac{\Gamma\pare{2-\alpha}}{\Gamma\pare{1-\frac{\alpha}{2}} \Gamma\pare{ \frac{\alpha}{2}} } \  \frac{1}{\av{\xi}^2 - \pare{ 1-\frac{\alpha}{2} }^2 } \frac{\Gamma\pare{\frac{\alpha}{2} +\av{\xi}}}{\Gamma\pare{ 1 - \frac{\alpha}{2} +\av{\xi}}}
\, . \label{eq:Malpha}
\end{align}
The map $   j \mapsto \sT^\mathsf{1}_\alpha \pare{\av{j}}
$ is a Fourier  multiplier in $ \wt \Gamma^{\max\pare{0, \alpha-1}}_0 $ and 
$ j \mapsto  \sT^\mathsf{2}_\alpha \pare{\av{j}} $ 
is a Fourier  multiplier in $ \wt \Gamma^{\alpha-1}_0  $.
\end{lemma}

By the previous lemma, in Fourier, the linear equation
\eqref{flin0} amounts to the decoupled
scalar equations
\begin{equation}\label{lineqomegan}
\partial_t \hat f\pare{j} + \ii \,  \omega_\alpha \pare{j}  \hat f\pare{j} = 0  \, , \qquad 
j \in \bZ \setminus \{0\} \, ,
\end{equation}
with linear frequencies of oscillations
$  \omega_\alpha \pare{j} \defeq j  L_\alpha \pare{\av{j}} $.

\begin{proof}[Proof of \Cref{lem:linearization}]
By differentiating
$ \nabla E_\alpha \pare{f} $  in \eqref{eq:gradient-pseudoenergy} we deduce that
\begin{equation}
\label{eq:linearization_gradient_pseudoenergy_1}
\begin{aligned}
\dd \nabla E_\alpha (0) \ \phi = & \  \frac{c_\alpha}{2\pare{1-\frac{\alpha}{2}}}  \fint \frac{2\phi\pare{y} - \pare{\phi\pare{x} + \phi\pare{y}}\cos\pare{x-y}}{ \bra{ 2\pare{1-\cos\pare{x-y}}}^{\frac{\alpha}{2}}} \ \dd y \\
& \ + \frac{c_\alpha}{2\pare{1-\frac{\alpha}{2}}}  \fint \frac{\phi'\pare{y} \sin\pare{x-y} }{ \bra{ 2\pare{1-\cos\pare{x-y}}}^{\frac{\alpha}{2}}} \ \dd y
- \frac{\alpha}{4}\frac{c_\alpha}{2\pare{1-\frac{\alpha}{2}}}  \fint \frac{\phi\pare{x} + \phi\pare{y}}{ \bra{ 2\pare{1-\cos\pare{x-y}}}^{\frac{\alpha}{2}-1}} \ \dd y  \\
= & \  - \frac{c_\alpha}{2\pare{1-\frac{\alpha}{2}}}  \fint \frac{  \phi\pare{x} - \phi\pare{y} }{ \bra{ 2\pare{1-\cos\pare{x-y}}}^{\frac{\alpha}{2}}} \ \dd y
 +\frac{c_\alpha}{2\pare{1-\frac{\alpha}{2}}}  \fint \frac{\phi'\pare{y} \sin\pare{x-y} }{ \bra{ 2\pare{1-\cos\pare{x-y}}}^{\frac{\alpha}{2}}} \ \dd y \\
& \ + \frac{ c_\alpha}{4 }  \fint \frac{  \phi\pare{y} }{ \bra{ 2\pare{1-\cos\pare{x-y}}}^{\frac{\alpha}{2}- 1}} \ \dd y  + \frac{ c_\alpha}{4 }  \fint \frac{ \dd y}{ \bra{ 2\pare{1-\cos\pare{x-y}}}^{\frac{\alpha}{2}- 1}} \ \phi\pare{x} 
=:  \ \sum_{\kappa=1}^4 L_{\nabla E_\alpha, \kappa} \phi \, .
\end{aligned}
\end{equation} 
We now compute these operators.

\begin{step}[Evaluation of $ L_{\nabla E_\alpha , 1} $]\label{step:1}
We claim that
\begin{equation}\label{eq:step:1}
\pare{L_{\nabla E_\alpha, 1} \phi} \pare{x} \defeq
- \frac{c_\alpha}{2\pare{1-\frac{\alpha}{2}}}  \fint \frac{  \phi\pare{x} - \phi\pare{y} }{ \bra{ 2\pare{1-\cos\pare{x-y}}}^{\frac{\alpha}{2}}} \ \dd y
=  
 - \frac{c_\alpha}{2\pare{1-\frac{\alpha}{2}}}      \sT_\alpha^1  \pare{\av{D}} \phi \, .
\end{equation}
Indeed, setting  $ y=x-z $ we have
\begin{equation}\label{eq:step:1-}
L_{\nabla E_\alpha, 1} \phi = - \frac{c_\alpha}{2\pare{1-\frac{\alpha}{2}}}  \fint \frac{\phi\pare{x}-\phi\pare{x-z}  }{ \bra{ 2\pare{1-\cos z}}^{\alpha/2}}\dd z   = - \frac{c_\alpha}{2\pare{1-\frac{\alpha}{2}}}  \fint \frac{\phi\pare{x}-\phi\pare{x-z} }{ \av{ 2  \sin \pare{ \frac{z}{2}}}^{\alpha }}\dd z \, .
\end{equation}
We compute the action of  $L_{\nabla E_\alpha, 1}$  on
$ \phi\pare{x} =  \ \sum_{j \in\bZ} \hat{\phi}\pare{j} e^{\ii j x} = \hat{\phi}(0) + \underbrace{\sum_{j\geq 1} \hat{\phi}\pare{j} e^{\ii j x}}_{=: \phi_+\pare{x}} + \underbrace{\sum_{j \geq 1} \overline{\hat{\phi}\pare{j}} e^{-\ii j x}}_{=: \phi_-\pare{x}}  $.

By \eqref{eq:step:1-} we immediately get $ L_{\nabla E_\alpha, 1} \hat{\phi}(0) = 0 $.
Moreover, by \eqref{eq:step:1-},
\begin{equation}\label{svifour1}
L_{\nabla E_\alpha, 1} \phi_+ \pare{x} =  \  -\frac{c_\alpha}{2\pare{1-\frac{\alpha}{2}}}  
\sum_{j \geq 1}  \hat{\phi}\pare{j} e^{\ii j x}  \fint \frac{ 1-e^{- \ii  jz} }{\av{ 4 
\sin^2\pare{z/2}}^{\alpha / 2}} \dd z \, .
\end{equation}
We compute
\begin{multline} 
 \frac{1}{2\pi} \int_0^{2\pi} \frac{1-e^{- \ii  jz }}{\av{ 4 \sin^2\pare{z/2}}^{\alpha/2}} \dd z =   \  \frac{1}{2\pi} \int_0^{2\pi} \frac{1-e^{- \ii  jz}}{\av{1-e^{- \ii z}}^{\alpha}} \dd z    \\
  =  \  \frac{1}{ \pi} \int_0^{\pi} \frac{1-e^{-2 \ii  jz}}{\av{1-e^{-2 \ii z}}} \av{1-e^{-2 
  \ii z}}^{1-\alpha} \dd z
   = - \frac{2^{1-\alpha}}{ \ii \pi} \sum_{k=0}^{j-1} \int_0^{\pi} e^{- \ii  z \pare{2k+1}} \pare{\sin z}^{1-\alpha} \dd z \label{intcas1L}
\end{multline}
having also written
$ \av{1-e^{-2 \ii z}} = -  \ii\pare{1-e^{-2 \ii z}}e^{\ii z} $, for any $  z \in \bra{0, \pi} $.
We use now the identity (cf. \cite[p. 8]{MO1943})
\begin{align}\label{eq:magic_identity_1}
\int_0^\pi \sin^X\pare{z} e^{\ii Y z}\dd z = \frac{\pi e^{\ii \frac{Y\pi}{2}}\Gamma\pare{X+1}}{2^X \Gamma\pare{1 + \frac{X+Y}{2}} \Gamma\pare{1+\frac{X-Y}{2}}} \, , && \pare{X, Y}\in \pare{-1, \infty}\times\bR \, .
\end{align}
Setting $ X=1-\alpha $ and $ Y=-\pare{2k+1} $,
and using $ e^{- \ii \pare{2k+1}\frac{\pi}{2}} = - \ii \pare{-1}^k $,
we obtain
\begin{equation}\label{passint1}
\int_0^{\pi} e^{- \ii  z \pare{2k+1}} \pare{\sin z}^{1-\alpha} \dd z = \frac{
- \ii \pare{-1}^k \pi
 \Gamma\pare{2-\alpha} }{2^{1-\alpha}\Gamma\pare{1-k-\frac{\alpha}{2}} 
 \Gamma\pare{2+k-\frac{\alpha}{2}}} \, .
\end{equation}
The following consequence of Euler's reflection formula (cf. \cite{NIST})
\begin{equation}
\label{eq:consequence_Euer_duplication}
\Gamma\pare{z-j} = \pare{-1}^{j-1} \frac{\Gamma\pare{-z}\Gamma\pare{1+z}}{\Gamma\pare{j+1-z}} \, , \qquad 
z\in \bR\setminus \bZ \, , \quad j \in \bZ \, ,
\end{equation}
implies,  setting $ z=1-\frac{\alpha}{2} $, $ j =k $, and since
$ \Gamma\pare{1+y}= y \ \Gamma\pare{y} $,
\begin{equation}\label{passint2}
 \Gamma\pare{1-k-\frac{\alpha}{2}} =
 \pare{-1}^{k-1} \frac{\Gamma\pare{\frac{\alpha}{2} - 1 }\Gamma\pare{2-\frac{\alpha}{2}}}{\Gamma\pare{\frac{\alpha}{2} + k}}
 =  \pare{-1}^k \frac{\Gamma\pare{1-\frac{\alpha}{2}}\Gamma\pare{\frac{\alpha}{2}}}{\Gamma\pare{\frac{\alpha}{2} + k}} \, .
\end{equation}
By  \eqref{passint1}-\eqref{passint2} we deduce
\begin{equation}\label{passint3}
\int_0^{\pi} e^{- \ii z \pare{2k+1}} \pare{\sin z}^{1-\alpha} \dd z = \frac{- \ii \pi }{2^{1-\alpha}} \frac{ \Gamma\pare{2-\alpha} }{\Gamma\pare{1-\frac{\alpha}{2}} \Gamma\pare{\frac{\alpha}{2}}  } \frac{\Gamma\pare{\frac{\alpha}{2}+k}}{\Gamma\pare{2-\frac{\alpha}{2} + k}} \, .
\end{equation}
Consequently, by \eqref{intcas1L} and \eqref{passint3}, we conclude that for any $ j \geq 1  $
\begin{align*}
\frac{1}{2\pi} \int_0^{2\pi} \frac{1-e^{- \ii  jz}}{\av{ 4 \sin^2\pare{z/2}}^{\alpha/2}} \dd z =  \frac{ \Gamma\pare{2-\alpha} }{\Gamma\pare{1-\frac{\alpha}{2}} \Gamma \pare{ \frac{\alpha}{2}  } }  \sum_{k=0}^{j-1} \frac{\Gamma \pare{ \frac{\alpha}{2} +k }}{\Gamma \pare{1-\frac{\alpha}{2} +k} } \frac{1}{1-\frac{\alpha}{2} + k}
=     \sT_\alpha^1 \pare{j} 
\end{align*}
defined in \eqref{eq:T1alpha},
which in turn, recalling \eqref{svifour1}, implies that
$ L_{\nabla E_\alpha, 1} \phi_+ \pare{x}  =  - \frac{c_\alpha}{2\pare{1-\frac{\alpha}{2}}}    \sum_{j \geq 1 }  \sT_\alpha^1\pare{j}  \hat{\phi}\pare{j} e^{\ii j x} $. 
Since $ L_{\nabla E_\alpha, 1} $ is a real operator
$   \sT_\alpha^1 \pare{-j} =   \overline{ \sT_\alpha^1 \pare{j}} =   \sT_\alpha^1 \pare{j} $
which, combined with    $ L_{\nabla E_\alpha, 1} \hat{\phi}(0) = 0 $, gives us \eqref{eq:step:1}.
\end{step}

\begin{step}[Evaluation of $ L_{\nabla E_\alpha , 2} $]\label{step:2}
We claim that
\begin{equation}\label{eq:Malpha1}
\pare{L_{\nabla E_\alpha, 2} \, \phi} \pare{x} =  \frac{c_\alpha}{2\pare{1-\frac{\alpha}{2}}}  \ \av{D}^2 \  \pare{\sM_{\alpha}  \pare{\av{D}} \phi } (x) \text{ .}
\end{equation}
Setting  $ y=x-z $ we have
$ (L_{\nabla E_\alpha , 2} \, \phi) (x) =
\frac{c_\alpha}{2\pare{1-\frac{\alpha}{2}}}  \fint \frac{\phi'\pare{x-z} \sin z }{ \bra{ 2\pare{1-\cos z}}^{\frac{\alpha}{2}}} \ \dd z $. 
As
$ \frac{ \sin z }{ \bra{ 2\pare{1-\cos z}}^{\frac{\alpha}{2}}} = \frac{1}{2\pare{1-\frac{\alpha}{2}}} \partial_z \pare{ \bra{2\pare{1-\cos z}}^{1-\frac{\alpha}{2}}  } $, integrating by parts
\begin{equation}\label{LabE2a}
(L_{\nabla E_\alpha , 2}\, \phi)(x) =  \  -
\frac{c_\alpha}{\bra{2 \pare{1-\frac{\alpha}{2}}}^2} \fint \frac{\partial_z \big[ \phi'\pare{x-z} \big]}{\bra{2\pare{1-\cos z}}^{\frac{\alpha}{2} - 1} } \dd z  =
\frac{c_\alpha}{\bra{2 \pare{1-\frac{\alpha}{2}}}^2} \fint \frac{ \phi'' \pare{x-z}}{\bra{2\pare{1-\cos z}}^{\frac{\alpha}{2} - 1} } \dd z \, .
\end{equation}
For $ j \geq 0 $ we compute
\begin{equation*}
\begin{aligned}
I_j \defeq \fint \frac{e^{-\ii j z}}{ \bra{ 2\pare{1-\cos z }}^{\frac{\alpha}{2}- 1}}  \ \dd z
= &  \frac{1}{\pi} \int_0^{\pi} e^{- \ii j z}
{ \bra{ 2\pare{1-\cos z }}^{\frac{\alpha}{2}- 1}}  \ \dd z
=  \frac{2^{2-\alpha} }{\pi} \int_0^{\pi}  e^{- \ii \ 2j z} \pare{\sin z} ^{2-\alpha } \dd z
\end{aligned}
\end{equation*}
and applying \cref{eq:magic_identity_1} with
$ X = 2 - \alpha $, $ Y = - 2 j  $, and  using $ \Gamma (x+1)
= x \Gamma (x) $, we obtain
\begin{equation*}
\fint \frac{e^{- \ii j z}}{ \bra{ 2\pare{1-\cos z }}^{\frac{\alpha}{2}- 1}}  \ \dd z
=
\pare{-1}^j  \frac{(2-\alpha) \Gamma\pare{2-\alpha }}{\Gamma\pare{2 -j-\frac{\alpha}{2}}\Gamma\pare{2 +j-\frac{\alpha}{2}}} \, .
\end{equation*}
We use now the  identities
\begin{align*}
\Gamma\pare{2-\frac{\alpha}{2} +j} =  \pare{1-\frac{\alpha}{2} + j} 
\  \Gamma\pare{1-\frac{\alpha}{2} +j} \, ,
&&
\Gamma\pare{2-\frac{\alpha}{2} - j} = \pare{1-\frac{\alpha}{2} -j}\pare{-1}^j \frac{\Gamma\pare{1-\frac{\alpha}{2}} \Gamma\pare{\frac{\alpha}{2}} }{\Gamma\pare{\frac{\alpha}{2} +j}} \, ,
\end{align*}
which follows by  $ \Gamma\pare{1+z}= z \ \Gamma\pare{z} $ and \eqref{eq:consequence_Euer_duplication}
(with $ z = - \alpha / 2 $ and $ j \leadsto j-  1 $),
to deduce,  for any $ j \geq 0 $, 
\begin{equation*}
I_j = \fint \frac{e^{- \ii jz}}{ \bra{ 2\pare{1-\cos z }}^{\frac{\alpha}{2}- 1}}  \ \dd z
=
\frac{\Gamma\pare{2-\alpha}}{\Gamma\pare{1-\frac{\alpha}{2}} \Gamma\pare{ \frac{\alpha}{2}} } \frac{2-\alpha}{\pare{ 1-\frac{\alpha}{2} }^2 -j^2 } \frac{\Gamma\pare{\frac{\alpha}{2} +j}}{\Gamma\pare{ 1 - \frac{\alpha}{2} +j}} =
 - 2\pare{1- \frac{\alpha}{2}} \sM_\alpha\pare{j}
\end{equation*}
with $ \sM_\alpha $ defined in \eqref{eq:Malpha}.
Since  $ I_{j}  = I_{-j} $  we conclude that
\begin{equation}
\label{eq:sMalpha_integral}
\fint \frac{e^{- \ii j z}}{ \bra{ 2\pare{1-\cos z }}^{\frac{\alpha}{2}- 1}}  \ \dd z
=
- 2\pare{1- \frac{\alpha}{2}} \sM_\alpha\pare{\av{ j }} \, .
\end{equation}
By \eqref{LabE2a} and
 \eqref{eq:sMalpha_integral} we deduce  \eqref{eq:Malpha1}.
\end{step}

\begin{step}[Evaluation of $ L_{\nabla E_\alpha , 3} $]\label{step:3}
The action of the operator $L_{\nabla E_\alpha , 3} $ in \eqref{eq:linearization_gradient_pseudoenergy_1}
on a function $ \phi (x) = \sum_{j \in \Z} \hat{\phi}\pare{j}\ e^{\ii jx}   $ is, 
setting $ y  = x - z $ and using \eqref{eq:sMalpha_integral},
\begin{equation}\label{eq:step:3}
\pare{L_{\nabla E_\alpha , 3}\phi}(x)  =  \ \frac{ c_\alpha}{4 } \sum_{j \in \Z} 
\hat{\phi}\pare{j}\ e^{\ii jx} \ \fint \frac{e^{- \ii jz}}{ \bra{ 2\pare{1-\cos z }}^{\frac{\alpha}{2}- 1}}  \ \dd z= \ -  \frac{c_\alpha}{2\pare{1-\frac{\alpha}{2}}}  \pare{1- \frac{\alpha}{2}}^2 \  
\pare{\sM_\alpha\pare{\av{ D }}  \phi}(x)  \, .
\end{equation}
\end{step}

\begin{step}[Evaluation of $ L_{\nabla E_\alpha , 4} $]
By \eqref{eq:sMalpha_integral} with $ j = 0 $ and \eqref{eq:Malpha} it results
that $L_{\nabla E_\alpha , 4} $ in \eqref{eq:linearization_gradient_pseudoenergy_1}
is
\begin{equation}\label{eq:Step3est1}
\pare{L_{\nabla E_\alpha , 4}\phi}(x)  =
\frac{ c_\alpha}{4 }  \fint \frac{ \phi\pare{x} }{ \bra{ 2\pare{1-\cos z }}^{\frac{\alpha}{2}- 1}} \ \dd z = \frac{c_\alpha}{2\pare{1-\frac{\alpha}{2}}}     \frac{\Gamma\pare{2-\alpha }}{\Gamma\pare{1 - \frac{\alpha}{2}}^2} \ \phi\pare{x} 	\, .
\end{equation}
In conclusion, by \eqref{eq:linearization_gradient_pseudoenergy_1},  \eqref{eq:step:1}, \eqref{eq:Malpha1}, \eqref{eq:step:3}, \eqref{eq:Step3est1} we deduce
that  $ \dd \nabla E_\alpha (0)$
is equal to $ L_\alpha\pare{\av{D}}  $ in  \eqref{eq:Lalpha}.
\end{step}

\begin{step}[$ \sT^1_\alpha\in \wt \Gamma^{\max\pare{0,\alpha-1}}$
and  $ \sT^2_\alpha\in  \wt \Gamma^{\alpha-1} $]
 We start with $ \sT^2_\alpha (|\xi|) $  in \eqref{eq:T2alpha}
 which is defined on $\mathbb{R}$.   We recall
 the   asymptotic expansion  for $ \av{\xi} \to\infty $, see
   \cite[Eq. (5.11.13)]{NIST},
  \begin{multline}\label{eq:asymptgamma00}
  \frac{\Gamma\pare{\xi + a}}{\Gamma \pare{ \xi + b }} - \xi^{a-b} \pare{\sum_{\kappa=0}^{N}\frac{G_\kappa\pare{a,b}}{\xi^\kappa}} = \cO\pare{\av{\xi}^{ \pare{a-b} - \pare{ 1 + N }}}\,
\\
\begin{aligned}
G_0\pare{a, b} = 1 \, , \ G_1\pare{a,b} = \frac{\pare{ a-b }\pare{a+b-1}}{2} \, ,
&&
\forall N\in \bN \, , \  \av{\textnormal{Arg} \ \xi }  < \pi \, , 
\end{aligned}
\end{multline}
  which involves holomorphic functions. We can focus on the case $\text{Re} (\xi)>0$. We claim that  formula
 \eqref{eq:asymptgamma00} implies automatically the estimates for the derivatives
 \begin{equation}\label{simboloN}
\av{\partial_\xi^\mu
\pare{\frac{\Gamma\pare{\xi+a}}{\Gamma\pare{\xi+b}}}}\lesssim_\mu \xi^{a-b-\mu}
\quad \text{ for large $\xi>0$ and for any $\mu \in \bN$} \, .
\end{equation}
Case $ \mu = 0 $ of \eqref{simboloN} follows trivially from \eqref{eq:asymptgamma00}.
For any   $ \mu \in \bN\setminus \set{0} $ and for $ N_1 \gg \mu \geq 1 $, let us set
\begin{align*}
M_1\pare{\xi} \defeq  \xi^{a-b} \sum_{\kappa=0}^{N} \frac{G_\kappa\pare{a, b}}{\xi^\kappa} ,
&&
M_2 \pare{\xi} \defeq  \xi^{a-b} \sum_{\kappa=N+1}^{N + N_1} \frac{G_\kappa\pare{a, b}}{\xi^\kappa} ,
&&
E \pare{\xi} \defeq  \frac{\Gamma\pare{\xi + a}}{\Gamma \pare{ \xi + b }} - \pare{M_1\pare{\xi} + M_2\pare{\xi}} \, .
\end{align*}
Obviously $ M_1\pare{\xi}\in\wt \Gamma^{a-b  }_0 $ and  $ M_2\pare{\xi}\in\wt\Gamma^{a-b -\pare{N+1}}_0 $. For $   \xi \gg 1 $, $ E $ is   holomorphic  in $ B\pare{\xi, 2} $. Thus
$ \partial_\xi^\mu E\pare{\xi} = \frac{c_\mu}{2\pi \ii} \int _{\partial B \pare{\xi, 1}} \frac{E\pare{\zeta}}{\pare{ \zeta -\xi}^{1+\mu}} \ \dd \zeta $
by  the Cauchy formula.
Moreover \eqref{eq:asymptgamma00} is true   in $ B\pare{\xi, 2} $ and so, by $\av {\zeta }\sim   \av {\xi } $,
\begin{equation*}
\av{\partial_\xi^\mu E\pare{\xi}}\lesssim_\mu \int _{\partial B\pare{\xi, 1}}\frac{\av{\zeta}^{a-b-\pare{1+N + N_1}}}{\av{\zeta-\xi}^{1+\mu}} \ \av{\dd \zeta } \lesssim_\mu \av{\xi}^{a-b-\pare{1+N + N_1}} \lesssim_\mu \av{\xi}^{a-b-\pare{1+N + \mu }} ,
\end{equation*}
which implies 
\begin{equation*}
\av{\partial_\xi^\mu \pare{\frac{\Gamma\pare{\xi + a}}{\Gamma \pare{ \xi + b }} - M_1\pare{\xi}}} \leq \av{\partial_\xi^\mu M_2\pare{\xi}} + \av{\partial_\xi^\mu E \pare{\xi}} \lesssim_\mu \av{\xi}^{a-b-\pare{1+N + \mu }} \, .
\end{equation*}
This proves  \eqref{simboloN}.
From \eqref{simboloN} we conclude that
$ \sT^2_\alpha \pare{\av{j}}$ in \eqref{eq:T2alpha}  is a Fourier multiplier of order $ \alpha-1 $.

We now consider $ \sT^1_\alpha \pare{\av{j}}$ defined in \eqref{eq:T1alpha}. 
For any  $  j \in \bN_0 $,
 the discrete derivative of $ \sT^1_\alpha \pare{j} $  is
 $$
( \Delta \sT^1_\alpha) \pare{j} \defeq
\sT^1_\alpha\pare{j+1} - \sT^1_\alpha\pare{j} =
\frac{\Gamma\pare{2-\alpha}}{\Gamma\pare{1-\frac{\alpha}{2}}\Gamma\pare{\frac{\alpha}{2}}} \frac{\Gamma\pare{\frac{\alpha}{2} + j}}{\Gamma\pare{1-\frac{\alpha}{2} +j}}
\frac{1}{1- \frac{\alpha}{2}+ j} =
\frac{\sT^2_\alpha \pare{j}}{1- \frac{\alpha}{2}+ j} \, .
 $$
 Since $\sT^2_\alpha  $ is a symbol of order $ \alpha - 1 $ we deduce that 
 $ \av{\sT^1_\alpha\pare{j}}\lesssim  1 + j^{\alpha-1} $ and,  
 for any $ \ell  \in \bN $,
the discrete derivatives satisfy
$ \av{ \pare{ \Delta^\ell  \sT^1_\alpha } \pare{ j }}  \lesssim j^{\alpha-1- \ell}  $.
By \cite[Lemma 7.1.1]{SV2002} there exists a $ \cC^\infty $ extension of $ \sT^1_\alpha $ to the whole $ \bR $ which is a symbol of order  $ \max\pare{  \alpha - 1 , 0 }$.
\end{step}
The proof of \Cref{lem:linearization} is complete. 
\end{proof}

\begin{rem}\label{rem:nuovo}
For $ \alpha \neq 1 $ the Fourier multiplier
 $ \sT^1_\alpha \pare{\av{j}}$  in \eqref{eq:T1alpha} is equal to 
\begin{equation*} 
\sT^1_\alpha \pare{\av{j}}  
=
\frac{ \Gamma\pare{2-\alpha} }{\Gamma\pare{1-\frac{\alpha}{2}} \Gamma \pare{ \frac{\alpha}{2}  } } 
\frac{1}{\alpha -1}
\pare{\frac{\Gamma ( \tfrac{\alpha}{2} + \av{j}) }{\Gamma ( 1 + \av{j} - \tfrac{\alpha}{2}) } -  
\frac{\Gamma ( \tfrac{\alpha}{2} ) }{\Gamma ( 1  - \tfrac{\alpha}{2}) } }
 \end{equation*}
 as follows by induction.
 \end{rem}

 \begin{rem} \label{rem:regmis}  The first linear frequency
 $ \omega_\alpha (1)=0$.    This is equivalent to prove that $ L_\alpha (1)=0$, that,
in view of \eqref{eq:Lalpha}-\eqref{eq:T2alpha}, amounts to show that
\begin{equation*}
   \sT^1_\alpha\pare{1} - \sT^2_\alpha\pare{1} - \frac{\Gamma\pare{2-\alpha }}{\Gamma\pare{1 - \frac{\alpha}{2}}^2}  \\ =
   \frac{\Gamma\pare{2-\alpha}}{\Gamma\pare{1-\frac{\alpha}{2}}} \bra{     \frac{\Gamma \pare{ \frac{\alpha}{2}   }}{\Gamma \pare{ \frac{\alpha}{2}  }\Gamma \pare{1-\frac{\alpha}{2} } \pare{1-\frac{\alpha}{2}  }}  -       \frac{ \Gamma\pare{\frac{\alpha}{2} +1} }{ \Gamma \pare{ \frac{\alpha}{2}  } \Gamma\pare{2-\frac{\alpha}{2}  } } - \frac{1}{\Gamma\pare{1 - \frac{\alpha}{2}} }} =0 \, .
\end{equation*}
This holds true  because, using the identity $\Gamma (y+1)= y \ \Gamma (y ) $,
\begin{equation*}
         \frac{1}{ \Gamma \pare{1-\frac{\alpha}{2} } \pare{1-\frac{\alpha}{2}  }}  -       \frac{  \Gamma\pare{\frac{\alpha}{2}  }  \frac{\alpha}{2}   }{ \Gamma \pare{ \frac{\alpha}{2}  } \Gamma\pare{1-\frac{\alpha}{2}  } \pare{1-\frac{\alpha}{2}  } } - \frac{1}{\Gamma\pare{1 - \frac{\alpha}{2}} }   \\
          =
 \frac{1}{ \Gamma \pare{1-\frac{\alpha}{2} } \pare{1-\frac{\alpha}{2}  }}   \bra{   1  -  \frac{\alpha}{2}    -\pare{1 - \frac{\alpha}{2}} } =0 \, .
\end{equation*}
The fact that  the first frequency  $ \omega_\alpha (1) =0 $
is zero has  a dynamical proof.  Indeed, 
in view of the translation invariance of the problem,
the patch equation \eqref{eq:SQG_Hamiltonian} possesses the vector prime integral
\begin{equation}
\label{eq:primeintegral}
\int_{\bT} \pare{\sqrt{1+2f\pare{x}}-1} \vec \gamma\pare{x} \ \dd x   =
\int_{\bT} f\pare{x} (\cos x, \sin x) \ \dd x + O\pare{ \norm{f}^2 }  \, .
\end{equation}
Let us consider a dynamical system
$ \dot f = Y\pare{f}$ with $  Y(0) = 0 $ and $  A \defeq \dd Y (0) $.
If $ b \pare{f} $ is a prime integral then
$ \nabla b\pare{f} \cdot Y \pare{f} = 0 $, $  \forall f $.
Hence, differentiating and since $ Y (0) = 0 $, we obtain
$ \nabla b(0) \cdot A f  = 0 $, $ \forall f $.
If $ A $ is non singular then $ \nabla b(0) = 0 $, i.e. the prime integral
$ b$ is quadratic at  $ f = 0 $. Here the linear operator $ A  $ (cf. \eqref{lineqomegan})
is degenerate in the one-Fourier mode on which \eqref{eq:primeintegral} has a linear component in $ f $.
\end{rem}

The  other linear frequencies $ \omega_\alpha\pare{j}   $, $ j \neq 0, \pm 1 $, are all
different from zero.

\begin{lemma} [Convexity of $ \omega_\alpha \pare{j} $]\label{lem:conve}
Let $ \alpha\in \pare{0, 2} $.
The frequency map
$ j \mapsto \omega_\alpha\pare{j} = j \, L_\alpha\pare{\av{j}} $,
$ j \in\bZ $,
where $ L_\alpha $ is  computed in \Cref{lem:linearization}, is odd and
satisfies the  convexity property
\begin{equation}
\label{eq:dUC1}
\triangle^2 \omega_\alpha \pare{j}\defeq\omega_\alpha \pare{j+1}+\omega_\alpha \pare{j-1}-2\ \omega_\alpha \pare{j}    =     \frac{ \Gamma\pare{2-\alpha} }{2^{1-\alpha}\Gamma ^2\pare{1-\frac{\alpha}{2}}   }   \frac{\Gamma\pare{\frac{\alpha}{2}-1 + j   }}{\Gamma\pare{2- \frac{\alpha}{2} + j }}  \alpha j >0  \, , \quad \forall \ j \ge 1 \, .
\end{equation}
The linear frequencies $ \omega_\alpha  \pare{j} $ are different from zero
for any $ \av{j}\geq 2 $,
in particular $ \omega_\alpha \pare{j} > 0 $ and increasing for any $ j \geq 2 $.
\end{lemma}

\begin{proof}
 In view of  \Cref{lem:linearization}, for any $ j \geq 1 $, 
and the identity $ \Gamma (1+y) = y \Gamma (y) $,  
the second discrete derivative  $ \triangle^2 \omega_\alpha \pare{j} $ is equal to
\begin{multline}
\label{eq:NR_est0}
\frac{c_\alpha}{2\pare{1-\frac{\alpha}{2}}} 
 \frac{\Gamma\pare{2-\alpha}}{\Gamma\pare{1-\frac{\alpha}{2}}
 \Gamma\pare{\frac{\alpha}{2}}}
\left\lbrace\pare{j+1} \sum_{k=0}^j \frac{\Gamma \pare{ \frac{\alpha}{2} +k }}{\Gamma \pare{2-\frac{\alpha}{2} +k} } + \pare{j-1} \sum_{k=0}^{j-2} \frac{\Gamma \pare{ \frac{\alpha}{2} +k }}{\Gamma \pare{2-\frac{\alpha}{2} +k} }
-2  j \sum_{k=0}^{j-1} \frac{\Gamma \pare{ \frac{\alpha}{2} +k }}{\Gamma \pare{2-\frac{\alpha}{2} +k} }\right.  \\
- \left. \pare{j+1} \frac{\Gamma\pare{\frac{\alpha}{2} + j +1}}{\Gamma\pare{1- \frac{\alpha}{2} + j +1}}
- \pare{j-1} \frac{\Gamma\pare{\frac{\alpha}{2} + j -1}}{\Gamma\pare{1- \frac{\alpha}{2} + j -1}}
+2j \frac{\Gamma\pare{\frac{\alpha}{2} + j }}{\Gamma\pare{1- \frac{\alpha}{2} + j }}
\right\rbrace
\, .
\end{multline}
The first term inside  the above bracket is equal to
\begin{equation}\label{primal}
\begin{aligned}
 I = & \  \pare{j+1} \frac{\Gamma \pare{ \frac{\alpha}{2} +j }}{\Gamma \pare{2-\frac{\alpha}{2} +j} } - \pare{j-1}\frac{\Gamma \pare{ \frac{\alpha}{2} + j -1 }}{\Gamma \pare{2-\frac{\alpha}{2} +j-1} }   \\
= & \ \frac{\Gamma \pare{ \frac{\alpha}{2} + j -1 }}{\Gamma\pare{2-\frac{\alpha}{2} + j}} \bra{\pare{j+1}\pare{\frac{\alpha}{2} + j -1 } -\pare{j-1}\pare{1-\frac{\alpha}{2} + j}}
=  \ \frac{\Gamma \pare{ \frac{\alpha}{2} + j -1 }}{\Gamma\pare{2-\frac{\alpha}{2} +j}} \ \alpha j  \, .
\end{aligned}
\end{equation}
Writing the  terms in the 2nd line of the bracket in \eqref{eq:NR_est0} as
\begin{equation}\label{secondlin}
\begin{aligned}
- \pare{j+1} \frac{\Gamma\pare{\frac{\alpha}{2} + j +1}}{\Gamma\pare{1- \frac{\alpha}{2} + j +1}}
= & \  - \frac{\Gamma\pare{\frac{\alpha}{2} + j -1 }}{\Gamma\pare{2- \frac{\alpha}{2} + j }} \ \pare{j+1}\pare{ \tfrac{\alpha}{2} +j}\pare{ \tfrac{\alpha}{2} +j -1 }, \\
- \pare{j-1} \frac{\Gamma\pare{\frac{\alpha}{2} + j -1}}{\Gamma\pare{1- \frac{\alpha}{2} + j -1}} = & \ - \frac{\Gamma\pare{\frac{\alpha}{2} + j -1 }}{\Gamma\pare{2- \frac{\alpha}{2} + j }} \ \pare{j-1}\pare{ - \tfrac{\alpha}{2} +j}\pare{ 1 - \tfrac{\alpha}{2} +j  } , \\
2 j \frac{\Gamma\pare{\frac{\alpha}{2} + j }}{\Gamma\pare{1- \frac{\alpha}{2} + j }} = & \ \frac{\Gamma\pare{\frac{\alpha}{2} + j -1 }}{\Gamma\pare{2- \frac{\alpha}{2} + j }}
\ 2 j \pare{j^2 - \pare{1-\tfrac{\alpha}{2}}^2} \, ,
\end{aligned}
\end{equation}
we conclude by \cref{eq:NR_est0,primal,secondlin} and
since $ c_\alpha = \frac{\Gamma\pare{\frac{\alpha}{2}}}{2^{1-\alpha}\Gamma\pare{1-\frac{\alpha}{2}}} $ (cf. \eqref{eq:calpha}), that
\begin{equation*}
	   \triangle^2 \omega_\alpha \pare{j}= \frac{1}{2^{1-\alpha}\pare{2-\alpha}}   \frac{ \Gamma\pare{2-\alpha} }{\Gamma ^2\pare{1-\frac{\alpha}{2}}   }   \frac{\Gamma\pare{\frac{\alpha}{2}-1 + j   }}{\Gamma\pare{2- \frac{\alpha}{2} + j }} \  X_\alpha\pare{j}
\end{equation*}
where
\begin{equation*}
X_\alpha\pare{j} \defeq \alpha j - \pare{j+1} \pare{\tfrac{\alpha}{2} +j}\pare{\tfrac{\alpha}{2} -1+j}- \pare{j-1}\pare{ - \tfrac{\alpha}{2} +j}\pare{ 1 - \tfrac{\alpha}{2} +j  }  +2j \pare{j^2- \pare{1-\tfrac{\alpha}{2}} ^2  } =  (2- \alpha  )\alpha j \, .
\end{equation*}	
This proves \eqref{eq:dUC1}. 
The positivity of $ \triangle^2 \omega_\alpha \pare{j} $ in \eqref{eq:dUC1} follows because
 the function $ \Gamma $ is positive on positive numbers.
Finally, the convexity property \eqref{eq:dUC1}
and $ \omega_\alpha\pare{0} = \omega_\alpha\pare{1} =0$ (cf. Remark \ref{rem:regmis}) imply that
$ \omega_\alpha \pare{j} > 0 $ and increasing for any $ j \geq 2 $. 
\end{proof}

The next lemma 
is crucial for the normal form construction of Section \ref{sec:quadratic_normal_forms}.

\begin{lemma}[Absence of three wave interactions]
\label{lem:nonres_cond}
Let $ \alpha\in\pare{0, 2} $. 
For any  $ n,j,k\in\bZ\setminus \set{0} $ satisfying
$ k = j + n $, it results
\begin{equation}
\label{eq:nonres_cond}
\av{\omega_\alpha\pare{k} -  \omega_\alpha\pare{j} -  \omega_\alpha\pare{n} } 
\geq \omega_\alpha\pare{2} > 0   \, .
\end{equation}
\end{lemma}

\begin{proof}
Since 
the map $ j \mapsto \omega_\alpha (j) $ is odd
and strictly increasing for  $ j \in \bN $,  
it is sufficient to consider the case $ k \geq j \geq n \geq 1 $, $ k = j + n $. 
Then, using that 
$\omega_\alpha (0) = \omega_\alpha\pare{1} =0 $,  
defining $  A_\alpha(\ell )\defeq \omega_\alpha (\ell)-\omega_\alpha (\ell-1) $,
we write by a  telescoping expansion, 
\begin{align*}
\omega_\alpha \pare{k}-\omega_\alpha \pare{j}-\omega_\alpha \pare{n}&= \sum _{q=1}^{j+n}\pare{\omega_\alpha \pare{q}-\omega_\alpha \pare{q-1}}-\sum _{q=1}^{j}\pare{\omega_\alpha \pare{q}-\omega_\alpha \pare{q-1}}-\sum _{q=1}^{n}\pare{\omega_\alpha \pare{q}-\omega_\alpha \pare{q-1}} \notag \\
& =  \sum _{q=1}^{n}\left (A_\alpha \pare{ q+j }-A_\alpha\pare{ q }  \right )  =  \sum _{q=1}^{n} \sum _{q'=1}^{j}   \left (A_\alpha\pare{ q+q' }-A_\alpha\pare{ q+q'-1   }  \right )   \notag \\
& =
 \sum _{q=1}^{n} \sum _{q'=1}^{j} \triangle^2 \omega_\alpha \pare{q+q'-1} \geq 
 \triangle^2 \omega_\alpha \pare{1} =  \omega_\alpha\pare{2} > 0 
\end{align*}
by \eqref{eq:dUC1} and 
\Cref{lem:conve}. 
This proves  \eqref{eq:nonres_cond}. 
\end{proof}

We finally prove an asymptotic expansion  of
the  frequencies $ \omega_\alpha \pare{j} $.
 We use the notation $ \sum_{j=p_1}^{p_2} a_j \equiv 0 $ if $ p_2 < p_1 $.
We   denote $ m_\beta $ 
a real Fourier multiplier of order $   \beta\in\bR $, and $ c_\alpha^\kappa $  real constants, which  may vary from line to line.

\begin{lemma}[Asymptotic behavior of $ L_\alpha \pare{\av{j}} $]
\label{prop:Lalpha_asymptotic}
Let 
\begin{equation}
\label{eq:bValpha}
\bV_\alpha \defeq
\system{
\begin{aligned}
&
  \frac{ \alpha c_\alpha}{2-\alpha} 
  \frac{\Gamma\pare{1-\alpha}}{\Gamma\pare{1-\frac{\alpha}{2}}^2} 
& \alpha & \neq 1 \, ,  \\
 &\frac{1}{\pi} \set{   \pare{\gamma_{\textnormal{EM}} - \frac{\pi^2}{12}- 2} + \sum_{k=1}^{\infty} \bra{ \frac{1}{\frac{1}{2} +k} - \frac{1}{k}\pare{1-\frac{1}{2k}}} }
 & \alpha & =1 \, ,
\end{aligned}}
\end{equation}
where  
$ \gamma_{\textnormal{EM}} \defeq \pare{ \lim_{n \to + \infty }
\sum_{k=1}^n \frac{1}{k} }- \log n
$ is the Euler-Mascheroni constant. 

Then the symbol  $ L_\alpha \pare{\av{j}}$    in \Cref{lem:linearization} has the following
asymptotic  expansion: for any $ \cK\in \bN, \ \cK\geq 3 $, 
\begin{itemize}
\item If $ \alpha\in\pare{0, 1}\cup\pare{1, 2} $   there exists real constants $ c^\kappa_\alpha, \ \kappa \in \set{3, \ldots , \cK-1}$ and 
a Fourier multiplier $ m_{\alpha-\cK} $ of order $ \alpha-\cK $ such that
\begin{equation}\label{expaLapha}
L_\alpha\pare{\av{j}} =  \bV_\alpha
+
\underbrace{\frac{c_\alpha}{2\pare{1-\frac{\alpha}{2}}} \frac{\Gamma\pare{3-\alpha}}{\Gamma\pare{1-\frac{\alpha}{2}}\Gamma\pare{\frac{\alpha}{2}}}  \frac{1}{\alpha-1}}_{\defeq c^1_\alpha} \  \av{j}^{\alpha - 1}
+ \sum_{\kappa =3}^{\cK-1} c_\alpha^\kappa \av{j}^{\alpha- \kappa}
+ m_{\alpha-\cK}\pare{ \av{j} } \, .
\end{equation}

\item If $ \alpha =1 $  there exists real constants $ c^\kappa_1, \ \kappa \in \set{3, \ldots , \cK-1} $ and a Fourier multiplier $ m_{1-\cK} $ of order $ 1-\cK $ such that
\begin{equation*}
L_1\pare{\av{j}} = \bV_1 +  \frac{1}{\pi} \log\av{j}   + \sum_{\kappa=3}^{\cK-1} c_1^\kappa \av{j}^{1-\kappa} + m_{1-\cK}\pare{ \av{j} } \, .
\end{equation*}
\end{itemize}
\end{lemma}

Note that in the expansion \eqref{expaLapha} there is not a term as
$ c_\alpha^2 |j|^{\alpha-2}$ and that  $\frac{1}{\alpha-1} \  \av{j}^{\alpha - 1} $ is,
for $ \alpha \in (1,2) $,  positive and tends to infinity,
whereas, for $ \alpha\in\pare{0, 1} $, it is  negative and tends to zero. 

We provide for completeness the expansion also in the cases $ \alpha =  1 $ and
$ \alpha \in (0,1) $, although  not needed for the proof of Theorem \ref{thm:main}.

  \Cref{prop:Lalpha_asymptotic} is a direct consequence of
 \eqref{eq:T2alpha}, \eqref{eq:Lalpha} and \eqref{eq:calpha}  
 % $ \Gamma (1) = 1 $, $ \Gamma ( \frac12) = \sqrt{\pi } $,
and the following lemma. 

\begin{lemma}\label{lem:asympt}
For any $ \cK\in \bN, \cK\geq 3  $, the following holds:
\begin{itemize}
\item  if $ \alpha\in \pare{0, 1}\cup\pare{1, 2} $, there exist real constants $ c_\alpha^\kappa $, $ \ \kappa \in\set{3, \ldots, \cK-1} $ such that 
\begin{equation}\label{eq:T1_second_expansion}
\sT^1_\alpha \pare{\av{j}} =
  \frac{ \Gamma\pare{1-\alpha} }{\Gamma\pare{1-\frac{\alpha}{2}}^2  }  
+ \frac{\Gamma\pare{2-\alpha}}{\Gamma\pare{1-\frac{\alpha}{2}} \Gamma\pare{\frac{\alpha}{2}}}  \frac{1}{\alpha-1} \ \av{j}^{\alpha-1}  + \sum_{\kappa=3}^{\cK-1} c_\alpha^\kappa \ \av{j}^{\alpha-\kappa}
  + m_{\alpha-\cK} \pare{\av{j}} \, . 
\end{equation}
\item if $ \alpha=1 $ there exist real constants $ c_1 ^\kappa, \ \kappa \in\set{3, \ldots, \cK-1} $ such that
\begin{equation}
\label{eq:T1_third_expansion}
\sT^1_1\pare{\av{j}} = \frac{1}{\pi} \set{ \log\av{j} +  \pare{\gamma_{\textnormal{EM}} - \frac{\pi^2}{12}} + \sum_{k=1}^{\infty} \bra{ \frac{1}{\frac{1}{2} +k} - \frac{1}{k}\pare{1-\frac{1}{2k}}} }  + \sum_{\kappa =3}^{\cK-1} c_1 ^\kappa \ \av{j}^{-\kappa}
  + m_{1-\cK} \pare{\av{j}} ; 
\end{equation}

\item if $ \alpha \in \pare{0, 2} $ there exist real constants $ c_\alpha^\kappa, \ \kappa\in\set{3, \ldots, \cK-1} $ such that
 \begin{equation}
\label{eq:asymptotic_Malpha2}
\sM_{\alpha }\pare{\av{j}}
=
\frac{\Gamma\pare{2-\alpha}}{\Gamma\pare{1-\frac{\alpha}{2}} \Gamma\pare{ \frac{\alpha}{2}} }  \frac{1}{\av{j}^2 - \pare{ 1-\frac{\alpha}{2} }^2  } \ \bra{ \av{j}^{\alpha-1}
+ \sum_{\kappa=3}^{\cK-1} c_\alpha^\kappa \ \av{j}^{\alpha-\kappa }
  + m_{\alpha-\cK} \pare{\av{j}} }.
\end{equation}
\end{itemize}
\end{lemma}

\begin{proof}
By the proof of  \Cref{lem:linearization} (below \eqref{eq:asymptgamma00}) we know that 
 \begin{equation*}%\label{eq:asymptgamma00symb}
 \frac{\Gamma\pare{\xi + a}}{\Gamma \pare{ \xi + b }} - \xi^{a-b} \sum_{\kappa=0}^{N} \frac{G_\kappa \pare{a, b}}{\xi^\kappa} = m_{a-b-\pare{ 1+N }} \pare{\xi} \, ,
  \end{equation*}
  where $ m_{a-b-\pare{ 1+N }} \pare{\xi}$ is a Fourier multiplier in $ \wt \Gamma^{a-b-\pare{ 1+N }}_0 $,
and therefore,  for any $ \cK\geq 3 $, 
 \begin{equation}  \label{eq:Gammaratios0}
 \frac{\Gamma\pare{\frac{\alpha}{2} + \av{j} }}{\Gamma\pare{ 1 - \frac{\alpha}{2} +\av{j} }} =  \ \av{j}^{\alpha-1}
+ \sum_{\kappa=2}^{\cK-1} G_\kappa \pare{\frac{\alpha}{2}, 1-\frac{\alpha}{2}} \ \av{j}^{\alpha-\pare{1+\kappa}}
  + m_{\alpha-\cK} \pare{\av{j}}   \, ,
 \end{equation}
 where we exploited  that $ G_0 (a,b) = 1 $ and 
 $ G_1\pare{\frac{\alpha}{2} , 1 - \frac{\alpha}{2}}=0$, by \eqref{eq:asymptgamma00}.
 By Remark \ref{rem:nuovo} and \eqref{eq:Gammaratios0} we deduce \eqref{eq:T1_second_expansion} for $ \cK = 3 $. 
Finally  \eqref{eq:T1_third_expansion} for 
$ \alpha = 1 $ follows by the asymptotic estimate  
of the harmonic numbers 
$
\sum_{k=1}^{j} k^{-1} = 
%\zeta\pare{\nu} + \frac{j^{1-\nu}}{1-\nu} +\frac{j^{-\nu}}{2} +m_{-\pare{1+\nu}}\pare{j}
% \qquad \text{if} \quad \nu \neq 1 \, , \\
\gamma_{\textnormal{EM}}  + \log \pare{j} + \frac{1}{2j} + m_{-2} \pare{j} $. 
\end{proof}

\section{Paralinearization of the  Hamiltonian scalar field}
\label{sec:paralinearization}

The main result of this section is the following.

\begin{theorem}[Paralinearization of the $ \alpha $-SQG patch equation]
\label{prop:paralinearization_1}
Let $ \alpha\in\pare{0,1}\cup\pare{1,2} $ .  
Let $ N \in \bN $ and $ \rho \geq 0 $.
For any  $ K\in\bN_0 $,
there exist $  s_0 > 0 $,  $  \epsilon_0 > 0 $ such that,
if $ f\in \Ball{K}{s_0} $   solves    \cref{eq:SQG_Hamiltonian} then
\begin{equation}
\label{eq:paralinearized_1}
\pa_t f
 +\partial_x \circ \OpBW{\pare{ 1+\nu\pare{f; x} }  L_{\alpha}\pare{\av{\xi}} + V\pare{f; x}  + P \pare{f; x , \xi}  }  \  f \\
 = R \pare{f} f
\end{equation}
where
\begin{itemize}
\item $  L_{\alpha}\pare{\av{\xi}} $ is the
real valued Fourier multiplier   of order $ \max\{ 0,\alpha-1 \} $
defined   in \Cref{lem:linearization};

\item
$\nu\pare{f; x}, V \pare{f; x} $ are real valued functions
in $ \Sigma \cF^\bR_{K, 0, 1}\bra{\epsilon_0 , N} $ (see \Cref{def:functions});

\item
$ P \pare{f; x , \xi} $ is a symbol in
$ \Sigma \Gamma^{-1}_{K, 0, 1}\bra{\epsilon_0 , N} $ (see \Cref{def:symbols}) satisfying  \eqref{areal};

\item $ R \pare{f} $ is a real smoothing operator in
$ \Sigma \dot \cR^{-\rho}_{K, 0,1}\bra{\epsilon_0 , N}  $ (see \Cref{def:smoothing}).
\end{itemize}
 \end{theorem}

Note that, since the symbol  
$ \pare{ 1+\nu\pare{f; x} }  L_{\alpha}\pare{\av{\xi}} + V\pare{f; x}  $
is real, 
the vector field in \eqref{eq:paralinearized_1} 
 is linearly Hamiltonian up to zero order operators. 
 
 \subsection{Isolating the integral terms}
 
 \noindent
{\bf Notation.}
 In this section  we use the following auxiliary functions
\begin{align}\label{eq:notation}
r = r\pare{f;x} \defeq \sqrt{1+2f\pare{x}} \, , && \delta_z f \defeq f\pare{x}-f\pare{x-z},
&&  \Delta_z f \defeq \frac{\delta_z f}{2\sin\pare{z/2}} \, , \ 
\forall z \in \T \setminus \{0\}  \, .
\end{align}
We shall denote by $ P\pare{f;x, \xi} $ a symbol  in $ \Sigma \Gamma^{-1}_{K,0,1}\bra{\epsilon_0, N} $
 (see \Cref{def:symbols})  
by $  R\pare{f} $ a smoothing operator in $ \Sigma \cR^{-\rho}_{K,0,1}\bra{\epsilon_0, N} $
 (see \Cref{def:smoothing}) 
and by $ R \pare{f;z} $ 
a Kernel-smoothing operator in 
$ \Sigma \KR^{-\rho,0}_{K,0,1}\bra{\epsilon_0, N} $ (see \Cref{def:KR}), whose explicit expression may vary from line to line.

 Note that  $ r\pare{f;x} $ is a  function in $ \Sigma \cF^\bR_{K,0,0}\bra{\epsilon_0, N} $
 and,   according to \Cref{def:KM},  
\begin{equation}\label{eq:deltaDelta}
\delta_z \in \wt \KM^{\ 1,1}_0 \, .
\end{equation}
In view of \eqref{eq:gradient-pseudoenergy}
and performing the change of variable $ y = x - z $, 
the gradient $ \nabla E_\alpha \pare{f}$   can be decomposed as
\begin{equation}\label{eq:deco_Hamiltoniana}
\begin{aligned}
\nabla E_\alpha \pare{f} = & \ \nabla E_\alpha^{\pare{1}} \pare{f} + \nabla E_\alpha^{\pare{2}} \pare{f}  \, ,  \\
\nabla E_\alpha^{\pare{1}} \pare{f} \defeq &  \frac{c_\alpha}{2\pare{1-\frac{\alpha}{2}}} \fint \frac{1+2f\pare{x-z} -\sqrt{1+2f\pare{x}} \ \sqrt{1+2f\pare{x-z}}\cos z  }
{ \bra{2
\pare{1+f\pare{x} + f\pare{x-z} - \sqrt{1+2f\pare{x}}\sqrt{1+2f\pare{x-z}}\cos z}
}^{\frac{\alpha}{2}} } \ \dd z \, ,
\\
\nabla E_\alpha^{\pare{2}} \pare{f} \defeq &  \frac{c_\alpha}{2\pare{1-\frac{\alpha}{2}}} \fint \frac{ \sqrt{\frac{1+2f\pare{x}}{1+2f\pare{x-z}} }  \ f '\pare{x-z} \sin z }{ \bra{2
\pare{1+f\pare{x} + f\pare{x-z} - \sqrt{1+2f\pare{x}}\sqrt{1+2f\pare{x-z}}\cos z}
}^{\frac{\alpha}{2}} } \ \dd z  \, .
\end{aligned}
\end{equation}
Then, recalling the notation in \eqref{eq:notation}, we write
\begin{equation} \label{eq:nablaHalpha(1)}
\nabla E_\alpha^{\pare{1}} \pare{f} =  \ \frac{c_\alpha}{2\pare{1-\frac{\alpha}{2}}} \fint \frac{ r^2 - 2\delta_z f -r \ \sqrt{r^2 - 2\delta_z f}\cos z  }
{ \bra{ 2\pare{ r^2 - \delta_z f  - r\sqrt{r^2 - 2\delta_z f}\cos z
} }^{\frac{\alpha}{2}} } \ \dd z  = \ \frac{c_\alpha}{2\pare{1-\frac{\alpha}{2}}} \ r^{2-\alpha} 
\fint
\mathsf{G}^1 _{\alpha , z}\pare{\frac{\delta_z f}{r^2}}  \ \dd z  
\end{equation}
with
\begin{equation}\label{eq:G1}
\mathsf{G}^1 _{\alpha , z} \pare{\sX} \defeq \ \frac{ 1 - 2 \sX - \sqrt{1 - 2\sX }\cos z  }
{ \bra{ 2\pare{  1 -\sX  -  \sqrt{1 - 2\sX}\cos z }
}^{\frac{\alpha}{2}} }  \, ,
\end{equation}
and
\begin{equation}\label{defHa2}
\nabla E^{\pare{2}}_\alpha \pare{f} =  
\ \frac{c_\alpha}{2\pare{1-\frac{\alpha}{2}}} \ 
\frac{1}{r^\alpha} 
\underbrace{\fint \sG^2 _{\alpha, z} \pare{\frac{\delta_z f}{r^2}} 
\  f'\pare{x-z}\sin z \ \dd z  }_{=: \cJ \pare{f}  }
=   \frac{c_\alpha}{2\pare{1-\frac{\alpha}{2}}} \ \frac{1}{r^\alpha}  \ \cJ \pare{f} 
\end{equation}
with
\begin{equation}
\label{eq:sG2}
\sG^2 _{\alpha, z}\pare{\sX}
\defeq   
 \frac{\frac{1 }{\sqrt{1-2 \sX }} }{\bra{2\pare{1- \sX  - \sqrt{1-2\sX  }\cos z}}^{\frac{\alpha}{2}}}.
\end{equation}
By \eqref{eq:deco_Hamiltoniana},  recalling \eqref{eq:Pi0bot} and that
$ \nabla E_{\alpha}^{\pare{1}}\pare{0} $ is a constant, using 
 \eqref{eq:nablaHalpha(1)}, 
\eqref{defHa2}, the  equation \eqref{eq:SQG_Hamiltonian}  
can be written as 
\begin{align}
\pa_t f
& =  \ \partial_x \bra{ \pare{ \nabla E_{\alpha}^{\pare{1}}\pare{f} - \nabla E_{\alpha}^{\pare{1}}\pare{0} } +  \nabla E_{\alpha}^{\pare{2}}\pare{f}  }  \notag 
\\
& = 
\frac{c_\alpha}{2\pare{1-\frac{\alpha}{2}}} 
\ 
\partial_x
\bra{
\pare{ r^{2-\alpha}  \Delta I \pare{f} }    + \fint  \mathsf{G}^1_{\alpha , z}\pare{0} \dd z  \ 
  \pare{r^{2-\alpha}-1}
+ 
\frac{1}{r^\alpha}  \ \cJ \pare{f} 
 }  \label{eq:fIJ}
\end{align} 
where 
\begin{equation}\label{eq:DeltaI_def}
\Delta I \pare{f} \defeq  \fint \mathsf{G}^1_{\alpha , z}\pare{\frac{\delta_z f}{r^2}} 
- \mathsf{G}^1_{\alpha , z}\pare{0} \dd z \, .  
\end{equation}
By \eqref{eq:Step3est1}, \eqref{eq:G1} we get 
\begin{equation}\label{eq:intG1(0)}
\fint \sG^1_z\pare{0} \dd z = \frac{1}{2} \fint \bra{2\pare{1-\cos z}}^{1-\frac{\alpha}{2}} \dd z = \frac{1}{1-\frac{\alpha}{2}} \frac{\Gamma\pare{2-\alpha}}{\Gamma\pare{1-\frac{\alpha}{2}}^2} \, . 
\end{equation}
The terms $ \Delta I \pare{f} $ and $ \cJ\pare{f} $ are yet not in a suitable form  to be paralinearized, since  the nonlinear convolution kernels need to be desingularized
at $ z = 0 $. 

\begin{lemma}
The term $ \Delta I \pare{f} $ in \eqref{eq:DeltaI_def} can be written as 
\begin{equation}\label{eq:DeltaI_as_cI}
 \Delta I \pare{f} =    
\cI\pare{f}   + R \pare{f} f 
\end{equation}
where $ R\pare{f} $ is a real smoothing operator in 
$ \Sigma {\cR}^{-\rho}_{K, 0, 1}\bra{\epsilon_0, N} $, and  
\begin{equation}
\cI\pare{f}\defeq \label{eq:cI}
\fint \OpBW{\sK^1_{\alpha, z}\pare{\frac{\Delta_z f}{r^2}}} {  \frac{\delta_z f}{r^2\av{2\sin\pare{z/2}}^\alpha }} \dd z 
\end{equation}
%{\color{blue} is integrationally invariant as per \Cref{def:int_inv}}
with
\begin{equation}\label{eq:sK1} 
\begin{aligned}
\sK^1 _{\alpha, z}\pare{\sX}  &
\  \defeq \pare{ \sG^1 _{\alpha, z} }' \pare{2 \sX \sin\pare{z/2} } \av{2\sin\pare{z/2}}^\alpha 
\\
& \   \begin{multlined}
 = \left[ -\frac{2-\frac{\cos z}{\sqrt{1-4\sX \sin\pare{z/2}}}}{\bra{2\pare{1-2 \sX \sin\pare{z/2}-\sqrt{1-4\sX \sin\pare{z/2}}\cos z}}^{\frac{\alpha}{2}}} \right.  \\
 \  \ \left.  + \alpha \frac{\pare{1-\frac{\cos z}{\sqrt{1-4\sX \sin\pare{z/2}}}} \pare{1-4\sX \sin\pare{z/2}-\sqrt{1-4\sX \sin\pare{z/2} }\cos z}}{\bra{2\pare{1-2\sX \sin\pare{z/2}-\sqrt{1-4\sX \sin\pare{z/2}}\cos z}}^{\frac{\alpha}{2}+1}}\right] \av{2\sin\pare{z/2}}^\alpha \, .
\end{multlined}
\end{aligned}
 \end{equation}
The term $\cJ \pare{f}  $ in \eqref{defHa2} 
 can be written as 
\begin{equation}\label{eq:cJ}
 \cJ \pare{f} 
=
\fint \sK^2 _{\alpha, z} \pare{\frac{\Delta_z f}{r^2}} \  f'\pare{x-z}\frac{\sin z}{\av{2\sin\pare{z/2}}^\alpha} \ \dd z 
\end{equation}
where
\begin{equation}
\label{eq:sK2}
\sK^2_{\alpha, z}\pare{\sX} \defeq \sG^2_{\alpha, z}\pare{\sX \ 2\sin\pare{z/2}} \ \av{2\sin\pare{z/2}}^\alpha
=
\frac{\frac{1 }{\sqrt{1-4\sX \sin\pare{z/2} }} \av{2\sin\pare{z/2}}^\alpha }{\bra{2\pare{1-2 \sX \sin\pare{z/2} - \sqrt{1-4\sX \sin\pare{z/2} }\cos z}}^{\frac{\alpha}{2}}}\,.
\end{equation}
The functions $ z\mapsto \sK^{\mathsf{j}}_{\alpha, z}\pare{\frac{\Delta_zf}{r^2}} $, $ \mathsf{j}=1,2, $
are $ {2\pi} $-periodic.
\end{lemma}

\begin{proof}
Applying \Cref{lem:Bony_paralinearization_W}  to \eqref{eq:DeltaI_def} we get 
\begin{align}
\label{eq:DeltaI_paralinearizzazione_1.1}
 \Delta I \pare{f} = \fint \OpBW{\pare{ \sG^1_{\alpha, z} }'\pare{\frac{\delta_z f}{r^2}}}  \frac{\delta_z f}{r^2} \dd z +  \fint R \pare{\frac{\delta_z f}{r^2}}  \frac{\delta_z f}{r^2}  \dd z 
\end{align}
where $ R $ is a smoothing operator in $ \Sigma \cR^{-\rho}_{K, 0, 1}\bra{\epsilon_0, N} $ and, recalling \eqref{eq:G1},  
\begin{equation}
 \label{eq:G1'}
\pare{ \sG^1 _{\alpha, z} }' \pare{\sX} =
-\frac{2-\frac{\cos z}{\sqrt{1-2\sX}}}{\bra{2\pare{1-\sX-\sqrt{1-2\sX}\cos z}}^{\frac{\alpha}{2}}} + \alpha \frac{\pare{1-\frac{\cos z}{\sqrt{1-2\sX}}} \pare{1-2\sX-\sqrt{1-2\sX }\cos z}}{\bra{2\pare{1-\sX-\sqrt{1-2\sX}\cos z}}^{\frac{\alpha}{2}+1}}.
\end{equation}
In view of \eqref{eq:G1'}, \eqref{eq:sK1} 
we have that  
$$ 
\sK^1_{\alpha, z}\pare{\sX} = \pare{ \sG^1 _{\alpha, z} }' \pare{2 \sX \sin\pare{z/2} } \av{2\sin\pare{z/2}}^\alpha 
$$ 
so that the first term on the right hand side of \eqref{eq:DeltaI_paralinearizzazione_1.1} is equal to $ \cI\pare{f} $ in \eqref{eq:cI}. 
Notice that, since $ \Delta_{z+2\pi} f = -\Delta_z f $ and $ \sK^1_{\alpha, z+2\pi}\pare{ - \mathsf{X}} = \sK^1_{\alpha, z}\pare{  \mathsf{X}} $ (cf. \eqref{eq:sK1}), 
the map $ z\mapsto  \sK^1_{\alpha, z}\pare{\frac{\Delta_z f}{r^2}} $ is $ 2\pi $-periodic. 
Similarly  $ z\mapsto \sK^2_{\alpha, z}\pare{\frac{\Delta_z f}{r^2}} $ is $ 2\pi $-periodic.  
 
We now prove that 
\begin{align}
\label{eq:reg_integral}
 \fint R \pare{\frac{\delta_z f}{1+2f}}\pare{ \frac{\delta_z f}{1+2f } } \dd z
= R\pare{f}f \qquad \text{where} \qquad 
 R \pare{f} \in \Sigma  \cR^{-\rho}_{K, 0, 1}\bra{\epsilon_0, N} \, . 
\end{align} 
We write 
$$
R \pare{\frac{\delta_z f}{1+2f}}\pare{ \frac{\delta_z f}{1+2f } }
 =
 R \pare{M \pare{f;z} f} M \pare{f;z} f 
$$
where 
$$ 
M \pare{f;z} = M_{1}\pare{f} + M_{2}\pare{f;z} \, , \quad
 M_{1}\pare{f}  \defeq \frac{1}{1+2f} \, , \quad 
  M_{2}\pare{f;z}  = -\frac{\st _{-z} }{1+2f} \, . 
$$ 
Remark \ref{item:propMop2}
 shows that $ M_{1} \pare{f}\in \Sigma \cM^{0,0}_{K, 0, 0}\bra{\epsilon_0, N} $ and  \Cref{prop:comp_z}, \cref{item:MM_ext_z} proves   that $ M_{2}\pare{f;z}
 $ belongs to $ \Sigma \KM^{0,0}_{K, 0, 0}\bra{\epsilon_0, N} $. Thus  
$ M \pare{f; z}
 \in \Sigma \KM^{0, 0}_{K, 0, 0}\bra{\epsilon_0, N} $ and  
\Cref{prop:comp_z}, \cref{item:MM_int_z,item:MM_ext_z} give that
$$
  R \pare{\frac{\delta_z f}{1+2f}}\pare{ \frac{\delta_z f}{1+2f } }
 =
  R \pare{M \pare{f;z} f} M \pare{f;z} f =  R\pare{f;z} f 
$$
for some  Kernel-smoothing  operator 
$ R \pare{f;z}$ in $ \Sigma \KR^{-\rho, 0}_{K, 0, 1}\bra{\epsilon_0, N} $. Finally \Cref{prop:action_z_smoothing} implies \eqref{eq:reg_integral}.
\end{proof}

Plugging \eqref{eq:DeltaI_as_cI}  in \eqref{eq:fIJ} we obtain
\begin{equation}\label{eq:fcIcJ}
\pa_t f = 
\frac{c_\alpha}{2\pare{1-\frac{\alpha}{2}}} 
\ 
\partial_x
\bra{
 r^{2-\alpha}  \pare{   \cI\pare{f} 
 + R \pare{f} f}    + \fint  \mathsf{G}^1_{\alpha , z}\pare{0} \dd z \   \pare{r^{2-\alpha}-1}
+
  \frac{1}{r^\alpha}  \ \cJ \pare{f} 
}. 
\end{equation}

\subsection{Analysis of the nonlinear convolution kernels}

The goal of this section is to 
represent the nonlinear convolution kernels in \eqref{eq:cI}
and 
\eqref{eq:cJ} as Kernel-functions 
according to \Cref{def:kernel_functions}. In \Cref{sec:paralinearization_cJ} 
we shall consider 
 as well the convolution kernel
\begin{multline}
\label{eq:sK3}
\sK^3 _{\alpha, z}\pare{\sX}
\defeq  \pare{\sG^2_{\alpha, z}}' \pare{\sX \ 2\sin\pare{z/2}} \  \sin z \av{2 \sin (z/2) }^\alpha
 =
 \left[ \frac{\frac{1}{\pare{1-4\sX \sin\pare{z/2}}^{3/2}}}{\bra{2\pare{1-2\sX \sin\pare{z/2}-\sqrt{1-4\sX \sin\pare{z/2}}\cos z}}^{\frac{\alpha}{2}}}\right. \\
 \left.
 + \alpha \frac{\pare{1-\frac{\cos z}{\sqrt{1-4\sX \sin\pare{z/2}}}} \frac{1}{\sqrt{1-4\sX \sin\pare{z/2}}}}{\bra{2\pare{1-2\sX \sin\pare{z/2}-\sqrt{1-4\sX \sin\pare{z/2}}\cos z}}^{\frac{\alpha}{2}+1}} \right] \av{2\sin\pare{z/2}}^\alpha \ \sin z \, .  
\end{multline}

\begin{lemma}\label{lem:characterization_Kernels}
Let  $\sK^\mathsf{j}_{\alpha, z} (\sX) $, $ \alpha\in\pare{0,2}, \  \mathsf{j} = 1, 2, 3, $ be the functions 
defined  in \eqref{eq:sK1}, \eqref{eq:sK2} and \eqref{eq:sK3}. Then 
\begin{equation}\label{eq:Kernel_in_kernelfunction}
\sK^\mathsf{j}_{\alpha, z} \pare{\frac{\Delta_z f}{r^2}}  = 
\sK^\mathsf{j}_{\alpha, z} \pare{\frac{\Delta_z f}{1+2f}}
\in 
 \Sigma\KF ^{0}_{K, 0, 0}\bra{\epsilon_0, N} 
\end{equation}
is a Kernel function, which admits the expansion 
\begin{equation}
  \label{eq:Taylor_expansion_kernels0}
 \sK^\mathsf{j}_{\alpha, z}\pare{\frac{\Delta_z f}{r^2}} = \sK^{\mathsf{j},0}_{\alpha}\pare{f;x} +  \sK^{\mathsf{j},1}_{\alpha}\pare{f;x} \ \sin z
+ \sK^{\mathsf{j},2}_\alpha  \pare{f;x} \pare{2\sin\pare{z/2}}^2
  +\varrho^{\mathsf{j},3}_\alpha\pare{f;x,z} \,  ,
 \end{equation}
 where
 \begin{align}\label{prop:fr}
   \sK^{\mathsf{j}, l}_{\alpha} \pare{f;x}
\in \Sigma \cF^\bR_{K,0, \underline{p}\pare{\mathsf{j},l} }\bra{\epsilon_0 , N} 
\, , 
&& 
\varrho^{\mathsf{j},3}_\alpha\pare{f;x,z} \in \Sigma\KF ^{3}_{K, 0, \underline{q} \pare{\mathsf{j}}}\bra{\epsilon_0, N} \, 
&&
\underline{q}\pare{ \mathsf{j} } \defeq \system{
\begin{aligned}
&1 && \text{ if } \  \mathsf{j} = 1, 2 \, , \\
& 0 && \text{ if } \ \mathsf{j} = 3 \, , 
\end{aligned}
} 
 \end{align}
with  $ \underline{p}\pare{\mathsf{j},l} \in \set{0, 1} $ 
and {constant 
functions}
 \begin{align}
 \label{eq:coefficients_Kernels_zero}
 \pare{
\begin{array}{ccc}
\sK^{1,0}_{\alpha}\pare{0;x} & \sK^{2,0}_{\alpha}\pare{0;x} & \sK^{3,0}_{\alpha}\pare{0;x} \\
\sK^{1,1}_{\alpha}\pare{0;x} & \sK^{2,1}_{\alpha}\pare{0;x} & \sK^{3,1}_{\alpha}\pare{0;x} \\
\sK^{1,2}_{\alpha}\pare{0;x} & \sK^{2,2}_{\alpha}\pare{0;x} & \sK^{3,2}_{\alpha}\pare{0;x}
\end{array}
 }
 =
 \pare{
\begin{array}{ccc}
-1 & 1 & 0 \\
0 & 0 & 1+\frac{\alpha}{2} \\
-\frac{1}{2}\pare{1-\frac{\alpha}{2}} & 0 & 0 \\
\end{array}
 } .
 \end{align}
\end{lemma}

\begin{proof}
The statement \eqref{eq:Kernel_in_kernelfunction}
follows by 	\eqref{eq:Taylor_expansion_kernels0}-\eqref{prop:fr} that we now prove.
We first claim that for any $R> 0$, there exists $\varepsilon _R>0$  such that  the functions
 \begin{equation} \label{def:Jfun}
\mathsf{J}^\mathsf{j}_{\alpha , w}\pare{\mathsf{x}, \mathsf{y}}
\defeq
\sK^\mathsf{j}_{\alpha, z}\pare{\frac{\mathsf{y}}{1+2\mathsf{x} }} \, , \qquad
w\defeq2\sin\pare{z/2}  \, , 
\end{equation}
where $ \sK^\mathsf{j}_{\alpha, z} \pare{ \cdot } $, $\mathsf{j} = 1, 2 ,  3 $  are defined 
in 
\eqref{eq:sK1}, \eqref{eq:sK2} and \eqref{eq:sK3},
are
 analytic in $(\mathsf{x}, \mathsf{y},w)$  in the domain 
\begin{equation}
\label{eq:dominiow}
   |\mathsf{x}|\le  \varepsilon _R \, , \qquad 
   |\mathsf{y}|\le  \varepsilon _R \,, \qquad |w|\le R \, , 
\end{equation}
and there exists $C_R>0$ such that 
$\av{ \mathsf{J}^\mathsf{j}_{\alpha , w} \pare{\mathsf{x},\mathsf{y}} } \le C_R $ in this domain. 
Let us prove the analiticity of 
$ \mathsf{J}^\mathsf{1}_{\alpha , w}\pare{\mathsf{x}, \mathsf{y}} $.
Substituting $ \sX = \frac{\mathsf{y}}{1+2\mathsf{x}} $, $w = 2\sin\pare{z/2} $ and
 $ \cos (z) = 1 - \frac{w^2}{2} $ in \eqref{eq:sK1} we have  
 \begin{subequations}
 \begin{align}
\mathsf{J}^\mathsf{1}_{\alpha , w}\pare{\mathsf{x}, \mathsf{y}}  
&  =
   -\frac{2-\frac{1 - \frac{w^2}{2}}{\sqrt{1- 2 \sX w}}}{\bra{2
  \pare{1-  \sX w -\sqrt{1- 2 \sX w}
  + \sqrt{1- 2 \sX w}  \frac{w^2}{2}  }}^{\frac{\alpha}{2}}} \av{w}^\alpha  \label{J1pr} \\
 &  \ \  + \alpha 
\frac{1}{\sqrt{1- 2 \sX w }} 
 \frac{\pare{\sqrt{1- 2 \sX w } -
 1 + \frac{w^2}{2}   } 
 \pare{1- 2 \sX w - \sqrt{1- 2 \sX w } 
 + \sqrt{1- 2 \sX w }  \frac{w^2}{2} 
 }  }{\bra{2\pare{1-\sX w
 - \sqrt{1- 2 \sX w } 
 + \sqrt{1- 2 \sX w }  \frac{w^2}{2}
  }}^{\frac{\alpha+2}{2}}} \av{w}^\alpha \, .\label{J1pr2}
\end{align}
 \end{subequations}
Since the function 
$  1-  \sX  w - \sqrt{1- 2 \sX  w} = \pare{\sX w}^2 + O \pare{\sX w}^3  $ is analytic in $ \sX w $  small, 
the function in \eqref{J1pr} is analytic in the domain 
\eqref{eq:dominiow} for $ \varepsilon_R $ small enough. 
Furthermore, noting that  
the functions 
$\sqrt{1- 2 \sX w } - 1 = - \sX w + O(\sX w)^2 $ and 
$ 1 - 2 \sX w  - \sqrt{1- 2 \sX w }  = - \sX w + O(\sX w)^2$  are analytic  in $ \sX w $  small, 
we deduce that also \eqref{J1pr2} is analytic 
in $(\mathsf{x}, \mathsf{y},w)$   in the domain 
\eqref{eq:dominiow}. 
The analiticity of 
$ \mathsf{J}^\mathsf{2}_{\alpha , w}\pare{\mathsf{x}, \mathsf{y}} $ 
and  $ \mathsf{J}^\mathsf{3}_{\alpha , w}\pare{\mathsf{x}, \mathsf{y}} $
follow similarly.

Then by Cauchy integral formula, 
\begin{equation}\label{tayJex}
   \mathsf{J}^\mathsf{j}_{\alpha , w}\pare{\mathsf{x}, \mathsf{y}} 
   = \sum_{p_1,p_2=0}^{\infty} \underbrace{\frac{1}{p_1! p_2!}   \partial _{\mathsf{x}} ^{p_2} \partial _{\mathsf{y}} ^{p_1}  \mathsf{J}^\mathsf{j}_{\alpha , w}\pare{0, 0}}_{ =: \mathsf{k}_{\alpha; p_1, p_2}^\mathsf{j}\pare{z}   } 
    {\mathsf{x}} ^{p_2} {\mathsf{y}} ^{p_1} \quad \text{where} \quad 
    \av{\partial _{\mathsf{x}} ^{p_2} \partial _{\mathsf{y}} ^{p_1}  \mathsf{J}^\mathsf{j}_{\alpha , w}\pare{0, 0}}\le p_1! p_2! C_R \varepsilon _R ^{-p_1-p_2} \, . 
\end{equation}
In view of \eqref{def:Jfun} we have 
$ \mathsf{J}^\mathsf{j}_{\alpha , w}\pare{\mathsf{x}, 0}
\defeq
\sK^\mathsf{j}_{\alpha, z}\pare{0} $ for any $ \mathsf{x} $, and therefore 
\begin{equation}\label{quasizero}
\mathsf{k}^{\mathsf{j}} _{\alpha;0,0}\pare{z} =  \sK^\mathsf{j}_{\alpha, z}\pare{0}  \, , \qquad
\mathsf{k}^{\mathsf{j}} _{\alpha;0, p_2}\pare{z}\equiv 0 \, , \ \forall  p_2\geq 1 \, , 
\end{equation}
and, by \eqref{eq:sK1}, \eqref{eq:sK2}, \eqref{eq:sK3}, one computes that 
 \begin{equation}\label{eq:K0} 
    \sK^1_{\alpha, z}\pare{0} =- \frac{1}{2} \pare{1-\frac{\alpha}{2}} \pare{ 2\sin \pare{\frac{z}{2}} } ^2-1 \, ,  \qquad  \sK^2_{\alpha, z}\pare{0}  =1 \, ,  \qquad 
      \sK^3_{\alpha, z}\pare{0} =  \pare{1+\frac{\alpha}{2}} \sin \pare{z}  \, . 
 \end{equation}
Choosing above $ R = 4 $  and   $ \underline{{\varepsilon}}>0$ such that
$ \av{\sin\pare{z/2}}\le 2$ for $\av{\Im z}\le \underline{{\varepsilon}}$, we 
deduce that
each $ \mathsf{k}_{\alpha; p_1, p_2}^\mathsf{j}\pare{z} \defeq  \frac{1}{p_1! p_2!}   \partial _{\mathsf{x}} ^{p_2} \partial _{\mathsf{y}} ^{p_1}  \mathsf{J}^\mathsf{j}_{\alpha , w}\pare{0, 0} $, $\mathsf{j} = 1, 2 ,  3 $,    
satisfies
\begin{equation}\label{eq:jzineq}
\av{\partial_z^{l } \mathsf{k}_{\alpha; p_1, p_2}^\mathsf{j}\pare{z}} \leq          l!   \underline{{\varepsilon}}  ^{-l}    C_R \varepsilon _R ^{-p_1-p_2} \, , \qquad
 \forall  z\in \bR \, , \  l\ge 0 \, ,  p_1, p_2 \geq 0 \, .
\end{equation}
Now, in view of \eqref{def:Jfun}, \eqref{tayJex} and  \eqref{quasizero} we expand 
\begin{equation}
\label{eq:decomposition_Kernel_homogeneous}
\begin{aligned}
\sK^\mathsf{j}_{\alpha, z}\pare{\frac{\Delta_z f}{1+2f}}  = 
{ \mathsf{J}^\mathsf{j}_{\alpha , w}\pare{f , \Delta_z f } } & =
   \sK^\mathsf{j}_{\alpha, z}\pare{0}  +  \sum_{p\geq 1} 
   \underbrace{\sum_{\substack{p_1 \geq 1 \\ p_1 + p_2 = p }} \mathsf{k}_{\alpha; p_1, p_2}^\mathsf{j}\pare{z} \ f^{p_2} \pare{\Delta_z f}^{p_1} }_{ {=: 
   \wt \sK^{\mathsf{j},p}_{\alpha} \pare{f;x, z} } }\\
 & =   \sK^\mathsf{j}_{\alpha, z}\pare{0} +  \sum_{p = 1}^N \wt \sK^{\mathsf{j},p}_{\alpha} \pare{f;x, z} + \underbrace{\sum_{p > N} \wt \sK^{\mathsf{j},p}_{\alpha} \pare{f;x, z}  }_{ =: \wt \sK^{\mathsf{j},> N}_{\alpha} \pare{f;x, z}} \, . 
\end{aligned}
\end{equation} 
We claim that  for any $  p\in\bN $ and  $  \ell =0,\ldots , 7 $,
\begin{align}\label{eq:claim_Khomog}
 \partial_z^\ell \wt \sK^{\mathsf{j},p}_\alpha\pare{f; x, z}  \ & \in \wt \KF^0_p 
 \\
 \label{eq:tail_homogeneous_deco_kernels}
   \partial_z^\ell \wt \sK^{\mathsf{j}, > N}_\alpha\pare{f;x, z}   & \in \KF^0_{K, 0, N +1}\bra{\epsilon_0} \, . 
\end{align}

\begin{step}
[Proof of \eqref{eq:claim_Khomog}] 
We  expand in Fourier 
\begin{equation}\label{laFdae}
\partial_z^l \wt \sK^{\mathsf{j},p}_\alpha\pare{ f ; x, z} = \sum_{\vec\jmath_p \in \pare{\bZ\setminus \set{0}}^p } \partial_z^l\hat{\sK}^{\mathsf{j}, p}_{\vec\jmath_p}\pare{z} \ f_{j_1} \cdots f_{j_p} e^{\ii \pare{j_1+\cdots + j_p} x}  
\end{equation}
where, in view of \eqref{eq:decomposition_Kernel_homogeneous} and \eqref{eq:notation}, 
\begin{equation*}
\hat{\sK}^{\mathsf{j}, p}_{\vec\jmath_p}\pare{z} \defeq \sum_{\substack{p_1\geq 1 \\ p_1+p_2 =p}} \mathsf{k}_{\alpha;p_1,p_2}^\mathsf{j} \pare{z} \prod_{q=1}^{p_1}\triangle_{j_q}\pare{z} \, , \qquad \quad
\triangle_{j_q}\pare{z} \defeq \frac{1-e^{-\ii j_q z}}{2\sin\pare{z/2}} \, . 
\end{equation*}
For any $ l \in \N_0 $ we have
\begin{equation}\label{eq:Kappas_in_Fourier}
\partial_z^l \hat{\sK}^{\mathsf{j}, p}_{\vec\jmath_p}\pare{z} 
= 
\sum_{\substack{l_2+l_{1,1} + \cdots + l_{1,p_1} = l}} 
\sum_{\substack{p_1\geq 1 \\ p_1+p_2 =p}}  \ \partial_z^{l_2} \mathsf{k}_{\alpha;p_1,p_2}^\mathsf{j} \pare{z}   \prod_{q=1}^{p_1} \partial_z^{l_{1, q}}\triangle_{j_q}\pare{z}. 
\end{equation}
For any $ j \in \Z \setminus \{ 0 \} $, we may write 
$
\triangle_j\pare{z} = \frac{1-e^{-\ii jz}}{2\sin\pare{z/2}}
= \ii\,    \sum_{j'=0}^{\av{j}-1}  e^{-\ii   \sgn\pare{j}\pare{j' + \sgn\pare{j} \frac{1}{2}} z}
$
and then, for any $ l \in \N_0 $
\begin{align} \label{eq:bound_finite_difference_Fourier}
  \av{ \partial_z^{l}\triangle_j\pare{z}} \lesssim 
  \sum_{j'=1}^{\av{j} }  \pare{ j' }^l 
  \lesssim_l \av{j}^{l+1}  \, ,  \quad \forall z\in\bR \, .
\end{align}
By \eqref{eq:jzineq},  \eqref{eq:bound_finite_difference_Fourier}, 
we estimate \eqref{eq:Kappas_in_Fourier} as (the constant $ \varepsilon_{4} $ 
is the one in \eqref{eq:jzineq} for $ R=4 $)
\begin{align}
\av{\partial_z^l \hat{\sK}^{\mathsf{j}, p}_{\vec\jmath_p}\pare{z} }
& \leq  \ 
\sum_{\substack{l_2+l_{1,1} + \cdots + l_{1,p_1} = l}} 
\sum_{\substack{p_1\geq 1 \\ p_1+p_2 =p}}  C_{l} \  \varepsilon _{4} ^{-p_1-p_2} \prod_{q=1}^{p_1}   \av{j_q}^{l_{1,q}+1} \notag  \\
&  \lesssim_{l}   \   p^2  \pare{ \frac{l}{\varepsilon_{4}} }^{p}\av{\vec\jmath_p}^l \prod_{q=1}^{p_1} \av{j_q}  \ \leq  C_l^p \av{\vec\jmath_p}^l  \ \prod_{q=1}^{p} \av{j_q} \label{eq:est_Fouirer_homog_K}
\end{align} 
for some constant  $ C_l > 0 $, for any $ z \in \bT $.
The bound \eqref{eq:est_Fouirer_homog_K} 
implies, recalling \Cref{def:kernel_functions}, 
  the claim  \eqref{eq:claim_Khomog}.   
\end{step}

\begin{step}[Proof of \eqref{eq:tail_homogeneous_deco_kernels}]
Recalling \eqref{eq:decomposition_Kernel_homogeneous} and 
\eqref{laFdae}, we have, 
for any $ 0\leq k\leq K $, $ l, \gamma \in \N_0 $,  
  \begin{align*}
  \av{\partial_t^k \partial_x^\gamma \partial_z^l \wt \sK^{\mathsf{j}, > N}_\alpha\pare{f;x, z}}
  \leq & \ 
  \sum_{p> N}
  \sum_{\vec\jmath_p \in \pare{\bZ\setminus \set{0}}^p }
  \av{\vec\jmath_p}^\gamma \av{\partial_z^l \hat \sK^{\mathsf{j}, p}_{\vec\jmath_p}\pare{z}} \av{\partial_t^k \pare{f_{j_1}\cdots f_{j_p}}}
  \\
  %--------------------------------------------
  \leq & \ 
   \sum_{p> N}
  \sum_{\vec\jmath_p \in \pare{\bZ\setminus \set{0}}^p }
  \sum_{k_1+\ldots + k_p = k }
  \av{\vec\jmath_p}^\gamma \av{\partial_z^l \hat \sK^{\mathsf{j}, p}_{\vec\jmath_p}\pare{z}} \prod _{q=1}^p \av{\partial_t^{k_q} f_{j_q}} \\
    \leq & \
  \sum_{p> N}
  \sum_{\vec\jmath_p \in \pare{\bZ\setminus \set{0}}^p }
  \sum_{k_1+\ldots + k_p = k } C_l^p
  \av{\vec\jmath_p}^{\gamma + l}  \prod _{q=1}^p  \av{j_q} \av{\partial_t^{k_q} f_{j_q}}  
  \end{align*}
using \eqref{eq:est_Fouirer_homog_K}. 
Then, assuming with no loss of generality 
that $  \av{\vec \jmath _p} = \max\{ \av{j_1}, \ldots, \av{j_p} \} = \av{j_1}  $ we have  
  \begin{align*}
   \av{\partial_t^k \partial_x^\gamma \partial_z^l \wt \sK^{\mathsf{j}, > N}_\alpha\pare{f;x, z}}
  \leq
  & \
  \sum_{p> N}
  \sum_{\vec\jmath_p \in \pare{\bZ\setminus \set{0}}^p } \sum_{k_1+\ldots + k_p = k } 
    C_l ^p 
    \pare{ \prod _{q=2}^p  \av{j_q}^{-2}
     \norm{\partial_t^{k_q} f}_3 } \ \av{j_1}^{-2} \norm{\partial_t^{k_1} f}_{3+\gamma + l},
  \\
  \leq &  \ \sum_{p> N} \ C_l ^p \
\sum_{k_1, \ldots , k_p =0}^k  \pare{  \prod _{q=2}^p  \norm{\partial_t^{k_q} f}_3 } 
   \norm{\partial_t^{k_1}  f }_{3+\gamma + l} 
  \\
  \leq &  \ \sum_{p> N} \pare{ C_l k }^p \ \norm{ f}_{k, 3 + \alpha k} ^{p-1} \norm{ f }_{k, 3+\gamma + l+\alpha k} 
  \end{align*}
  recalling \eqref{Knorm}.  
  Summing in $ p $ and  
 setting $ s_0 \defeq 11 +\alpha k  $, we get,  for any $ l\leq 8 $, for any 
  $ 0\leq \gamma \leq s-s_0 $,   
$$
   \av{\partial_t^k \partial_x^\gamma \partial_z^l \wt \sK^{\mathsf{j}, > N}_\alpha\pare{f;x, z}}
   \leq  C_k   ^{N+1}    \norm{f}_{k,s_0}^{N}
    \norm{f}_{k, s} \, , \quad 
    \forall x, z\in\bT \, ,  
$$
  which, recalling \cref{def:kernel_functions}, proves the claim in \eqref{eq:tail_homogeneous_deco_kernels}. 
  \end{step}
   \Cref{eq:claim_Khomog,eq:tail_homogeneous_deco_kernels}
 thus prove  \eqref{eq:Kernel_in_kernelfunction}. 
\\[1mm]
[Proof  of  
\eqref{eq:Taylor_expansion_kernels0}-\eqref{eq:coefficients_Kernels_zero}]
In view of \eqref{eq:decomposition_Kernel_homogeneous},
in order to expand $ \sK^\mathsf{j}_{\alpha, z}\pare{\frac{\Delta_z f}{r^2}} $ 
as in \eqref{eq:Taylor_expansion_kernels0}, 
we perform a Taylor expansion 
in $ z $ of the functions $\sK^\mathsf{j}_{\alpha, z}\pare{0}  $ and 
 $ \wt \sK^{\mathsf{j},p}_{\alpha} \pare{f;x, z} $,  
for any $ p \geq 1 $. 
By \eqref{eq:K0} we have   
\begin{equation}\label{exKj0}
\sK^\mathsf{j}_{\alpha, z}\pare{0}  = 
\sK^{\mathsf{j}, 0}_{\alpha} \pare{0;x}  +
\sK^{\mathsf{j}, 1}_{\alpha} \pare{0;x}  \sin z +
\sK^{\mathsf{j}, 2}_{\alpha} \pare{0;x}  \pare{2\sin\pare{z/2}}^2 
%\sum_{l=0}^2  \sK^{\mathsf{j}, l}_{\alpha} \pare{0;x}  \pare{2\sin\pare{z/2}}^l 
 + \varrho ^{\mathsf{j},3}_\alpha\pare{0;x,z}
\end{equation}
with $\sK^{\mathsf{j}, l}_{\alpha} \pare{0;x}  $, $ \mathsf{j} = 1,2, 3 $, $ l = 0, 1,2 $,
are the constants computed in \eqref{eq:coefficients_Kernels_zero} and  
\begin{equation}\label{perf=0}
\varrho ^{\mathsf{1},3}_\alpha\pare{0;x,z} = \varrho ^{\mathsf{2},3}_\alpha\pare{0;x,z} = 0 
\qquad \text{and} \qquad 
\varrho^{\mathsf{j},3}_\alpha\pare{0;x,z} \in \wt \KF^3_0  
\end{equation}
is $ x $-independent. 
Then,
for any $ p \geq 1 $, we expand 
\begin{equation}\label{eq:TaylorHomogeneousComponents}
\begin{aligned}
\wt \sK^{\mathsf{j},p}_{\alpha} \pare{f;x, z}
= & \ \sum_{l=0}^2 \wt  \sK^{\mathsf{j},p, l}_{\alpha} \pare{f;x} z^l + {\sR}^{\mathsf{j},p,3}_\alpha\pare{f;x,z}
\\
= & \   \wt  \sK^{\mathsf{j},p, 0}_{\alpha} \pare{f;x} + \wt  \sK^{\mathsf{j},p, 1}_{\alpha} \pare{f;x} \sin z + \wt  \sK^{\mathsf{j},p, 2}_{\alpha} \pare{f;x} \pare{2\sin\pare{z/2}}^2 + \varrho ^{\mathsf{j},p,3}_\alpha\pare{f;x,z},
\end{aligned}
\end{equation}
where, for $ l = 0, 1, 2 $, 
\begin{equation}
\label{eq:Ktildes}
\wt \sK^{\mathsf{j},p,l}_{\alpha} \pare{f;x} \defeq \
\frac{1}{l!}
\left. \partial_z^l \wt \sK^{\mathsf{j},p}_{\alpha} \pare{f;x, z}\right|_{z=0} 
\end{equation}
and 
\begin{equation}
\label{eq:expression_remainders_z_homogeneous}
\begin{aligned}
\varrho ^{\mathsf{j},p,3}_\alpha\pare{f;x,z}
\defeq 
&
\
{\sR}^{\mathsf{j},p,3}_\alpha\pare{f;x,z} + \tilde{\sR}^{\mathsf{j},p,3}_\alpha\pare{f;x,z}
\\
{\sR}^{\mathsf{j},p,3}_\alpha\pare{f;x,z}
\defeq & \ \frac{1}{2!}  \int_0^1 \pare{1-\vartheta}^2
 \partial_z ^3\wt \sK^{\mathsf{j},p}_{\alpha} \pare{f;x, \vartheta z}
 \dd \vartheta \ z^3
\\
\tilde{\sR}^{\mathsf{j},p,3}_\alpha\pare{f;x,z}
\defeq & \
 \wt \sK^{\mathsf{j},p, 1}_{\alpha} \pare{f;x } \pare{ z - \sin z } + \wt \sK^{\mathsf{j},p, 2}_{\alpha} \pare{f;x } \pare{ z^2 - \pare{2\sin\pare{z/2}}^2 } .
\end{aligned}
\end{equation}
Notice that $ z\mapsto \varrho ^{\mathsf{j},p,3}_\alpha\pare{f;x,z} $ is $ 2\pi $-periodic thanks to \eqref{eq:TaylorHomogeneousComponents}. 
In view of \eqref{exKj0} and \eqref{eq:TaylorHomogeneousComponents}, 
 we obtain the expansion \eqref{eq:Taylor_expansion_kernels0} 
 with,  for any $ \mathsf{j}=1,2,3 $,  
 \begin{equation}\label{eq:SHC1}
 \begin{aligned}
&  \sK^{\mathsf{j},l}_\alpha \pare{f; x} \defeq \ 
{\sK^{\mathsf{j},l}_{\alpha}\pare{0;x} } 
+ 
\sum_{p=1 }^{N} \wt  \sK^{\mathsf{j},p, l}_{\alpha} \pare{f;x}  
+  \wt \sK^{\mathsf{j},> N , l}_{\alpha} \pare{f;x}  \, , 
\qquad  l = 0, 1, 2 \, , \\  
& 
 \varrho^{\mathsf{j},3}_\alpha \pare{f;x,z} 
 \defeq { \varrho^{\mathsf{j},3}_\alpha \pare{0;x,z}  } + 
 \sum_{p=1}^N  \varrho^{\mathsf{j},p,3}_\alpha \pare{f;x,z} + \varrho^{\mathsf{j},>N,3}_\alpha \pare{f;x,z}  \, , 
 \end{aligned}
 \end{equation}
and
 \begin{equation} \label{eq:expression_kernels_NH}
 \wt \sK ^{\mathsf{j}, >N, l}_\alpha \pare{f;x} \defeq  \ \sum_{p>N} \wt \sK ^{\mathsf{j}, p, l}_\alpha\pare{f;x} \, , \qquad 
  \varrho^{\mathsf{j}, >N, 3}_\alpha \pare{f;x,z} \defeq  \ \sum_{p>N} \varrho^{\mathsf{j}, p, 3}_\alpha\pare{f;x,z} \, .
 \end{equation}
 Let us prove  \eqref{prop:fr}. 
 We deduce that  each 
$ \wt \sK^{\mathsf{j},p,l}_{\alpha} \pare{f;x} = 
\frac{1}{l!}
   \partial_z^l \wt \sK^{\mathsf{j},p,l}_\alpha\pare{f;x, 0} $, $ p \geq 1 $,  
is a homogenous function  in $ \wt\cF^\bR_p $ by 
\eqref{eq:claim_Khomog} and 
Remark \ref{rem:firstremaKernelfunctions}.   
Analogously the non-homogenous term
$ \wt \sK ^{\mathsf{j}, >N, l}_\alpha \pare{f;x} $ 
is in $ \cF^\bR_{K, 0, N+1}\bra{\epsilon_0} $ by \eqref{eq:tail_homogeneous_deco_kernels}. 
Next,  by \eqref{eq:claim_Khomog} an integration in $ z $ 
give that
$ \varrho^{\mathsf{j},p,3}_\alpha \pare{f;x,z} $, $ p \geq 1 $,  defined in 
\eqref{eq:expression_remainders_z_homogeneous}
is a homogenous Kernel-function in $  \wt \KF ^3_p $ and, by \eqref{eq:tail_homogeneous_deco_kernels},   
the non-homogenous 
term $ \varrho^{\mathsf{j}, >N, 3}_\alpha \pare{f;x,z} $
 in \eqref{eq:expression_kernels_NH} is a Kernel function in 
$ \KF ^3 _{K, 0, N+1}\bra{\epsilon_0} $.

Finally the zero-homogenous functions 
$ {\sK^{\mathsf{j},l}_{\alpha}\pare{0;x} } $ are the constants 
in \eqref{eq:coefficients_Kernels_zero} (cf. \eqref{exKj0})
and the Kernel functions $ \varrho^{\mathsf{j},3}_\alpha \pare{0;x,z} $
are in \eqref{perf=0}.
\end{proof}

\subsection{Paralinearization of the quasilinear integral term $ \cI \pare{f} $}

In this section we paralinearize $ \cI \pare{f} $.

\begin{lemma}\label{lem:paralinearization_cI}
The term  $ \cI\pare{f} $ defined in \eqref{eq:cI} can be written as 
\begin{equation}\label{eq:paralinearization_cI}
\cI\pare{f} =  \OpBW{  - \pare{1+\nu_\cI\pare{f;x}} L_\cI \pare{\av{\xi}} + 
\ii
 S_{\cI, \alpha-2}\pare{f;x, \xi} +   V\bra{\cI}\pare{f;x}  + P\pare{f;x,\xi}} f +R\pare{f}f 
\end{equation}
 where 
\begin{itemize}
\item $\nu_\cI \pare{f; x} $ is the real function 
\begin{equation}
\label{eq:LcI}
\nu_\cI \pare{f; x} \defeq   \ - \pare{  r^{-2} \sK^{1,0}_{\alpha}\pare{f;x} +1 }   \in \Sigma \cF^\bR_{K, 0, 1}\bra{\epsilon_0, N} ;  
\end{equation}
\item $ L_\cI \pare{\av{\xi}}\defeq  \    \sT_\alpha^1  \pare{\av{\xi}}  +\frac{\Gamma\pare{2-\alpha}}{\Gamma\pare{1-\frac{\alpha}{2}}^2}  + \pare{1-\frac{\alpha}{2}}^2 \sM_\alpha\pare{\av{\xi}} $ is a real Fourier multiplier in 
$ \wt \Gamma^{\max \{ 0,\alpha-1 \} }_0 $ (the Fourier multipliers \ $  \sT_\alpha^1 \pare{\av{\xi}}  $ and $ \sM_\alpha  \pare{\av{\xi}}  $
are defined in \Cref{lem:linearization});
\item $ S_{\cI, \alpha-2}\pare{f;x, \xi}\defeq  \ 
 - \frac{1}{2 } \pare{  \nu_\cI\pare{f;x} }_x\ \partial_\xi L_\cI \pare{\av{\xi}} +  \ r^{-2}\sK^{1,1}_{\alpha}\pare{f;x}  \sM_{\alpha}\pare{\av{\xi}}\xi
$ is a real symbol in $ \Sigma \Gamma^{\alpha-2} _{K, 0, 1}\bra{\epsilon_0, N} $; 
\item $ V\bra{\cI}\pare{f;x} $ is a function in $ \Sigma \cF^\bR_{K, 0, 1}\bra{\epsilon_0, N} $; 
\item $ P \pare{f;x, \xi} $ is a symbol in $ \Sigma \Gamma^{-1} _{K, 0, 1}\bra{\epsilon_0, N}$ satisfying \eqref{areal}; 
\item  $ R \pare{f} $ is a real smoothing operator in 
$ \Sigma \cR^{-\rho} _{K, 0, 1}\bra{\epsilon_0, N}  $. 
\end{itemize}
 \end{lemma}

The rest of this section is devoted to prove \Cref{lem:paralinearization_cI}.

By \Cref{lem:paraproduct_Weyl}  we have 
\begin{equation*}
 \frac{\delta_z f}{r^2} = \OpBW{r^{-2}}\delta_z f + \OpBW{\delta_z f}\pare{r^{-2}-1} + R_1
 \pare{r^{-2}-1}\delta_z f + R_2 \pare{\delta_z f}\pare{r^{-2}-1} 
\end{equation*}
with smoothing operators $ R_1, R_2 $ in $ \wt \cR^{-\rho}_1 $ for any $ \rho \geq 0 $.
Hence, recalling the definition of $ \cI\pare{f} $ in \eqref{eq:cI}, we write 
\begin{equation}
\label{eq:cI_decomposition}
\begin{aligned}
\cI\pare{f} = & \ \sum_{{\mathsf{j}}=1}^4 \cI_{\mathsf{j}}\pare{f} \, , 
\\
\cI_1 \pare{f} \defeq & \ \fint \OpBW{\sK^1_{\alpha, z}\pare{\frac{\Delta_z f}{r^2}}} \OpBW{r^{-2}} \frac{\delta_z f}{\av{2\sin\pare{z/2}}^\alpha } \dd z \, , 
\\
\cI_2 \pare{f} \defeq & \ \fint \OpBW{\sK^1_{\alpha, z}\pare{\frac{\Delta_z f}{r^2}}} \OpBW{\delta_z f} \frac{1}{\av{2\sin\pare{z/2}}^\alpha } \dd z \  \pare{ r^{-2}-1 } \, , 
\\
%-------------------------------------------
\cI_3 \pare{f} \defeq & \ \fint \OpBW{\sK^1_{\alpha, z}\pare{\frac{\Delta_z f}{r^2}}} R_1
\pare{r^{-2}-1} \frac{\delta_z f}{\av{2\sin\pare{z/2}}^\alpha } \dd z \, , 
\\
%-------------------------------------------
\cI_4 \pare{f} \defeq & \ \fint \OpBW{\sK^1_{\alpha, z}\pare{\frac{\Delta_z f}{r^2}}} R_2
\pare{\frac{\delta_z f}{\av{2\sin\pare{z/2}}^\alpha }}  \dd z \ \pare{r^{-2}-1} \, . 
\end{aligned}
\end{equation}

\begin{step}[Paralinearization of $ \cI_1 $ in \eqref{eq:cI_decomposition}]
By \eqref{eq:Taylor_expansion_kernels0} and 
isolating by \eqref{eq:coefficients_Kernels_zero} 
the zero-order components in $ f $, we write
\begin{align}
&  \OpBW{\sK^1_{\alpha, z}\pare{\frac{\Delta_z f}{r^2}}} \OpBW{r^{-2}} 
= \notag \\ 
&  
\OpBW{ \sK^{1,0}_{\alpha}\pare{f;x} +  \sK^{1,1}_{\alpha}\pare{f;x} \, \sin z
+ \sK^{1,2}_\alpha  \pare{f;x} \pare{2\sin\pare{z/2}}^2
  +\varrho^{1,3}_\alpha\pare{f;x,z}} \OpBW{r^{-2}}  \notag   \\
& 
 =  
 -1 -\frac{1}{2}\pare{1-\frac{\alpha}{2}} \pare{2\sin\pare{z/2}}^2 \notag 
 \\
 & \ + \ 
\OpBW{ \pare{ \sK^{1,0}_{\alpha}\pare{f;x} +1 } +  \sK^{1,1}_{\alpha}\pare{f;x} \, \sin z 
+ \pare{ \sK^{1,2}_\alpha  \pare{f;x} + \frac{1}{2}\pare{1-\frac{\alpha}{2}} } \pare{2\sin\pare{z/2}}^2
  +\varrho^{1,3}_\alpha\pare{f;x,z}} \notag 
  \\
 & \  + \
\OpBW{ \sK^{1,0}_{\alpha}\pare{f;x} +  \sK^{1,1}_{\alpha}\pare{f;x} \, \sin z 
+ \sK^{1,2}_\alpha  \pare{f;x} \pare{2\sin\pare{z/2}}^2
  +\varrho^{1,3}_\alpha\pare{f;x,z}} \OpBW{r^{-2}-1}  \label{eq:cI1_1}
 \end{align}
where  $ \varrho^{1,3}_\alpha\pare{f;x,z} $ is a Kernel function in $ \Sigma\KF ^{3}_{K, 0, 1}\bra{\epsilon_0, N} $  by \eqref{prop:fr}. 
We now expand the last line \eqref{eq:cI1_1}.
 By Proposition \ref{prop:composition_BW} 
 there exists a smoothing operator 
 $ R \pare{f} $ in $ \Sigma \cR^{-\rho}_{K,0,1}\bra{\epsilon_0, N} $  such that
\begin{multline}
\label{eq:cI1_2}
\OpBW{ \sK^{1,0}_{\alpha}\pare{f;x} +  \sK^{1,1}_{\alpha}\pare{f;x} \, \sin z
+ \sK^{1,2}_\alpha  \pare{f;x} \pare{2\sin\pare{z/2}}^2
  } \OpBW{r^{-2}-1}
  \\
  = \OpBW{ \pare{r^{-2}-1} \pare{  \sK^{1,0}_{\alpha}\pare{f;x} + 
  \sK^{1,1}_{\alpha}\pare{f;x} \, \sin z
+ \sK^{1,2}_\alpha  \pare{f;x} \pare{2\sin\pare{z/2}}^2 }
  } \\
   + R \pare{f} + \underbrace{ \pare{\sin\pare{z} + \pare{ 2\sin\pare{z/2} }^2} R\pare{f}}_{\defeq R_{\pare{1}}\pare{f;z} \in  \Sigma \KR ^{-\rho, 1}_{K,0,1}\bra{\epsilon_0, N} } \, . 
\end{multline}
Moreover due to \Cref{prop:comp_z}, \cref{item:OpOp_ext_z}, 
there exists 
a Kernel-smoothing operator 
$ R_1 \pare{f;z} $ in $ \Sigma\ \KR^{-\rho,3}_{K, 0, 1}\bra{\epsilon_0, N} $   such that
\begin{align}\label{eq:cI1_3}
\OpBW{\varrho^{1,3}_\alpha\pare{f;x,z}} \OpBW{r^{-2}-1} =
\OpBW{\pare{r^{-2}-1}\varrho^{1,3}_\alpha\pare{f;x,z}} + R_1\pare{f;z}  \, . 
\end{align}
Plugging \eqref{eq:cI1_2} and \eqref{eq:cI1_3} in \eqref{eq:cI1_1} we get 
\begin{align}
&  \OpBW{\sK^1_{\alpha, z}\pare{\frac{\Delta_z f}{r^2}}} \OpBW{r^{-2}} 
  =  
 -1 -\frac{1}{2}\pare{1-\frac{\alpha}{2}} \pare{2\sin\pare{z/2}}^2 \notag 
 \\
 & \ + \ 
\OpBW{ \pare{ r^{-2} \sK^{1,0}_{\alpha}\pare{f;x} +1 } + r^{-2}  \sK^{1,1}_{\alpha}\pare{f;x} \, \sin z
+ \pare{ r^{-2}\sK^{1,2}_\alpha  \pare{f;x} + \frac{1}{2}\pare{1-\frac{\alpha}{2}} } \pare{2\sin\pare{z/2}}^2
 } \notag
  \\
 & \  + \ \OpBW{r^{-2}\varrho^{1,3}_\alpha\pare{f;x,z}} + 
 R\pare{f} + R_{\pare{1}}\pare{f;z}  \label{eq:cI1_4}
 \end{align}
where  $ R_{\pare{1}}\pare{f;z} $ is a 
 Kernel smoothing operator in $ \Sigma \ \KR^{-\rho,1}_{K, 0, 1}\bra{\epsilon_0,N} $. 
 Inserting the decomposition \eqref{eq:cI1_4} in 
the expression of $ \cI_1 \pare{f} $ 
in \eqref{eq:cI_decomposition} we obtain that
\begin{equation}\label{eq:deco_cI1}
\cI_1 \pare{f} = - \fint   \frac{\delta_z f}{\av{2\sin\pare{z/2}}^\alpha } \dd z 
- \frac{1}{2}\pare{1-\frac{\alpha}{2}}\fint   \frac{ \delta_z f }{\av{2\sin\pare{z/2}}^{\alpha-2} } \  \dd z
+
 \sum_{{\mathsf{j}}=1}^5 \cI_{1,{\mathsf{j}}} \pare{f}
\end{equation}
where
\begin{equation}\label{ultedivi}
\begin{aligned}
\cI_{ 1,1 } \pare{f} \defeq & \ \OpBW{ r^{-2} \sK^{1,0}_{\alpha}\pare{f;x} +1} \fint   \frac{\delta_z f}{\av{2\sin\pare{z/2}}^\alpha } \dd z \, , \\
\cI_{ 1,2 } \pare{f} \defeq & \
 \OpBW{ r^{-2}\sK^{1,1}_{\alpha}\pare{f;x}} \fint   \frac{\sin z}{\av{2\sin\pare{z/2}}^\alpha } \ \delta_z f\  \dd z \, , 
\\
\cI_{ 1,3 } \pare{f} \defeq & \  \OpBW{  \pare{ r^{-2}\sK^{1,2}_\alpha  \pare{f;x} + \frac{1}{2}\pare{1-\frac{\alpha}{2}} } } \fint   \frac{\delta_z f}{\av{2\sin\pare{z/2}}^{\alpha-2} } \  \dd z \, , 
\\
\cI_{ 1,4 } \pare{f} \defeq & \  \fint \OpBW{r^{-2}\varrho^{1,3}_\alpha\pare{f;x,z}}  \frac{\delta_z f}{\av{2\sin\pare{z/2}}^{\alpha-2} } \  \dd z \, , 
\\
\cI_{ 1,5 } \pare{f} \defeq & \ \fint \pare{R\pare{f} + R_{\pare{1}}\pare{f;z}} \frac{\delta_z f}{\av{2\sin\pare{z/2}}^{\alpha}} \ \dd z \, . 
\end{aligned}
\end{equation}
By recalling \eqref{eq:notation} and \eqref{eq:step:1} we have 
\begin{equation}\label{primofm}
 \fint   \frac{\delta_z f}{\av{2\sin\pare{z/2}}^\alpha } \dd z 
 =
 \OpBW{ \sT_\alpha^1  \pare{\av{\xi}}} f \, .
\end{equation}
Next, by \cref{eq:step:3,eq:Step3est1} we deduce that
\begin{equation}\label{eq:costante_linearizzato_1}
\frac{1}{2}\pare{1-\frac{\alpha}{2}}\fint   \frac{ \delta_z f }{\av{2\sin\pare{z/2}}^{\alpha-2} } \  \dd z
=
\frac{\Gamma\pare{2-\alpha}}{\Gamma\pare{1-\frac{\alpha}{2}}^2} f\pare{x} + \pare{1-\frac{\alpha}{2}}^2 \sM_\alpha\pare{\av{D}} f \, . 
\end{equation}
By \eqref{primofm}, 
using also Proposition \ref{prop:composition_BW} and \eqref{asharpb},
and that $ \sT_\alpha^1  \pare{\av{\xi}} $ is a symbol of order $ \alpha - 1 $, we have  
\begin{align}\label{eq:cI11'}
\cI_{ 1,1 } \pare{f} & =   
\OpBW{ r^{-2} \sK^{1,0}_{\alpha}\pare{f;x} +1} \OpBW{ \sT_\alpha^1  \pare{\av{\xi}}} f 
 \\
& = \OpBW{ \pare{ r^{-2} \sK^{1,0}_{\alpha}\pare{f;x} +1 } \sT_\alpha^1  \pare{\av{\xi}} + \frac{\ii}{2 }\partial_x\pare{ r^{-2} \sK^{1,0}_{\alpha}\pare{f;x} +1 }\ \partial_\xi \sT_\alpha^1  \pare{\av{\xi}} + P\pare{f;x, \xi}} f  \notag
+R\pare{f}f  
\end{align}
where $  P\pare{f;x, \xi} $  is a symbol in 
$  \Sigma \Gamma^{-1} _{K, 0, 1}\bra{\epsilon_0, N} $. 

In order to compute $ \cI_{ 1,2 } \pare{f} $   in \eqref{ultedivi} 
 we need the following lemma.
\begin{lemma}\label{lem:mult2}
We have
\begin{equation}\label{ausilia1}
\fint \frac{\sin z}{\av{2\sin\pare{z/2}}^\alpha }\delta_z \phi\ \dd z =
\ii \ \sM_{\alpha}\pare{\av{D}} D \  \phi    ,
\end{equation}
where $ \sM_{\alpha} \pare{\av{\xi}} $ is defined in \eqref{eq:Malpha}.
\end{lemma}

\begin{proof}
By oddness 
$ \fint \frac{\sin z }{\av{2\sin\pare{z/2}}^\alpha }\dd z
= 0 $ and thus, integrating by parts, 
\begin{align*}
\fint \frac{\sin z}{\av{2\sin\pare{z/2}}^\alpha }\delta_z \phi\ \dd z
& =    - \fint \frac{\sin z}{\bra{2\pare{1-\cos z}}^{\alpha / 2} } \phi\pare{x-z} \ \dd z 
=    -  \fint \partial_z \pare{ \frac{\bra{2\pare{1-\cos z}}^{1- \frac{\alpha}{2}}}{2\pare{1-\frac{\alpha}{2}}} } \phi\pare{x-z} \, \dd z \\
& 
 =  - \frac{1}{2\pare{1-\frac{\alpha}{2}}} \fint   \frac{\phi'\pare{x-z}}{\bra{2\pare{1-\cos z}}^{\frac{\alpha}{2}-1}} \ \dd z 
= \ii \ \sM_{\alpha}\pare{\av{D}} D \  \phi  
\end{align*}
using    \eqref{eq:sMalpha_integral}. This proves \eqref{ausilia1}.
\end{proof}
\Cref{lem:mult2} and \Cref{prop:composition_BW} and since 
$ \sM_{\alpha}\pare{\av{\xi}} $ is a symbol of order $ \alpha - 3 $
gives that
\begin{equation}\label{eq:cI12}
\cI_{1,2}\pare{f} = \OpBW{ \ii \ r^{-2}\sK^{1,1}_{\alpha}\pare{f;x}  \sM_{\alpha}\pare{\av{\xi}} \xi + P\pare{f;x,\xi}} f +R\pare{f}f 
\end{equation}
where $ P\pare{f;x,\xi} $ is a symbol in 
$  \Sigma \Gamma^{-1} _{K, 0, 1}\bra{\epsilon_0, N} $ satisfying \eqref{areal}. 

Let us now compute $ \cI_{1,3} \pare{f} $ in \eqref{ultedivi}. 
Applying \Cref{prop:reminders_integral_operator} we deduce that
\begin{equation}\label{eq:cI13}
\cI_{1, 3}\pare{f} = \OpBW{V\bra{\cI_{1, 3}}\pare{f;x} + P\pare{f;x, \xi}} f   
\end{equation}
 where 
\begin{equation}\label{eq:VI13}
V\bra{\cI_{1,3}}\pare{f;x} \defeq   \pare{ r^{-2}\sK^{1,2}_\alpha  \pare{f;x} + \frac{1}{2}\pare{1-\frac{\alpha}{2}} }  \fint   \frac{1}{\av{2\sin\pare{z/2}}^{\alpha-2} } \  \dd z
\end{equation}
is a function in $ \Sigma \cF^{\bR}_{K,0,1}\bra{\epsilon_0, N} $ 
and $ P\pare{f; x, \xi} $ is a symbol in  
$ \Sigma \Gamma^{- 1}_{K, 0, 1}\bra{\epsilon_0, N} $,  being 
$ \alpha \in (0,2)$.  

Similarly, applying \Cref{prop:reminders_integral_operator}, 
\begin{equation}\label{eq:cI14}
\cI_{1, 4}\pare{f} = \OpBW{V\bra{\cI_{1, 4}}\pare{f;x} + P\pare{f;x, \xi}} f 
\end{equation}
where 
\begin{equation}\label{eq:VI14}
V\bra{\cI_{1,4}}\pare{f;x} \defeq 
 \fint   \frac{r^{-2}\varrho^{1,3}_\alpha\pare{f;x,z}}{\av{2\sin\pare{z/2}}^{\alpha-2} } \  \dd z 
\end{equation}
is a function  in $ \Sigma \cF^{\bR}_{K,0,1}\bra{\epsilon_0, N} $  
by \cref{prop:integral_kernelfunction}, 
and $ P\pare{f; x, \xi} $ is a symbol in  
$ \Sigma \Gamma^{-1}_{K, 0, 1}\bra{\epsilon_0, N} $. 

Finally, the last term in \eqref{ultedivi} is, 
applying \Cref{prop:action_z_smoothing} since $ \frac{R_{\pare{1}}\pare{f;z}}{\av{2\sin\pare{z/2}}^\alpha} \in \Sigma \KR^{-\rho, 1-\alpha}_{K, 0, 1}\bra{\epsilon_0, N} $,  
\begin{equation}\label{eq:cI15}
\cI_{1,5}\pare{f} = R\pare{f} \OpBW{ \sT_\alpha^1  \pare{\av{\xi}}} f + \tilde{R}\pare{f} f \, , \qquad R\pare{f}, \tilde{R}\pare{f} \in 
\Sigma \cR^{-\rho}_{K,0,1}\bra{\epsilon_0, N} \, . 
\end{equation}
We thus plug \eqref{primofm}, \eqref{eq:costante_linearizzato_1}
\eqref{eq:cI11'},  \eqref{eq:cI12}, \eqref{eq:cI13}, \eqref{eq:cI14},  \eqref{eq:cI15}, 
in \Cref{eq:deco_cI1} and obtain
\begin{multline}\label{eq:I1_paralinearization}
\cI_{1}\pare{f} = -  \sT_\alpha^1  \pare{\av{D}} f -\frac{\Gamma\pare{2-\alpha}}{\Gamma\pare{1-\frac{\alpha}{2}}^2} \ f - \pare{1-\frac{\alpha}{2}}^2 \sM_\alpha\pare{\av{D}} f
\\
+
\OpBW{ \pare{ r^{-2} \sK^{1,0}_{\alpha}\pare{f;x} +1 } \sT_\alpha^1  \pare{\av{\xi}} + \frac{\ii}{2 }\partial_x\pare{ r^{-2} \sK^{1,0}_{\alpha}\pare{f;x} +1 }\ \partial_\xi \sT_\alpha^1  \pare{\av{\xi}}+ \ii \ r^{-2}\sK^{1,1}_{\alpha}\pare{f;x}  \sM_{\alpha}\pare{\av{\xi}} \xi} f
\\
+
\OpBW{  V\bra{\cI_1}\pare{f;x}  + P\pare{f;x,\xi}} f +R\pare{f}f 
\end{multline}
where  $V\bra{\cI_1}\pare{f;x} $ is the function  (cf. \cref{eq:VI14,eq:VI13})
\begin{equation}\label{eq:VI1}
V\bra{\cI_1}\pare{f;x} \defeq V\bra{\cI_{1,3}}\pare{f;x} + V\bra{\cI_{1,4}}\pare{f;x} \in \Sigma \cF^{\bR}_{K,0,1}\bra{\epsilon_0, N} \, .
\end{equation}
\end{step}

\begin{step}[Paralinearization of $ \cI_2 $ in \eqref{eq:cI_decomposition}]
Since $ \sK^1_{\alpha, z}\pare{\frac{\Delta_z f}{r^2}}\in \Sigma \KF^0_{K, 0, 0}\bra{\epsilon_0, N} $ (cf. \Cref{lem:characterization_Kernels}) and $ \delta_zf\in \wt \KF ^{1}_1 $ we apply \Cref{prop:comp_z}, \cref{item:OpOp_ext_z} and obtain that, 
for some $ R_2\pare{f;z} \in \Sigma \ \KR^{-\rho, 1-\alpha}_{K, 0, 1}\bra{\epsilon_0, N} $
\begin{align*}
 \fint \OpBW{\sK^1_{\alpha, z}\pare{\frac{\Delta_z f}{r^2}}} \OpBW{\delta_z f} \frac{\dd z}{\av{2\sin\pare{z/2}}^\alpha }  &  =
\fint \OpBW{\sK^1_{\alpha, z}\pare{\frac{\Delta_z f}{r^2}}\delta_z f} \frac{\dd z}{\av{2\sin\pare{z/2}}^\alpha } \dd z 
+\fint R_2 \pare{f;z} \dd z \\
& = \OpBW{\tilde{V}\bra{\cI_2}} + R\pare{f}
\end{align*}
where $ R \pare{f} $ is a smoothing operator in 
$ \Sigma \cR^{-\rho}_{K, 0, 1}\bra{\epsilon_0, N} $ (by \Cref{prop:action_z_smoothing}) 
and 
\begin{equation*}
\tilde V\bra{\cI_2} \pare{f;x} \defeq \fint \sK^1_{\alpha, z}\pare{\frac{\Delta_z f}{r^2}}\delta_z f \frac{1}{\av{2\sin\pare{z/2}}^\alpha } \dd z \, 
\end{equation*}
is a function in  $ \Sigma \cF^{\bR}_{K, 0, 1}\bra{\epsilon_0, N} $, by 
\Cref{prop:integral_kernelfunction}  and since 
since $ \sK^1_{\alpha, z}\pare{\frac{\Delta_z f}{r^2}}\delta_z f \ \frac{1}{\av{2\sin\pare{z/2}}^\alpha } $ is in $ \Sigma\ \KF^{1-\alpha}_{K, 0, 1}\bra{\epsilon_0, N} $.

Finally, using the identity (cf. \Cref{lem:Bony_paralinearization_W})
\begin{equation}\label{eq:parabeta}
r^{\beta} - 1 = \OpBW{\beta r^{\beta-2}} f + R\pare{f}f \, , \qquad 
\forall\beta\in\bR \, , 
\end{equation}
 we write the term  $ \cI_2  \pare{f} $ in \eqref{eq:cI_decomposition} as,
 using  \Cref{prop:composition_BW,compositionMoperator},  
\begin{align}
\cI_2 \pare{f} & = \pare{\OpBW{\tilde{V}\bra{\cI_2} \pare{f;x}} + R\pare{f}}\pare{\OpBW{-2r^{-4}}  f + R \pare{f}f}  \notag \\
& =\OpBW{V\bra{\cI_2} \pare{f;x}}f + R\pare{f}f \label{eq:I2_paralinearization}
\end{align}
where 
\begin{equation}\label{eq:VI2}
V\bra{\cI_2} \pare{f;x} \defeq -2r^{-4} \tilde{V}\bra{\cI_2}  \pare{f;x} 
\in \Sigma \cF^{\bR}_{K, 0, 1}\bra{\epsilon_0, N} \, . 
\end{equation} 
\end{step}

\begin{step}[Paralinearization of $ \cI_3 $ in \eqref{eq:cI_decomposition}]
We first note that, 
in view of \eqref{eq:parabeta}, the fact that
$ \OpBW{\beta r^{\beta-2}} $ and $ R\pare{f} $ 
are $ 0 $-operators,  
and Proposition \ref{compositionMoperator}-(ii), we deduce that 
 $$ 
 R_1\pare{r^{-2}-1} = \breve{R}\pare{f}  \in \Sigma \cR^{-\rho}_{K, 0, 1}\bra{\epsilon_0, N}
 $$ 
 is a smoothing operator, that we may also regard as a Kernel-smoothing operator  
 in $ \Sigma \KR^{-\rho,0}_{K, 0, 1}\bra{\epsilon_0, N} $. 
 Furthermore by \eqref{eq:deltaDelta}
$ \frac{\delta_z }{\av{2\sin\pare{z/2}}^\alpha }  $ is in $ \widetilde{\KM}^{\ 1,1-\alpha}_{0} $
and $  \sK^\mathsf{j}_{\alpha, z} \pare{\frac{\Delta_z f}{r^2}}  $ is a Kernel function in 
$  \Sigma\KF ^{0}_{K, 0, 0}\bra{\epsilon_0, N} $
by \eqref{eq:Kernel_in_kernelfunction}. 
Therefore by \Cref{prop:comp_z} {\color{blue} \cref{item:MM_ext_z,item:OpR_ext_z}} and 
\Cref{prop:action_z_smoothing} we obtain that 
\begin{equation}\label{eq:I3_paralinearization}
\cI_3\pare{f} =  \fint R\pare{f;z} f \dd z = R\pare{f}f 
\end{equation}
where $ R \pare{f} $ is a  smoothing operator
 in $ \Sigma \cR^{-\rho}_{K,0,p}\bra{\epsilon_0, N} $. 
\end{step}

\begin{step}[Paralinearization of $ \cI_4 $ in \eqref{eq:cI_decomposition}]
Reasoning as in the previous step  there is a 
   smoothing operator $ \breve R \pare{f} $ 
 in $ \Sigma \cR^{-\rho}_{K,0,p}\bra{\epsilon_0, N} $ such that  
\begin{equation}\label{eq:I4_paralinearization}
\cI_4\pare{f} = \breve R\pare{f}\pare{r^{-2}-1} 
= R\pare{f}f 
\end{equation}
(use \eqref{eq:parabeta})  where 
$  R \pare{f} $ 
 is a smoothing operator in  $ \Sigma \cR^{-\rho}_{K,0,p}\bra{\epsilon_0, N} $. 
\end{step}

\begin{step}[Conclusion]
Inserting 
\cref{eq:I1_paralinearization,eq:I4_paralinearization,eq:I3_paralinearization,eq:I2_paralinearization} 
in \cref{eq:cI_decomposition}, recalling the definition of $  L_{\cal I } (|\xi | ) $
in \Cref{lem:paralinearization_cI},  
and that 
$ \nu_\cI \pare{f; x} \defeq  - \big(  r^{-2} \sK^{1,0}_{\alpha}\pare{f;x} +1  \big) $
(cf. \cref{eq:LcI})
we obtain 
\begin{multline}\label{eq:I_paralinearization}
\cI \pare{f} = - L_{\cal I } (|\xi | ) f 
+
\OpBW{ - \nu_\cI \pare{f; x}  \sT_\alpha^1  \pare{\av{\xi}} - 
\frac{\ii}{2 } \pare{\nu_\cI \pare{f; x}}_x 
\ \partial_\xi \sT_\alpha^1  \pare{\av{\xi}}+ \ii \ r^{-2}\sK^{1,1}_{\alpha}\pare{f;x}  \sM_{\alpha}\pare{\av{\xi}} \xi} f
\\
+
\OpBW{  V\bra{\cI_1} \pare{f;x}  + V\bra{\cI_2} \pare{f;x}  + P\pare{f;x,\xi}} f +R\pare{f}f  \, . 
\end{multline}
Finally, substituting 
$\sT^1_\alpha\pare{\av{\xi}} = L_\cI \pare{\av{\xi}} - \frac{\Gamma(2- \alpha)}{
\Gamma(1- \frac{\alpha}{2})^2 } - (1- \frac{\alpha}{2})^2 M_\alpha (|\xi| ) $, 
we deduce that \eqref{eq:I_paralinearization}  
is the paralinearization \eqref{eq:paralinearization_cI} with 
(cf. \cref{eq:VI1,eq:VI2})
\begin{equation}\label{eq:VI}
V\bra{\cI} \pare{f;x}  \defeq V\bra{\cI_1} \pare{f;x}  + V\bra{\cI_2} \pare{f;x} 
+ \nu_\cI \pare{f; x} \frac{\Gamma(2- \alpha)}{
\Gamma(1- \frac{\alpha}{2})^2 } 
 \in \Sigma \cF^{\bR}_{K,0,1}\bra{\epsilon_0, N} 
\end{equation}
and another symbol 
$ P \pare{f; x , \xi} $ in
$ \Sigma \Gamma^{-1}_{K, 0, 1}\bra{\epsilon_0 , N} $  satisfying  \eqref{areal}.
\end{step}

\subsection{Paralinearization of the quasilinear integral term  $ \cJ\pare{f} $}
\label{sec:paralinearization_cJ}

In this section we paralinearize  $ \cJ\pare{f} $. 

\begin{lemma}\label{lem:paralinearization_cJ}
The term $ \cJ \pare{ f }  $ defined in \eqref{eq:cJ} can be written as 
\begin{equation}
\label{eq:paralinearization_cJ}
 \cJ\pare{f}
=
  \OpBW{ - \pare{1+\nu_{\cJ}\pare{f;x}} L_\cJ \pare{\av{\xi}}  + \ii \  S_{\cJ, \alpha-2}\pare{f;x, \xi} +  V\bra{\cJ}\pare{f;x} + P\pare{f;x, \xi}} f + R\pare{f}f 
\end{equation}
where
\begin{itemize}
\item $\nu_\cJ \pare{f; x} $ is the real function 
\begin{equation}
\label{eq:LcJ}
\nu_\cJ \pare{f; x} \defeq  \ \pare{  \sK^{2,0}_{\alpha}\pare{f;x}-1} +  \frac{1}{\alpha-1}\frac{f'(x) }{r^2}\sK^{3, 0}_\alpha\pare{f;x} \  \in \Sigma \cF^\bR_{K, 0, 1}\bra{\epsilon_0, N}  ; 
\end{equation}
\item 
$ L_\cJ \pare{\av{\xi}} \defeq  - \av{\xi}^2 \  \sM_{\alpha}\pare{\av{\xi}} $ is a real Fourier multiplier in $ \wt\Gamma^{\alpha-1}_0 $
(the Fourier multiplier $ \sM_\alpha  \pare{\av{\xi}}  $
is defined in \Cref{lem:linearization});  
\item 
$ S_{\cJ, \alpha-2}\pare{f;x, \xi}\defeq 
-\frac{\partial_x}{2}  \pare{\nu_{\cJ}\pare{f;x}}  \partial_\xi L_\cJ \pare{\av{\xi}} 
  +\pare{  \pare{\alpha - 2}  \sK^{2 ,1}_{\alpha}\pare{f;x} + 
\frac{1}{r^2}\pare{f'\sK^{3, 1}_\alpha\pare{f;x} - f''\sK^{3, 0}_\alpha\pare{f;x}} }  \sM_\alpha\pare{\av{\xi}}\xi $ is a real symbol in
$ \Sigma \Gamma^{\alpha-2} _{K, 0, 1}\bra{\epsilon_0, N} $;
\item $ V\bra{\cJ}\pare{f;x} $ is a real function in $ \Sigma \cF^\bR_{K, 0, 1}\bra{\epsilon_0, N} $; 
\item $ P \pare{f;x, \xi} $ is a symbol in $ \Sigma \Gamma^{-1} _{K, 0, 1}\bra{\epsilon_0, N}$ satisfying \eqref{areal}; 
\item  $ R\pare{f} $ is a real smoothing operator in 
$ \Sigma \cR^{-\rho} _{K, 0, 1}\bra{\epsilon_0, N}  $. 
\end{itemize}
\end{lemma}

The rest of this section is devoted to the proof of \Cref{lem:paralinearization_cJ}. 

By  \Cref{lem:paraproduct_Weyl}  we obtain that
\begin{multline}\label{eq:K2_Bony_paraproduct}
 \sK^2 _{\alpha, z} \pare{\frac{\Delta_z f}{r^2}}  f'\pare{x -z} 
 =  \ \OpBW{ \sK^2 _{\alpha, z} \pare{\frac{\Delta_z f}{r^2}} } f'\pare{x -z}
+ \OpBW{ f'\pare{x -z} } \  \pare{ \sK^2 _{\alpha, z} \pare{\frac{\Delta_z f}{r^2}} - \sK^2_{\alpha, z}\pare{0} } 
\\
+R_1\pare{ \sK^2 _{\alpha, z} \pare{\frac{\Delta_z f}{r^2}}- \sK^2_{\alpha, z}\pare{0} }f'\pare{x-z} 
+R_2\pare{ f'\pare{x-z} } \pare{ \sK^2 _{\alpha, z} \pare{\frac{\Delta_z f}{r^2}}- \sK^2_{\alpha, z}\pare{0} }
\end{multline}
where $ R_1, R_2 $ are smoothing operators in $ \wt \cR^{-\rho}_1 $ . 
Hence, recalling the definition of
 $  \cJ\pare{f} $   in \eqref{eq:cJ}, we have
\begin{equation}\label{eq:cJ_decomposition}
\begin{aligned}
 \cJ\pare{f} \defeq & \ \sum_{\mathsf{j}=1}^4 \cJ_\mathsf{j}\pare{f} \, , \\
\cJ_1 \pare{f} \defeq & \ \fint \OpBW{ \sK^2 _{\alpha, z} \pare{\frac{\Delta_z f}{r^2}}  } f'\pare{x -z} \frac{\sin z}{\av{2\sin\pare{z/2}}^\alpha} \ \dd z \, , \\
\cJ_2 \pare{f} \defeq & \ \fint \OpBW{ f'\pare{x -z} } \  \pare{ \sK^2 _{\alpha, z} \pare{\frac{\Delta_z f}{r^2}} - \sK^2_{\alpha, z}\pare{0} }  \frac{\sin z}{\av{2\sin\pare{z/2}}^\alpha} \ \dd z \, , \\
\cJ_3 \pare{f} \defeq & \ \fint R_1\pare{ \sK^2 _{\alpha, z} \pare{\frac{\Delta_z f}{r^2}}- \sK^2_{\alpha, z}\pare{0}}f'\pare{ x -z}  \frac{\sin z}{\av{2\sin\pare{z/2}}^\alpha} \ \dd z \, , \\
\cJ_4 \pare{f} \defeq & \ \fint R_2\pare{ f'\pare{x-z} } \pare{ \sK^2 _{\alpha, z} \pare{\frac{\Delta_z f}{r^2}}- \sK^2_{\alpha, z}\pare{0} }  \frac{\sin z}{\av{2\sin\pare{z/2}}^\alpha}  \ \dd z \, . 
\end{aligned}
\end{equation}

\begin{step}[Paralinearization of $ \cJ_1 $ in \eqref{eq:cJ_decomposition}]
By \eqref{eq:Taylor_expansion_kernels0}
and \eqref{eq:coefficients_Kernels_zero} we obtain that
\begin{equation}\label{svilJ14}
\begin{aligned}
\cJ_1\pare{f}\defeq & \ \fint  f'\pare{x -z} \frac{\sin z}{\av{2\sin\pare{z/2}}^\alpha} \ \dd z +   \sum_{\mathsf{j}=1}^3 \cJ_{1, \mathsf{j}}\pare{f} \, , \\
\cJ_{1,1}\pare{f}\defeq & \ \OpBW{ \sK^{2,0} _{\alpha} \pare{f;x}-1  } \fint  f'\pare{x -z} \frac{\sin z}{\av{2\sin\pare{z/2}}^\alpha} \ \dd z \, , \\
\cJ_{1,2}\pare{f}\defeq & \ \OpBW{ \sK^{2,1} _{\alpha} \pare{f;x} } \fint  f'\pare{x -z} \frac{ \sin^2 z}{\av{2\sin\pare{z/2}}^\alpha} \ \dd z \, , 
 \\
\cJ_{1,3}\pare{f}\defeq & \  \fint 
{\OpBW{ \varrho^{\bra{3-\alpha}}_\alpha\pare{f;x,z}  }  f'\pare{x -z}} \ \dd z \, , 
\end{aligned}
\end{equation}
where, by  
\eqref{prop:fr}, \eqref{eq:coefficients_Kernels_zero} and Remark \ref{rem:operation_z}, 
\begin{equation}\label{varrho22}
\varrho^{\bra{3-\alpha}}_\alpha\pare{f;x,z} 
\defeq { \pare{\sK^{2,2}_\alpha  \pare{f;x} \pare{2\sin\pare{z/2}}^2
  +\varrho^{2,3}_\alpha\pare{f;x,z}} \frac{ \sin z}{\av{2\sin\pare{z/2}}^\alpha} }
   \in \Sigma \ \KF^{3 - \alpha}_{K, 0, 1}\bra{\epsilon_0, N} \, . 
\end{equation}
Now,  by \eqref{eq:Malpha1}, the first  term in \eqref{svilJ14} is 
\begin{equation}\label{J10}
\fint  f'\pare{x -z} \frac{\sin z}{\av{2\sin\pare{z/2}}^\alpha} \ \dd z 
=  \av{D}^2 \  \sM_{\alpha}\pare{\av{D}} f 
\end{equation}
and, using Proposition \ref{prop:composition_BW}, 
\begin{align}
\cJ_{1, 1}\pare{f} & 
=  \   \OpBW{  \pare{  \sK^{2,0}_{\alpha}\pare{f;x}-1}} \av{D}^2  \sM_{\alpha}\pare{\av{D}} f 
\label{J11} \\
& = \ \OpBW{
  \pare{  \sK^{2,0}_{\alpha}\pare{f;x}-1} \  \av{\xi}^2  \sM_{\alpha}\pare{\av{\xi}}
  +\frac{\ii}{2} \partial_x \pare{\sK^{2,0}_{\alpha}\pare{f;x}} \  \partial_\xi\pare{\av{\xi}^2 \  \sM_{\alpha}\pare{\av{\xi}} +  P \pare{f;x, \xi}}
} f + R\pare{f}f  \notag 
\end{align}
where $ P\pare{f;x, \xi} $ is a symbol 
in $ \Sigma \Gamma^{-1}_{K,0, 1}\bra{\epsilon_0 , N} $. 
In order to expand $\cJ_{1, 2}\pare{f}$ in
\eqref{svilJ14} we write
\begin{equation}\label{decofuJ1}
\frac{ \sin^2 z}{\av{2\sin\pare{z/2}}^\alpha} = 
\frac{\cos^2 (z/2) }{\av{2\sin\pare{z/2}}^{\alpha-2}} = 
\frac{1}{\av{2 (1 - \cos\pare{z}}^{\frac{\alpha}{2}-1}}  + \varrho_{1,2} (z) \, , \quad \varrho_{1,2} (z) \in 
\wt{\KF}^{3-\alpha}_0 \, . 
\end{equation}
As a consequence of \eqref{decofuJ1}, using also \eqref{eq:step:3}, 
Propositions \ref{prop:reminders_integral_operator} and 
\ref{prop:composition_BW},  for any $ \alpha \in (0,2) $, we get 
\begin{align}
\cJ_{1, 2}\pare{f} & =  
\ \OpBW{ \sK^{2,1} _{\alpha} \pare{f;x} } \fint  f'\pare{x -z} 
\frac{\dd z }{\av{2 (1 - \cos\pare{z}}^{\frac{\alpha}{2}-1}} + 
 \fint  \OpBW{ \sK^{2,1} _{\alpha} \pare{f;x} \varrho_{1,2} (z)  }  f'\pare{x -z}   \ \dd z  \notag 
\\
& = 
\ \OpBW{ \sK^{2,1} _{\alpha} \pare{f;x} } \ii (\alpha - 2 ) M_\alpha (|D|)  D f + 
\OpBW{ a\pare{f;x, \xi}} \pa_x f  \notag 	\\ 
 & = \OpBW{ \ii  \pare{\alpha - 2}  \sK^{2 ,1}_{\alpha}\pare{f;x} \ \sM_\alpha\pare{\av{\xi}}\xi + P\pare{f;x, \xi}} f + R \pare{f}{f} \label{J12}
\end{align}
where $ a\pare{f;x, \xi} $ is a symbol in 
$ \Sigma \Gamma^{4-\alpha}_{K,0, 1}\bra{\epsilon_0 , N} $ and
$ P\pare{f;x, \xi} $ is a symbol 
in $ \Sigma \Gamma^{-1}_{K,0, 1}\bra{\epsilon_0 , N} $ satisfying \eqref{areal}.
Furthermore, by \eqref{varrho22}, Propositions \ref{prop:reminders_integral_operator}
and \ref{prop:composition_BW}, 
for any $ \alpha \in (0,2) $, 
the last term in \eqref{svilJ14} is
\begin{equation}\label{J12J13}
\cJ_{1, 3}\pare{f} =  \  \OpBW{P\pare{f;x, \xi}}f + { R \pare{f} f }  \, . 
\end{equation}
In conclusion, by \cref{J10,J11,J12,J12J13,eq:coefficients_Kernels_zero} defining 
\begin{equation}\label{eq:nu2}
\nu_2\pare{f;x} \defeq  \sK^{2,0}_{\alpha}\pare{f;x}-1 
\in \Sigma \cF^\bR_{K,0,1}\bra{\epsilon_0, N}
 \, , 
\end{equation} 
the term  $ \cJ_1\pare{f} $  in \eqref{svilJ14} is 
\begin{multline}\label{eq:J1}
\cJ_1\pare{f} 
 =   \OpBW{\pare{1+\nu_2\pare{f;x}} \av{\xi}^2  \sM_{\alpha}\pare{\av{\xi}} } f
\\
+\ii \ \OpBW{\frac{\partial_x}{2} \pare{\nu_2\pare{f;x}} \  \partial_\xi \pare{\av{\xi}^2   \sM_{\alpha}\pare{\av{\xi}} } + \pare{\alpha - 2}  \sK^{2 ,1}_{\alpha}\pare{f;x} \ \sM_\alpha\pare{\av{\xi}}\xi + P\pare{f;x, \xi}} f + R\pare{f}f \, .   
\end{multline}
\end{step}
\begin{step}[Paralinearization of $ \cJ_2 $ in \eqref{eq:cJ_decomposition}]
Using 
\eqref{eq:sK2}, \eqref{eq:notation}, the paralinearization formula 
\eqref{lem:Bony_paralinearization_W}  and  \eqref{eq:sK3}, 
we write
\begin{equation}\label{paraK2K0}
\begin{aligned}
\sK^2_{\alpha, z}\pare{\frac{\Delta_z f}{r^2}} - \sK^2_{\alpha, z}\pare{0} 
= & \ 
\pare{\sG^2_{\alpha, z}\pare{\frac{\delta_zf}{r^2}} - \sG^2_{\alpha, z}\pare{0}} \av{2\sin\pare{z/2}}^\alpha
\\
%------------------------------------------------------
= & \ \OpBW{\pare{\sG^2_{\alpha, z}}'\pare{\frac{\delta_z f }{r^2}}\av{2\sin\pare{z/2}}^\alpha}\frac{\delta_z f }{r^2}
+ R \pare{\frac{\delta_z f }{r^2}}\frac{\delta_z f }{r^2} \ \av{2\sin\pare{z/2}}^\alpha
\\
%------------------------------------------------------
= & \ 
\OpBW{\sK^3_{\alpha, z}\pare{\frac{\Delta_z f}{r^2}}}
\frac{\delta_zf}{r^2 \sin (z) }  
+ R \pare{\frac{\delta_z f }{r^2}}\frac{\delta_z f }{r^2} \ \av{2\sin\pare{z/2}}^\alpha
\end{aligned}
\end{equation}
where $ R $ is a smoothing operator in 
$\Sigma \cR^{-\rho}_{K,0,1}\bra{\epsilon_0, N}$ for any $ \rho $. 
 By \cref{eq:sK3} it results 
$ \sK^3_{\alpha, z}\pare{\mathsf{X}} = \sK^3_{\alpha, z+2\pi}\pare{\mathsf{- X}} $ 
and the  map $ z\mapsto \sK^3_{\alpha, z}\pare{\frac{\Delta_z f}{r^2}} $ is $ 2\pi $-periodic.
 Therefore, by \eqref{eq:cJ_decomposition} and \eqref{paraK2K0} we obtain that
\begin{subequations}
\begin{align}
\cJ_2\pare{f} 
 = & \  \fint \OpBW{f'\pare{x-z}}\OpBW{\sK^3_{\alpha, z}\pare{\frac{\Delta_z f}{r^2}}} \  \frac{\delta_zf}{r^2 \av{2\sin\pare{z/2}}^\alpha}  \ \dd z \label{inJ1} \\
&  + \fint \OpBW{f'\pare{x-z}}R \bra{\frac{\delta_z f }{r^2}}\frac{\delta_z f }{r^2} \ \sin z \  \dd z \, . \label{eq:J2_1}
\end{align}
\end{subequations}
By \eqref{eq:deltaDelta} and % we have that $ \delta_z \in \wt \KM ^{1, 1}_0 $.
\Cref{item:propMop2} we deduce that 
 $  M_2\pare{f;z} \defeq 
 r^{-2} \delta_z $ 
 is an operator
   in $ \Sigma \KM^{1,1}_{K, 0, 0}\bra{\epsilon_0, N} $. 
 As a consequence  by \Cref{prop:comp_z} we obtain that 
 \begin{equation}\label{Rfraz}
 R \pare{\frac{\delta_z f }{r^2}}\frac{\delta_z f }{r^2}   = R\pare{f;z} f \qquad 
  \text{where} \qquad R\pare{f;z} \in \Sigma \KR^{-\rho, 1}_{K, 0, 1}\bra{\epsilon_0, N} \, ,
 \end{equation}
  and, by  \Cref{prop:comp_z}, \cref{item:OpR_ext_z},  
\Cref{prop:action_z_smoothing}, being $ \alpha \in (0,2) $, we deduce 
that the integral \eqref{eq:J2_1} is 
\begin{equation}\label{la587}
\fint \OpBW{f'\pare{x-z} \sin z } R\pare{\frac{\delta_zf}{r^2}} \frac{\delta_zf}{r^2}  \dd z 
= R\pare{f}f 
\end{equation}
where $ R\pare{f} $ is a smoothing operator 
in $ \Sigma \cR^{-\rho}_{K,0,1}\bra{\epsilon_0, N} $. 

We now consider the term \eqref{inJ1}. By 
\Cref{lem:paraproduct_Weyl} we write
$$
\frac{\delta_z f}{r^2} = \OpBW{r^{-2}}  \delta_z f +
\OpBW{\delta_z f } (r^{-2} -1)+ R_1 \pare{ r^{-2} -1 }\delta_z f + 
R_2 \pare{ \delta_z f   }[r^{-2} -1]
$$
where $ R_1, R_2 $ are smoothing operators  
 in $ \wt \cR^{-\rho}_1 $ for any $ \rho \geq 0 $, and thus
 \begin{subequations}
  \begin{align}
 \eqref{inJ1} 
 & = \fint \OpBW{f'\pare{x-z}}
 \OpBW{\sK^3_{\alpha, z}
 \pare{\frac{\Delta_z f}{r^2}}} 
 \  \OpBW{r^{-2}}  \delta_z f \ \frac{\dd z}{\av{2\sin\pare{z/2}}^\alpha}  \label{inJ11}   \\
 & + \fint \OpBW{f'\pare{x-z}}
 \OpBW{\sK^3_{\alpha, z}
 \pare{\frac{\Delta_z f}{r^2}}} 
 \  \OpBW{\delta_z f } (r^{-2}-1) 
 \ \frac{\dd z}{\av{2\sin\pare{z/2}}^\alpha}   \label{inJ12}  \\
 & + \fint \OpBW{f'\pare{x-z}}
 \OpBW{\sK^3_{\alpha, z}
 \pare{\frac{\Delta_z f}{r^2}}} 
  \pare{R_1 \pare{ r^{-2} -1 }\delta_z f + 
R_2 \pare{ \delta_z f   }[r^{-2} -1]}  \, \frac{\dd z}{\av{2\sin\pare{z/2}}^\alpha}.  \label{inJ13}
 \end{align}
 \end{subequations}  
 \Cref{prop:comp_z}   
give that
$$
\eqref{inJ11} = \fint \OpBW{r^{-2} f'\pare{x-z} \sK^3_{\alpha, z}\pare{\frac{\Delta_z f}{r^2}}} \frac{\delta_zf}{ \av{2\sin\pare{z/2}}^\alpha}  \ \dd z + \underbrace{\fint R\pare{f;z} f \  \dd z}_{= R \pare{f}f }  
$$
where $ R\pare{f;z} $ is a Kernel-smoothing operator in 
$ \Sigma \cR^{-\rho, 1-\alpha }_{K,0, 1}\bra{\epsilon_0, N} $
and, since $ \alpha \in (0,2) $, the operator 
$ R \pare{f} $ is in $  \Sigma \cR^{-\rho}_{K,0, 1}\bra{\epsilon_0, N} $ by \Cref{prop:action_z_smoothing}.
Then by   \Cref{prop:comp_z,eq:parabeta,prop:integral_kernelfunction}, we get 
$$
 \eqref{inJ12} = 
 \OpBW{V_1 \bra{\cJ_2}\pare{f;x}} f 
+  R\pare{f} f 
$$
where $ V_1 \bra{\cJ_2} \pare{f;x} $ is a real function in $  \Sigma \cF^\bR_{K, 0, 1}\bra{\epsilon_0, N} $. Finally  \eqref{inJ13} is a smoothing term $ R\pare{f} f  $
and we deduce that 
\begin{equation}\label{eq:J2_2}
\eqref{inJ1} =
\fint \OpBW{r^{-2} f'\pare{x-z} \sK^3_{\alpha, z}\pare{\frac{\Delta_z f}{r^2}}} \frac{\delta_zf}{ \av{2\sin\pare{z/2}}^\alpha}  \ \dd z
+\OpBW{V_1 \bra{\cJ_2}\pare{f;x}} f 
+ R\pare{f} f \, .  
\end{equation}
Then we write
$$
 f'(x-z) = f'(x) - f''(x) \  \sin z + {\varrho}_{1} \pare{f;x,z} 
$$
where $ {\varrho}_{1} \pare{f;x,z} $ is a homogenous Kernel function in 
$  \widetilde{ \KF}^2_{1}  $ and, using \eqref{eq:Taylor_expansion_kernels0}, we deduce that
\begin{equation*}
r^{-2} f'\pare{x-z} \sK^3_{\alpha, z}\pare{\frac{\Delta_z f}{r^2}}
=
r^{-2} f' (x) \sK^{3, 0}_\alpha\pare{f;x} + r^{-2}\pare{f' (x) \sK^{3, 1}_\alpha\pare{f;x} - f''(x) \sK^{3, 0}_\alpha\pare{f;x}} \sin z
+\breve{\varrho}^{3, 2}_\alpha\pare{f;x, z}, 
\end{equation*}
where $ \breve{\vr}^{3, 2}_\alpha \pare{f;x, z} $ is a Kernel function 
in $ \Sigma \KF^2_{K, 0, 1}\bra{\epsilon_0, N} $, and,   
by also \eqref{eq:step:1}, \cref{lem:mult2,prop:reminders_integral_operator} and
\Cref{prop:composition_BW},  
we get   
\begin{align}
& \fint \OpBW{r^{-2} f'\pare{x-z} \sK^3_{\alpha, z}\pare{\frac{\Delta_z f}{r^2}}} \frac{\delta_zf}{ \av{2\sin\pare{z/2}}^\alpha}  \ \dd z = 
 \OpBW{r^{-2} f' (x) \sK^{3, 0}_\alpha\pare{f;x}}\OpBW{\sT^1_\alpha\pare{\av{\xi}}} f 
 \notag \\
& +\ii \ \OpBW{ r^{-2}\pare{f'(x) \sK^{3, 1}_\alpha\pare{f;x} - f''(x) \sK^{3, 0}_\alpha\pare{f;x}}} \OpBW{\sM_\alpha\pare{\av{\xi}}\xi} f + \OpBW{P\pare{f;x, \xi}}f \notag \\
& =
\OpBW{ r^{-2} f'(x) \sK^{3, 0}_\alpha\pare{f;x} \sT^1_\alpha\pare{\av{\xi}}} f \notag 
\\
& +\ii \ \OpBW{\frac{1}{2}\partial_x\pare{\frac{f'(x)}{r^2}\sK^{3, 0}_\alpha\pare{f;x}} \ \partial_\xi \sT^1_\alpha\pare{\av{\xi}}
+
r^{-2} \pare{f'(x) \sK^{3, 1}_\alpha\pare{f;x} - f''(x) \sK^{3, 0}_\alpha\pare{f;x}} \ \sM_\alpha\pare{\av{\xi}}\xi
 }f \notag 
 \\
 & + \OpBW{P\pare{f;x, \xi}}f + R\pare{f}f \, .  \label{eq:J2_3}
\end{align}
By \Cref{lem:asympt}  we have 
$ \sT^1_\alpha\pare{\av{\xi}} = \frac{1}{\alpha-1} \av{\xi}^2 \sM_\alpha\pare{\av{\xi}} + \tilde{\bV}_\alpha + m_{\alpha-3}\pare{|\xi|} $
and so, defining the function 
\begin{equation}
\label{eq:nu3}
\nu_3 \pare{f;x}\defeq \frac{1}{\alpha-1}\frac{f'\pare{x}}{r^2}\sK^{3, 0}_\alpha\pare{f;x} \in \Sigma\cF^\bR_{K, 0, 1}\bra{\epsilon_0, N}, 
\end{equation}
the equation \eqref{eq:J2_3} becomes 
\begin{multline}\label{eq:J2_4}
\fint \OpBW{r^{-2} f'\pare{x-z} \sK^3_{\alpha, z}\pare{\frac{\Delta_z f}{r^2}}} \frac{\delta_zf}{ \av{2\sin\pare{z/2}}^\alpha}  \ \dd z
=
\OpBW{\nu_3 \pare{f;x} \av{\xi}^2 \sM_\alpha\pare{\av{\xi}}} f
\\
+\ii \ \OpBW{\frac{1}{2} \partial_x \pare{\nu_3 \pare{f;x}}\ \partial_\xi \pare{\av{\xi}^2 \sM_\alpha\pare{\av{\xi}}}
+ r^{-2}\pare{f' (x) \sK^{3, 1}_\alpha\pare{f;x} - f'' (x) \sK^{3, 0}_\alpha\pare{f;x}} \ \sM_\alpha\pare{\av{\xi}}\xi
 }f
 \\
  + \OpBW{ V_2\bra{\cJ_2}\pare{f;x} +  P\pare{f;x, \xi}}f + R\pare{f}f
\end{multline}
where  $ V_2\bra{\cJ_2} \pare{f;x} \defeq \tilde{\bV}_\alpha \, \nu_3 \pare{f; x}  $ is 
a function in $ \Sigma \cF^\bR_{K, 0, 1}\bra{\epsilon_0, N} $. 
By \eqref{la587}, \eqref{eq:J2_2}, \eqref{eq:J2_4}
we deduce that $ \cJ_2\pare{f}  $ in \eqref{inJ1}-\eqref{eq:J2_1} is
\begin{multline}\label{eq:J2}
\cJ_2\pare{f} =
\OpBW{\nu_3 \pare{f;x} \av{\xi}^2 \sM_\alpha\pare{\av{\xi}}} f
\\
+\ii \ \OpBW{\frac{\partial_x}{2} \pare{\nu_3 \pare{f;x}} \ \partial_\xi \pare{\av{\xi}^2 \sM_\alpha\pare{\av{\xi}}}
+ r^{-2}\pare{f' (x) \sK^{3, 1}_\alpha\pare{f;x} - f'' (x) \sK^{3, 0}_\alpha\pare{f;x}} \ \sM_\alpha\pare{\av{\xi}}\xi
 }f
 \\
  + \OpBW{V\bra{\cJ_2}\pare{f;x} + P\pare{f;x, \xi}}f + R\pare{f}f 
\end{multline}
where the real function $ V\bra{\cJ_2} \defeq V_1 \bra{\cJ_2} + V_2\bra{\cJ_2} $ is  in 
$ \Sigma \cF^\bR_{K, 0, 1}\bra{\epsilon_0, N}   $.
\end{step}

\begin{step}[Paralinearization of $ \cJ_3 $ in \eqref{eq:cJ_decomposition}]
By \cref{paraK2K0,Rfraz,prop:comp_z,rem:OpBWisanMop-z} and 
$ \frac{\delta_z}{r^2 \sin z}\in \Sigma \KM^{1, 0}_{K, 0, 0}\bra{\epsilon_0, N} $
(which follows by \cref{eq:deltaDelta,item:propMop2,prop:comp_z}), and 
 since 
$ R_1 $ is a smoothing operator in $ \wt \cR^{-\rho}_1 $,
we deduce that 
\begin{equation}\label{eq:R1compostoK}
R_1\bra{ \sK^2 _{\alpha, z} \pare{\frac{\Delta_z f}{r^2}} - \sK^2_{\alpha, z}\pare{0}}  \in \Sigma \KR^{-\rho, 0}_{K, 0, 1}\bra{\epsilon_0, N}. 
\end{equation}
Furthermore $ f'\pare{x-z} = \partial_x \circ \st _{-z} f $  
and $ \partial_x \circ \st _{-z} $ is in $ \wt \KM ^{1, 0}_{0} $. By \Cref{rem:operation_z,prop:comp_z} we obtain (after relabeling  $ \rho $) that
\begin{equation*}
 R_1\bra{ \sK^2 _{\alpha, z} \pare{\frac{\Delta_z f}{r^2}}- \sK^2_{\alpha, z}\pare{0}} \frac{\sin z}{\av{2\sin\pare{z/2}}^\alpha} \partial_x \circ \st _{-z}
 \defeq R^\star \pare{f;z}
 \in \Sigma \KR^{-\rho, 1-\alpha}_{K, 0, 1}\bra{\epsilon_0, N} \ .
\end{equation*}
Finally \Cref{prop:action_z_smoothing} implies  that 
\begin{equation}\label{eq:J3}
\cJ_3\pare{f} =
\fint  R^\star \pare{f;z} f \ \dd z
= R\pare{f}f \qquad 
\text{where}
\qquad 
R\pare{f}\in\Sigma\cR^{-\rho}_{K,0,1}\bra{\epsilon_0, 1} \, . 
\end{equation}
\end{step}
\begin{step}[Paralinearization of $ \cJ_4 $ in \eqref{eq:cJ_decomposition}] 
We similar arguments 
one obtains 
\begin{equation}\label{eq:J4}
\cJ_4\pare{f} = R\pare{f}f \qquad 
\text{where}
\qquad 
R\pare{f}\in\Sigma\cR^{-\rho}_{K,0,1}\bra{\epsilon_0, 1} 
\, . 
\end{equation}
\end{step}
\begin{step}[Conclusion]
We plug \Cref{eq:J1,eq:J2,eq:J3,eq:J4} in \cref{eq:cJ_decomposition} and, 
recalling that $ L_\cJ\pare{\av{\xi}} = -\av{\xi}^2 \sM _\alpha \pare{\av{\xi}} $, 
defining  the real functions $ V\bra{\cJ}\defeq V\bra{\cJ_2} $  in 
$ \Sigma \cF^{\bR}_{K, 0, 1}\bra{\epsilon_0, N} $   and  
$ \nu_{\cJ} \defeq  \nu_2 + \nu_3 $ in $ \Sigma \cF^\bR_{K, 0, 1}\bra{\epsilon_0, N}
$ (cf. \cref{eq:nu2,eq:nu3})  
we obtain the paralinearization formula \eqref{eq:paralinearization_cJ} stated in \Cref{lem:paralinearization_cJ}. 
\end{step}

\subsection{Proof of \Cref{prop:paralinearization_1} }

We now paralinearize the scalar field in \Cref{eq:fcIcJ}. 
We apply \Cref{lem:paraproduct_Weyl}
\begin{align*}
 r^{2-\alpha}\cI\pare{f}  = &  \OpBW{r^{2-\alpha}}  \cI\pare{f}  + \OpBW{\cI\pare{f}} \pare{r^{2-\alpha}-1} + R_1 \pare{r^{2-\alpha}-1} \cI\pare{f} + R_2\pare{\cI\pare{f}} \pare{r^{2-\alpha}-1}, 
\\
 r^{-\alpha}\cJ\pare{f}  = &   \OpBW{r^{-\alpha}}  \cJ\pare{f}  + \OpBW{\cJ\pare{f}} \pare{r^{-\alpha}-1} + R_1 \pare{r^{-\alpha}-1} \cJ \pare{f} + R_2 \pare{\cJ\pare{f}} \pare{r^{-\alpha}-1}. 
\end{align*}
We thus apply \eqref{eq:parabeta},  \cref{lem:MtoFunctions,lem:paralinearization_cI,lem:paralinearization_cJ,prop:composition_BW,compositionMoperator}
and obtain that there exist real functions $ \tilde{V}_\cI, \tilde{V}_\cJ $ in 
$ \Sigma \cF^{\bR}_{K, 0, 1}\bra{\epsilon_0, N} $ such that
\begin{equation}
\label{eq:buuh8}
\begin{aligned}
 r^{2-\alpha}\cI\pare{f}  = & \  \OpBW{r^{2-\alpha}}  \cI\pare{f} +\OpBW{\tilde{V}_\cI\pare{f;x}} f + 
 R \pare{f}f , \\
%---------------------------------------------
 r^{-\alpha}\cJ\pare{f}  = &  \  \OpBW{r^{-\alpha}}  \cJ\pare{f}  +\OpBW{\tilde{V}_\cJ \pare{f;x}} f + 
 R\pare{f}f , 
\end{aligned}
\end{equation}
for some smoothing operator $ R\pare{f} $ in $ \Sigma \cR^{-\rho}_{K, 0, 1}\bra{\epsilon_0, N} $. 

A key fact proved in the  next lemma is that the imaginary part of the symbol in  \eqref{eq:buuh2_1} has order at most $ -1 $. This is actually an 
effect % reflex 
of the linear Hamiltonian structure, see Remark \ref{LHAM}.
\begin{lemma}
It results
\begin{multline}\label{eq:buuh2_1}
\OpBW{r^{2-\alpha}}  \cI\pare{f} + \OpBW{r^{-\alpha}}  \cJ\pare{f} 
=
\ - \OpBW{ \pare{1+\tilde \nu_\cI\pare{f;x}}L_\cI\pare{\av{\xi}} 
+
\pare{1+\tilde \nu_\cJ\pare{f;x}}L_\cJ\pare{\av{\xi}}
} f 
\\
\ + \OpBW{V_\cI \pare{f;x} + V_\cJ\pare{f;x} + P\pare{f;x,\xi}} f + R\pare{f}f
\end{multline}
where 
\begin{itemize}
\item $ L_\cI\pare{\av{\xi}} $ and $ L_{\cJ} \pare{\av{\xi}}  $ are the real Fourier multipliers defined in  \Cref{lem:paralinearization_cI,lem:paralinearization_cJ}; 
\item  $ \tilde{\nu}_\cI \pare{f;x} $, 
$ \tilde{\nu}_{\cJ} \pare{f;x} $, $  V_\cI \pare{f;x} $, 
$ V_\cJ \pare{f;x} $ are real functions
in $ \Sigma \cF^\bR_{K, 0, 1}\bra{\epsilon_0, N} $; 
\item $ P\pare{f;x,\xi}$ is a symbol  in $ \Sigma \Gamma^{-1}_{K, 0, 1}\bra{\epsilon_0, N} $; 
\item $ R\pare{f} $ is a smoothing operator  in
 $ \Sigma \cR^{-\rho}_{K, 0, 1}\bra{\epsilon_0, N} $.  
\end{itemize}
\end{lemma}

\begin{proof}
\Cref{prop:composition_BW,lem:paralinearization_cI,lem:paralinearization_cJ}
give that
\begin{align*}
\OpBW{r^{2-\alpha}} \cI\pare{f} 
= & \ \OpBW{r^{2-\alpha}\pare{-\pare{1+\nu_\cI\pare{f;x}}L_\cI\pare{\av{\xi}} + \ii \ S_{\cI, \alpha-2}\pare{f;x, \xi}  }} f
\\
& \ + \OpBW{\frac{1}{2\ii}  \pare{r^{2-\alpha}}_x \pare{1+\nu_\cI\pare{f; x}}\partial_\xi L_\cI\pare{\av{\xi}}} f
\\
& \ + \OpBW{V_\cI\pare{f;x} + P\pare{f;x,\xi}} f + R\pare{f}f
,
\\
\OpBW{r^{-\alpha}}  \cJ\pare{f} 
= & \ \OpBW{r^{-\alpha}\pare{-\pare{1+\nu_\cJ\pare{f;x}}L_\cJ\pare{\av{\xi}} + \ii \ S_{\cJ, \alpha-2}\pare{f;x, \xi}  }} f
\\
& \ + \OpBW{\frac{1}{2\ii}  \pare{r^{-\alpha}}_x \pare{1+\nu_\cJ \pare{f; x}}\partial_\xi L_\cJ\pare{\av{\xi}}} f 
\\
& \ + \OpBW{V_\cJ\pare{f;x} + P\pare{f;x,\xi}} f + R\pare{f}f, 
\end{align*}
so that, defining
$ \tilde{\nu}_\cI \pare{f;x} \defeq r^{2-\alpha} \pare{1+\nu_\cI\pare{f;x}} -1 $ and 
$ \tilde{\nu}_\cJ \pare{f;x} \defeq r^{-\alpha} \pare{1+\nu_\cJ\pare{f;x}} -1 $, we get 
\begin{subequations}
\label{ciochev}
\begin{align} 
\OpBW{r^{2-\alpha}}   \cI\pare{f} & + \OpBW{r^{-\alpha}}  \cJ\pare{f} \nonumber  \\
= 
&
\ - \OpBW{ \pare{1 + \tilde{\nu}_\cI\pare{f;x}}
L_\cI\pare{\av{\xi}} 
+ \pare{1+ \tilde{\nu}_\cJ\pare{f;x}}L_\cJ\pare{\av{\xi}}
} f   \notag 
\\  
& \ + \ii \ \OpBW{r^{2-\alpha} \ S_{\cI, \alpha-2}\pare{f;x, \xi}  + r^{-\alpha} \ S_{\cJ, \alpha-2}\pare{f;x, \xi}  } f 
\label{eq:pain1}
\\
& \ + 
\frac{1}{2\ii} \ \OpBW{
\pare{r^{2-\alpha}}_x \pare{1+\nu_\cI\pare{f; x}}\partial_\xi L_\cI\pare{\av{\xi}}
+
 \pare{r^{-\alpha}}_x \pare{1+\nu_\cJ \pare{f; x}}\partial_\xi L_\cJ\pare{\av{\xi}}
}f
\label{eq:pain2}
\\
& \ + \OpBW{V_\cI \pare{f;x} + V_\cJ\pare{f;x} + P\pare{f;x,\xi}} f + R\pare{f}f \, . \notag 
\end{align}
\end{subequations}
We now prove that the sum of \eqref{eq:pain1} and \eqref{eq:pain2}   
give a paradifferential term of order $ - 1$. 
We first note that, 
by \cref{lem:asympt},  we have the 
asymptotic expansions
\begin{equation}
\begin{aligned}\label{asi0}
& \av{\xi}^2 \sM_\alpha\pare{\av{\xi}} =  \  \breve{c}_\alpha \av{\xi}^{\alpha-1} + m_{\alpha-3}\pare{\av{\xi}},  && \xi \sM_\alpha\pare{\av{\xi}} =  \ \breve{c}_\alpha \av{\xi}^{\alpha-3} \xi + m_{\alpha-4}\pare{\av{\xi}}, 
\\
& \sT^1_\alpha\pare{\av{\xi}} =   \frac{1}{\alpha-1} \breve{c}_\alpha\av{\xi}^{\alpha-1} + \tilde \bV_\alpha + m_{\alpha-3}\pare{\av{\xi}}, && 
\text{where} \quad 
\breve{c}_\alpha \defeq \frac{\Gamma\pare{2-\alpha}}{\Gamma\pare{1-\frac{\alpha}{2}}\Gamma\pare{\frac{\alpha}{2}}} \, , 
\end{aligned}
\end{equation}
so that
$$
\partial_\xi L_\cI\pare{\av{\xi}} = \breve{c}_\alpha \av{\xi}^{\alpha-3} \xi + m_{\alpha-4}\pare{\av{\xi}}, \qquad 
\partial_\xi L_\cJ\pare{\av{\xi}} = -\pare{\alpha-1}\breve{c}_\alpha \av{\xi}^{\alpha-3} \xi  + m_{\alpha-4}\pare{\av{\xi}} \, . 
$$
By the explicit definition of the symbols $ S_{\cI, \alpha-2} $ and $ S_{\cJ, \alpha-2} $  in \Cref{lem:paralinearization_cI,lem:paralinearization_cJ} and \eqref{asi0}
we have the expansion of the symbol in \eqref{eq:pain1}
\begin{align}
& \ii \pare{r^{2-\alpha} \ S_{\cI, \alpha-2}\pare{f;x, \xi}  + r^{-\alpha} \ S_{\cJ, \alpha-2}\pare{f;x, \xi}}  \label{siespa1} \\
& = \ii \bra{- r^{2-\alpha} \  \frac{1}{2} \pare{ \nu_\cI }_x \pare{f;x} + \pare{\alpha-1} r^{-\alpha} \  \frac{1}{2} \pare{ \nu_\cJ }_x \pare{f;x}
+ A_{\alpha, 1}\pare{f;x}
} \breve{c}_\alpha \av{\xi}^{\alpha-3}\xi
+ \ii P\pare{f;x, \xi} , \notag 
\end{align}
where   
\begin{equation}\label{eq:A1}
A_{\alpha, 1}\pare{f;x} \defeq 
\frac{1}{r^{\alpha}}
\bra{ \sK^{1,1}_{\alpha}\pare{f;x } +  \pare{\alpha - 2}  \sK^{2 ,1}_{\alpha}\pare{f;x } + \frac{1}{r^2} \pare{ f ' \ \sK^{3,1}_{\alpha} \pare{f;x} - f'' \ \sK^{3,0}_{\alpha} \pare{f;x} }  }
\end{equation}
is a function in $  \Sigma \cF^\bR_{K,0,1}\bra{\epsilon_0,N} $, recalling \eqref{eq:coefficients_Kernels_zero}.  
Then the sum of  \eqref{eq:pain1} and \eqref{eq:pain2}  gives 
\begin{multline}\label{eq:pain3}
  r^{2-\alpha} \ S_{\cI, \alpha-2}\pare{f;x, \xi}  + r^{-\alpha} \ S_{\cJ, \alpha-2}\pare{f;x, \xi} 
\\
 -\frac{1}{2}
\bra{ \Big. \pare{r^{2-\alpha}}_x \pare{1+\nu_\cI\pare{f; x}}\partial_\xi L_\cI\pare{\av{\xi}}
+
\ \pare{r^{-\alpha}}_x \pare{1+\nu_\cJ \pare{f; x}}\partial_\xi L_\cJ\pare{\av{\xi}}} \\
=
\set{\frac{1}{2} 
\underbrace{\bra{ -r^{2-\alpha}\pare{1+\nu_\cI\pare{f;x}} + \pare{\alpha-1}r^{-\alpha}\pare{1+\nu_\cJ\pare{f;x}}}_x}_{= \pare{ A_{\alpha, 0}\pare{f;x}}_x }
+
A_{\alpha, 1}\pare{f;x}}\breve{c}_\alpha \av{\xi}^{\alpha-3}\xi
+ P\pare{f;x, \xi} , 
\end{multline}
where, having 
substituting the explicit values of $ \nu_\cI $, $ \nu_\cJ $ in \eqref{eq:LcI}, \eqref{eq:LcJ}, 
we define 
\begin{equation}\label{eq:A0}
A_{\alpha, 0}\pare{f;x} \defeq   \frac{1}{r^{\alpha}} \bra{\sK^{1,0}_{\alpha}\pare{f;x} +  \pare{\alpha-1} \sK^{2,0}_{\alpha}\pare{f;x} +   \frac{ f' }{r^{2}} \ \sK^{3,0}_{\alpha} \pare{f;x}  } + \pare{2-\alpha}  
\end{equation}
which is a function in $ \Sigma \cF^\bR_{K,0,1}\bra{\epsilon_0,N} $.  
We finally write 
\begin{equation}\label{eq:buuh1}
 \eqref{eq:pain3} 
=
\bra{ \frac{1}{2} \pare{ A_{\alpha, 0}\pare{f;x} }_x
+
A_{\alpha, 1}\pare{f;x}}\breve{c}_\alpha \av{\xi}^{\alpha-3}\xi
+ P\pare{f;x, \xi} = P\pare{f;x, \xi}  
\end{equation}
in view of the key cancellation
\begin{equation}
\label{eq:Hamiltonian_identity}
A_{\alpha , 1}\pare{f;x}  + \frac{1}{2} \ \pare{A_{\alpha,  0}\pare{f;x}}_x = 0 \, .
\end{equation}
 proved in \Cref{sec:Hamiltonian_identity}.  
By \eqref{eq:buuh1} we deduce that \eqref{ciochev} has the form 
\eqref{eq:buuh2_1}.
\end{proof}

\begin{rem}\label{LHAM}
The algebraic 
reason of the cancellation \eqref{eq:Hamiltonian_identity}
is that a symbol of the form $ \ii g\pare{f;x} |\xi|^{\alpha-3} \xi $, as in \eqref{siespa1},  
with a real function $ g\pare{f;x} $, 
does not 
respect the Hamiltonianity condition \eqref{Hamassy}.
\end{rem}

The next lemma enables to highlight the quasilinear structure of the vector field in \eqref{eq:fcIcJ}.
 
 \begin{lemma}
It results 
 \begin{multline}\label{eq:buuh6}
 \OpBW{r^{2-\alpha}}   \cI\pare{f} + \OpBW{r^{-\alpha}}   \cJ\pare{f}
+
\fint  \mathsf{G}^1_{\alpha , z}\pare{0} \dd z  \  \pare{r^{2-\alpha}-1}
\\
 =
- \pare{\frac{c_\alpha}{2\pare{1-\frac{\alpha}{2}}}}^{-1}
\OpBW{\pare{1+\nu\pare{f;x}}L_\alpha\pare{\av{\xi}}} f + \OpBW{\tilde{\tilde{ V}}\pare{f;x} + P\pare{f;x, \xi}} f + R\pare{f}f 
\end{multline}
where  $L_\alpha\pare{\av{\xi}}$ is the Fourier multiplier defined in \cref{lem:linearization} and 
 \begin{itemize}
 \item  $ \nu\pare{f;x}, \tilde{\tilde{V}} \pare{f;x} $ are 
 real functions in  $ \Sigma \cF^\bR_{K,0,1}\bra{\epsilon_0, N} $; 
 \item $ P \pare{f;x,\xi} $ is a symbol  in  $\Sigma \Gamma^{-1}_{K,0,1}\bra{\epsilon_0, N} $; 
 \item $ R\pare{f} $ is 
 a smoothing operator  in  $ \Sigma \cR^{-\rho}_{K,0,1}\bra{\epsilon_0, N} $. 
 \end{itemize}
 \end{lemma}

\begin{proof}
By  \eqref{eq:intG1(0)} and \eqref{eq:parabeta} we have 
\begin{equation}\label{eq:buuh3}
 \fint  \mathsf{G}^1_{\alpha , z}\pare{0} \dd z \  \pare{r^{2-\alpha}-1}
 =
 2 \ \frac{\Gamma\pare{2-\alpha}}{\Gamma\pare{1-\frac{\alpha}{2}}^2 } \pare{ f  + \OpBW{r^{-\alpha}-1} f } + R\pare{f}f . 
\end{equation}
Notice now, from \cref{lem:paralinearization_cI,lem:paralinearization_cJ} and \Cref{lem:linearization},  that
\begin{equation}\label{eq:buuh4}
L_\cI\pare{\av{\xi}} + L_\cJ\pare{\av{\xi}} -  2 \ \frac{\Gamma\pare{2-\alpha}}{\Gamma\pare{1-\frac{\alpha}{2}}^2}
=
\pare{\frac{c_\alpha}{2\pare{1-\frac{\alpha}{2}}}}^{-1}
L_\alpha\pare{\av{\xi}} \ .
\end{equation}
Now we claim that
\begin{equation}\label{eq:buuh5}
\tilde{\nu}_\cI\pare{f;x} L_\cI\pare{\av{\xi}} + \tilde{\nu}_\cJ \pare{f;x} L_\cJ\pare{\av{\xi}}
=
\pare{\frac{c_\alpha}{2\pare{1-\frac{\alpha}{2}}}}^{-1}
\nu\pare{f;x}L_\alpha\pare{\av{\xi}} + \tilde{V}\pare{f;x} + P\pare{f;x, \xi} , 
\end{equation}
for a suitable real functions $ \nu, \tilde{V} $ in 
$ \Sigma\cF^\bR_{K, 0, 1}\bra{\epsilon_0, N} $ and a symbol $ P $ in  
$ \Sigma \Gamma^{-1}_{K,0,1}\bra{\epsilon_0, N} $.
% Let us comment on the derivation of \cref{eq:buuh5}.
 From \Cref{lem:paralinearization_cI,lem:paralinearization_cJ} 
 and the asymptotic decomposition   of $ \sT^1_\alpha $ and $ \sM_\alpha $  in
 \Cref{lem:asympt}, 
 we have that
\begin{equation*}
\begin{aligned}
\text{l.h.s. of \eqref{eq:buuh5}}
& =
\tilde{\nu}_\cI \pare{f;x} \sT^1_\alpha\pare{\av{\xi}}
-
\tilde{\nu}_\cJ\pare{f;x} \av{\xi}^2 \sM_\alpha\pare{\av{\xi}} + V\pare{f; x} + P\pare{f; x, \xi} \\
& = 
\frac{\Gamma\pare{2-\alpha}}{\Gamma\pare{1-\frac{\alpha}{2}}\Gamma\pare{\frac{\alpha}{2}}} \frac{\av{\xi}^{\alpha-1}}{\alpha-1} \ \pare{\tilde{\nu}_\cI\pare{f;x} - \pare{\alpha-1}\tilde{\nu}_{\cJ}\pare{f;x}}
+ V\pare{f; x} + P\pare{f; x, \xi} \, . 
\end{aligned}
\end{equation*}
Defining  
$ \nu\pare{f; x}\defeq  \frac{\tilde{\nu}_I\pare{f;x} - \pare{\alpha-1}\tilde{\nu}_{\cJ}\pare{f;x}}{2-\alpha}$ and using the identity
 $ \Gamma\pare{3-\alpha} = \pare{2-\alpha} \Gamma\pare{2-\alpha} $, we get 
\begin{equation}\label{eq:QL2}
\text{l.h.s. of \eqref{eq:buuh5}}
=  \  \pare{\frac{c_\alpha}{2\pare{1-\frac{\alpha}{2}}}}^{-1} \frac{c_\alpha}{2\pare{1-\frac{\alpha}{2}}} \frac{\Gamma\pare{3-\alpha}}{\Gamma\pare{1-\frac{\alpha}{2}}\Gamma\pare{\frac{\alpha}{2}}} \frac{\av{n}^{\alpha-1}}{\alpha-1} \ {\nu}\pare{f;x} 
+ V\pare{f; x} + P\pare{f; x, \xi} \, .
\end{equation}
By \Cref{prop:Lalpha_asymptotic} we have 
\begin{equation}\label{eq:QL1}
\frac{c_\alpha}{2\pare{1-\frac{\alpha}{2}}} \frac{\Gamma\pare{3-\alpha}}{\Gamma\pare{1-\frac{\alpha}{2}}\Gamma\pare{\frac{\alpha}{2}}} \frac{\av{\xi}^{\alpha-1}}{\alpha-1} \ {\nu}\pare{f;x} 
= \nu\pare{f; x} L_\alpha\pare{\av{\xi}} + V\pare{f; x} + P\pare{f; x, \xi} \, .
\end{equation}
Finally  plugging \eqref{eq:QL1} in \eqref{eq:QL2} we deduce \eqref{eq:buuh5}. 

 \Cref{eq:buuh2_1,eq:buuh3,eq:buuh4,eq:buuh5} give that for suitable $ \nu, \tilde{\tilde{ V}}\in \Sigma\cF^\bR_{K, 0, 1}\bra{\epsilon_0, N}  $ the desired decomposition provided in \Cref{eq:buuh6}.
\end{proof}

We can finally paralinearize \Cref{eq:fcIcJ}. Using \eqref{eq:buuh8} and \eqref{eq:buuh6} we have 
\begin{multline*}
 r^{2-\alpha}  \pare{  \cI\pare{f} 
 + R \pare{f} f}    + \fint  \mathsf{G}^1_{\alpha , z}\pare{0} \dd z \  \pare{r^{2-\alpha}-1}
+ r^{-\alpha}  \, \cJ \pare{f} 
\\
 =
 - \pare{\frac{c_\alpha}{2\pare{1-\frac{\alpha}{2}}}}^{-1}
\OpBW{\pare{1+\nu\pare{f;x}}L_\alpha\pare{\av{\xi}} + V \pare{f;x} + P\pare{f;x, \xi}} f 
+ R\pare{f}f  
\end{multline*}
where $ V \pare{f;x} $ is a real function in $ \Sigma\cF^\bR_{K, 0, 1}\bra{\epsilon_0, N} $.
This, combined with the observation that $ \partial_x\circ R\pare{f} \in \Sigma \dot{\cR}^{-\rho+1}_{K, 0, 1}\bra{\epsilon_0, N} $, proves that  \Cref{eq:fcIcJ} 
has the form \eqref{eq:paralinearized_1}. 
\hfill $ \Box $

\section{Birkhoff normal form reduction up to cubic terms}
\label{sec:constant_coeff}

In this section we 
reduce the equation 
 \eqref{eq:paralinearized_1} 
to its Birkhoff normal form up to a cubic smoothing vector field, from which Theorem 
\ref{thm:main} easily follows.  
From now on we consider $ \alpha \in (1,2) $. 

\begin{prop}[Cubic Birkhoff normal form]
\label{lem:BNF1}
Let $ \alpha\in\pare{1,2}  $ and  $ N \in \bN $.
There exists $ \underline{\rho} \defeq\underline{\rho}\pare{N, \alpha} $, such that for any $ \rho\geq \underline{\rho} $ there exists    $ \underline{K'} \defeq \underline{K'} \pare{\rho, \alpha} > 0 $ such that for any 
$ K\geq \underline{K'} $  there is $ \underline{s_0} > 0 $ such that for any $ s\geq \underline{s_0} $, there is  $  \underline{\epsilon_0}\pare{s} > 0 $ such that
 for any $ 0<\epsilon_0 \leq \underline{\epsilon_0}\pare{s} $ and any solution $ f \in B^K_{\underline{s_0} , \bR}\pare{I;\epsilon_0} \cap \Cast{K}{s} $ 
 of the equation \eqref{eq:paralinearized_1} the following holds:
\begin{itemize}
\item
there exists a real
invertible operator $ \underline{\Psi}\pare{f;t} $ on $ H^s_0 (\T, \R) $ satisfying the following:
for any $ s \in \R $ there are % a constant
$ C \defeq C (s,\epsilon_0,K)$ and $ \epsilon_0' (s) \in (0, \epsilon_0) $, 
such that for any
$ f \in B^K_{\underline{s_0} , \bR}\pare{I;\epsilon_0'(s)} $ and
$ v \in C^{K-\underline{K'}}_{*} \pare{  I; H^s_0 (\T, \R) } $,
for any $ 0 \leq k \leq K - \underline{K'} $, $ t \in I $, 
\begin{equation}\label{equivfg}
\Big\| \partial_t^k \Big( \underline{\Psi}\pare{f;t}  v \Big) \Big\|_{s- k} +
\Big\| \partial_t^k \Big( \underline{\Psi}\pare{f;t}^{-1}  v \Big) \Big\|_{s- k}
\leq C \| v \|_{k,s} \, ;
\end{equation}
\item
the  variable $ \cy \defeq \underline{\Psi}\pare{f;t} f  $  solves the equation
 \begin{equation}\label{BNF12}
 \pa_{t} \cy +   \ii \  \omega_\alpha \pare{D} \cy +  \ii \OpBW{  d \pare{ f; t, \x} } \cy 
 =   R_{\geq 2}\pare{f;t}  \cy  
 \end{equation}
 where

 $ \bullet $ $ \omega_\alpha\pare{\xi} = \xi  L_\alpha\pare{\av{\xi}} $,
 with $ L_\alpha\pare{\av{\xi}} $ 
 %the Fourier multiplier of order $ \alpha - 1 $
 defined in \Cref{lem:linearization},  is a Fourier multiplier 
 of order $ \alpha $;

$ \bullet $ $ d\pare{ f; t, \x}  $ is a symbol in 
$ \Sigma \Gamma^{\alpha}_{K, \underline{K'}, 2}\bra{\epsilon_0, N} $ independent of $ x $,  satisfying \eqref{areal},  
with % imaginary part 
$ \Im \, d\pare{ f; t, \x}  $ in the space $ \Sigma \Gamma^{0}_{K, \underline{K'}, 2}\bra{\epsilon_0, N} $;

$ \bullet $
 $ R_{\geq 2}\pare{f;t} $ is a real smoothing operator  in 
$ \Sigma \dot{\mathcal{R}}^{-( \rho -\underline{\rho} -\alpha )}_{K, \underline{K'}, 2}\bra{\epsilon_0, N} $.
\end{itemize}
\end{prop}
The bounds \eqref{equivfg} imply in particular that  for any $ s\geq s_0 $,  there exists $ C \defeq C_{s,K,\alpha} > 0  $ such that
\begin{equation}
\label{eq:equivalence_yf}
C^{-1} \norm{ f \pare{t} }_s \leq \norm{ \cy\pare{t} }_s \leq C \norm{ f \pare{t}  }_s \, , \quad \forall t \in I \,  . 
\end{equation}
Note that 
the $ x$-independent symbol $ d \pare{ f; t, \xi } $ in \eqref{BNF12} has homogeneity 
at least $ 2 $ %(as stated in Items {\it ii})-{\it iii}) of Proposition \ref{prop:cc_ao}),
by Remark \ref{rem:symbol}.

\paragraph{Reduction to constant coefficients up to a smoothing operator.}
The first step % section is the following proposition which reduces 
is to reduce the  symbol of the paradifferential operator in \eqref{eq:paralinearized_1} 
to a constant coefficient one, up to a smoothing operator.

\begin{prop}[Reduction to constant coefficients up to smoothing operators]
\label{prop:cc_ao}
Let $ \alpha\in\pare{1,2}  $ and  $ N \in \bN $.
There exists $ \underline{\rho} \defeq\underline{\rho}\pare{N, \alpha} $, such that for any $ \rho\geq \underline{\rho} $ there exists    $ \underline{K'} \defeq \underline{K'} \pare{\rho, \alpha} > 0 $
 such that for any 
$ K\geq \underline{K'} $  
 there are $ s_0 > 0 $, $ \epsilon_0 > 0 $ such that
 for any solution $ f \in B^K_{s_0, \bR}\pare{I;\epsilon_0} $ 
 of \eqref{eq:paralinearized_1} the following holds:
\begin{itemize}
\item
there exists a real
invertible operator $ \Psi \pare{f;t} $ on $ H^s_0 (\T, \R) $ satisfying \eqref{equivfg};
\item
the  variable $ \cg \defeq \Psi\pare{f;t} f  $  solves the equation
\begin{equation}
\label{eq:paralinearized_7}
\pa_t \cg + \partial_x \circ \OpBW{ \pare{1+\mathpzc{c}_0\pare{f}} 
L_\alpha\pare{\av{\xi}} + \mathsf{H}_{\alpha}\pare{f; t, \xi} } \cg = R \pare{f;t} \cg
\end{equation}
 where

$ \bullet $ $  L_\alpha\pare{\av{\xi}} $ is the Fourier multiplier
of order $ \alpha - 1 $
 defined in \Cref{lem:linearization};

$ \bullet $   $  \mathpzc{c}_0\pare{f}  $ is a  $ x $-independent  real  function in
$  \Sigma \cF^\bR_{K,0,2}\bra{\epsilon_0,N} $;

$ \bullet $ $ \sH_\alpha \pare{f; t,  \xi} $ 
is an $ x $-independent symbol
in $ \Sigma \Gamma^{0}_{K,\underline{K'}, 2}\bra{\epsilon_0, N} $
 satisfying \eqref{areal}, with  
$\  \Im \sH_\alpha\pare{f; t, \xi} $ 
in $ \Sigma \Gamma^{-1}_{K,\underline{K'}, 2}\bra{\epsilon_0, N}  $;

$ \bullet $  $ R \pare{f;t} $ is a real smoothing operator in $  \Sigma  \dot{\cR}^{-(\rho-\underline{\rho})}_{K,\underline{K'},1}\bra{\epsilon_0, N} $.
\end{itemize}
\end{prop}

\Cref{prop:cc_ao} relies
on general results (given in \Cref{sec:action_paradiff})
that describe how paradifferential operators
are conjugated under the flow generated by a  paradifferential operator, 
which is Hamiltonian  up to  zero order operators. 
We shall use repeatedly the following result.

 \begin{lemma}[Flows of Hamiltonian operators up to order zero]
 \label{lem:existence_flow}
  Let  $ p, N \in \N $, $ 0\leq K'\leq K $
  and $  \delta \geq 0 $.  Let us consider a 
    ``{\it Hamiltonian operators up to order zero}"
 \begin{equation*}
 \Lambda \pare{f, \tau ; t } \defeq \partial_x \circ  \OpBW{\lambda\pare{f, \tau ; t, x, \xi}}
 \end{equation*}
 where $ \lambda\pare{f, \tau ; t, x, \xi} $ is a  symbol in  
 $  \,  \Sigma \Gamma^{-\delta}_{K, K',p}\bra{\epsilon_0,N} $, 
  uniformly in $ \av{\tau}\leq 1 $, with 
  $ \, {\rm Im} \, \lambda\pare{f, \tau ; t, x, \xi} \in 
  \Sigma \Gamma^{-1}_{K, K',p}\bra{\epsilon_0,N}  $ satisfying \eqref{areal}. 
 Then there exists $ s_0 > 0 $ such that,
  for any $ f\in\Ball{K'}{s_0} $,  the equation 
\begin{equation}
 \label{eq:operatorsfamily_Lambda}
 \frac{\dd }{\dd \tau} \Phi_\Lambda \pare{f,\tau; t} =  \Lambda \pare{f,\tau; t} \ \Phi_\Lambda \pare{f,\tau; t} \, , \qquad  \Phi_\Lambda \pare{f,0; t} = \Id \,  \, ,
 \end{equation}
 has a unique solution 
 $ \Phi_\Lambda\pare{f, \tau} \defeq \Phi_\Lambda\pare{f, \tau; t} $ 
 satisfying the following properties:
  for any $ s \in \R $ the linear map $ \Phi_\Lambda\pare{f, \tau; t} $ is 
  bounded and invertible  on
  $ H^s_0 (\T, \R) $  and 
  there are a constant $ C \defeq C (s,\epsilon_0,K)$ 
  and 
 $ \epsilon_0' (s) \in (0, \epsilon_0) $ such that, for any $ f\in B_{s_0, \bR}^K (I; \epsilon_0' (s)) $,
for any $ 0 \leq k \leq K - K' $, $ v \in C_*^{K-K'} (I; H^s_0(\T, \R))$,  $ t \in I $, 
\begin{equation} \label{eq:transport_estimate0}
\Big\| \partial_t^k \Big( \Phi_\Lambda\pare{f, \tau; t}  v \Big) \Big\|_{s- k} +
\Big\| \partial_t^k \Big( \Phi_\Lambda\pare{f, \tau; t}^{-1}  v \Big) \Big\|_{s- k}
\leq C \| v \|_{k,s} 
\end{equation} 
uniformly in $ |\tau | \leq 1 $. 
%Furthermore
%$ \Phi_\Lambda\pare{f, \tau;t}^\pm - \Id $ is a map in 
%$ \Sigma\dot{\mathcal{M}}_{K,K',p}[\epsilon_0,N] $. 
\end{lemma}

\begin{proof}
Since the imaginary part of the symbol $ \lambda $ has order $ -1 $, 
the flow $ \Phi_\Lambda $ of \eqref{eq:operatorsfamily_Lambda}
is well-posed and satisfies
\eqref{eq:transport_estimate0} arguing 
as in \cite[Lemma 3.22]{BD2018}. 
Moreover it preserves the subspace of real functions  since
$ \lambda\pare{f, \tau ; t, x, \xi} $ satisfies  \eqref{areal}.  
\end{proof}

In the proof of Proposition \ref{prop:cc_ao}
 it is convenient to preserve the linear 
Hamiltonian structure of \eqref{eq:paralinearized_1} up to order zero along the reduction  
which leads to \eqref{eq:paralinearized_7}, since it 
guarantees that the  symbol  $ \pare{1+\mathpzc{c}_0\pare{f}} 
L_\alpha\pare{\av{\xi}} + \mathsf{H}_{\alpha}\pare{f; t, \xi} $, as well as those obtained in the intermediate reduction steps,  are real, at least up to order $ - 1$. 
\\[1mm]
{ \bf Reduction to constant coefficients at principal order.}
We first  reduce to constant coefficients the highest order
paradifferential operator in \eqref{eq:paralinearized_1}.
We conjugate  \eqref{eq:paralinearized_1} via the transformation
\begin{equation} \label{eq:conjugated_f}
f^{\bra{1}} \defeq  \Phi_B \pare{f,1} f
\end{equation}
  where $ \Phi_B \pare{f, \tau}  $ is the flow
 generated as in \Cref{lem:existence_flow}  
 by the Hamiltonian operator
  \begin{equation}\label{eq:Btheta}
 B\pare{f,\tau} \defeq  \partial_x \circ \OpBW{b\pare{f, \tau; x}} \, ,
\qquad 
 b\pare{f, \tau ;  x}
 \defeq \frac{\beta\pare{f;x}}{1+\tau\  \partial_x\pare{ \beta\pare{f;x} }} \, ,
 \end{equation}
where   $ \beta \pare{ f;x } $ is a real function to be chosen.

\begin{lemma}[Reduction to constant coefficients at principal order]
\label{lem:cc_po}
Let
$ \beta \pare{ f;x } \in \Sigma \cF^\bR_{K,0,1}\bra{\epsilon_0 , N} $ 
be the periodic function
of the diffeomorphism   $ x \mapsto x + \beta \pare{ f;x } $ of  $\,  {\mathbb T} $
whose  inverse diffeomorphism 
is $ y \mapsto y + \breve \beta \pare{ f;y } $, 
where
\begin{align}\label{eq:pzc0}
\breve{\beta}\pare{f;  y}  \defeq \partial_y^{-1}
\bra{\pare{\frac{ 1+ \mathpzc{c}_0\pare{f}}{ 1+\nu\pare{f;  y} }}^{\frac{1}{\alpha}} -1}
\in \Sigma \cF^\bR_{K,0,1}\bra{\epsilon_0 , N} \, ,
&&
 \mathpzc{c}_0\pare{f} \defeq \pare{ \fint   \pare{1+\nu\pare{f;  y}} ^{-\frac{1}{\alpha}}   \dd y        } ^{-\alpha} - 1 \, ,
\end{align}
and  $ \nu\pare{f;  y} $ is the real function defined in Theorem 
\ref{prop:paralinearization_1}.
Then, if $ f $ solves \eqref{eq:paralinearized_1},
the variable $ f^{\bra{1}} $ 
defined
in \eqref{eq:conjugated_f} satisfies the equation
\begin{equation}
\label{eq:paralinearized_4}
\pa_t f^{\bra{1}} 
  +\partial_x \circ \OpBW{      \pare{ 1+ \mathpzc{c}_0\pare{f} } \
 L_\alpha\pare{\av{\xi}}    +  V^1\pare{f; t, x}  +
P \pare{f;  x ,\xi}   \Big.    }  f^{\bra{1}}
=   R \pare{f;t} f^{\bra{1}}
\end{equation}
where
\begin{itemize}

\item $ \mathpzc{c}_0 \pare{f} $ is the $ x$-independent function
in $ \Sigma \cF^\bR_{K,0,1}\bra{\epsilon_0 , N} $   defined in \eqref{eq:pzc0};

\item   $ V^1 \pare{f;t, x} $ is a real function in $ \Sigma \cF^\bR_{K, 1, 1} \bra{\epsilon_0, N} $;
\item $ P\pare{ f; x, \xi } $ is a symbol in $  \Sigma\Gamma^{-1}_{K,0,1}\bra{\epsilon_0 , N} $
satisfying \eqref{areal};
\item  $ R \pare{f;t} $ is a real smoothing operator in
$  \Sigma \dot{\cR}^{-\pare{ \rho -N }}_{K,1,1}\bra{\epsilon_0 , N} $.
\end{itemize}
\end{lemma}

\begin{proof}
If $ f $ solves \eqref{eq:paralinearized_1} then,
the variable $ f^{\bra{1}} \defeq \Phi_B\pare{f, 1} f \defeq  \Phi_B\pare{1} f  $  satisfies, using  also
the expansion 
$ L_\alpha (|\xi|)  =  \bV_\alpha +  c^1_\alpha    \av{\xi}^{\alpha - 1} + 
m_{\alpha-3}\pare{\av{\xi}} $ 
in \eqref{expaLapha}, 
and $   \partial_t \Phi_B \pare{1}\circ   \Phi_B \pare{1}^{-1}
= - \Phi_B \pare{1}\circ \pare{ \partial_t  \Phi_B \pare{1}^{-1} } $, the equation
\begin{multline}
\label{eq:paralinearized_3}
\pa_t f^{\bra{1}}  +\Phi_B \pare{1} \circ \partial_x \circ \OpBW{\pare{ 1+\nu\pare{f;x} }
\pare{ c^1_\alpha \av{\xi}^{\alpha-1} +
 \bV_\alpha + m_{\alpha-3}\pare{\av{\xi}}}   + 
 V\pare{f;x}   + P\pare{f;x,\xi}  } \circ \Phi_B \pare{1}^{-1}  \  f^{\bra{1}} \\
  +
 \Phi_B \pare{1}\circ \pare{ \partial_t  \Phi_B \pare{1}^{-1} } f^{\bra{1}} 
 = \Phi_B \pare{1}\circ R\pare{f} \circ \Phi_B \pare{1}^{-1}  f^{\bra{1}} \, .
\end{multline}
By \eqref{eq:conj_Ham_symbol}, \eqref{eq:simbolo_principale_trasformato}
 the principal order operator in \eqref{eq:paralinearized_3}  is
\begin{align}
& \Phi_B \pare{1} \circ \partial_x \circ \OpBW{ \pare{ 1+\nu \pare{f;x} }    c_\alpha^1 \av{\xi}^{\alpha-1}  } \circ \Phi_B \pare{1}^{-1}  \label{eq:transformed_symbols} 
\\
  & = \ \partial_x \circ \OpBW{ c^1_\alpha \
\pare{1+\nu\pare{f;y}}
\pare{1+\left. \partial_y\breve{\beta}\pare{f; y}}^\alpha\right|_{y=x+\beta(f;x)}   \  \av{\xi}^{\alpha-1} + P_1 \pare{f;x,\xi}  } 
+  R \pare{f}  \notag 
\end{align}
where $ y \mapsto y + \breve \beta \pare{ f;y } $ is 
the inverse diffeomorphism of  $ x \mapsto x + \beta \pare{ f;y } $ 
given by \Cref{lem:LemA3}, 
$ P_1 \pare{f;x,\xi} $ is a symbol 
in $ \Sigma \Gamma^{\alpha-3}_{K, 0, 1}\bra{\epsilon_0, N} $
and $ R \pare{f} $  is a smoothing operator 
in $ \Sigma\dot{\cR}^{-\pare{ \rho-N } }_{K,0,1}\bra{\epsilon_0, N} $. 
By \eqref{eq:pzc0} we deduce that
the symbol of highest order 
in \eqref{eq:transformed_symbols}
 is independent of the variable $ x $,
 that is
 \begin{multline}
\label{eq:transformed_symbols1}
\Phi_B \pare{1} \partial_x \OpBW{ \pare{ 1+\nu \pare{f;x} }
c_\alpha^1 \av{\xi}^{\alpha-1}  }  \Phi_B \pare{1}^{-1} \!
 = \! \partial_x  \OpBW{ c^1_\alpha
\pare{ 1+ \mathpzc{c}_0\pare{f} }    \av{\xi}^{\alpha-1} \!+\!
P_1 \pare{f;x,\xi}  } + R_1 \pare{f} \, .
\end{multline}
The lower order conjugated operator in \eqref{eq:paralinearized_3} is,
by \eqref{eq:conj_Ham_symbol} and \Cref{lem:closure_comp_symbols},
\begin{multline}\label{eq:commutator_pat0}
 \Phi_B \pare{1} \circ \partial_x \circ \OpBW{\pare{ 1+\nu \pare{f;x} }  
 \pare{  \bV_\alpha + m_{\alpha-3}\pare{\av{\xi}} 
}+ V\pare{f; x} + P\pare{f;x,\xi} } \circ \Phi_B \pare{1}^{-1}  \\
 = \partial_x \circ \OpBW{ \bV_\alpha + m_{\alpha-3} (\av{\xi})
 + \tilde V^1\pare{f;x} + P_2 \pare{f;x,\xi}}  + R \pare{f}
\end{multline}
where $ \tilde V^1 \pare{f; x} $ is a function  in 
$ \Sigma \cF^\bR_{K,0,1}\bra{\epsilon_0 , N}$,  
$ P_2 \pare{f; x, \xi} $ is a symbol in $ \Sigma \Gamma^{-1}_{K, 0, 1}\bra{\epsilon_0, N} $, since $ \alpha < 2 $
(note that $ m_{\alpha-3} 
\pare{|\xi| \pare{1+ \breve \beta \pare{ f;y } }_{|y=x+\beta (f; x)}} 
- m_{\alpha-3} \pare{|\xi|} $ is a symbol in
$ \Sigma \Gamma^{\alpha -3}_{K, 0, 1}\bra{\epsilon_0, N} $)
and $ R \pare{f} $ 
is a smoothing operator in $ \Sigma \dot{\cR}^{-\rho}_{K,0,1}\bra{\epsilon_0 , N} $, by 
renaming  $ \rho $.
Finally by \eqref{eq:Egorov_timeder}
 there exists a real function $  \fV \pare{f;t,x} $ in $ \Sigma \cF^\bR_{K,1,1}\bra{\epsilon_0 , N}$ and  
 a smoothing operator $ R \pare{f;t} $ in $ \Sigma\dot{\cR}^{-\rho}_{K,1,1}\bra{\epsilon_0 , N} $ such that
\begin{equation}
\label{eq:commutator_pat}
\Phi_B \pare{1}\circ \pare{ \partial_t  \Phi_B \pare{1}^{-1} } 
= \partial_x\circ \OpBW{\fV \pare{f;t,x}}  + R \pare{f;t} \, .
\end{equation}
\Cref{lem:cc_po}
follows by \eqref{eq:paralinearized_3}, \eqref{eq:transformed_symbols1}, \eqref{eq:commutator_pat0} and 
\eqref{eq:commutator_pat} 
  with $ V^1 \pare{f;x} \defeq \tilde V^1 \pare{f;x} + \fV \pare{f;x} -\mathpzc{c} _0\pare{f}\bV_\alpha $, which belongs to
$   \cF^\bR_{K, 1, 1}\bra{\epsilon_0 , N}  $, 
  and  
  $ P \pare{f;x, \xi} \defeq (P_1 + P_2)\pare{f;x, \xi} -\mathpzc{c} _0\pare{f} m_{\alpha-3} \pare{\av{\xi}} $ in $  \Sigma\Gamma^{-1}_{K,0,1}\bra{\epsilon_0 , N} $. 
\end{proof}

\noindent 
{\bf Reduction to constant coefficients at arbitrary-order.}
We now reduce \eqref{eq:paralinearized_4} 
to constant coefficients
up to a smoothing operator, implementing an inductive process which,
at each step,  regularizes the symbol   of $ \delta \defeq \alpha-1 > 0  $. 
We distinguish two regimes. 
\begin{lemma}[Reduction to constant coefficients up order $ 0 $]
\label{lem:reduction_positive_orders}
Let $ \delta \defeq \alpha-1 $ and\footnote{Note that 
$ \mathsf{j}_\ast  =
 \min \set{ \ \mathsf{j}\in\bN \ \middle| \ \pare{\mathsf{j}-1}\delta > 1 }. 
  $} $ \mathsf{j}_\ast\defeq \ceil{\xfrac{1}{\delta}} +1 $. For any $ \mathsf{j} \in\set{1, \ldots , \mathsf{j}_\ast -1 } $, there exist $ \rho_{\mathsf{j}} $ defined inductively as $ \rho_{1} \defeq N $ and $ \rho_{\mathsf{j} + 1}\defeq 
\rho_{\mathsf{j}} +N \pare{ 1-\mathsf{j}\delta} $ such that for any $ K\geq \mathsf{j} $ there exist $ s_0 > 0 $ and a  
\begin{itemize}
\item  symbol 
$  \mathpzc{d}^{\bra{\mathsf{j}}}\pare{f;t, \xi}\defeq \pare{1+\mathpzc{c}_0\pare{f}}L_\alpha\pare{\av{\xi}} + \sH^{\bra{\mathsf{j}}}_{\alpha}\pare{f;t,  \xi} $ where $ \sH^{\bra{\mathsf{j}}}_{\alpha}\pare{f;t, \xi} \in \Sigma \Gamma^{0}_{K, \mathsf{j}-1, 2}\bra{\epsilon_0, N} $, independent of $ x $, real, even in $ \xi $; 
\item  symbol $ r^{\bra{\mathsf{j}}}\pare{f;t, x, \xi} $ in
$ \Sigma \Gamma^{-\pare{ \mathsf{j}-1 }\delta}_{K, \mathsf{j}, 1}\bra{\epsilon_0, N} $, 
real and even in $ \xi $; 
\item  symbol $ P^{\bra{\mathsf{j}}} \pare{f;t, x, \xi} $ in 
$ \Sigma \Gamma^{-1}_{K, \mathsf{j}-1 , 1}\bra{\epsilon_0, N} $;
\item real smoothing operator 
$ R^{\bra{\mathsf{j}}}\pare{f;t} $  in $ \Sigma\cR^{-\pare{\rho - \rho_{\mathsf{j}} }}_{K, \mathsf{j}, 1}\bra{\epsilon_0, N} $;
\item  Hamiltonian operator $  W^{\bra{\mathsf{j}}}\pare{f} \defeq \partial_x\circ \OpBW{w^{\bra{\mathsf{j}}}\pare{f;t, x, \xi}}  $ where $  w^{\bra{\mathsf{j}}} $ is the real and even in $ \xi $ symbol 
\begin{equation}\label{eq:def_pc1}
 w^{\bra{\mathsf{j}}}\pare{f;t, x, \xi}
 \defeq
 -
  \partial_x^{-1} \bra{
\frac{ r^{\bra{\mathsf{j}}}\pare{f;t,x,\xi} - \fint r^{\bra{\mathsf{j}}}\pare{f;t,x,\xi} \dd x}{ \pare{1+\mathpzc{c}_0\pare{f}} \ c^1_\alpha \alpha \av{\xi}^{\alpha-1}   }  
  }
\in\Sigma\Gamma^{-\mathsf{j}\delta}_{K, \mathsf{j}, 1}\bra{\epsilon_0, N}  
   ;
\end{equation}

\end{itemize}
such that 
if $ f\in\Ball{K}{s_0} $ is a solution of   \eqref{eq:paralinearized_1} then $ f^{\bra{\mathsf{j}}} \defeq \prod_{\mathsf{j}' = 1}^{\mathsf{j}-1} \Phi_{W^{\bra{\mathsf{j}'}}}\pare{f; 1}^{-1} \circ \Phi_B\pare{f; 1} f $ solves
\begin{equation}\label{eq:fj}
\pa_t f^{\bra{\mathsf{j}}} + \partial_x\circ\OpBW{ \mathpzc{d}^{\bra{\mathsf{j}}}\pare{f;t, \xi} + r^{\bra{\mathsf{j}}}\pare{f; t, x, \xi} + P^{\bra{\mathsf{j}}}\pare{f;t, x, \xi} } f^{\bra{\mathsf{j}}} 
=
R^{\bra{\mathsf{j}}}\pare{f; t}f^{\bra{\mathsf{j}}}. 
\end{equation}
\end{lemma}

\begin{proof}
Note that \eqref{eq:paralinearized_4} 
has the form \eqref{eq:fj} for $ \mathsf{j} = 1  $
with $ \sH^{\bra{1}}_\alpha \pare{f;t, \xi} :=0 $, 
$ \mathpzc{d}^{\bra{1}}\pare{f; t, \xi}\defeq \pare{1+\mathpzc{c}_0\pare{f}} L_\alpha\pare{\av{\xi}} $,  
$ r^{\bra{1}}\pare{f;t, x, \xi}\defeq V^1\pare{f;t, x} $, 
$ P^{\bra{1}}\pare{f;t, x, \xi} \defeq P\pare{f; t, x, \xi} $ and 
$ R^{\bra{1}}\pare{f; t} \defeq R\pare{f; t} $. 
We now prove that, if 
$ f^{\bra{\mathsf{j}}} $ solves \eqref{eq:fj} then 
\begin{equation}\label{eq:fj+1}
f^{\bra{\mathsf{j}+1}} \defeq 
\Phi_{W^{\bra{\mathsf{j}}}}\pare{f, 1}^{-1} f^{\bra{\mathsf{j}}}  
\end{equation}
solves \eqref{eq:fj} with $ \mathsf{j}+1 $ instead of $ \mathsf{j} $. 
By conjugation, from  \eqref{eq:fj+1}, setting 
$ \Phi_{W^{\bra{\mathsf{j}}}}\pare{1} := \Phi_{W^{\bra{\mathsf{j}}}}\pare{f, 1} $, we have 
\begin{multline}
\label{eq:paralinearized_5}
\pa_t f^{\bra{\mathsf{j}+1}}   +\Phi_{W^{\bra{\mathsf{j}}}}\pare{1}^{-1}\circ \partial_x \circ \OpBW{ \mathpzc{d}^{\bra{\mathsf{j}}}\pare{f;t, \xi} + r^{\bra{\mathsf{j}}}\pare{f; t, x, \xi} + P^{\bra{\mathsf{j}}}\pare{f;t, x, \xi}  }\circ \Phi_{W^{\bra{\mathsf{j}}}}\pare{1}  \,  f^{\bra{\mathsf{j}+1}} \\
   -\partial_t \Phi_{W^{\bra{\mathsf{j}}}} \pare{1}^{-1} \circ \Phi_{W^{\bra{\mathsf{j}}}} \pare{1}
f^{\bra{\mathsf{j}+1}}  
 =  \Phi_{W^{\bra{\mathsf{j}}}} \pare{1}^{-1}\circ  R^{\bra{\mathsf{j}}} \pare{f;t}  \circ \Phi_{W^{\bra{\mathsf{j}}}}  \pare{1}  f^{\bra{\mathsf{j}+1}} .
\end{multline}
Using \eqref{eq:Taylor_conjugation}  we expand the highest order operator in
\eqref{eq:paralinearized_5} as
\begin{multline}
\label{eq:expansion_principal_by_auxiliary_flow}
  \ \Phi_{W^{\bra{\mathsf{j}}}} \pare{1}^{-1}\circ \partial_x \circ \OpBW{\mathpzc{d}^{\bra{\mathsf{j}}}\pare{f;t, \xi} }\circ \Phi_{W^{\bra{\mathsf{j}}}} \pare{1}
=
  \  \partial_x \circ \OpBW{\mathpzc{d}^{\bra{\mathsf{j}}}\pare{f;t, \xi} } \\
 -   \ \comm{\partial_x \circ \OpBW{w^{\bra{\mathsf{j}}} \pare{f;t,x,\xi} }}{ \partial_x \circ \OpBW{  \mathpzc{d}^{\bra{\mathsf{j}}}\pare{f;t, \xi}}}
+ \partial_x\circ \OpBW{Q_{-\pare{2\mathsf{j}-1}\delta} \pare{f;t, x,\xi}} + R \pare{f;t} 
\end{multline}
where, in view of \eqref{eq:def_pc1},   
 $ Q_{-\pare{2\mathsf{j}-1}\delta} $ is a real and even in $ \xi $ valued 
 symbol in $ \Sigma\Gamma^{-\pare{2\mathsf{j}-1}\delta}_{K, \mathsf{j}, 2}\bra{\epsilon_0 , N} $
and $ R\pare{f;t} $ is a smoothing operator in $ \Sigma \dot{\cR}^{-\rho}_{K, \mathsf{j}, 2}\bra{\epsilon_0 , N} $.
By symbolic calculus, \eqref{eq:def_pc1} 
  and since $ \mathpzc{d}^{\bra{\mathsf{j}}} $ is $ x $-independent we have
\begin{multline} \label{eq:expansion_principal_by_auxiliary_flow_2}
    \comm{\partial_x \circ \OpBW{w^{\bra{\mathsf{j}}}\pare{f;t,x,\xi}}}{ \partial_x \circ \OpBW{\mathpzc{d}^{\bra{\mathsf{j}}}\pare{f;t, \xi}}}
  =  \partial_x \circ \comm{\OpBW{ w^{\bra{\mathsf{j}}}\pare{f;t,x,\xi}}}{\OpBW{\ii  \xi \ \mathpzc{d}^{\bra{\mathsf{j}}}\pare{f;t, \xi}}}  \\
    = 
 \partial_x \circ \OpBW{ -  w^{\bra{\mathsf{j}}}_x \pare{f;t,x,\xi}  \  \partial_\xi\pare{\xi \ \mathpzc{d}^{\bra{\mathsf{j}}}\pare{f;t, \xi} }  + Q_{-2 -\pare{\mathsf{j}-1}\delta}\pare{f; t, x, \xi}
 }  + R\pare{f;t} 
\end{multline}
where $ Q_{-2-\pare{\mathsf{j}-1}\delta} \pare{f; t, x, \xi} $ 
is a  real and even symbol in $\Sigma\Gamma^{-2-\pare{\mathsf{j}-1}\delta}_{K, \mathsf{j},1}\bra{\epsilon_0, N} $ 
and $  R \pare{f;t} $ is a smoothing operator 
in $ \Sigma \dot{\cR}^{-\rho}_{K, \mathsf{j}, 1} \bra{\epsilon_0, N} $.
Using the asymptotic expansion \eqref{expaLapha} we have that
\begin{equation}
\label{eq:expansion_principal_by_auxiliary_flow_2-1}
\partial_\xi\pare{\xi \ \mathpzc{d}^{\bra{\mathsf{j}}}\pare{f;t, \xi} } 
=
\pare{1+\mathpzc{c}_0\pare{f}} \ c^1_\alpha \alpha \av{\xi}^{\alpha-1}
+\tilde{Q}^{\bra{\mathsf{j}}}\pare{f;t,\xi} \qquad
\text{where} \qquad 
\tilde{Q}^{\bra{\mathsf{j}}} 
\pare{f;t,\xi} \in\Sigma\Gamma^{0}_{K, \mathsf{j}-1, 0}\bra{\epsilon_0, N} \, . 
\end{equation}
So, by \eqref{eq:expansion_principal_by_auxiliary_flow}, 
\eqref{eq:expansion_principal_by_auxiliary_flow_2},
\eqref{eq:expansion_principal_by_auxiliary_flow_2-1},  \eqref{eq:Taylor_conjugation}, 
the definition of $ w^{\bra{\mathsf{j}}}\pare{f;t,x,\xi}  \in\Sigma \Gamma^{-\mathsf{j}\delta}_{K, \mathsf{j}, 1}\bra{\epsilon_0, N} $ provided in \cref{eq:def_pc1},   
\eqref{AdSwc}, 
we obtain
\begin{align}
 &  \Phi_{W^{\bra{\mathsf{j}}}}\pare{1}^{-1}\circ \partial_x \circ \OpBW{
  \mathpzc{d}^{\bra{\mathsf{j}}}\pare{f;t, \xi}
  +r^{\bra{\mathsf{j}}}\pare{f;t,x,\xi}
  }\circ \Phi_{W^{\bra{\mathsf{j}}}}\pare{1} \notag \\
& \ =  \partial_x \circ \OpBW{\mathpzc{d}^{\bra{\mathsf{j}}}\pare{f;t, \xi} 
+
w^{\bra{\mathsf{j}}}_x\pare{f;t,x,\xi}  \   \pare{1+\mathpzc{c}_0\pare{f}} \ c^1_\alpha \alpha \av{\xi}^{\alpha-1} + Q_{ -\mathsf{j}\delta}\pare{f; t, x, \xi}
 + r^{\bra{\mathsf{j}}}\pare{f;t,x,\xi}
}   + R \pare{f;t}  \notag \\
& \ =
   \partial_x \circ \OpBW{\mathpzc{d}^{\bra{\mathsf{j}}}\pare{f;t, \xi} + \fint r^{\bra{\mathsf{j}}}\pare{f;t,x,\xi} \dd x +   Q_{-\mathsf{j} \delta}\pare{f; t, x, \xi}  } +  R \pare{f;t} \label{finalid}
\end{align}
where $ Q_{-\mathsf{j}\delta}\pare{f; t, x, \xi} $ is a symbol 
in $ \Sigma \Gamma^{-\mathsf{j}\delta}_{K, \mathsf{j}, 1}\bra{\epsilon_0, N} $. 
 By 
 \eqref{eq:cc_pat}, 
\begin{equation}\label{conjtempo}
-\partial_t \Phi_{W^{\bra{\mathsf{j}}}} \pare{1}^{-1} \circ \Phi_{W^{\bra{\mathsf{j}}}} \pare{1} = \partial_x\circ\OpBW{T^{\bra{\mathsf{j}}}\pare{f;t, x, \xi}} + R\pare{f;t}
\quad \text{with} \quad  
T^{\bra{\mathsf{j}}}\pare{f;t, x, \xi} \in\Sigma \Gamma^{-\mathsf{j}\delta}_{K, \mathsf{j}+1, 1}\bra{\epsilon_0, N}
\end{equation}
real and even in $ \xi $. Furthermore by \eqref{eq:Taylor_conjugation}
\begin{equation}\label{anchePj} 
\Phi_{W^{\bra{\mathsf{j}}}}\pare{1}^{-1}\circ 
\OpBW{P^{\bra{\mathsf{j}}}\pare{f;t, x, \xi}}
\circ \Phi_{W^{\bra{\mathsf{j}}}}\pare{1} = 
\OpBW{P^{\bra{\mathsf{j+1}}} \pare{f;t, x, \xi}} \, , 
\end{equation}
up to a smoothing operator,   
with  a symbol 
$ P^{\bra{\mathsf{j+1}}} \pare{f;t, x, \xi} $  in $ \Sigma\Gamma^{-1}_{K, \mathsf{j}, 1}\bra{\epsilon_0, N} $. 
By  
  \eqref{finalid},  \eqref{conjtempo}, \eqref{anchePj}  
  and since 
  $ \Phi_{W^{\bra{\mathsf{j}}}}\pare{1}^{-1}\circ R^{\bra{\mathsf{j}}}\pare{f;t}\circ \Phi_{W^{\bra{\mathsf{j}}}}\pare{1} $ 
is in 
$ \Sigma\dot{\cR}^{-\pare{\rho - \rho_{\mathsf{j}} - N \pare{ 1-\mathsf{j}\delta}  }}_{K, \mathsf{j}, 1}\bra{\epsilon_0 , N}  $, 
we deduce that \eqref{eq:paralinearized_5} has the form
 \eqref{eq:fj} with $ \mathsf{j} $ replaced by $ \mathsf{j}+1 $
where 
$ \sH^{\bra{\mathsf{j}+1}}_{\alpha} \defeq \sH^{\bra{\mathsf{j}}}_{\alpha}  + \fint r^{\bra{\mathsf{j}}}\pare{f;t,x,\xi} \dd x $ and 
  $ r^{\bra{\mathsf{j}+1}} \defeq
T^{\bra{\mathsf{j}}}
+ Q_{-\mathsf{j} \delta} $.
  \end{proof}
  
Now, implementing an analogous  algorithmic procedure for the symbols of 
order $ \leq - 1 $, 
we reduce the equation \eqref{eq:fj} for $ \mathsf{j} = \mathsf{j}_\ast $ 
 to constant coefficients  up to a smoothing operator. 
 
  \begin{lemma}[Reduction to constant coefficients up to smoothing operators]
  \label{lem:reduction_negative_orders}
  For any integer $ \mathsf{j} \geq \mathsf{j}_\ast  $, 
  for any $ K\geq \mathsf{j} $ there exist a  
\begin{itemize}
\item  symbol 
$  \mathpzc{d}^{\bra{ \mathsf{j}}}\pare{f;t, \xi}\defeq \pare{1+\mathpzc{c}_0\pare{f}}L_\alpha\pare{\av{\xi}} + \sH^{\bra{ \mathsf{j}}}_{\alpha}\pare{f;t,  \xi} $ 
with $\sH^{\bra{ \mathsf{j}}}_{\alpha}\pare{f;t, \xi} \in 
\Sigma \Gamma^{0}_{K, \mathsf{j}-1, 2}\bra{\epsilon_0, N} $
and
$ \Im \sH^{\bra{ \mathsf{j}}}_{\alpha}\pare{f;t, \xi} $
in $ \Sigma \Gamma^{-1}_{K, \mathsf{j}-1, 2}\bra{\epsilon_0, N}  $, 
independent of $ x $ and satisfying \eqref{areal}; 
\item  symbol $ P^{\bra{\mathsf{j}}}\pare{f;t, x, \xi} $ in
$ \Sigma \Gamma^{-1- \pare{ \mathsf{j}-\mathsf{j}_\ast }\delta}_{K, \mathsf{j}, 1}\bra{\epsilon_0, N} $ satisfying \eqref{areal}; 
\item a real smoothing operator 
$ R^{\bra{\mathsf{j}}}\pare{f;t} $  in $ \Sigma\cR^{-\pare{ \rho- \rho_{\mathsf{j}_\ast} }}_{K, \mathsf{j}, 1}\bra{\epsilon_0, N} $;
\item bounded linear operators $  W^{\bra{ \mathsf{j}}}\pare{f} \defeq \partial_x\circ \OpBW{w^{\bra{ \mathsf{j}}}\pare{f;t, x, \xi}}  $ where 
% $  w^{\bra{\mathsf{j}_\ast + \mathsf{j}}} $ is the symbol 
\begin{equation}\label{eq:def_pc2}
 w^{\bra{\mathsf{j}}}\pare{f;t, x, \xi}
 \defeq
 -
  \partial_x^{-1} \bra{
\frac{ P^{\bra{\mathsf{j}}}\pare{f;t,x,\xi} - \fint P^{\bra{\mathsf{j}}}\pare{f;t,x,\xi} \dd x}{ \pare{1+\mathpzc{c}_0\pare{f}} \ c^1_\alpha \alpha \av{\xi}^{\alpha-1}   }  
  }
\in\Sigma\Gamma^{-1-\pare{  \mathsf{j}-\mathsf{j}_\ast  +1 }\delta}_{K,  \mathsf{j}, 1}\bra{\epsilon_0, N}  
   ;
\end{equation}

\end{itemize}
and $ s_0 > 0 $, such that 
if $ f\in\Ball{K}{s_0} $ is a solution of   \eqref{eq:paralinearized_1} then $ f^{\bra{\mathsf{j}}} \defeq \prod_{\mathsf{j}' = 1}^{ \mathsf{j}-1} \Phi_{W^{\bra{\mathsf{j}'}}}\pare{f; 1}^{-1} \circ \Phi_B\pare{f; 1} f $ solves
\begin{equation}\label{eq:fjast+j}
\pa_t f^{\bra{ \mathsf{j}}} + \partial_x\circ\OpBW{ \mathpzc{d}^{\bra{ \mathsf{j}}}\pare{f;t, \xi} + P^{\bra{\mathsf{j}}}\pare{f; t, x, \xi}  } f^{\bra{ \mathsf{j}}} 
=
R^{\bra{ \mathsf{j}}}\pare{f; t}f^{\bra{\mathsf{j}}}. 
\end{equation}
  \end{lemma}
  
We now conclude the proof of Proposition \ref{prop:cc_ao}. Let 
 $
 \mathsf{j}^\ast :=  \mathsf{j}^\ast \pare{\rho} 
 \defeq \min\set{\mathsf{j}\in\bN_0 \ \middle| \ \pare{\mathsf{j}-\mathsf{j}_\ast}\delta > \rho- \rho_{\mathsf{j}_\ast}}, 
 $
 which is explicitly 
 $
 \mathsf{j}^\ast \defeq \ceil{\frac{\rho - \rho_{\mathsf{j}_\ast} }{\alpha-1}} + \mathsf{j}_\ast
 =
 \ceil{\frac{\rho - \rho_{\mathsf{j}_\ast} }{\alpha-1}} + \ceil{ \frac{1}{\alpha-1}} + 1 
 $, 
so  that $ \OpBW{P^{\bra{\mathsf{j}^\ast}}\pare{f;t,x,\xi}} $ is a smoothing operator in 
$ \Sigma \cR ^{-\pare{\rho- \rho_{\mathsf{j}_\ast} }}_{K,  \mathsf{j}^\ast,1}\bra{\epsilon_0, N} $ by \Cref{rem:smoo}.  
Then the equation 
\eqref{eq:fjast+j} with  $\mathsf{j} = \mathsf{j}^* $ has the form 
\eqref{eq:paralinearized_7} with 
$$ 
\cg =  f^{\bra{\mathsf{j}^*}} = 
 \Psi\pare{f;t} f \, , \qquad 
 \Psi\pare{f;t} := 
 \prod_{ \mathsf{j} ' =1}^{\mathsf{j}^\ast - 1} \Phi_{W^{\bra{\mathsf{j}'}}}\pare{f, 1}^{-1} \circ \Phi_B\pare{f, 1} \, ,
$$
symbol $ \sH_\alpha\pare{f;t, \xi} \defeq \sH^{\bra{\mathsf{j}^\ast}}_{\alpha} \pare{f;t, \xi}  $, 
smoothing operator 
 $ R\pare{f;t}\defeq R^{\bra{\mathsf{j}^\ast}}\pare{f;t} + 
 \OpBW{P^{\bra{\mathsf{j}^\ast}}\pare{f;t,x,\xi} }   $, 
    and defining $ \underline{\rho} \pare{N, \alpha} \defeq  \rho_{\mathsf{j}_\ast} $ and 
    $ \underline{K'}\pare{\rho, \alpha} \defeq \mathsf{j}^\ast $. 
\hfill $ \Box $

\paragraph{Birkhoff normal form step.}
\label{sec:quadratic_normal_forms}
We now perform one step of Birkhoff normal form to cancel out  
the  quadratic term  in \eqref{eq:paralinearized_7} which, 
since $ c_0 \pare{f} $ and $ H_\alpha \pare{f; t, \xi }$ vanish quadratically at $ f = 0 $,  
comes only from  $ R\pare{f;t} \cg $. 

By Proposition \ref{lem:BNF1} and 
using Proposition  \ref{prop:composition_BW} we first rewrite 
\eqref{eq:paralinearized_7} as
 \begin{equation}
\label{eq:paralinearized_1bis}
\pa_t \cg  +  \ii \omega_\alpha \pare{D} \cg +
\ii \, \OpBW{ d\pare{ f; t, \x}  }  \, \cg = 
R_1 \pare{f} \cg + R_{\geq 2}\pare{f;t} \cg
\end{equation}
where  
\begin{enumerate}[i)]
\item
$ d\pare{ f; t, \x}  \defeq    {  \mathpzc{c}_0\pare{f} 
\omega_\alpha \pare{\xi} 
+ \xi \ \mathsf{H}_{\alpha}\pare{f; t, \xi} } $ is a symbol in 
$ \Sigma \Gamma^{\alpha}_{K, \underline{K'}, 2}\bra{\epsilon_0, N} $ independent of $ x $,  
with imaginary part 
$ \Im \, d\pare{ f; t, \x}  $ in $ \Sigma \Gamma^{0}_{K, \underline{K'}, 2}\bra{\epsilon_0, N} $;
\item
$ R_1\pare{f} $ is a real homogenous smoothing operator 
in $ \dot{\widetilde{\mathcal{R}}}^{ \ - ( \rho-\underline{\rho} ) }_1 $ , that we expand 
(cf.  \eqref{smoocara0}) as 
\be\label{Qsmo2}
R_1\pare{f}  v = \sum_{n,k, j \in \Z \setminus \{0\}, \atop n+j=k} 
\pare{r_1}_{n,j,k} f_n v_j e^{\ii k x}\, , \quad  
\quad \pare{r_1}_{n,j,k}  \in \C \, , 
\ee
and $ R_{\geq 2}\pare{f;t} $ is a real smoothing operator  in 
$ \Sigma \dot{\mathcal{R}}^{-(\rho-\underline{\rho})}_{K, \underline{K'}, 2}\bra{\epsilon_0, N} $.
\end{enumerate}

In order to remove $ R_1 \pare{f} $ 
we conjugate \eqref{eq:paralinearized_1bis} with  the  flow 
 \begin{equation}\label{BNFstep1}
\partial_{\tau} \mathcal{\Phi}_Q^{\tau}\pare{f}  = Q\pare{f} \mathcal{\Phi}_Q^{\tau}\pare{f} \, , 
\qquad  \mathcal{\Phi}_Q^{0}\pare{f} = {\rm Id} \, ,  
\end{equation}
generated by the $ 1$-homogenous smoothing operator
\begin{equation}\label{omoBNF5}
Q\pare{f} v = \sum_{n,k, j \in \Z \setminus \{0\}, \atop n+j=k} 
q_{n,j,k} f_{n} v_j e^{\ii k x}\, , \quad 
   q_{n,j,k} \defeq
  \frac{-( r_1)_{n,j,k}}{\ii \big( \omega_\alpha(k) -  \omega_\alpha\pare{j} -
   \omega_\alpha(n) \big)} \, , 
\end{equation} 
which is  well-defined
by Lemma  \ref{lem:nonres_cond}. Note also 
that by \eqref{eq:nonres_cond} and 
since (cf. \eqref{eq:reality_cond_reminder}, \eqref{eq:bound_fourier_representation_m_operators})
\begin{equation}\label{realsmo}
\overline{(r_1)_{n,j,k}} = (r_1)_{-n,-j,-k} \, , \qquad 
 | (r_1)_{n,j,k}|\leq  C \frac{{\max}_2 \pare{|n|,|j|}^{\mu}}{\max\pare{|n|,|j|}^{\rho-\underline{\rho}}} \, ,
 \end{equation}
also $ Q \pare{f} $  is a real smoothing operator in 
$ \dot{\widetilde{\mathcal{R}}}_1^{-\rho +\underline{\rho} } $ as $ R_1 \pare{f} $.

\begin{lemma}[Birkhoff step]\label{lem:BNF1step}
If $ \cg $ solves \eqref{eq:paralinearized_1bis} 
then   the variable
$ \cy \defeq \mathcal{\Phi}_Q^{1} \pare{f} \cg    $
solves the equation \eqref{BNF12}.
 \end{lemma}

\begin{proof}
To conjugate \eqref{eq:paralinearized_1bis}  we apply a Lie expansion 
(similarly to \Cref{lem:Lie_expansion}). We have 
   \begin{align}
  -\ii \mathcal{\Phi}_Q^{1} \pare{f} \, \omega_\alpha (D)  \, \big(\mathcal{\Phi}_Q^{1}\big)^{-1} & =
 -\ii \, \omega_\alpha (D)  \, + \comm{Q\pare{f}}{-\ii \, \omega_\alpha (D)}  \nonumber \\
 & \ +\int_{0}^{1} (1- \tau) 
\mathcal{\Phi}_Q^{\tau}\pare{f} \comm{ Q\pare{f}}{ 
\comm{Q\pare{f}}{-\ii \, \omega_\alpha (D)  } }\pare{\mathcal{\Phi}_Q^{\tau}\pare{ f}  }^{-1} d \tau \label{core4} \, .
\end{align}
Using that $ Q\pare{f}$ belongs to 
 $\dot{\widetilde{\mathcal{R}}}_{1}^{-\rho+\underline{\rho}} $
 the  term in \eqref{core4}
 is a smoothing operator  in $ \Sigma\dot{\mathcal{R}}^{-\rho+\underline{\rho}+ \alpha}_{K,0,2} \bra{\epsilon_0, N} $.
Similarly we obtain
 \begin{equation}\label{core2}
  -\ii \mathcal{\Phi}_Q^{1}\pare{f}\OpBW{ d\pare{ f; t, \x} } 
   \pare{\mathcal{\Phi}_Q^{1}\pare{f} }^{-1}=
  -\ii \OpBW{ d\pare{ f; t, \x} }
 \end{equation}
  up to a smoothing operator  in 
  $ \Sigma\dot{\mathcal{R}}^{-\rho+\underline{\rho}+\alpha}_{K,K_0',2} \bra{\epsilon_0, N} $, and 
 \begin{equation}\label{core3}
\mathcal{\Phi}_Q^{1}\pare{f}\big( R_1\pare{f} + R_{\geq 2}\pare{f;t}\big)
\pare{\mathcal{\Phi}_Q^{1}\pare{f}}^{-1}=   R_1\pare{f} 
 \end{equation}
plus a smoothing operator in 
$ \Sigma\dot{\mathcal{R}}^{-\rho+\underline{\rho}+\alpha}_{K,K_0',2}\bra{\epsilon_0, N}  $.
Next we consider the contribution coming from the conjugation of $ \pa_t $. 
By a Lie expansion (similarly to  
\Cref{lem:Lie_expansion}) we get
 \begin{multline}
    \pa_t \mathcal{\Phi}_Q^{1} \pare{f} \pare{\mathcal{\Phi}_Q^{1} \pare{f} }^{-1}    =  
  \pa_t Q\pare{f}   \\
    +
  \frac{1}{2} \comm{Q\pare{f}}{\pa_t Q\pare{f} }
 + \frac{1}{2}\int_0^{1} \!\! (1- \tau)^{2} \mathcal{\Phi}_Q^{\tau}\pare{f}
\comm{Q\pare{f}}{ \comm{ Q\pare{f} }{ \pa_t Q\pare{f}  } }
 \pare{\mathcal{\Phi}_Q^{\tau} \pare{f} }^{-1} d \tau \, . \label{core5}
 \end{multline}
Since the  \cref{eq:paralinearized_1}  can be written as  
$ \pa_t f  =-\ii \omega_\alpha (D) f +  M\pare{f} f $
where $  M\pare{f} $ is a real $ \alpha $-operator  
in  $\Sigma \dot{\mathcal{M}}^{\alpha}_{K,0,1} $ by \Cref{rem:smoo} and \ref{rem:OpBW_firstproperties}, we deduce by Proposition \ref{compositionMoperator} that 
\begin{equation}\label{patQs}
 \pa_t  Q\pare{f} =  Q \pare{ -\ii \omega_\alpha (D) f+ M\pare{f}f } =
  Q \pare{-\ii \omega_\alpha (D)  f}
\end{equation}
 up to a smoothing operator 
 in $ \Sigma\dot{\mathcal{R}}^{-\rho+\underline{\rho}+ \alpha}_{K,0,2} \bra{\epsilon_0, N} $. 
Since $ Q \pare{ -\ii \omega_\alpha (D) f } $ is in 
$ \dot{\widetilde{\mathcal{R}}}_{1}^{-\rho+\underline{\rho} + \alpha }   $
  we have that the line \eqref{core5} belongs to 
$ \Sigma\dot{\mathcal{R}}^{-\rho+\underline{\rho}+ \alpha}_{K,0,2} \bra{\epsilon_0, N} $.

We now prove  that $ Q\pare{f} $ % \in \dot{\widetilde{\mathcal{R}}}^{-\rho+\underline{\rho}}_1 $ 
solves 
the homological equation
\begin{equation}\label{omoBNF}
Q \pare{ -\ii \omega_\alpha (D) f} + \comm{Q \pare{ f} }{ -\ii \omega_\alpha (D) } + R_1\pare{f}=0 \, .
\end{equation}
Writing \eqref{omoBNF5} as
$ Q\pare{f} v = \sum_{k, j \in \Z \setminus \{0\}} \big[ Q\pare{f} \big]_k^j v_j e^{\ii k x}
$ with $ \big[ Q\pare{f} \big]_k^j \defeq q_{n,j,k} f_{n} $, 
we see that the homological 
equation \eqref{omoBNF} 
amounts to 
$
\bra{Q(-\ii \omega_\alpha (D) f)}^{j}_{k}+  \bra{Q\pare{f}}^j_k \big(
\ii  \omega_\alpha (k) - \ii \omega_\alpha \pare{j}   \big) + \bra{R_1 \pare{f}}^{j}_{k}=0 $, 
 for any $j,k\in \Z\setminus\{0\}$,  
and then, recalling \eqref{Qsmo2}, 
to 
$  q_{n,j,k} \, \ii \big(\omega_\alpha (k) -\omega_\alpha \pare{j}- \omega_\alpha (n) \big)+
   (r_1)_{n,j,k}=0 $.
This proves  \eqref{omoBNF}. 
 
In conclusion, by  \eqref{core4}, \eqref{core2},  \eqref{core3}, \eqref{core5}, \eqref{patQs}  
and \eqref{omoBNF}  we deduce \eqref{BNF12}
(after renaming $ \rho $). The bound \eqref{eq:equivalence_yf}
follows by standard theory of Banach space ODEs for the flow \eqref{BNFstep1}
and \eqref{eq:equivalence_yf}.
\end{proof}

In view of Lemma \ref{lem:BNF1step}, 
Proposition \ref{lem:BNF1} follows 
defining $ \underline{\Psi}\pare{f;t}\defeq \Phi^1_Q\pare{f}\circ \Psi\pare{f;t} $ where $ \Psi\pare{f;t} $ is defined in \Cref{prop:cc_ao} and $ \Phi^1_Q\pare{f} $ is defined in \eqref{omoBNF5}. We now easily deduce Theorem \ref{thm:main}.

\paragraph{Proof of \Cref{thm:main}.}

The following result, analogous to \cite[Lemma 8.2]{BMM2022}, 
enables to control the time derivatives $ \| \pa_t^k f(t) \|_{s-k\alpha} $ of a solution 
$  f(t) $ of \eqref{eq:paralinearized_1} via $ \| f(t)\|_s  $. 

\begin{lemma}\label{lem:equivalence_norms}
Let $ K\in \bN $. There exists $ s_0> 0 $ such that for any $ s\geq s_0 $, any $ \epsilon\in\pare{0, \overline{\epsilon_0}\pare{s}} $ small, if $ f $ belongs to
$ B^0_{s_0, \bR}\pare{I;\epsilon}\cap \Cast{0}{s} $ 
and solves \eqref{eq:paralinearized_1} then $ f\in\Cast{K}{s} $ and there exists $ C_1 \defeq C_1\pare{s, \alpha, K} \geq 1 $ such that
$   \norm{f\pare{t}}_s  \leq
 \norm{f\pare{t}}_{K, s}
 \leq C_1 \norm{f\pare{t}}_s $ \ for any $ t \in I $. 
\end{lemma}

The first step is to choose the parameters in \Cref{lem:BNF1}. 
Let $ N \defeq 1 $. 
In the statement of
 \Cref{lem:BNF1} we fix 
   $ \rho \defeq \underline{\rho}\pare{1,\alpha} +\alpha  $ and $ K\defeq \underline{K'}\pare{\rho,\alpha} $. Then 
   \Cref{lem:BNF1} gives us  $ \underline{s_0} >0 $. For any $ s\geq \underline{s_0}$ we fix $ 0< \epsilon_0 \leq   \min \{ \overline{ \epsilon_0}\pare{s}, \underline{ \epsilon_0}\pare{s} \} $ where $ \underline {\epsilon_0}\pare{s} $ is
   defined in \Cref{lem:BNF1} and $ \overline{ \epsilon_0}\pare{s} $ in \Cref{lem:equivalence_norms}. 

The key corollary of  \Cref{lem:BNF1} is the following energy estimate
where by the time-reversibility of %the equation
 $ \alpha $-SQG % in the proof of \Cref{thm:main}
 we may restrict  to positive times $ t > 0 $.

\begin{lemma}[Quartic  energy estimate]\label{lem:SE}
Let $ f(t) $ be a solution of 
   equation \eqref{eq:paralinearized_1} in 
   $  \Ball{K}{\underline{s_0}} \cap \Cast{K}{s} $.
Then  there exists  $ \bar C_2 \pare{s, \alpha} > 1 $ such that 
\begin{equation}\label{stimaenf}
\norm{f\pare{t}}_s^2 \leq \bar{C}_2 \pare{s, \alpha}  \pare{\norm{ f \pare{0}}_s^2 
+ \int_0^t\norm{ f\pare{\tau}}_s^4 \dd \tau } \,  , 
\quad \forall 0 < t < T  \, . 
\end{equation}
\end{lemma}

\begin{proof}
The variable $ \cy \defeq \underline{\Psi}\pare{f;t} f $ defined in 
Proposition \ref{lem:BNF1}  solves the equation \eqref{BNF12} 
where $ \Im \, d\pare{f;t, \xi} $ is a symbol in 
 $ \Gamma^0_{K, \underline{K'}, 2}\bra{\epsilon_0} $ and, being $ x $-independent, 
 $ \OpBW{d\pare{f;t, \xi}} $ commutes with $ \angles{D}^s $. Furthermore, 
for the above choice of $ \rho $ 
it results that $ R_{\geq 2}\pare{f;t} $ is in 
$ \cR^0_{K, \underline{K'}, 2}\bra{\epsilon_0} $. 
Then % \eqref{stimaconclu} 
by  \eqref{piove}, Lemmata \ref{prop:action}
and \ref{lem:equivalence_norms} we deduce 
%\begin{equation}\label{stimaconclu}
$$
\norm{ \cy \pare{t} }_{s}^2 \leq \norm{ \cy \pare{0} }_{s}^2 +  \bar C_1 (s, \alpha)   \int_0^t
\norm{ \cy\pare{\tau}}_{s}^4 \, \dd \tau \, , \quad \forall 0 < t < T \, ,
$$
and, % Thus \eqref{stimaconclu} and 
by  \eqref{eq:equivalence_yf}, we deduce \eqref{stimaenf}. 
\end{proof}

The energy estimate \eqref{stimaenf}, \eqref{eq:f-h} and 
the local existence result in  \cite{CCCG2012}
(which amounts to a local existence result 
 for the equation \eqref{eq:paralinearized_1}), 
imply, by a standard 
bootstrap argument,  \Cref{thm:main}. 
\hfill$ \Box $

\bigskip

\noindent
{\bf Acknowledgments.}
{\footnotesize We thank A. Maspero, R. Montalto, and F. Murgante for many discussions. 
M.B. and S.C. were supported by  PRIN 2020XB3EFL, {\it Hamiltonian and Dispersive PDEs}.
F. G. was partially supported by the MICINN (Spain), 
grants EUR2020-112271 and PID2020-114703GB-I00, by RED2022-134784-T funded by MCIN/AEI/10.13039/501100011033, by the Junta de Andalucía, grant P20-00566, and by the Fundacion de Investigación de la Universidad de Sevilla, grant FIUS23/0207. F. G. acknowledges support from IMAG, funded by MICINN through Maria de Maeztu
Excellence Grant CEX2020-001105-M/AEI/10.13039/501100011033.
S.S. is supported by PRIN 2022HSSYPN - {\it Turbulent Effects vs Stability in Equations from Oceanography}, MTM PID2022-141187NB-I00 and FIUS23/0207.} % acronym TESEO}. 

\appendix

\section{Proof of \Cref{eq:Hamiltonian_identity}}\label{sec:Hamiltonian_identity}

We now prove  the identity \cref{eq:Hamiltonian_identity} where
the functions $ A_{\alpha, 0}, A_{\alpha, 1} $ are defined in \eqref{eq:A0}, \eqref{eq:A1}, and 
 $ \sK^{j,l}_\alpha, \ j=1,2,3, \ l=0,1 $ are the $ l $-th order Taylor expansion in $ z $ of the function $ z\mapsto \sK^{j}_{\alpha, z}\pare{\frac{\Delta_z f}{r^2}} $ where the kernel functions $ \sK^{j}_{\alpha, z}\pare{\sx} $ are defined   in \eqref{eq:sK1}, \eqref{eq:sK2}, \eqref{eq:sK3}. 
The verification of \cref{eq:Hamiltonian_identity} 
can be {\it automated}. 
The next small program in {\tt SageMath}, a Python-based, open-source 
Computer-Algebra System, verifies \cref{eq:Hamiltonian_identity}.

\begin{lstlisting} 
x, X, z, a = var('x, X, z, a')
assume(0<a<2)

f(x) = function('f')(x)
Deltaf (x,z) = (f(x)-f(x-z))/(2*sin(z/2))

G1(X,z,a) = (1-2*X-sqrt(1-2*X)*cos(z))/((2*(1-X-sqrt(1-2*X)*cos(z)))^(a/2))
DXG1(X,z,a) = diff(G1(X,z,a),X)
K1(X,z,a)= DXG1(2*X*sin(z/2),z,a) * (2*(1-cos(z)))^(a/2)

G2(X,z,a) = (1/ (sqrt(1-2*X)) )/((2*(1-X-sqrt(1-2*X)*cos(z)))^(a/2))
K2 (X,z,a) =  G2( 2*X*sin(z/2) , z  , a) *  (2*(1-cos(z)))^(a/2)

DXG2(X,z,a) = diff(G2(X,z,a),X)
K3(X,z,a) = DXG2(X*2*sin(z/2), z, a)* (2*(1-cos(z)))^(a/2) * sin(z)

expansionf_K1(x,z,a) = taylor(K1( Deltaf (x,z) / (1+2*f(x)) , z, a ), z, 0, 1)
expansionf_K2(x,z,a) = taylor(K2( Deltaf (x,z) / (1+2*f(x)) , z, a ), z, 0, 1)
expansionf_K3(x,z,a) = taylor(K3( Deltaf (x,z) / (1+2*f(x)) , z, a ), z, 0, 1)

C10(x,a) = expansionf_K1.coefficient(z, n=0)
C11(x,a) = expansionf_K1.coefficient(z, n=1)
C20(x,a) = expansionf_K2.coefficient(z, n=0)
C21(x,a) = expansionf_K2.coefficient(z, n=1)
C30(x,a) = expansionf_K3.coefficient(z, n=0)
C31(x,a) = expansionf_K3.coefficient(z, n=1)

A0(x,a) = ((1+2*f(x))^(-a/2)) * ( C10(x,a) + (a-1)*C20(x,a) + (diff(f(x),x) / (1+2*f(x))) * C30(x,a) )
A1(x,a) = ((1+2*f(x))^(-a/2)) * ( C11(x,a) +(a-2)* C21(x,a) + ( 1 / (1+2*f(x))) * ( diff(f(x),x) * C31(x,a) - diff(f(x),x,x) * C30(x,a)))


bool(A1(x,a)  + 1/2 * diff(A0(x,a) , x)==0)

\end{lstlisting}

Here we comment the lines of code above.

\begin{itemize}

\item[1,2] Several variables are defined, so that
$ \pare{ \texttt{x,X,z,a}} = \pare{x, \sx, z, \alpha} $
accordingly to the notation of the present manuscript. The variable {\tt a}, which is the parameter $ \alpha $, is limited to the range $ \pare{0, 2} $.

\item[3,4] We define {\tt f} as an implicit function depending on the variable {\tt x} only, next we define {\tt Deltaf} as the periodic finite difference $ \Delta_z f $ defined in \eqref{eq:notation}.

\item[5-7] The function {\tt G1} is the function $ \sG^1_{\alpha, z} $ defined in \eqref{eq:G1}, the function {\tt DXG1} is the function $ \pare{ \sG^1_{\alpha, z} }' $ defined in \eqref{eq:G1'} and finally we define {\tt K1} as the function $ \sK^1_{\alpha, z} $ defined in \cref{eq:sK1}. 
% eq:sK2,eq:sK3}.

\item[8-11] We perform the same computations as $ \sK^1_{\alpha, z} $ for the kernels $ \sK^2_{\alpha, z} $, $ \sK^3_{\alpha, z} $ defined in \cref{eq:sK2,eq:sK3}.

\item[12-14] The asymptotic expansion in \cref{eq:Taylor_expansion_kernels0} is computed for the three kernels.

\item[15-20] We ask the computer to extract the coefficients of the expansions in \cref{eq:Taylor_expansion_kernels0} so that
$ \texttt{C}jl\texttt{(x,a)} = \sK^{j,l}_{\alpha}\pare{f;x} $, for any 
$ j=1,2,3, \ l =0,1 $. 

\item[21,22] We define the functions {\tt A0} and {\tt A1} as in \cref{eq:A0,eq:A1}.

\item[23] The last line, line 23, is a statement of truth, which asks the computer whether using algebraic simplifications it can prove that \cref{eq:Hamiltonian_identity} is true. 
\end{itemize}

 \section{Conjugation of paradifferential  operators under flows} 
\label{sec:action_paradiff}

The main results of this section concern 
transformation rules of 
paradifferential operators of the form 
$ \partial_x\circ \OpBW{\mathsf{a}} $ 
under the flow generated by  
paradifferential operators which are Hamiltonian,  
or Hamiltonian up to order zero.

\begin{prop}\label{prop:action_flow}
 Let $ q\in \bN, \ K'\leq K, \ N\in \bN $ with $ q\leq N $, $ \epsilon_0 > 0 $ and $ \rho \gg N $. Let 
  $ \beta \pare{ f; t, x } $ be a function in $
  \Sigma \cF^\bR_{K,K',1}\bra{\epsilon_0 , N} $ and
 $ \Phi_B\pare{f, \tau} $ be the flow generated by the Hamiltonian 
 operator $ B\pare{f, \tau} $ defined in \eqref{eq:Btheta}.
\begin{enumerate}[{\bf i}]

\item \label{item:action_flow-i}
{\bf (Conjugation of a paradifferential operator)}
 Let $ \sa\pare{f; t, x , \xi}  $ be a symbol in $ \Sigma \Gamma^{m}_{K,K', q}\bra{\epsilon_0 , N} $. 
 Then
\begin{equation}
\label{eq:conj_Ham_symbol}
\Phi_B \pare{f, 1} \circ \partial_x \circ \OpBW{ \sa\pare{f; t, x , \xi} } \circ \Phi_B \pare{f, 1}^{-1} 
  = \partial_x \circ \OpBW{a_0\pare{f, 1; t, x, \xi} + P \pare{f; t, x, \xi} }   + R\pare{f;t} 
\end{equation}
where
 \begin{equation}
 \label{eq:simbolo_principale_trasformato}
 a_0\pare{f, \tau ; t, x, \xi}   \defeq \left. \pare{1+  \partial_y  \breve{\beta}\pare{f, \tau; t, y} } \ \sa\pare{ f ; \, t, \, y , \,  \xi   \pare{1+  \partial_y  \breve{\beta}\pare{f, \tau; t, y}}  \Big. }  
 \right|_{y= x + \tau \beta\pare{f; t, x}}
 \end{equation}
  is a symbol in $ \Sigma \Gamma^m_{K,K',q}\bra{\epsilon_0 , N} $,
   $ P\pare{f; t, x, \xi} $ is a symbol in
   $ \Sigma \Gamma^{m-2}_{K,K',q+1}\bra{\epsilon_0 , N} $ and $ R \pare{f; t}
   $ is a smoothing operator in
   $ \Sigma \dot{\cR}^{-\rho +m + 1 + N}_{K,K',q+1}\bra{\epsilon_0 , N} $.
\item \label{item:action_flow-iii}
{\bf (Conjugation of $ \partial_t $)}
 There exists a function
 $  V\pare{f;t, x} $ in $ \Sigma \cF^\bR_{K, K'+1, 1}\bra{\epsilon_0 , N}$ and a smoothing operator 
 $ R \pare{f;t} $ in $ \Sigma \dot{\cR}^{-\rho}_{K,K'+1, 1}\bra{\epsilon_0 , N} $
such that
\begin{equation}
\label{eq:Egorov_timeder}
\begin{aligned}
\Phi_B \pare{f,1} \circ \pare{ \partial_t \Phi_B \pare{f,1}^{-1} }
=  \partial_x \circ \OpBW{V \pare{f;t,x} } + R\pare{f;t}.
\end{aligned}
\end{equation}
\item  \label{item:action_flow-ii}
{\bf (Conjugation of a smoothing operator)}
If  $ \,  R \pare{f;t} $ is a smoothing operator in
$  \Sigma \dot{\cR}^{-\rho}_{K, K', q}\bra{\epsilon_0 , N} $ then
the composed operator 
$ \Phi_B\pare{f, 1}\circ R \pare{f;t} \circ \Phi_B\pare{f, 1}^{-1} $ is in 
$   \Sigma \dot{\cR}^{-\rho+N}_{K, K', q}\bra{\epsilon_0 , N} $. 
\end{enumerate}
\end{prop}

We also prove an analogous result when the paradifferential operator which generates the flow has order strictly less than $ 1 $. 

\begin{prop}[Lie expansions]\label{lem:Lie_expansion}
Let $ q\in \bN, \ K'\leq K, \ N\in \bN $ with 
$ q\leq N $, $ \epsilon_0 > 0 $ and $ \rho \gg N $.  
Given a  symbol  $ w :=  w\pare{f; t, x, \xi} $  satisfying 
\begin{equation}\label{eq:assumption_w}
w\pare{f; t, x, \xi}  \in \Sigma \Gamma^{- {\mathtt d}}_{K,K',1} \bra{\epsilon_0 , N} \, , 
\ {\mathtt d} > 0 \, , \quad \Im w\pare{f; t, x, \xi} \in 
 \Gamma^{ - \max\{ 1, \mathtt d\}}_{K,K',1} \bra{\epsilon_0 , N} \, , 
\end{equation}
and \eqref{areal} 
and denote $  \Phi_W\pare{f, \tau} $  the flow generated by 
\begin{align}
 \label{eq:operatorsfamily2}
\partial_\tau \Phi_W \pare{f,\tau} = \partial_x\circ
\OpBW{w(f;t, x,\xi)} \ \Phi_W \pare{f, \tau} \, , \qquad \Phi_W \pare{0} = \Id \, . 
 \end{align}
\begin{enumerate}[{\bf i}]

\item \label{item:action_flow-iLie}
{\bf (Conjugation of a paradifferential operator)}
 Let $ \sa := \sa\pare{f; t, x , \xi}  $ be a symbol in 
 $ \Sigma \Gamma^{m}_{K,K', q}\bra{\epsilon_0 , N} $. 
Then  \begin{multline}
 \label{eq:Taylor_conjugation}
 \Phi_W\pare{f,1}^{-1}\circ \partial_x \circ  \OpBW{\sa\pare{f; t, x , \xi}} \circ  \Phi_W\pare{f,1}  
 = \\
  \partial_x \circ  \OpBW{\sa} - \comm{\partial_x \circ \OpBW{w}}{\partial_x
 \circ \OpBW{\sa}}
  +\partial_x \circ \OpBW{P\pare{f;t,x,\xi}}
+ R \pare{f;t}
 \end{multline}
 where
 $ P \pare{f;t,x,\xi} $ is a 
 symbol in $ \Sigma\Gamma^{m-2 {\mathtt d}}_{K,K',q+2}\bra{\epsilon_0 , N} $,  
 and $R \pare{f;t} $ is a smoothing operator in 
 $ \Sigma \dot{\cR}^{-\rho}_{K,K',{q+2}} \bra{\epsilon_0, N} $.
 If $ \sa, w$ are real and even in $ \xi $ then 
 $ \comm{\partial_x \circ \OpBW{w}}{\partial_x
 \circ \OpBW{\sa}} $ is Hamiltonian and $ P $ is real and even in $ \xi $. 
\item \label{item:action_flow-iiiLie}
{\bf (Conjugation of $ \partial_t $)}
There exists a symbol  $ T \pare{f;t, x, \xi} $ 
in 
$ \Sigma\Gamma^{-{\mathtt d} }_{K,K'+1,1}\bra{\epsilon_0 , N} $
satisfying \eqref{areal}, and a smoothing operator
$ R \pare{f;t}$  in
$ \Sigma \dot{\cR}^{-\rho}_{K, K'+1, 2}\bra{\epsilon_0 , N} $ such that 
 \begin{equation}
 \label{eq:cc_pat}
 -\partial_t \Phi_W\pare{f,1}^{-1}\circ \Phi_W\pare{f,1} = 
 \partial_x \circ \OpBW{T\pare{f;t, x, \xi}} + R\pare{f;t} \, .
 \end{equation}
If $ w $ is real and even in $ \xi $ then 
$ \partial_x \circ \OpBW{T\pare{f;t, x, \xi}}$ 
is Hamiltonian, i.e. $T  $ % \pare{f;t, x, \xi} $
  is real and even in $ \xi $.
\item  \label{item:action_flow-iiLie}
{\bf (Conjugation of a smoothing operator)}
If  $ \,  R \pare{f;t} $ is a smoothing operator in
$  \Sigma \dot{\cR}^{-\rho}_{K, K', q}\bra{\epsilon_0 , N} $ then
the composed operator 
$ \Phi_W\pare{f, 1}\circ R \pare{f;t} \circ \Phi_W \pare{f, 1}^{-1} 
$ is in 
$ \Sigma \dot{\cR}^{-\rho+ N \max\{0,\pare{1-{\mathtt d}}\}}_{K, K', q}\bra{\epsilon_0 , N} $.
\end{enumerate}
\end{prop}

The rest of this section is devoted to the proof of Propositions \ref{prop:action_flow}
and \ref{lem:Lie_expansion}.

\subsection*{Proof of Proposition \ref{prop:action_flow}}

The proof of Propositions \ref{prop:action_flow}
is inspired by the Egorov type analysis in  \cite[Section 3.5]{BD2018}.
The difference is that we highlight the Hamiltonian 
structure in  \eqref{eq:conj_Ham_symbol} and \eqref{eq:Egorov_timeder} of the conjugated operators.  

For simplicity  we avoid to track the dependence of $ \beta, b $ and $ \Phi_B $
 on the variable $ f $, as well as on $ t $,  and denote
$ \beta_x \pare{x} \defeq \partial_x\pare{\beta\pare{f;t,x}} $,
$ b_x \pare{\tau;t, x} \defeq \partial_x\pare{b\pare{f, \tau;t, x}} $
and  $ \Phi_B (\tau) \defeq   \Phi_B (f, \tau) $.
In the sequel $ \partial_x^{-1} $ is the Fourier multiplier with symbol $ \pare{\ii \xi}^{-1} $ that maps $ H^s_0 $ onto $ H^{s+1}_0 $ for any $ s\in\bR $.

\subsubsection*{Proof of item \ref{item:action_flow-i} : 
conjugation of a paradifferential operator}

The conjugated operator
\begin{equation}\label{defPtheta}
\cP\pare{\tau} \defeq \Phi_B \pare{\tau} \circ \partial_x \circ \OpBW{ \sa } \circ \Phi_B \pare{\tau}^{-1} \in \cL   \pare{ H^s_0 ; H^{s-1-m}_0  }  \, , \quad
\forall s \in \bR \, ,
\end{equation}
satisfies
$ \cP\pare{0} = \partial_x \circ \OpBW{ \sa}  $, and using that  
$ \partial_\tau \pare{ \Phi_B \pare{\tau}^{-1} } = -\Phi_B \pare{\tau}^{-1} \circ \partial_\tau \Phi_B \pare{\tau} \circ \Phi_B \pare{\tau}^{-1} $, it  
solves the Heisenberg equation
\begin{equation}
\label{eq:Heisenberg}
\partial_\tau \cP\pare{\tau} =  \comm{B(\tau)}{\cP\pare{\tau}} \, ,
\qquad  \cP\pare{0} = \partial_x \circ \OpBW{ \sa } \, .
\end{equation}

\begin{lemma}\label{lemmaA4}
The operator $ A \pare{ \tau  } \defeq \partial_x^{-1} \circ \cP\pare{\tau}
\in \cL\pare{H^s_0 ; H^{s-m }_0}  $   solves
\begin{equation}
\label{eq:approx_Atheta}
\system{
\begin{aligned}
&
\begin{multlined}
\partial_\tau A\pare{\tau} = \ii \comm{\OpBW{b\pare{\tau;x} \xi}}{A\pare{\tau}}
 - \tfrac12 \pare{  \OpBW {b_x\pare{\tau;x}  } A\pare{\tau} + A\pare{\tau} \OpBW{ b_x\pare{\tau;x}  } }  \\
 + R'\pare{\tau} \, A\pare{\tau}
- A\pare{\tau} \,  R\pare{\tau}
\end{multlined}
  \\
& A\pare{0} =  \OpBW{ \sa }
\end{aligned}
}
\end{equation}
where $ R\pare{\tau} $, $  R'\pare{\tau} $
are smoothing operators in $ \Sigma\dot{\cR}^{-\rho}_{K,K',1} $, uniformly 
in $ \av{\tau}\leq 1  $,
preserving the zero-average subspaces.
\end{lemma}

\begin{proof}
By \eqref{eq:Btheta}  and \Cref{prop:composition_BW} we have
 \begin{align}\label{eq:Btheta1}
&  B \pare{\tau}
 - \OpBW{\ii  \, b\pare{\tau;x} \, \xi +
 \tfrac12 b_x\pare{\tau;x} } = R\pare{\tau} \\
&  \label{eq:Btheta2}
 \partial_x^{-1}  \circ B ({\tau})   \circ \partial_x  =   \OpBW{b\pare{\tau;x}} \circ \partial_x  =    \OpBW{ \ii b\pare{\tau;x} \xi - \tfrac12
 b_x\pare{\tau;x} } + R' \pare{\tau}
\end{align}
where $ R\pare{\tau}, R'\pare{\tau} $ are smoothing operators in 
$  \Sigma\dot{\cR}^{-\rho}_{K,K',1} $ 
preserving the zero-average subspaces.
Then, by  \eqref{eq:Heisenberg}, \eqref{eq:Btheta1}, \eqref{eq:Btheta2} we get
\begin{align*}
\partial_\tau A (\tau) & =
 \partial_x^{-1}  \circ B(\tau)  \circ  \partial_x \circ A\pare{\tau} - A\pare{\tau} \circ B
 (\tau) \\ 
& =   \pare{  \OpBW{\ii b\pare{\tau;x} \xi-
\tfrac{1 }{2} b_x\pare{\tau;x}  }} A\pare{\tau} - A\pare{\tau} \pare{  \OpBW{\ii b\pare{\tau;x} \xi +
\tfrac{1 }{2} b_x\pare{\tau;x} }}
 +  R'\pare{\tau} A\pare{\tau}
- A\pare{\tau} R \pare{\tau}  \notag
\end{align*}
proving \eqref{eq:approx_Atheta}.
\end{proof}

We now look for an approximate solution  of \eqref{eq:approx_Atheta} of the form
\begin{equation}
\label{eq:ansatz_Atheta}
A^{(J)}\pare{\tau} = \sum_{\mathsf{j} = 0}^{J } \OpBW{a_\mathsf{j} \pare{\tau}} \, ,\qquad
a_0\pare{\tau} \in \Sigma\Gamma^{m}_{K,K',q}\bra{\epsilon_0 , N} \, , \quad 
a_\mathsf{j} \pare{\tau} \in \Sigma\Gamma^{m-2\mathsf{j}}_{K,K',q+1}\bra{\epsilon_0 , N} \, , 
\ \forall \mathsf{j} = 1, \ldots, J \, .
\end{equation}
We use the following asymptotic expansions
derived by Proposition \ref{prop:composition_BW} 
and \eqref{asharpb}.
\begin{lemma} \label{lem:comm}
Let  $ a $ be a symbol 
in $  \Sigma \Gamma^m_{K,K',q} \bra{\epsilon_0,N} $. Then the commutator
$$
\comm{\OpBW{ \ii b(\tau;x) \xi }}{  \OpBW {a} }  =
\OpBW{ \pbra{ b(\tau;x) \xi }{ a } }  +   \OpBW{ r_{-3}(b(\tau),a) }   + R\pare{\tau} 
$$
with  symbols
$$
 \pbra{  b(\tau;x) \xi }{  a }  \in \Sigma \Gamma^{m}_{K,K',q+1}\bra{\epsilon_0,N} \, ,
\qquad
r_{-3}(b(\tau),a) \in \Sigma \Gamma^{m-2}_{K,K',q+1}\bra{\epsilon_0,N} \, ,
$$
and  a smoothing operator
$ R \pare{\tau} $ in $ \Sigma \dot{\cR}^{-\rho + m  + 1}_{K,K',q +1} [r, N] $, 
uniformly in $ \tau $. 
 Moreover 
\begin{equation*}
\tfrac12 \OpBW{ b_x\pare{\tau;x} } \OpBW{a} +
\tfrac12 \OpBW{a}  \OpBW{b_x\pare{\tau;x} }  =
\OpBW{b_x\pare{\tau;x} a + r_{-2}(b(\tau),a) }  + R \pare{\tau}  
\end{equation*}
where $ r_{-2}(b(\tau),a) $ is a symbol in 
$ \Sigma \Gamma^{m-2}_{K,K',q+1} \bra{\epsilon_0 , N} $, and 
$ R\pare{\tau} $ is a smoothing operator 
in $ \Sigma\dot{\cR}^{-\rho+m}_{K,K',q+1}\bra{\epsilon_0 , N} $.
\end{lemma}

We shall also use the following lemma concerning solutions of a transport  equation.

\begin{lemma}
\label{lem:modif_subprinc_symbol_flow}
Let   $ \,  W\pare{f,\tau; x, \xi} $ by a symbol in $ \Sigma\Gamma^m_{K,K',q}\bra{\epsilon_0 , N} $ uniformly in $ \av{\tau}\leq 1 $.
Then the unique solution of
\begin{equation}
\label{eq:PDEtransport}
\system{
\begin{aligned}
& \partial_\tau Q \pare{f, \tau; x, \xi} =
\pbra{b\pare{f, \tau; x}\Big. \xi}{Q\pare{f, \tau; x, \xi}} -   b_x\pare{f, \tau; x} Q\pare{f, \tau; x, \xi} + W\pare{f, \tau; x, \xi} \\
& \left. Q \pare{f, \tau; x, \xi} \right|_{\tau=0} = Q_0\pare{f ; x, \xi} \in \Sigma\Gamma^m_{K,K',q}\bra{\epsilon_0 , N}
\end{aligned}
}
\end{equation}
has the form
\begin{align}
 Q\pare{f, \tau; x, \xi} & = \left. \pare{ 1+ \breve \beta_y \pare{ f, \tau; \ y } } Q_0\pare{f; \  y\ , \ \xi    \pare{  1+ \breve \beta_y \pare{ f, \tau; \ y } }  } \right|_{y= x+\tau\beta\pare{f;x}} \label{formqtheta}
 \\
 & +\left.
 \int_0^\tau
 \frac{1+\breve{\beta}_y\pare{f, \tau; \  y} }{1+\breve{\beta}_y\pare{f, \tau'; \  y}}
 \
 W\pare{f, \tau' ; \
y + \breve{\beta} \pare{f, \tau'; \  y} \ ,  \ \frac{\xi\pare{1+\breve{\beta}_y \pare{f, \tau; \  y}}}{1+\breve{\beta}_y\pare{f, \tau';\  y}}} \ \dd \tau' \right|_{y=x+\tau\beta\pare{f;x}} \, , \notag
\end{align}
which is a symbol in 
$ \Sigma\Gamma^m_{K,K',q}\bra{\epsilon_0 , N} $,
uniformly in $ \av{\tau}\leq 1 $.
\end{lemma}

\begin{proof}
The solution
$ (x(\tau), \xi(\tau)) = \phi^{0,\tau} (X,\Xi) $  of the characteristics system
\begin{equation}\label{sistode}
\frac{\dd }{\dd \tau} {x}\pare{\tau} = - b\pare{\tau; x\pare{\tau}} \, , \quad
\frac{\dd }{\dd \tau} {\xi}\pare{\tau} = b_x\pare{\tau; x\pare{\tau}} \ \xi\pare{\tau} \, , 
\end{equation}
with initial condition $ \left.  \pare{x\pare{\tau} , \xi\pare{\tau}} \right|_{\tau=0} = \phi^{0,0} (X,\Xi) = \pare{X, \Xi} $ is (cf. \cite[p. 83]{BD2018})
\begin{equation}\label{dirflow}
\pare{x\pare{\tau} , \xi\pare{\tau}} =
\phi^{0, \tau}\pare{X, \Xi} =    \pare{ X + \breve \beta \pare{ \tau, X },  \frac{\Xi}{1+ \breve \beta_y \pare{ \tau, X } } } \, .
\end{equation}
By \eqref{sistode} and \eqref{eq:PDEtransport} we get
$ \frac{\dd }{\dd \tau} \bra{
\xi (\tau)\big.  Q\pare{\tau;  x\pare{\tau} , \xi\pare{\tau}} } =
\xi (\tau)  W \pare{ \tau; x\pare{\tau} , \xi\pare{\tau} } $
and so, by integration,
\begin{equation}\label{integrata}
\xi(\tau) \
Q\pare{\tau; x\pare{\tau} , \xi\pare{\tau} } = \Xi Q\pare{0; X, \Xi}
+
\int_0^{\tau} \xi(\tau')
W \pare{ \tau'; x(\tau'), \xi(\tau') } \dd \tau' \, .
\end{equation}
The  inverse flow
$   \phi^{ \tau , 0}\pare{x, \xi}   $,
i.e.
$ (x, \xi) = \phi^{0, \tau}\pare{X,\Xi} $
 if and only if $ \pare{X,\Xi} =  \phi^{ \tau , 0} (x, \xi) $
is  (cf. \cite[p. 83]{BD2018})
\begin{equation}\label{eq:flux_dirinv}
\pare{X,\Xi} = \phi^{\tau,0} (x,\xi) = \pare{ x + \tau \beta (x),  \xi\pare{ 1+ \breve \beta_y \pare{ \tau;y } }_{|{y= x + \tau \beta (x)}} } \, .
\end{equation}
In addition,  by \eqref{dirflow} and \eqref{eq:flux_dirinv},
\begin{equation}\label{flowthetaprimo}
\pare{x(\tau'), \xi(\tau')}
=
\phi^{\tau, \tau'}\pare{x, \xi}
= \phi^{0, \tau'}\pare{\phi^{\tau, 0}\pare{x, \xi}}
=
\left. \pare{
y + \breve{\beta} \pare{\tau' y} \ ,  \ \frac{\xi\pare{1+\breve{\beta}_y \pare{\tau; y}}}{1+\breve{\beta}_y\pare{\tau'; y}}
}
\right|_{y=x+\tau\beta\pare{x}}.
\end{equation}
We deduce
\eqref{formqtheta}
inserting \eqref{eq:flux_dirinv} and \eqref{flowthetaprimo} in \eqref{integrata}.
Finally  $ Q \pare{f, \tau; x, \xi}  $ is a symbol in $ \Sigma\Gamma^m_{K,K',q}\bra{\epsilon_0 , N} $, by \eqref{formqtheta} and Lemmata \ref{lem:closure_comp_symbols} and \ref{lem:LemA3}.
\end{proof}

\noindent
{\bf Step i): Determination of the principal symbol $ a_0 $.}
From \eqref{eq:approx_Atheta}, \eqref{eq:ansatz_Atheta} and Lemma 
\ref{lem:comm}  the principal symbol $a_0 $ solves     the equation
\begin{equation}
\label{eq:ODEprincipalsymbol}
\system{
\begin{aligned}
& \partial_\tau a_0\pare{\tau ; x, \xi} = \pbra{b\pare{\tau; x}\xi}{a_0\pare{\tau; x, \xi}} - b_x\pare{\tau; x}\ a_0\pare{\tau; x, \xi}
\\
& a_0\pare{0; x, \xi} = \sa\pare{x, \xi} \, .
\end{aligned}
}
\end{equation}
By \Cref{lem:modif_subprinc_symbol_flow} with $ W=0 $ and $ Q_0=\sa $,
the solution of \eqref{eq:ODEprincipalsymbol}   is given by
\eqref{eq:simbolo_principale_trasformato}.
The operator
$ A^{(0)} \defeq A^{(0)}\pare{\tau}  \defeq \OpBW{a_0 \pare{\tau}}  $
solves approximately
\eqref{eq:approx_Atheta}
in the sense that, by \eqref{eq:ODEprincipalsymbol}
and \Cref{lem:comm},
\begin{equation}\label{app0}
\partial_\tau A^{(0)}  =
 \ii \comm{\OpBW{b\pare{\tau} \xi}}{A^{(0)}}
 - \OpBW{\tfrac{b_x\pare{\tau}}{2}} A^{(0)}  -
 A^{(0)} \OpBW{\tfrac{b_x\pare{\tau}}{2}} \\
 +  \OpBW{r^{(0)}(\tau)} +  R ^{\pare{0}}\pare{\tau}
\end{equation}
where $ r^{(0)} (\tau) \defeq - r_{-3} (b,a_0) - r_{-2} (b,a_0)  $ is a symbol in $ \Sigma\Gamma^{m-2}_{K,K',q+1}\bra{\epsilon_0 , N} $ and $ R^{\pare{0}}\pare{\tau} $ is a smoothing operator in $ \Sigma \dot{\cR}^{-\rho+m}_{K,K',q+1}\bra{\epsilon_0 , N} $, unifomly in 
$ \tau \in [0,1] $.
\\[1mm]
{\bf Step ii): Determination of the subprincipal symbol $\sum_{\mathsf{j}=1}^J a_\mathsf{j} $.}
We define  $ a_1 (\tau; x, \xi )$ as the solution of the
transport equation
\begin{equation}
\label{eq:ODE_a1}
\system{
\begin{aligned}
& \partial_\tau a_1 \pare{\tau ; x, \xi} = \pbra{b\pare{\tau; x}\xi}{a_1 \pare{\tau; x, \xi}} - b_x\pare{\tau; x}\ a_1 \pare{\tau; x, \xi} -
r^{(0)} \pare{\tau;x,\xi}
\\
& a_1\pare{0; x, \xi} = 0 \, .
\end{aligned}
}
\end{equation}
By \Cref{lem:modif_subprinc_symbol_flow}
the symbol
$ a_1\pare{\tau; x, \xi} $ is in $ \Sigma \Gamma^{m-2}_{K,K',q+1} $.
By  \Cref{app0,eq:ODE_a1,lem:comm}
$$
A^{(1)} (\tau) \defeq A^{(0)} (\tau) + \OpBW{a_1 (\tau)}
$$
is a better approximation of equation \eqref{eq:approx_Atheta} in the sense that
\begin{equation}\label{app1}
\partial_\tau A^{(1)}  =
\ii \comm{\OpBW{b\pare{\tau} \xi}}{A^{(1)}}
 - \OpBW{\tfrac{b_x\pare{\tau}}{2}} A^{(1)}  -
 A^{(1)} \OpBW{\tfrac{b_x\pare{\tau}}{2}} \\
 +  \OpBW{r^{(1)}(\tau)} +  R^{(1)}\pare{\tau}
\end{equation}
where $ r^{(1)} \defeq - r_{-3} (b,a_1) - r_{-2} (b,a_1)  $ is a symbol in $ \Sigma \Gamma^{m-4}_{K,K',q+1}\bra{\epsilon_0 , N} $ and $ R^{(1)}\pare{\tau}
$  are smoothing operators  in $ \Sigma \dot{\cR}^{-\rho +m}_{K,K',q+1}\bra{\epsilon_0 , N}$ uniformly in $ |\tau | \leq 1 $.

Repeating $ J $ times ($ J \sim \rho / 2 $) the above procedure, until the new paradifferential  term may be incorporated into the smoothing remainder, we obtain
an operator $ A^{\pare{J}}\pare{\tau} \defeq \sum_{\mathsf{j}=0}^J \OpBW{a_\mathsf{j} \pare{\tau}} $
as in \eqref{eq:ansatz_Atheta} solving
\begin{equation}\label{solapproJ}
 \system{
\begin{aligned}
&
\begin{multlined}
\partial_\tau A^{\pare{J}}\pare{\tau} = \ii \ \comm{\OpBW{b\pare{\tau} \xi}}{A^{\pare{J}}\pare{\tau}} - \OpBW{\tfrac{b_x\pare{\tau}}{2}} A^{\pare{J}}\pare{\tau} - A^{\pare{J}}\pare{\tau} \OpBW{\tfrac{b_x\pare{\tau}}{2}}  + R^{(J)}\pare{\tau}
\end{multlined}
  \\
& A^{\pare{J}}\pare{0} =  \OpBW{ \sa }
\end{aligned}
}
\end{equation}
where $ R^{(J)}\pare{\tau} $ are smoothing operators in $
\Sigma \dot{\cR}^{-\rho+m}_{K,K',q+1}\bra{\epsilon_0 , N}  $  uniformly in
$ |\tau | \leq 1 $.
\\[1mm]
{\bf Step iii) :  Analysis of the error. }
We finally estimate the  difference between the  conjugated
operator $  P \pare{\tau}  $ in
\eqref{defPtheta} and
$ P^{\pare{J}}\pare{\tau} \defeq \partial_x \circ A^{\pare{J}}\pare{\tau}  $.

\begin{lemma}
\label{lem:modif_regularizing_operator_flow}
$ P \pare{\tau} - P^{\pare{J}}\pare{\tau}  $
is a smoothing operator $ R\pare{\tau}  $ in $ \Sigma\dot{\cR}^{-\rho +m+1+N}_{K,K',q+1}\bra{\epsilon_0 , N} $
uniformly in $ \av{\tau}\leq 1 $.
\end{lemma}

\begin{proof}
In  view of \cref{solapproJ,eq:Btheta1,eq:Btheta2},
the operator
$ \cP^{(J)} (\tau) = \partial_x \circ A^{\pare{J}}\pare{\tau} $
solves an approximated Heisenberg equation (cf. \eqref{eq:Heisenberg})
\begin{equation}
\label{eq:Heisenbergapp}
\partial_\tau \cP^{(J)} \pare{\tau} =  \comm{B}{\cP^{(J)}\pare{\tau}} +
R \pare{\tau} \, ,
\qquad
R (\tau) \in \Sigma \dot{\cR}^{-\rho+m}_{K,K', q+1}
 \, .
\end{equation}
Recalling \eqref{defPtheta} we write
$$
P^{(J)}  \pare{\tau} -  P \pare{\tau} =
V(\tau)  \Phi_B (\tau)^{-1}  \quad \text{where} \quad
V(\tau)  \defeq
P^{(J)} \pare{\tau} \Phi_B (\tau) -   \Phi_B (\tau)
\circ \partial_x \circ \OpBW{\sa} \, .
$$
By \eqref{eq:Heisenbergapp} we have that
$ \partial_\tau V(\tau)  = B(\tau)  V(\tau)  + R
(\tau) \Phi_B (\tau)  $, $ V(0) = 0 $,
and therefore, by Duhamel and 
$ \partial_\tau \Phi_B = B \Phi_B $ we deduce
$ V(\tau) = \Phi_B (\tau) \int_0^{\tau} \Phi_B (\tau')^{-1}
R (\tau')\Phi_B (\tau') \, \dd  \tau' $
and thus
$$
P^{(J)}  \pare{\tau} -  P \pare{\tau} =
 \int_0^{\tau} \Phi_B (\tau) \circ \Phi_B (\tau')^{-1}  \circ
R (\tau') \circ \Phi_B (\tau') \circ
\Phi_B (\tau)^{-1}  \, \dd  \tau' \, .
$$
This is a smoothing operator in  arguing as in  \cite[Proof of Thm. 3.27]{BD2018}.
\end{proof}

\Cref{lem:modif_regularizing_operator_flow} implies that
$ P \pare{\tau} = \partial_x \circ A^{\pare{J}}\pare{\tau} + R (\tau)  $
concluding the proof of  \Cref{prop:action_flow}-\ref{item:action_flow-i}
with symbol $ P % \pare{f;x, \xi}
 = \sum_{\mathsf{j}=1}^J a_\mathsf{j} \pare{1}  $.
 \Cref{item:action_flow-iii} follows 
similarly as  in \cite[Lemma A.5]{BFP2018}. \Cref{item:action_flow-ii} 
is given  in \cite[Remark at page 89]{BD2018}.

\subsection*{Proof of Proposition \ref{lem:Lie_expansion}}

In view of \eqref{eq:assumption_w} and \Cref{lem:existence_flow} the flow 
$ \Phi_W\pare{\tau} := \Phi_W\pare{f, \tau} $ 
generated by \eqref{eq:operatorsfamily2} is well posed and 
\begin{equation}\label{paWt}
\frac{\di}{\di \tau } \pare{ \Phi_W\pare{\tau}^{-1} \circ \partial_x \circ \OpBW{\sa} \circ 
\Phi_W\pare{\tau} } =
-   \Phi_W\pare{\tau}^{-1} \comm{\partial_x \circ \OpBW{w}}{\partial_x \circ \OpBW{\sa}} \Phi_W\pare{\tau} 
\end{equation}
and a  Taylor expansion gives 
\begin{multline}\label{eq:Lie_symbol}
\Phi_W\pare{1}^{-1}\, \partial_x \circ \OpBW{\sa} \circ  \Phi_W\pare{1}
\\
=  \partial_x \circ \OpBW{\sa} - 
\comm{\partial_x  \circ \OpBW{w}}{\partial_x\circ  \OpBW{\sa}}  
+
\sum_{\ell=2}^L \frac{ \pare{-1}^\ell}{\ell!}
\Ad_{\partial_x \circ \OpBW{w}}^\ell \pare{\partial_x \circ \OpBW{\sa}}
\\
+ \frac{\pare{-1}^{L+1}}{L!} \int _0^1 \pare{1-\tau}^L
\Phi_W\pare{\tau}^{-1}\circ
\Ad_{\partial_x \circ \OpBW{w}}^{L+1} \pare{\partial_x \circ \OpBW{\sa}} \circ \Phi_W\pare{\tau}  \dd \tau \, .
\end{multline}
Since $  \partial_x \circ \OpBW{w} $ belongs to 
$ \Sigma \Gamma^{-{\mathtt d}}_{K,K',1} $, $  {\mathtt d} > 0 $, 
 then, by Proposition \ref{prop:composition_BW} 
each commutator $ \comm{ \partial_x \circ \OpBW{w}}{ \cdot }$ gains $  {\mathtt d} > 0$ 
unit of order and one degree of vanishing in $ f $ and  
\eqref{eq:Lie_symbol} is an expansion as in \eqref{eq:Taylor_conjugation} 
in operators with decreasing order and increasing degree of homogeneity  
with a symbol $P$ of order $ m - 2 {\mathtt d} $.  
Item \ref{item:action_flow-iiLie} follows as in \cite[Remark at page 89]{BD2018}, see also \cite{BMM2022}, by properties of the flow generated by paradifferential operators. 
Thus, Proposition \ref{prop:composition_BW}, 
give that the last term of \eqref{eq:Lie_symbol} belongs to 
$ \Sigma \dot{\cR}^{1+m - {\mathtt d} \pare{L+1} +   \max \{0, 1- {\mathtt d} \} N}_{K, K', q}\bra{\epsilon_0, N} $, 
hence if $ L +1 \geq \frac{\rho  +1+m+ \max \{0, 1- {\mathtt d} \}   N}{{\mathtt d}}  $
 it belongs to $  \Sigma \dot{\cR}^{-\rho}_{K, K', q}\bra{\epsilon_0, N}  $. 
 If $ w, \sa $ are real and even in $ \xi $, then 
 the operators $\partial_x \circ \OpBW{w} $ and $ \partial_x \circ \OpBW{\sa}$ are Hamiltonian
 (cf. \eqref{Hamassy}). 
The commutator of two Hamiltonian operators
 \begin{equation}\label{AdSwc}
 \Ad_{\partial_x \circ \OpBW{w}} 
 \pare{\partial_x \circ  \OpBW{\sa}}
= \partial_x \circ S \, , \quad
S := 
\OpBW{w} \circ \partial_x \circ \OpBW{\sa} - \OpBW{\sa}
 \circ \partial_x \circ \OpBW{w} \, ,
 \end{equation}
where $ S = S^* $, $ S = \overline{S} $, 
 is another Hamiltonian operator where, by 
 \Cref{prop:composition_BW},  
 the operator $ S = \OpBW{s} $ with 
 a real 
 symbol $ s $ in $ \Sigma\Gamma^{m-{\mathtt d}}_{K,K',q+1}\bra{\epsilon_0 , N} $
 even in $ \xi $ (cf. \eqref{Hamassy}), 
  up to a smoothing operator in  $ \Sigma \dot{\cR}^{-\rho}_{K,K',q+1} \bra{\epsilon_0, N} $, by renaming $ \rho $. Applying iteratively this result to 
 $ \Ad_{\partial_x \circ \OpBW{w}}^\ell \pare{\partial_x \circ \OpBW{\sa}} $ the 
 formula \eqref{eq:Taylor_conjugation} follows.

Let us prove \eqref{eq:cc_pat}. As in \eqref{paWt} we have that 
\begin{equation*}
\frac{\di}{\di \tau } \pare{ \Phi_W\pare{\tau}^{-1}\circ \partial_t  \circ \Phi_W\pare{\tau} }
 =-
 \Phi_W\pare{\tau}^{-1}\circ
 \comm{\partial_x \circ \OpBW{w}}{ \partial_t } \circ \Phi_W\pare{  \tau} =
 \Phi_W\pare{\tau}^{-1}\circ
\partial_x \circ \OpBW{w_t } \circ \Phi_W\pare{  \tau}
 \end{equation*}
% and, for any $ \ell \geq 2 $,
% $$
% \pare{\frac{\di}{\di \tau }}^\ell \pare{ \Phi_W\pare{\tau}^{-1}\circ \partial_t  \circ \Phi_W\pare{\tau} } 
% \partial_\tau^\ell \Psi (\tau)
% =
%(-1)^{\ell-1} \Phi_W\pare{\tau}^{-1}\circ  \Ad_{\partial_x \circ \OpBW{w}}^{\ell-1}
%\big( \partial_x \circ \OpBW{w_t }\big)  \circ \Phi_W\pare{  \tau} 
% $$
and a Taylor expansion gives
\begin{align}
\Phi_W\pare{1}^{-1}\circ \partial_t  \circ \Phi_W\pare{1}
& =
\partial_t
 + \partial_x \circ \OpBW{w_t } + \sum_{\ell = 2}^L \frac{(-1)^{\ell-1}}{\ell !} 
 \Ad_{\partial_x \circ \OpBW{w}}^{\ell-1}
\big( \partial_x \circ \OpBW{w_t }\big) \label{Lietime}
 \\
 & \ + \frac{(-1)^{L}}{L!} \int_0^1 
 \pare{1-\tau}^L 
  \Phi_W\pare{\tau}^{-1}\circ \Ad_{\partial_x \circ \OpBW{w}}^{L} \pare{ \partial_x \circ \OpBW{w_t} \Big. } \circ \Phi_W\pare{\tau}  \,  \dd \tau \, . \notag
 \end{align}
Since
$
\Phi_W\pare{1}^{-1}\circ \partial_t  \circ \Phi_W\pare{1}
= $ $ \partial_t + \Phi_W\pare{1}^{-1}\circ \pare{\partial_t  \Phi_W\pare{1}}
 = $ $ \partial_t  - \pare{\partial_t \Phi_W\pare{1}^{-1}} \circ   \Phi_W\pare{1}   $
we deduce by \eqref{Lietime} that 
$$
- \partial_t \Phi_W\pare{1}^{-1}\circ   \Phi_W\pare{1} =
\partial_x \circ \OpBW{w_t } +
 \sum_{\ell=2}^L \frac{(-1)^{\ell-1}}{\ell!} \Ad_{\partial_x \circ \OpBW{w}}^{\ell-1}
\big( \partial_x \circ \OpBW{w_t }\big) + R 
$$
 where, %  as for \eqref{eq:Taylor_conjugation} that, 
 if  $ L \gtrsim_{\mathtt d, N} \rho $,  then   $ R $ is in 
 $ \Sigma \dot{\cR}^{-\rho}_{K,K'+1,2}\bra{\epsilon_0 , N} $ (renaming $ \rho $).
 Then  \eqref{eq:cc_pat} follows arguing as for \eqref{eq:Taylor_conjugation} 
 and if $ w $ is real and even in $ \xi $ we also deduce that 
 $ T $ is real and even in $ \xi $.

	\begin{footnotesize}
%		\bibliography{references}
%		\bibliographystyle{plain}

	\end{footnotesize}

\end{document}